\documentclass{article}
\usepackage{amsmath}
\usepackage{amssymb}
\usepackage{natbib}
\usepackage[left=2cm,right=2cm,top=2cm,bottom=2cm]{geometry}
\usepackage{hyperref}
\usepackage{url}
\usepackage{placeins}
\usepackage{morefloats}
\usepackage[noend]{algpseudocode}
\usepackage[ruled,noend]{algorithm2e}
\usepackage{pgf,tikz}
\usepackage{mathtools}
\usepackage{authblk}
\usetikzlibrary{calc, positioning}

\newtheorem{theorem}{Theorem}

\newtheorem{corollary}{Corollary}

\newtheorem{definition}{Definition}
\newtheorem{example}{Example}

\newtheorem{lemma}{Lemma}

\newtheorem{proposition}{Proposition}
\newtheorem{remark}[theorem]{Remark}

\newenvironment{proof}[1][Proof]{\noindent\textbf{#1.} }{\ \rule{0.5em}{0.5em}}
\DeclareMathOperator{\pa}{pa}
\DeclareMathOperator{\doo}{do}
\DeclareMathOperator{\ch}{ch}
\DeclareMathOperator{\an}{an}
\DeclareMathOperator{\de}{de}
\DeclareMathOperator{\cn}{cn}
\DeclareMathOperator{\nd}{nd}
\DeclareMathOperator{\forb}{forb}
\DeclareMathOperator{\indir}{indir}
\DeclareMathOperator{\dir}{dir}
\DeclareMathOperator{\irrel}{irrel}

\newcommand{\cw}{\overset{d}{\rightarrow}}
\newcommand{\ort}{\perp \!\!\!\perp}
\newcommand{\A}{\mathbf{A}}

\author[1]{Andrea Rotnitzky \thanks{arotnitzky@utdt.edu}}
\affil[1]{Department of Economics, Universidad Torcuato Di Tella, CONICET and Harvard T. H. Chan School of Public Health}
\author[2]{Ezequiel Smucler \thanks{esmucler@utdt.edu}}
\affil[2]{Department of Mathematics and Statistics, Universidad Torcuato Di Tella}

\begin{document}

\title{Efficient adjustment sets for population average treatment effect estimation in
non-parametric causal graphical models}
\maketitle

\begin{abstract}

The method of covariate adjustment is often used for estimation of
population average causal treatment effects in observational studies.
Graphical rules for determining all valid covariate adjustment sets from an
assumed causal graphical model are well known. Restricting attention to
causal linear models, a recent article \citep{perkovic} derived two novel
graphical criteria: one to compare the asymptotic variance of linear
regression treatment effect estimators that control for certain distinct
adjustment sets and another to identify the optimal adjustment set that
yields the adjusted least squares treatment effect estimator with the
smallest asymptotic variance among consistent adjusted least squares
estimators. In this paper we show that the same graphical criteria can be
used in non-parametric causal graphical models when treatment effects are
estimated by contrasts involving non-parametrically adjusted estimators of
the interventional means. We also provide a graphical criterion for
determining the optimal adjustment set among the minimal adjustment sets, which is valid for both linear and non-parametric estimators. We
additionally provide a new graphical criterion for comparing 
time dependent adjustment sets, that is, sets comprised by covariates
that adjust for future treatments and that are themselves affected by
earlier treatments. We show by example that uniformly optimal time dependent
adjustment sets do not always exist. In addition, for point interventions,
we provide a sound and complete graphical criterion for determining when a
non-parametric optimally adjusted estimator of an interventional mean, or of
a contrast of interventional means, is as efficient as an efficient
estimator of the same parameter that exploits the information in the
conditional independencies encoded in the non-parametric causal graphical
model. The algorithm also checks for possible simplifications of the
efficient influence function of the parameter. We find an interesting
connection between identification and efficient covariate adjustment
estimation. Specifically, we show that if there exists an identifying
formula for an interventional mean that depends only on treatment, outcome
and mediators, then the non-parametric optimally adjusted estimator can
never be globally efficient under the non-parametric causal graphical model.
\end{abstract}

\section{Introduction}

Estimating total, population average, causal treatment effects by
controlling for, that is, conditioning on, a subset of covariates is known as
the method of covariate adjustment. Assuming a causal directed acyclic graph (DAG) model, the
back-door criterion \citep{pearl-causality} is a popular graphical criterion
that gives sufficient conditions for a covariate set to be such that control
for this set yields consistent estimators of total treatment effects.
\cite{shpitser-adjustment} gives a necessary and sufficient graphical criterion
for a subset of covariates to qualify for adjustment.

The graphical criteria of Pearl and Shpitser et al. are particularly useful for
designing observational studies. Specifically,  investigators planning
an observational study might be prepared to hypothesize a
causal diagram and apply the aforementioned criteria to aid them in
selecting the covariates to measure in order to control for confounding.
When many covariate adjustment sets are available, a natural question is
which one should be selected.

\cite{perkovic} gave an answer to this question under the following assumptions:
(i) the causal DAG model is linear,
 that is, each vertex in the
DAG stands for a random variable that follows a  linear regression model on its parents in the DAG, with an
independent error that has an arbitrary distribution and (ii) the total
treatment effects are estimated with the coefficients associated with
treatments in the ordinary least squares (OLS) fit of the outcome on treatments
and a set of valid adjustment covariates. They derive
a graphical criterion that identifies the optimal covariate
adjustment set in the sense that this set yields the OLS treatment
effect estimator which has the smallest asymptotic
variance among all OLS estimators of treatment effects that control for
valid adjustment sets.

Our first contribution, see Section \ref{sec:optimal_adj_point}, is to establish that the same criterion holds for identifying the optimal
valid covariate adjustment set when (i) the causal DAG model is non-parametric in the sense that no assumptions
are made on the conditional distribution of each node given is parents, and,  (ii) the treatment effects are estimated non-parametrically, that is,
without exploiting the conditional indepencences in the data generating law encoded in the causal DAG model. 
For instance, the treatment effects could be estimated by inverse probability weighting
with the propensity score estimated non-parametrically \citep{hirano, cattaneo}, or by
doubly-robust or double-machine learning approaches \citep{chernozhukov2018double, smucler}.
Our second contribution is to provide a graphical criterion for identifying
the optimal adjustment set among the class of minimal adjustment sets. A
minimal adjustment set is a valid adjustment set such that removal of any
vertex from the set yields a non-valid adjustment set. We note that our criterion
holds for non-parametric causal DAG models and estimators as well as linear causal DAG models
and estimators.

A second important contribution of \cite{perkovic} is a
graphical criterion, assuming linear DAG models and OLS estimators, to
compare certain pairs of valid adjustment sets which is more broadly
applicable than earlier existing criteria \citep{kuroki-miyawa, kuroki-cai}. Building on their criterion Henckel et al. also provided a simple
procedure that, for a valid adjustment set, returns a
pruned valid adjustment set that yields OLS estimators of treatment effects
with smaller asymptotic variance. The procedure was conjectured to yield
improved efficiency in \cite{vanderweele2011}. The contribution of \cite{perkovic} was to rigorously show that the conjecture
is valid for causal linear models and OLS estimators of treatment effects.
Our third contribution is to prove that both the graphical criterion and the
pruning procedure of \cite{perkovic} also apply for non-parametric causal DAG models and estimators.

\cite{perkovic} considered not only DAGs but also (linear) completed partially directed acyclic graphs (CPDAGs) and maximal
PDAGs. A CPDAG \citep{meek, andersson, spirtes, chickering} represents, under causal sufficiency and faithfulness, the
Markov equivalence class of DAGs that can be deduced from the conditional independences in the observed data distribution. A maximal PDAGs is a
maximally oriented partially directed acyclic graph that maximally refines
the Markov equivalence class when the orientation of some edges are known
a-priori \citep{meek, scheines, hoyer, hauser, eigenmann, wang}. 
\cite{perkovic} derived graphical criteria for identifying the optimal adjustment
set and for comparing certain adjustment sets under linear CPDAGs and maximal PDAGs, assuming treatment effects are estimated by least squares. These criteria are consequences of the corresponding criteria for DAGs.
This is because the criteria are based solely on d-separation conditions on CPDAGs and maximal PDAGs, and d-separations that hold on CPDAGs and maximal PDAGs hold on all possible DAGs represented by them. Because, as indicated earlier, we show that the graphical criteria developed by \cite{perkovic} for linear DAGs and estimators also holds for non-parametric DAGs and estimators, we conclude that the criteria derived by \cite{perkovic} for linear CPDAGs and maximal PDAGs using linear estimators of treatment effects, also hold for non-parametric CPDAGs and maximal PDAGs when non-parametric estimators of treatment effects are used. To avoid repetitions we do not expand on this topic in the present paper and refer the reader to \cite{perkovic}.

The  aforementioned graphical criterion of \cite{perkovic} for comparing certain
adjustment sets in DAGs applies to OLS estimators of the causal effects of
both point and joint interventions. However, for joint interventions, 
the criterion makes the restrictive assumption that the adjustment
sets are time independent. As \cite{perkovic} pointed out, 
time independent covariate adjustment sets for joint interventions do not 
always exist. In contrast, time \textit{dependent} covariate adjustment sets,
which are comprised by covariates that are needed to adjust for future treatments
but are themselves affected by earlier treatments, always exist.
 The g-formula \citep{robins1986}, is the
generalization of the adjustment formula from time independent to time dependent covariate
adjustment sets. This raises the
question of whether it is possible to generalize the results obtained for
comparing time independent covariate adjustment sets to time dependent covariate adjustment
sets. The
answer is mixed. Specifically, in Section \ref{sec:optimal_adj_time} we establish a result (Theorem \ref{theo:compare_adj_time}) that allows the comparison of certain time dependent covariate adjustment sets
and which generalizes the results obtained for non-parametric models and
estimators in Theorem \ref{theo:compare_adj} of the present article from time independent to
time dependent covariate adjustment sets. However, in that section we also exhibit a DAG in which no uniformly optimal time dependent covariate
adjustment set exists. We do so by exhibiting two data generating laws, both
satisfying the restrictions implied by the non-parametric causal DAG, such that a given
time dependent covariate adjustment set dominates all others for one law, in
the sense of yielding non-parametric estimators of the g-formula with
smallest asymptotic variance, but for the second law a different time
dependent covariate adjustment set dominates the rest. 

Next we investigate the following problem. If we could measure all the
variables of the causal DAG, we could then exploit the 
conditional independences encoded in the non-parametric causal DAG model to
efficiently estimate the total treatment effects. 
For a point exposure, we can also estimate each treatment
effect by the method of covariate adjustment using the optimal time independent
covariate adjustment set. A natural question then is under which DAG configurations, if any,
do the two procedures result in estimators with the same asymptotic efficiency?
From a practical perspective this question is
interesting for the planning of observational studies since for DAGs for
which no efficiency loss is incurred by non-parametric optimal covariate
adjustment estimation, then the optimal covariate adjustment set, the treatment and the
outcome are all the variables that one needs to measure not only for
consistent but also for efficient estimation of treatment effects. 
In Section \ref{sec:efficient_est} we provide a sound and complete algorithm that
answers this question. 
The completeness of our algorithm and of the ID algorithm \citep{id-orig,shpitser-id} imply the  following interesting result, linking
identification and efficient covariate adjustment estimation: if
there exists an identifying formula for an interventional mean that depends
only on treatment, outcome and mediators, then the non-parametric optimally adjusted
estimator can never be globally efficient under the causal DAG model.

When the
optimal covariate adjustment estimator is not efficient, it may nevertheless
be the case that not all the variables in the DAG enter into the calculation
of an efficient estimator. As such, from the perspective of planning a
study, it is useful to learn which variables are irrelevant for efficient
estimation since such variables need not be measured.
In Section \ref{sec:semiparam_effi} we
review a general one-step estimation strategy for computing semiparametric efficient
estimators. We argue that only variables entering the efficient influence
function of a interventional mean under the non-parametric causal graphical model
are required for computing the one-step estimator of treatment effects. As such, all variables that do
not enter into the efficient influence function are irrelevant for efficient
estimation. The aforementioned algorithm conducts sound checks for variables
that do not enter into the efficient influence function.  In addition, the
algorithm conducts sound checks for possible simplifications of the formula
for the efficient influence function. As we indicate in Section \ref{sec:semiparam_effi}, such
simplifications not only facilitate the computation of the one-step
estimator but also relax the requirements on smoothness or complexity of
certain conditional expectations for the convergence of the estimator.

In Section \ref{sec:back} we review the basic concepts of causal graphical models. In Section \ref{sec:optimal_adj} we provide the main results concerning optimal adjustment sets. In Section \ref{sec:efficient_est} we provide an algorithm for determining if a non-parametric optimally
adjusted estimator is efficient under the Bayesian Network implied by the causal graphical model. Section \ref{sec:disc} concludes with a list of open problems. Proofs of all the results stated
in the main text are given in the Appendix.

\section{Background}\label{sec:back}
In this section we review some elements of the theory of causal graphical models.

\subsection{Definitions and notation}
\label{sec:def}

\textbf{Directed graph}. A directed graph $\mathcal{G}=(\mathbf{V},\mathbf{E}%
)$ consists of a finite node set $\mathbf{V}$ and a set of directed edges $%
\mathbf{E}$. A directed edge between two nodes $V$, $W$ is represented by $%
V\rightarrow W$. Given a set of nodes $\mathbf{Z}\subset \mathbf{V}$ the
induced subgraph $\mathcal{G}_{\mathbf{Z}}=(\mathbf{Z},\mathbf{E}_{Z})$ is
the graph obtained by considering only nodes in $\mathbf{Z}$ and edges
between nodes in $\mathbf{Z}$.

\noindent\textbf{Paths}. Two nodes are adjacent if there exists an edge
between them. A path from a node $V$ to a node $W$ in graph $\mathcal{G}$ is
a sequence of nodes $(V_1, \dots, V_{j})$ such that $V_{1}=V$, $V_{j}=W$ and 
$V_i$ and $V_{i+1}$ are adjacent in $\mathcal{G}$ for all $i \in \{1, \dots,
j-1\}$. Then $V$ and $W$ are called the endpoints of the path. A path $(V_1,
\dots, V_{j})$ is directed or causal if $V_{i}\to V_{i+1}$ for all $i \in
\lbrace 1,\dots,j-1\rbrace$.

\noindent \textbf{Ancestry}. If $V\rightarrow W$, then $V$ is a parent of $W$
and $W$ is a child of $V$. If there is a directed path from $V$ to $W$, then 
$V$ is an ancestor of $W$ and $W$ a descendant of $V$. We follow the
convention that very node is an ancestor and a descendant of itself. The
sets of parents, children, ancestors and descendants of $V$ in $\mathcal{G}$
are denoted by $\pa_{\mathcal{G}}(V)$, $\ch_{\mathcal{G}}(V)$, $\an_{%
\mathcal{G}}(V)$, $\de_{\mathcal{G}}(V)$. The set of non-descendants of a
vertex $V$ is defined as $\nd_{\mathcal{G}}(V)\equiv \de^{c}_{\mathcal{G}%
}(V) $.

\noindent \textbf{Colliders and forks}. A node $V$ is a collider on a path $%
\delta $ if $\delta $ contains a subpath $(U,V,W)$ such that $U\rightarrow
V\leftarrow W$. A node $V$ is called a fork on $\delta $ if $\delta $
contains a subpath $(U,V,W)$ such that $U\leftarrow V\rightarrow W$.

\noindent \textbf{Directed cycles, DAGs}. A directed path from $V$ to $W$,
together with the edge $W\rightarrow V$ forms a directed cycle. A directed
graph without directed cycles is called a directed acyclic graph (DAG). The
nodes $(V_{k_{1}},\dots ,V_{k_{s}})$ are said
to follow a topological order relative to a DAG $\mathcal{G}$ if $V_{k_{j}}$
is not an ancestor of $V_{k_{j^{\prime }}}$ in $\mathcal{G}$ whenever $%
j>j^{\prime }$.

\noindent \textbf{d-separation} \citep{pearl-causality}. Consider a DAG $\mathcal{G}$ and distinct
sets of nodes $\mathbf{U},\mathbf{W},\mathbf{Z}$. A path $\delta $ between $%
U\in \mathbf{U}$ and $W\in \mathbf{W}$ is blocked by $\mathbf{Z}$ in $%
\mathcal{G}$ if one of the following holds:

\begin{itemize}
\item[1.] $\delta $ contains a node that is not a collider and is a member
of $\mathbf{Z}$, or

\item[2.] If there exists a collider $C$ in $\delta $ such that neither $C$
nor its descendants are in $\mathbf{Z}$.
\end{itemize}

$\mathbf{U},\mathbf{W}$ are d-separated by $\mathbf{Z}$ in $\mathcal{G}$
(denoted as $\mathbf{U}\perp \!\!\!\perp _{\mathcal{G}}\mathbf{W}\mid 
\mathbf{Z)}$ if for any $U\in \mathbf{U}$ and $W\in \mathbf{W,}$ all paths
between $U$ and $W $ are blocked given $\mathbf{Z}$.

\noindent \textbf{Marginal DAG model} \citep{evans-marginal}. Let $\mathcal{G%
}$ be a DAG with vertices $\mathbf{V}\overset{\cdot }{\cup }\mathbf{U}$, and 
$\mathcal{V}$ a state-space for $\mathbf{V}$. Define the \textit{marginal
DAG model }$\mathcal{M}\left( \mathcal{G},\mathbf{V}\right) $ by the
collection of probability distributions $P$ over $\mathbf{V}$ such that
there exist

\begin{enumerate}
\item some state-space $\mathcal{U}$ for $\mathbf{U}$,

\item a probability measure $Q$ $\in \mathcal{M}\left( \mathcal{G},\mathbf{V}%
\right) $ over $\mathcal{V}\times \mathcal{U}$
\end{enumerate}
and $P$ is the marginal distribution of $Q$ over $\mathbf{V}$.

\noindent \textbf{Exogenized DAG} \citep{evans-marginal}. Let $\mathcal{G}$
be a DAG and let $U$ be a vertex of $\mathcal{G}$ with a single child $R$.
Define the exogenized DAG $\tau \left( \mathcal{G},U\right) $ as follows:
take the vertices and edges of $\mathcal{G}$, and then (i) add an edge $H$ $%
\rightarrow $ $R$ from every $H$ $\in \pa_{\mathcal{G}}\left( U\right) $ to $%
R$, and (ii) delete $U$ and any edge $H$ $\rightarrow $ $U$ for $H$ $\in \pa%
_{\mathcal{G}}\left( U\right) $. All other edges and vertices are as in $%
\mathcal{G}$. In words, to exogenize a DAG $\mathcal{G}$ relative to a
vertex $U$ with a single child, we join all parents of $U$ to the child of $U
$ with directed edges, and then remove $U$ and all edges into and out of $U$.

\medskip

Throughout we use standard set theory notation. For a DAG with node set $%
\mathbf{V}$ and for $\mathbf{U},\mathbf{W}\subset \mathbf{V}$ we have $%
\mathbf{U}^{c}=\mathbf{V}\setminus \mathbf{U}$, $\mathbf{U}\setminus \mathbf{%
W}=\mathbf{U}\cap \mathbf{W}^{c}$ and $\mathbf{U}\bigtriangleup \mathbf{W}%
=\left( \mathbf{U}\setminus \mathbf{W}\right) \cup \left( \mathbf{W}%
\setminus \mathbf{U}\right) $. For a vector 
$
\mathbf{U}=(U_{0},\dots, U_{r}) \subset \mathbf{V}
$
and $j\leq r$ we let 
$$
\overline{\mathbf{U}}_{j}\equiv \left(U_{0},\dots,U_{j} \right).
$$
If $U$ and $V$ are independent random variables defined on a common probability space we write $U\ort V$.

\subsection{Causal graphical models}

Given a DAG $\mathcal{G}$ with a vertex set $\mathbf{V}$ that represents a random vector defined on a given probability space, a law $P$ for $\mathbf{V}$ is said to satisfy the Local Markov Property relative to $\mathcal{G}$ if and only if
\begin{equation}
    V
\perp \!\!\!\perp \nd_{\mathcal{G}}\left( V\right) \text{ }|\text{ }\pa_{%
\mathcal{G}}\left( V\right) \text{ under } P \text{ for all }V\in\mathbf{V} .
\nonumber
\end{equation}
The Bayesian Network represented by DAG $\mathcal{G}$ \citep{pearl-causality} is defined as the collection
\begin{equation*}
\mathcal{M}\left( \mathcal{G}\right) \equiv \left\{ P: P \text{ satisfies the Local Markov Property relative to } \mathcal{G} 
\right\}.
\end{equation*}
\cite{verma-pearl} and \cite{geiger} 
show that for any disjoint sets $\mathbf{A,B,C}$ included in $\mathbf{V}$%
\begin{equation*}
\mathbf{A}\text{ }\mathbf{\perp \!\!\!\perp }_{\mathcal{G}}\text{ }\mathbf{B}%
\text{ }|\text{ }\mathbf{C} \Leftrightarrow \mathbf{A\perp \!\!\!\perp
B}\text{ }|\text{ }\mathbf{C} \text{ under $P$ for all } P\in \mathcal{M}%
\left( \mathcal{G}\right).
\end{equation*}

A causal (agnostic) graphical model \citep{spirtes, mediation} represented by $\mathcal{G}$ assumes that the law of $\mathbf{V}\equiv (V_{1},\dots, V_{s})$ belongs to $\mathcal{M(G)}$ and that for
any $\mathbf{A=}\left\{ A_{1},\dots,A_{p}\right\} \mathbf{\subset }\mathbf{V},$
the post-intervention density (with respect to a dominating measure) $f\left[ \mathbf{v}\mid \doo(\mathbf{a})\right]$ of $\mathbf{V}$ when $\mathbf{A}$ is set to $\mathbf{a}$ on the entire population satisfies
\begin{equation}
f\left[ \mathbf{v}\mid \doo(\mathbf{a})\right] = 
\begin{cases}
\prod\limits_{V_{j}\in \mathbf{V}\setminus \mathbf{A}} f(v_{j} \mid \pa_{\mathcal{G}}(V_{j})) &\text{if } \mathbf{A}=\mathbf{a} \\
\quad 0 &\text{otherwise}.
\end{cases}
\label{eq:g-form}
\end{equation}
Formula \eqref{eq:g-form} is known as the g-formula \citep{robins1986}, the manipulated density formula \citep{spirtes} and the truncated factorization formula \citep{pearl-causality}. 

The non-parametric structural equations model with independent errors (NPSEM-IE, \citealt{pearl-causality}) is a sub-model of the causal agnostic graphical model that additionally assumes the existence of counterfactuals. Specifically, the model associates each vertex $V\in \mathbf{V}$ with a factual
random variable satisfying 
\begin{equation*}
V=g_{V}\left( \pa_{\mathcal{G}}\left( V\right) ,\varepsilon _{V}\right) 
\text{ for all }V\in \mathbf{V}
\end{equation*}%
where $\left\{ \varepsilon _{V}\right\} _{V\in \mathbf{V}}$ are mutually
independent and $\left\{ g_{V}\right\} _{V\in \mathbf{V}}$ are arbitrary
functions. The model also assumes that for
any $\mathbf{A=}\left\{ A_{1},\dots,A_{p}\right\} \mathbf{\subset }\mathbf{V,}$
the counterfactual vector $\mathbf{V}_{\mathbf{a}}$ that would be observed
had $\mathbf{A}$ been set to $\mathbf{a}$ exists, and is generated according to%
\begin{eqnarray*}
V_{\mathbf{a}} &=&g_{V}\left( \pa_{\mathcal{G}}\left( V_{\mathbf{a}}\right)
,\varepsilon _{V}\right) \text{ for all }V\in \mathbf{V\backslash A} \\
A_{\mathbf{a},k} &=&a_{k} \quad \quad \quad \quad \quad \quad \quad \:\text{%
for all }k=1,\dots,p.
\end{eqnarray*}
The finest fully randomized causally interpretable  structured tree graph model (FFRCISTG, \citealt{robins1986}) makes the same assumptions as the NPSEM-IE model, except that it relaxes
the assumption that the $\left\{ \varepsilon _{V}\right\} _{V\in \mathbf{V}}$ are mutually independent. We note that the only restriction that the NPSEM-IE and the FFRCISTG models place on the law $P$ of the factual random vector $\mathbf{V}$, is that $P\in\mathcal{M(G)}$. Furthermore, \eqref{eq:g-form} remains valid under both models. See \cite{swig} for more details.

The results that we will derive in this paper rely solely on the assumption that $P\in\mathcal{M(G)}$  and on the validity of \eqref{eq:g-form}. Therefore, the results hold for the causal agnostic graphical models, the NPSEM-IE, and the FFRCISTG.

A causal (agnostic) graphical linear model represented by $\mathcal{G}$ is the submodel of the causal (agnostic) graphical model which additionally imposes the restriction that $\mathbf{V}=(V_{1},\dots, V_{s})$ satisfies
$$
V_{i}= \sum\limits_{V_{j}\in \pa_{\mathcal{G}}(V_{i})} \alpha_{ij}V_{j} + \varepsilon_{i},
$$
for $i\in \lbrace 1,\dots,S\rbrace$, where $\alpha_{ij}\in\mathbb{R}$ and $\varepsilon_{1},\dots,\varepsilon_{p}$ are jointly independent random variables with zero mean and finite variance.

Throughout this paper we let $\mathbf{V}_{\mathbf{a}}$ be a random vector with density $f\left[ \mathbf{v}\mid \doo(\mathbf{a})\right]$. In particular for $Y\in \mathbf{V}$ we let $Y_{\mathbf{a}}$ be the corresponding component of $\mathbf{V}_{\mathbf{a}}$. We call $E\left[ Y_{\mathbf{a}}\right]= E\left[Y \mid \doo(\mathbf{A}) \right]$ the interventional mean under $\mathbf{A}=\mathbf{a}$. Note that  $\mathbf{V}_{\mathbf{a}}$ is not a counterfactual random vector if only the causal agnostic graphical model is assumed.

\subsection{Interventional mean}

Under the causal graphical model, for any $\mathbf{A=}\left\{ A_{0},\dots,A_{p}\right\} 
\mathbf{\subset }\mathbf{V}$ topologically ordered, where each $A_{k}$ a discrete random variable 
and $Y\in \mathbf{V\backslash A}$, the interventional mean on the outcome $Y$ satisfies
\begin{equation*}
E\left[ Y_{\mathbf{a}}\right] =E_{P}\left[ \prod\limits_{k=0}^{p}\left\{ \frac{%
I_{a_{k}}\left( A_{k}\right) }{P\left( A_{k}=a_{k}|\pa_{\mathcal{G}}\left(
A_{k}\right) \right) }\right\} Y\right].
\end{equation*}
This is an immediate consequence of formula \eqref{eq:g-form}. The Local Markov Property for $P\in\mathcal{M(G)}$ further implies that 
$$
E\left[Y_{\mathbf{a}} \right]=E_{P}\left\lbrace\left\{ E_{P}\left\{ E_{P}\left[ E_{P}\left[ Y|\mathbf{a},%
\overline{\pa_{\mathcal{G}}(A_{p})}\right] |\overline{\mathbf{a}}_{p-1},\overline{\pa_{\mathcal{G}}(A_{p-1})}\right] |\overline{\mathbf{a}}_{p-2},%
\overline{\pa_{\mathcal{G}}(A_{p-2})}\right\} \cdots \mid {\mathbf{a}}_{0},%
\overline{\pa_{\mathcal{G}}(A_{0})} \right\}\right\rbrace,
$$
where for every $j\in\lbrace 0,\dots,p \rbrace$
$$
\overline{\pa_{\mathcal{G}}(A_{j})}=\bigcup\limits_{k=0}^{j} \pa_{\mathcal{G}}(\mathcal{A}_{j}).
$$
In particular, if $\mathbf{A}$ is a point intervention, so that it is a single variable $A,$ then 
\begin{eqnarray}
E\left[ Y_{a}\right] &=&E_{P}\left[ \frac{I_{a}\left( A\right) }{P\left( A=a|\pa%
_{\mathcal{G}}\left( A\right) \right) }Y\right]  \label{eq:countermeanbinary}
\\
&=&E_{P}\left[ E\left[ Y|A=a,\pa_{\mathcal{G}}\left( A\right) \right] \right] . 
\notag
\end{eqnarray}%
For a binary point intervention $A$, the average treatment effect (ATE),  $ATE\equiv E\left[
Y_{a=1}\right] -E\left[ Y_{a=0}\right] $, quantifies the effect on the mean of the outcome of setting $A=1$ versus $A=0$ on the entire population.
Under a causal graphical model, equation \eqref{eq:g-form} implies
\begin{equation*}
ATE=E_{P}\left[ E_{P}\left[ Y|A=1,\pa_{\mathcal{G}}\left( A\right) \right] \right] -E_{P}
\left[ E_{P}\left[ Y|A=0,\pa_{\mathcal{G}}\left( A\right) \right] \right].
\end{equation*}

\subsection{Adjustment sets}

\begin{definition}[Time dependent covariate adjustment set]
\label{def:time_dep_adj}
$ $\newline
Let $\mathcal{G}$ be a DAG with vertex set $\mathbf{V}$ let $\mathbf{A=}%
\left( A_{0},\dots,A_{p}\right) \mathbf{\subset V}$ be topologically ordered
and $Y\in  \mathbf{V} \backslash \mathbf{A}$. We say that $\mathbf{Z}\equiv \left( \mathbf{Z}%
_{0},\mathbf{Z}_{1},\dots,\mathbf{Z}_{p}\right) \subset \mathbf{V}\backslash \left\{ 
\mathbf{A},Y\right\} $ where ${\mathbf{Z}_{0},\mathbf{Z}_{1},\dots}$
and ${\mathbf{Z}_{p}}$ are disjoint, is a time dependent covariate adjustment
set relative to $\left( \mathbf{A},Y\right) $ in $\mathcal{G}$ if under all $%
P\in \mathcal{M}\left( \mathcal{G}\right) $  and all $y\in \mathbb{R}$
\begin{eqnarray*}
&&E_{P}\left[ \prod\limits_{k=0}^{p}\left\{ \frac{I_{a_{k}}\left(
A_{k}\right) }{P\left( A_{k}=a_{k}|\pa_{\mathcal{G}}\left( A_{k}\right)
\right) }\right\} I_{(-\infty,y]}(Y)\right] \\
&=&E_{P}\left\lbrace \left\{ E_{P}\left\{ E_{P}\left[ E_{P}\left[ I_{(-\infty,y]}(Y)|\mathbf{A}=\mathbf{a},%
\mathbf{Z}\right] |\overline{\mathbf{A}}_{p-1}=\overline{\mathbf{a}}_{p-1},\overline{%
\mathbf{Z}}_{p-1}\right] |\overline{\mathbf{A}}_{p-2}=\overline{\mathbf{a}}_{p-2},%
\overline{\mathbf{Z}}_{p-2}\right\} \cdots \mid {\mathbf{A}_{0}}=\mathbf{a}_{0}, \mathbf{Z}_{0}\right\}\right\rbrace.
\label{eq:time_dep_adj}
\end{eqnarray*}%
\end{definition}
The preceding definition extends the following definition of covariate adjustment set of \cite{shpitser-adjustment} and \cite{maathuis2015}. We use the appellatives time dependent and time independent to distinguish the two definitions.
\begin{definition}[Time independent covariate adjustment set]
\citep{shpitser-adjustment,maathuis2015}
Let $\mathcal{G}$ be a DAG with vertex set $\mathbf{V},$ let $\mathbf{A}\subset \mathbf{V}$ and $Y\in \mathbf{V}\setminus \mathbf{A}$. A set $\mathbf{Z\subset V\backslash }%
\left\{ \mathbf{A},Y\right\} $ is a time independent adjustment set relative to $\left( \mathbf{A},Y\right) $
in $\mathcal{G}$ if under all $P\in \mathcal{M}\left( \mathcal{G}\right) $ 
\begin{equation}
E_{P}\left[ E_{P}\left[ I_{(-\infty, y]}(Y)|\mathbf{A}=\mathbf{a},\pa_{\mathcal{G}}\left( \mathbf{A}\right) \right] %
\right] =E_{P} \left[ E_{P}\left[ I_{(-\infty, y]}(Y)|\mathbf{A}=\mathbf{a},\mathbf{Z}\right] \right] \quad \text{for all } y\in\mathbb{R}.
\label{eq:colombo}
\end{equation}
\end{definition}

Note that $\mathbf{Z}$ is a time independent adjustment set if and only if $\widetilde{\mathbf{Z}}=\left( \mathbf{Z}_{0},\dots, \mathbf{Z}_{p}\right)$ with $\mathbf{Z}_{0}=\mathbf{Z}$ and $\mathbf{Z}_{j}=\emptyset$ for $j=1\dots,p$ is a time dependent adjustment set. 

The back-door criterion \citep{pearl-causality} is a sufficient
graphical condition for $\mathbf{Z}$ to be a time independent adjustment set. 
\cite{shpitser-adjustment} gives a necessary and sufficient graphical
condition for $\mathbf{Z}$ to be a time independent covariate adjustment set. 
These authors also show that if $\mathbf{Z}$ is a time independent covariate adjustment
set, then there exists $\mathbf{Z}_{sub}\subset \mathbf{Z}$ such that $%
\mathbf{Z}_{sub}$ is a time independent adjustment set and it satisfies the back-door
criterion. On the other hand \cite{prob-eval} provides a sufficient graphical criterion for $\mathbf{Z}$ to be a time dependent adjustment set. \cite{robinsaddendum} derives analogous sufficient conditions assuming the causal diagram  represents a non-parametric structural equations model. See also \cite{swig}.

 When $\mathbf{A}$ is a point intervention $A$, a time independent adjustment sets always exist. For instance, $\mathbf{Z}=\pa_{\mathcal{G}}(A)$ is one such set. However, for $\mathbf{A=}\left( A_{0},\dots,A_{p}\right)$ a joint intervention, a time independent covariate adjustment set $\mathbf{Z}$
may not exist in some graphs, as noted in \cite{perkovic}. In contrast, a time
dependent adjustment sets always exists, since $\mathbf{Z}\equiv \left( \mathbf{Z}_{0},\mathbf{Z}_{1},\dots,\mathbf{Z}_{p}\right) $
where $\mathbf{Z}_{0}\equiv \pa_{\mathcal{G}}\left( A_{0}\right) $ and $%
\mathbf{Z}_{k}\equiv \pa_{\mathcal{G}}\left( A_{k}\right) \backslash \left[
\cup _{j=0}^{k-1}\pa_{\mathcal{G}}\left( A_{j}\right) \right] ,k=1,\dots,p$ is
a time dependent adjustment set.

\begin{example}
In the DAG of Figure \ref{fig:examp_time},
there is no time independent adjustment set relative to $(\mathbf{A},Y)$ for $\mathbf{A}=(A_0,A_1)$.
For instance, $\mathbf{Z}=\left( \mathbf{Z}_0, \mathbf{Z}_{1} \right)$ with 
$\mathbf{Z}_{0}=\left\{ L_{0}\right\} $ and $\mathbf{Z}%
_{1}=\left\{ L_{1}\right\}$, and $\widetilde{\mathbf{Z}}=\left( \widetilde{\mathbf{Z}}_0, \widetilde{\mathbf{Z}}_{1} \right)$, with $\widetilde{\mathbf{Z}}_{0}=\left\{ L_{0}\right\} $ and $\widetilde{\mathbf{Z}}_{1}=\left\{ L_1, U\right\} $, are two time dependent adjustment sets \citep{robinsaddendum}. 
\end{example}

\begin{figure}[ht!]
\begin{center}
\begin{tikzpicture}[>=stealth, node distance=1.5cm,
pre/.style={->,>=stealth,ultra thick,line width = 1.4pt}]
  \begin{scope}
    \tikzstyle{format} = [circle, inner sep=2.5pt,draw, thick, circle, line width=1.4pt, minimum size=6mm]
    \node[format] (L0) {$L_{0}$};
        \node[format, right of=L0] (A0) {$A_0$};
       \node[format, right of=A0] (L1) {$L_1$};
         \node[format, right of=L1] (A1) {$A_1$};
         \node[format, right of=A1] (Y) {$Y$};
        \node[format, above of=A1] (U) {$U$};
                 \draw (L0) edge[pre, black] (A0);
      \draw (A0) edge[pre, black] (L1);
     \draw (L1) edge[pre, black] (A1);
     \draw (A1) edge[pre, black] (Y);
    \draw (U) edge[pre, black] (L1);
    \draw (U) edge[pre, black] (Y);
        \draw (L1) edge[pre, black, out=30, in=155] (Y);
        \draw (L0) edge[pre, black, out=300,in=230] (Y);
     \draw (A0) edge[pre, black, out=300,in=200] (Y);
  \end{scope} 
  \end{tikzpicture}
\end{center}
\caption{A DAG with two possible time dependent adjustment sets and no time independent adjustment sets.}
\label{fig:examp_time}
\end{figure}
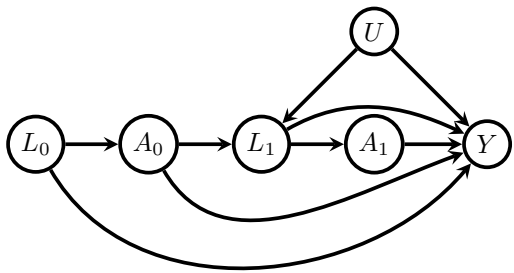

\FloatBarrier

We also have the following definition.
\begin{definition}[Minimal covariate adjustment set]
Let $\mathcal{G}$ be a DAG with vertex set $\mathbf{V},$ let $\mathbf{A}$ and $Y\in \mathbf{V}\setminus \mathbf{A}$. A set $\mathbf{Z\subset V\backslash }%
\left\{ \mathbf{A},Y\right\} $ is a minimal time dependent (independent) adjustment set relative to $\left(
\mathbf{A},Y\right) $ in $\mathcal{G}$ if $\mathbf{Z}$ is a time dependent (independent) adjustment set and no
proper subset of $\mathbf{Z}$ is a time dependent (independent) adjustment set.
\end{definition}

\subsection{Non-parametric estimation of an interventional mean}

In order to discuss the non-parametric estimation of an interventional mean $E[Y_{\mathbf{a}}]$, we begin by reviewing some elements of the theory of asymptotic inference.

An estimator $\widehat{\gamma }$ of a scalar parameter $\gamma \left( P\right) $
based on $n$ i.i.d. copies $\mathbf{V}_{1},\dots,\mathbf{V}_{n}$ of $\mathbf{V}
$ is asymptotically linear at $P$ if there exists a random variable $\varphi
_{P}\left( \mathbf{V}\right) $ with mean zero and finite variance such
that under $P$ 
\begin{equation*}
n^{1/2}\left\{ \widehat{\gamma } -\gamma \left( P\right) \right\} =\frac{1}{%
n^{1/2}}\sum_{i=1}^{n}\varphi
_{P}\left( \mathbf{V}_{i}\right) +o_{p}(1).
\end{equation*}
The random variable $\varphi_{P}\left( \mathbf{V}\right)$ is called
the influence function of $\gamma(P)$ at $P$. By the Central Limit
Theorem any asymptotically linear estimator is consistent and asymptotically
normal (CAN) with asymptotic variance equal to $var_{P}\left[ \varphi
_{P}\left( \mathbf{V}_{i}\right) \right] $, provided that $var_{P}\left[\varphi
_{P}\left( \mathbf{V}_{i}\right) \right]<\infty $. Furthermore any two asymptotically
linear estimators, say $\widehat{\gamma }_{1}$ and $\widehat{\gamma }_{2}$, with
the same influence function are asymptotically equivalent in the sense that $%
n^{1/2}\left( \widehat{\gamma }_{1}\,-\ \widehat{\gamma }_{2}\right)
=o_{p}\left( 1\right) $

Given a collection of probability laws $\mathcal{M}$ for $\mathbf{V}$,
an estimator of $\widehat{\gamma }$ of $\gamma \left( P\right) $ is regular in $\mathcal{M}$ at $P$ if its convergence to $\gamma \left( P\right) $ is
locally uniform \citep{van1998asymptotic}. Regularity is a necessary
condition for a nominal $1-\alpha $ level Wald interval centered at the
estimator to be an honest confidence interval in the sense that there exists
a sample size $n^{\ast }$ such that for all $n>n^{\ast }$ the interval
attains at least its nominal coverage over all laws in  $\mathcal{M}$.

Suppose that $\mathbf{A}$ is a vector of variables taking values on a finite set $\mathcal{A}$ and one
is interested in estimating some contrast
$$
\Delta \equiv
\sum_{ \mathbf{a} \in \mathcal{A}}c_{\mathbf{a}}E[Y_{\mathbf{a}}]
$$
for given constants $c_{\mathbf{a}}$, $\mathbf{a}\in\mathcal{A}$.
In particular if $\mathbf{A}=A$ is binary and $c_{1}=1$ and $c_{0}=-1$ the preceding linear combination is equal to $ATE$.
Suppose that, having postulated a causal graphical model, one finds that  time independent adjustment sets exist. Having decided on one adjustment
set $\mathbf{Z,}$ one estimates
$$
\Delta(P;\mathcal{G})\equiv \sum_{ \mathbf{a} \in \mathcal{A}}c_{\mathbf{a}}  \chi _{\mathbf{a}}\left( P;%
\mathcal{G}\right),
$$
where
$$
\chi _{\mathbf{a}}\left( P;%
\mathcal{G}\right) \equiv E_{P}\left[ E_{P}\left[ Y|\A=\mathbf{a},\mathbf{Z}\right] \right]
=E_{P}\left[ \pi _{\mathbf{a}}\left( \mathbf{Z};P\right) ^{-1}I_{\mathbf{a}}\left( \A\right) Y%
\right],
$$
by estimating each $\chi _{\mathbf{a}}\left( P;%
\mathcal{G}\right)$ under a model $\mathcal{M}$ that makes at most smoothness
or complexity assumptions on 
$$
b_{\mathbf{a}}\left( \mathbf{Z};P\right) \equiv E_{P}%
\left[ Y|\A=\mathbf{a},\mathbf{Z}\right] 
$$
and/or
$$
\pi _{\mathbf{a}}\left( \mathbf{Z};P\right)
\equiv P\left[ \A=\mathbf{a}|\mathbf{Z}\right].
$$
Examples of such estimating strategies are the inverse probability weighted estimator $$
\widehat{\chi }%
_{\mathbf{a},IPW}=\mathbb{P}_{n}\left[ \widehat{\pi }_{\mathbf{a}}\left( \mathbf{Z}\right)
^{-1}I_{\mathbf{a}}\left( \A\right) Y\right] $$
where $\widehat{\pi }_{\mathbf{a}}\left( \cdot
\right) $ is a series or kernel estimator of $P\left[ \A=\mathbf{a}|\mathbf{Z=\cdot }%
\right] $ \citep{hirano}, the outcome regression estimator $\mathbb{P}_{n}%
\left[ \widehat{b}_{\mathbf{a}}\left( \mathbf{Z}\right) \right] $ where $\widehat{b}%
_{\mathbf{a}}\left( \cdot \right) $ is a smooth estimator of $b_{\mathbf{a}}\left( \mathbf{Z}%
;P\right) $ \citep{hahn} or the doubly-robust estimator 
\citep{vanderlaan,
chernozhukov2018double, smucler}.

This estimation strategy effectively uses the causal model solely to provide
guidance on the selection of the adjustment set but otherwise ignores the
information about the interventional means $\chi _{\mathbf{a}}\left( P;\mathcal{G}\right) $ encoded
in the causal model. This is a strategy frequently followed in applications %
\citep{cattaneo, bottou, hernan}. It is well known \citep{robins1994} that
estimators $\widehat{\chi }_{\mathbf{a},\mathbf{Z}}$ of $\chi _{\mathbf{a}}\left( P;\mathcal{G}%
\right) $ based on the adjustment set $\mathbf{Z}$ that are regular and
asymptotically linear under a model $\mathcal{M}$ that imposes at most
smoothness or complexity assumptions on $b_{\mathbf{a}}\left( \mathbf{Z};P\right) $
and/or $\pi _{\mathbf{a}}\left( \mathbf{Z};P\right) $ have a unique influence
function equal to 
\begin{equation}
\psi _{P,\mathbf{a}}\left( \mathbf{Z};\mathcal{G}\right) \equiv \frac{I_{\mathbf{a}}\left(
\mathbf{A}\right) }{\pi _{\mathbf{a}}\left( \mathbf{Z};P\right) }\left( Y-b_{\mathbf{a}}\left( \mathbf{Z%
};P\right) \right) +b_{\mathbf{a}}\left( \mathbf{Z};P\right) -\chi _{\mathbf{a}}\left( P;%
\mathcal{G}\right) ,  \label{eq:inf_fc_sing}
\end{equation}%
where to avoid overloading the notation in $\psi _{P,\mathbf{a}}$ we do not
explicitly write its dependence on $\left( Y,\mathbf{A}\right) $. 

Consequently,
estimators $%
\widehat{\Delta }_{\mathbf{Z}}\equiv \sum_{\mathbf{a}\in \mathcal{A}}c_{\mathbf{a}}\widehat{%
\chi }_{\mathbf{a},\mathbf{Z}}$ 
of $\Delta(P;\mathcal{G})$
have a unique influence function
equal to 
$$
\psi _{P,\Delta}\left( \mathbf{Z};\mathcal{G}\right) =\sum_{\mathbf{a}\in\mathcal{A}}c_{\mathbf{a}}\psi
_{P,\mathbf{a}}\left( \mathbf{Z};\mathcal{G}\right).
$$ 
For simplicity, we refer to asymptotically
linear estimators of $\chi _{\mathbf{a}}\left( P;\mathcal{G}\right) $ with influence function $\psi _{P,\mathbf{a}}\left( \mathbf{Z};\mathcal{G%
}\right) $  as non-parametric estimators that use the adjustment set $\mathbf{Z}$ and
we abbreviate them with NP-$\mathbf{Z.}$ 

The preceding discussion implies that any NP-$\mathbf{Z}$ estimator $\widehat{\chi }_{\mathbf{a},%
\mathbf{Z}}$ satisfies 
\begin{equation*}
\sqrt{n}\left\{ \widehat{\chi }_{\mathbf{a},\mathbf{Z}}-\chi _{\mathbf{a}}\left( P;\mathcal{G}%
\right) \right\} \cw  N\left( 0,\sigma _{\mathbf{a},\mathbf{Z}}^{2}\left(
P\right) \right)
\end{equation*}%
where $\sigma _{\mathbf{a},\mathbf{Z}}^{2}\left( P\right) \equiv var_{P}\left[ \psi
_{P,\mathbf{a}}\left( \mathbf{Z};\mathcal{G}\right) \right] .$
Likewise, $\sqrt{n}%
\left\{ \widehat{\Delta}_{\mathbf{Z}}-\Delta(P;\mathcal{G}) \right\rbrace \cw
N\left( 0,\sigma _{\Delta,\mathbf{Z}}^{2}\right) $ where 
\begin{equation*}
\sigma _{\Delta,\mathbf{Z}}^{2}\left( P\right) \equiv var_{P}\left[ \psi
_{P,\Delta}\left( \mathbf{Z};\mathcal{G}\right) \right] .
\end{equation*}

Two natural questions of practical interest arise. The first is whether any two given time independent covariate adjustment sets, say $\mathbf{Z}, \mathbf{Z}^{\prime}$, are comparable in the sense that either
$$
\sigma _{\Delta,\mathbf{Z}}^{2} \leq \sigma _{\Delta,\mathbf{Z}^{\prime}}^{2} \text{ for all } P \in \mathcal{M(G)} \text{ or } \sigma _{\Delta,\mathbf{Z}^{\prime}}^{2}\leq \sigma _{\Delta,\mathbf{Z}}^{2} \text{ for all } P \in \mathcal{M(G)}.
$$
The second is whether an optimal time independent adjustment set $\mathbf{O}
$ exists such that for any other time independent adjustment set $\mathbf{Z,}$
\begin{equation}
\sigma _{\Delta,\mathbf{O}}^{2}\left( P\right) \leq \sigma _{\Delta,\mathbf{Z}%
}^{2}\left( P\right) .  \label{eq:var_compare}
\end{equation}%
These questions were answered by \cite{perkovic} under (i) a linear causal graphical 
model, (ii) when 
$$
\Delta= E\left[Y_{\mathbf{a}}-Y_{\mathbf{a}^{\prime}} \right]
$$ 
where $\mathbf{a}-\mathbf{a}^{\prime}$ is the vector with all coordinates equal to zero except for coordinate $j$ which is equal to one, and (iii)
when $\Delta$ is estimated as the ordinary least squares estimator of the coefficient of $A_{j}$ in the linear regression of $Y$ on $\mathbf{A}$ and $\mathbf{Z}$ and $\sigma^{2}_{\Delta, \mathbf{Z}}$ is the asymptotic variance of such estimators.
These authors showed that not all time independent covariate adjustment sets are comparable. However, they provided a graphical criterion to compare certain pairs of time independent covariate adjustment sets. 
They also provided a graphical criterion for characterizing the set $
\mathbf{O}$, whenever a valid time independent covariate adjustment set exists. In particular, the criterion always returns an optimal valid time independent covariate adjustment set for $\mathbf{A}=A$ a point interventions.

In Section \ref{sec:optimal_adj_point} we prove that the same
graphical criteria remain valid for comparing time independent covariate adjustment sets and for characterizing the set $\mathbf{O}$ that
satisfies $\left( \ref{eq:var_compare}\right) $ under an arbitrary, not necessarily linear, causal graphical model and for NP-$\mathbf{Z}$ estimators of an arbitrary contrast $\Delta$.
Moreover, for $\mathbf{A}=A$ a point intervention, we further show that there exists a minimal adjustment set $\mathbf{O}%
_{\min }$ included in $\mathbf{O}$ such that $\mathbf{O}_{\min } $ is optimal among
the minimal adjustment sets; that is, for any other minimal adjustment set $%
\mathbf{Z}_{\min },$ 
\begin{equation}
\sigma _{\Delta,\mathbf{O}_{\min }}^{2}\left( P\right) \leq \sigma _{\Delta,%
\mathbf{Z}_{\min }}^{2}\left( P\right),   \label{eq:sigmaopt}
\end{equation}
where  $\sigma _{\Delta,%
\mathbf{Z}_{\min }}^{2}(P)$ stands for either the asymptotic variance of the NP-$\mathbf{Z}_{min}$ estimator or the asymptotic variance of the OLS estimator of treatment effect of \cite{perkovic}.
In addition, we provide a graphical criterion for identifying $\mathbf{O}_{\min }$.
Using the tools developed in \cite{vanfinding}, $\mathbf{O}$ and $%
\mathbf{O}_{\min }$ it can be shown that can be computed in polynomial time.

Consider next the case in which $\mathbf{A}=\left( A_{0},\dots,A_{p}\right) $
is a joint intervention  with $p>0.$ In analogy with the time independent covariate adjustment case we consider in Section \ref{sec:optimal_adj_time} the setting in which
one uses the causal model to identify the collection of time dependent adjustment sets, but
then for any given time dependent adjustment set $\mathbf{Z,}$ one estimates each
 $E\left[ Y_{\mathbf{a}}\right] $ ignoring the conditional
independences encoded in the causal graphical model. For instance, for $p=1, $ we study the asymptotic efficiency of estimators of 
\begin{eqnarray*}
\chi _{a_{0},a_{1}}\left( P;\mathcal{G}\right) &\mathbf{\equiv }%
&E_{P}\left\{ E_{P}\left[ E_{P}\left[ Y|A_{0}=a_{0},A_{1}=a_{1},\mathbf{Z}%
_{0},\mathbf{Z}_{1}\right] |A_{0}=a_{0},\mathbf{Z}_{0}\right] \right\} \\
&\mathbf{=}&E_{P}\left[ \frac{I_{a_{0}}\left( A_{0}\right) }{P\left[
A_{0}=a_{0}|\mathbf{Z}_{0}\right] }\frac{I_{a_{1}}\left( A_{1}\right) }{P%
\left[ A_{1}=a_{1}|A_{0}=a_{0},\mathbf{Z}_{0},\mathbf{Z}_{1}\right] }Y\right]
\end{eqnarray*}%
for different time dependent adjustment sets $\left( \mathbf{Z}_{0},\mathbf{Z}_{1}\right)$, under a model $\mathcal{M}$ that makes at most smoothness or
complexity assumptions on 
\begin{align*}
 &b_{a_{0},a_{1}}\left( \mathbf{Z}_{0},\mathbf{Z}%
_{1};P\right) \equiv E_{P}\left[ Y|A_{0}=a_{0},A_{1}=a_{1},\mathbf{Z}_{0},%
\mathbf{Z}_{1}\right] ,   \\
&b_{a_{0}}\left( \mathbf{Z}_{0};P\right) \equiv
E_{P}\left[ b_{a_{0},a_{1}}\left( \mathbf{Z}%
_{0},\mathbf{Z}_{1};P\right)|A_{0}=a,\mathbf{Z}_{0}\right] 
\end{align*}
 and/or 
\begin{align*}
 &   \pi
_{a_{0},a_{1}}\left( \mathbf{Z}_{0},\mathbf{Z}_{1};P\right) \equiv P\left[
A_{1}=a_{1}|A_{0}=a_{0},\mathbf{Z}_{0},\mathbf{Z}_{1}\right], \\
&\pi
_{a_{0}}\left( \mathbf{Z}_{0};P\right) \equiv P\left[ A_{0}=a_{0}|\mathbf{Z}%
_{0}\right].
\end{align*} 
See \cite{vanderlaan}. 
Just as for the case of time independent adjustment sets, not all time dependent adjustment sets are comparable in terms of their asymptotic variance uniformly for all $P\in\mathcal{M(G)}$.
However, in Section \ref{sec:optimal_adj_time} we generalize the aforementioned graphical criterion that allows the comparison of certain time dependent
adjustment sets. Nevertheless we show by example that unlike the case of time independent adjustment sets, even though a time dependent adjustment set always exists, there are DAGs in which no uniformly optimal  time dependent adjustment set exists.

\section{Comparison of adjustment sets}
\label{sec:optimal_adj}

In Section  \ref{sec:optimal_adj_point} we show that the graphical criteria for comparing time independent adjustment sets and for identifying the optimal time independent adjustment set of \cite{perkovic} is valid also when treatment effects are estimated non-parametrically. In Section \ref{sec:optimal_adj_time} we provide results for time dependent adjustment sets.
\subsection{Time independent adjustment sets}

\label{sec:optimal_adj_point}

\begin{lemma}[Supplementation with time independent precision variables]
\label{lemma:supplementation} Let $\mathcal{G}$ be a DAG with vertex set $%
\mathbf{V},$ let $\mathbf{A}\subset \mathbf{V}$ and $Y\in \mathbf{V}%
\setminus \mathbf{A}$ with $\mathbf{A}$ a random vector taking values on a
finite set. Suppose $\mathbf{B\subset V\backslash }\left\{ \mathbf{A}%
,Y\right\} $ is a time independent adjustment set relative to $\left( 
\mathbf{A},Y\right) $ in $\mathcal{G}$ and suppose $\mathbf{G}$ is a
disjoint set with $\mathbf{B}$ that satisfies 
\[
\mathbf{A}\text{ }\mathbf{\perp \!\!\!\perp }_{\mathcal{G}}\text{ }\mathbf{G}%
\text{ }\mathbf{|}\text{ }\mathbf{B}.
\]%
Then $\left( \mathbf{G,B}\right) $ is also a time independent adjustment set
relative to $\left( \mathbf{A},Y\right) $ in $\mathcal{G}$ and for all $P\in 
\mathcal{M}\left( \mathcal{G}\right) $%
\begin{equation}
\sigma _{\mathbf{a},\mathbf{B}}^{2}\left( P\right) -\sigma _{\mathbf{a},%
\mathbf{G},\mathbf{B}}^{2}\left( P\right) =E_{P}\left[ \left\{ \frac{1}{\pi
_{\mathbf{a}}\left( \mathbf{B};P\right) }-1\right\} var_{P}\left[ \left. b_{%
\mathbf{a}}(\mathbf{G,B};P)\right\vert \mathbf{B}\right] \right] \geq 0.
\label{eq:suplemento}
\end{equation}%
Furthermore,%
\[
\sigma _{\Delta ,\mathbf{B}}^{2}\left( P\right) -\sigma _{\Delta ,\mathbf{G},%
\mathbf{B}}^{2}\left( P\right) =\mathbf{c}^{T}var_{P}\left( \mathbf{Q}%
\right) \mathbf{c\geq }0
\]%
where $\mathbf{c}\equiv \left( c_{\mathbf{a}}\right) _{\mathbf{a}\in 
\mathcal{\mathbf{A}}}$  and $\mathbf{Q\equiv }\left[ Q_{\mathbf{a}}\right]
_{\mathbf{a}\in \mathcal{\mathbf{A}}}$ with 
\begin{align*}
&Q_{\mathbf{a}}\equiv \left\{ 
\frac{I_{\mathbf{a}}(\mathbf{A})}{\pi _{\mathbf{a}}(\mathbf{G},\mathbf{B};P)}%
-1\right\} \left\{ b_{\mathbf{a}}(\mathbf{G},\mathbf{B};P)-b_{\mathbf{a}}(%
\mathbf{B};P)\right\},\\
&var_{P}\left( Q_{\mathbf{a}}\right) =E_{P}\left[
\left\{ \frac{1}{\pi _{\mathbf{a}}(\mathbf{B};P)}-1\right\} var_{P}(b_{%
\mathbf{a}}(\mathbf{G,B};P)\mid \mathbf{B})\right], \\
&\text{and } cov_{P}\left[ Q_{%
\mathbf{a}},Q_{\mathbf{a}^{\prime }}\right] =-E_{P}\left[ cov_{P}\left\{ b_{%
\mathbf{a}}(\mathbf{G},\mathbf{B};P),b_{\mathbf{a}^{\prime }}(\mathbf{G},%
\mathbf{B};P)|\mathbf{B}\right\} \right]  \text{ for } \mathbf{a\not=a}^{\prime }.
\end{align*}
In particular,
\begin{eqnarray*}
\sigma _{ATE,\mathbf{B}}^{2}\left( P\right) -\sigma _{ATE,\mathbf{G},\mathbf{%
B}}^{2}\left( P\right)  &=&E_{P}\left[ \left\{ \frac{1}{\pi _{a=1}(\mathbf{B}%
;P)}-1\right\} var_{P}(b_{a=1}(\mathbf{G,B};P)\mid \mathbf{B})\right]  \\
&&+E_{P}\left[ \left\{ \frac{1}{\pi _{a=0}(\mathbf{B};P)}-1\right\}
var_{P}(b_{a=0}(\mathbf{G,B};P)\mid \mathbf{B})\right]  \\
&&-2E_{P}\left[ cov_{P}\left\{ b_{a=1}(\mathbf{G},\mathbf{B};P),b_{a=0}(%
\mathbf{G},\mathbf{B};P)|\mathbf{B}\right\} \right].
\end{eqnarray*}
\end{lemma}

For the special case in
which $\mathbf{B=\emptyset}$, formula \eqref{eq:suplemento} was derived in 
\cite{AIDS} and \cite{hahn}.
The formula quantifies the reduction in
variance associated with supplementing an adjustment set with `precision'
variables, i.e. variables that may help predict the outcome within 
treatment levels but are not associated with treatments after controlling for
the already existing adjustment set. Notice that $var_{P}\left[ \left. b_{a}(%
\mathbf{G,B};P)\right\vert \mathbf{B}\right] $ quantifies the additional
explanatory power carried by $\mathbf{G}$ for $Y$ after adjusting for $%
\mathbf{B.}$ In the DAG represented in Figure \ref{fig:supp}, $\mathbf{B=}%
\left\{ B\right\} $ and $\mathbf{G=}\left\{ G\right\} $ satisfy the
conditions of Lemma \ref{lemma:supplementation}. In that DAG, $var_{P}\left[
\left. b_{a}(\mathbf{G,B};P)\right\vert \mathbf{B}\right] $ increases as the
strength of the association encoded in the red edge increases and the one
encoded in the green edge decreases. In contrast, $\left\{ 1/{\pi _{a}\left( 
\mathbf{B};P\right) }-1\right\} $ is always greater than 0, and it is more
variable, and thus tends to have larger values, the stronger the marginal
association of $\mathbf{B}$ with $A.$ In the DAG in Figure \ref{fig:supp},
this association is represented by the blue edge.

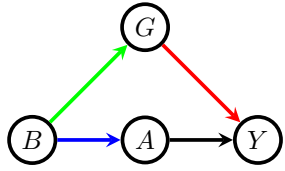
\begin{figure}[ht]
\begin{center}
\begin{tikzpicture}[>=stealth, node distance=1.5cm,
pre/.style={->,>=stealth,ultra thick,line width = 1.4pt}]
  \begin{scope}
    \tikzstyle{format} = [circle, inner sep=2.5pt,draw, thick, circle, line width=1.4pt, minimum size=6mm]
    \node[format] (B) {$B$};
        \node[format, right of=B] (A) {$A$};
       \node[format, right of=A] (Y) {$Y$};
    \node[format, above of=A] (G) {$G$};

                 \draw (A) edge[pre, black] (Y);
                 \draw (B) edge[pre, blue] (A);
                 \draw (B) edge[pre, green] (G);
                  \draw (G) edge[pre, red] (Y);
  \end{scope} 
  \end{tikzpicture}
\end{center}
\caption{A DAG illustrating Lemmas \ref{lemma:supplementation} and \ref{lemma:deletion}.}
\label{fig:supp}
\end{figure}

\begin{lemma}[Deletion of time independent overadjustment variables]
\label{lemma:deletion} Let $\mathcal{G}$ be a DAG with vertex set $\mathbf{V}%
,$ let $\mathbf{A}\subset \mathbf{V}$ and $Y\in \mathbf{V}\setminus \mathbf{A%
}$ with $\mathbf{A}$ a random vector taking values on a finite set. Suppose $%
\left( \mathbf{G}\cup \mathbf{B}\right) \mathbf{\subset V\backslash }\left\{ 
\mathbf{A},Y\right\} $ is a time independent adjustment set relative to $%
\left( \mathbf{A},Y\right) $ in $\mathcal{G}$ with $\mathbf{G}$ and $\mathbf{%
B}$\ disjoint and suppose 
\[
Y\text{ }\mathbf{\perp \!\!\!\perp }_{\mathcal{G}}\text{ }\mathbf{B}\text{ }%
\mathbf{|}\text{ }\mathbf{G},\mathbf{A}.
\]%
Then $\mathbf{G}$ is also an adjustment set relative to $\left( \mathbf{A}%
,Y\right) $ in $\mathcal{G}$ and for all $P\in \mathcal{M}\left( \mathcal{G}%
\right) $%
\begin{equation}
\sigma _{\mathbf{a},\mathbf{G},\mathbf{B}}^{2}\left( P\right) -\sigma _{%
\mathbf{a},\mathbf{G}}^{2}\left( P\right) =E_{P}\left[ \pi _{\mathbf{a}%
}\left( \mathbf{G};P\right) var_{P}\left( Y|\mathbf{A}=\mathbf{a},\mathbf{G}%
\right) var_{P}\left( \left. \frac{1}{\pi _{\mathbf{a}}\left( \mathbf{G,B}%
;P\right) }\right\vert \mathbf{A}=\mathbf{a},\mathbf{G}\right) \right] \geq
0.  \label{eq:deletion}
\end{equation}%
Furthermore,%
\[
\sigma _{\Delta ,\mathbf{G},\mathbf{B}}^{2}\left( P\right) -\sigma _{\Delta ,%
\mathbf{B}}^{2}\left( P\right) =\sum_{a\in \mathcal{A}}c_{\mathbf{a}%
}^{2}E_{P}\left\{ \pi _{\mathbf{a}}(\mathbf{G};P)var_{P}(Y\mid \mathbf{A}=%
\mathbf{a},\mathbf{G})var_{P}\left[ \frac{1}{\pi _{\mathbf{a}}(\mathbf{G,B}%
;P)}\mid \mathbf{A=a},\mathbf{G}\right] \right\} \mathbf{\geq }0.
\]%
In particular,
\begin{eqnarray*}
\sigma _{ATE,\mathbf{B}}^{2}\left( P\right) -\sigma _{ATE,\mathbf{G},\mathbf{%
B}}^{2}\left( P\right)  &=&E_{P}\left\{ \pi _{a=0}(\mathbf{G}%
;P)var_{P}(Y\mid A=0,\mathbf{G})var_{P}\left[ \frac{1}{\pi _{a=0}(\mathbf{G,B%
};P)}\mid A=0,\mathbf{G}\right] \right\}  \\
&&+E_{P}\left\{ \pi _{a=1}(\mathbf{G};P)var_{P}(Y\mid A=1,\mathbf{G})var_{P}%
\left[ \frac{1}{\pi _{a=0}(\mathbf{G,B};P)}\mid A=1,\mathbf{G}\right]
\right\}.
\end{eqnarray*}
\end{lemma}

Formula $\left( \ref{eq:deletion}\right) $ quantifies the increase in
variance incurred by keeping `overadjustment' variables that are marginally
associated with treatment but that do not help predict the outcome within
levels of treatment and the remaining adjusting variables. Notice that $%
var_{P}\left( Y|A=a,\mathbf{G}\right) $ is zero if $\mathbf{G}$ is a perfect
predictor of $Y$. In such extreme case, the formula indicates that it is
irrelevant whether one keeps the overadjustment variables $\mathbf{B}$. In
general, $\mathbf{B}$ is more harmful the weaker the association between $%
\mathbf{G}$ and $Y$ within levels of $A$ is. For example, in the causal
diagram in Figure \ref{fig:supp}, the penalty for keeping overadjustment variables
increases as the strength of the association represented in the red arrow
decreases. Furthermore, the quantity $var_{P}\left( \left. 1/{\pi _{a}\left( 
\mathbf{G,B};P\right) }\right\vert A=a,\mathbf{G}\right) $ indicates that $%
\mathbf{B}$ is also more harmful the weaker the association between $%
\mathbf{G}$ and $\mathbf{B}$ within levels of $A$, and the stronger the
association between $\mathbf{B}$ and $A$ within levels of $\mathbf{G.}$
For instance, in the causal diagram in Figure \ref{fig:supp}, $\mathbf{B}$ is also
more harmful the weaker the association represented by the green arrow is
and the stronger the association represented by the blue arrow is.

\begin{theorem}
\label{theo:compare_adj} Let $\mathcal{G}$ be a DAG with vertex set $\mathbf{%
V},$ let $\mathbf{A}\subset \mathbf{V}$ and $Y\in \mathbf{V}\setminus 
\mathbf{A}$ with $\mathbf{A}$ a random vector taking values on a finite set.
Suppose $\mathbf{G\mathbf{\subset V\backslash }\left\{ \mathbf{A},Y\right\} }
$ and $\mathbf{B\subset V\backslash }\left\{ \mathbf{A},Y\right\} $ are two
time independent adjustment sets relative to $\left( \mathbf{A},Y\right) $
in $\mathcal{G}$ such that%
\begin{equation}
\mathbf{A}\text{ }\mathbf{\perp \!\!\!\perp }_{\mathcal{G}}\text{ }\left[ 
\mathbf{G\backslash B}\right] \text{ }\mathbf{|}\text{ }\mathbf{B}
\label{eq:cond_indep_A}
\end{equation}%
\begin{equation}
Y\text{ }\mathbf{\perp \!\!\!\perp }_{\mathcal{G}}\text{ }\left[ \mathbf{%
B\backslash G}\right] \text{ }\mathbf{|}\text{ }\mathbf{G},\mathbf{A}.
\label{eq:cond_indep_Y}
\end{equation}%
Then, 
\begin{eqnarray*}
\sigma _{\mathbf{a},\mathbf{B}}^{2}\left( P\right) -\sigma _{\mathbf{a},%
\mathbf{G}}^{2}\left( P\right)  &=&E_{P}\left[ \left\{ \frac{1}{\pi _{%
\mathbf{a}}\left( \mathbf{B};P\right) }-1\right\} var_{P}\left[ \left. b_{%
\mathbf{a}}(\mathbf{G,B};P)\right\vert \mathbf{B}\right] \right]  \\
&&+E_{P}\left[ \pi _{\mathbf{a}}\left( \mathbf{G};P\right) var_{P}\left( Y|%
\mathbf{A}=\mathbf{a},\mathbf{G}\right) var_{P}\left( \left. \frac{1}{\pi
_{a}\left( \mathbf{G,B};P\right) }\right\vert \mathbf{A}=\mathbf{a},\mathbf{G%
}\right) \right] 
\end{eqnarray*}%
and
\begin{align*}
\sigma _{\Delta ,\mathbf{B}}^{2}\left( P\right) -\sigma _{\Delta ,\mathbf{G}%
}^{2}\left( P\right) &=\mathbf{c}^{T}var_{P}\left( \mathbf{Q}\right) \mathbf{%
c}\\
&+\sum_{a\in \mathcal{A}}c_{\mathbf{a}}^{2}E_{P}\left\{ \pi _{\mathbf{a}}(%
\mathbf{G};P)var_{P}(Y\mid \mathbf{A}=\mathbf{a},\mathbf{G})var_{P}\left[ 
\frac{1}{\pi _{\mathbf{a}}(\mathbf{G,B};P)}\mid \mathbf{A=a},\mathbf{G}%
\right] \right\} ,
\end{align*}
where $\mathbf{Q}$ is defined as in Lemma \ref{lemma:supplementation}. In
particular, 
\begin{eqnarray*}
\sigma _{ATE,\mathbf{B}}^{2}\left( P\right) -\sigma _{ATE,\mathbf{G}%
}^{2}\left( P\right)  &=&E_{P}\left[ \left\{ \frac{1}{\pi _{a=1}(\mathbf{B}%
;P)}-1\right\} var_{P}(b_{a=1}(\mathbf{G,B};P)\mid \mathbf{B})\right]  \\
&&+E_{P}\left[ \left\{ \frac{1}{\pi _{a=0}(\mathbf{B};P)}-1\right\}
var_{P}(b_{a=0}(\mathbf{G,B};P)\mid \mathbf{B})\right]  \\
&&-2E_{P}\left[ cov_{P}\left\{ b_{a=1}(\mathbf{G},\mathbf{B};P),b_{a=0}(%
\mathbf{G},\mathbf{B};P)|\mathbf{B}\right\} \right]  \\
&&+E_{P}\left\{ \pi _{a=0}(\mathbf{G};P)var_{P}(Y\mid A=0,\mathbf{G})var_{P}%
\left[ \frac{1}{\pi _{a=0}(\mathbf{G,B};P)}\mid A=0,\mathbf{G}\right]
\right\}  \\
&&+E_{P}\left\{ \pi _{a=1}(\mathbf{G};P)var_{P}(Y\mid A=1,\mathbf{G})var_{P}%
\left[ \frac{1}{\pi _{a=0}(\mathbf{G,B};P)}\mid A=1,\mathbf{G}\right].
\right\} 
\end{eqnarray*}
\end{theorem}
\medskip
\begin{proof}
Write $\sigma _{\mathbf{a},\mathbf{B}}^{2}-\sigma _{\mathbf{a},\mathbf{G}}^{2}=\sigma _{\mathbf{a},\mathbf{B}%
}^{2}-\sigma _{\mathbf{a},\mathbf{B\cup }\left( \mathbf{G\backslash B}\right)
}^{2}+\sigma _{\mathbf{a},\mathbf{G\cup }\left( \mathbf{B\backslash G}\right)
}^{2}-\sigma _{\mathbf{a},\mathbf{G}}^{2}$ and apply Lemmas \ref{lemma:supplementation}
and \ref{lemma:deletion}. The derivations for the expressions for $\sigma _{\Delta ,\mathbf{B}}^{2}\left( P\right) -\sigma _{\Delta ,\mathbf{G}%
}^{2}\left( P\right)$ and  $\sigma _{ATE,\mathbf{B}}^{2}\left( P\right) -\sigma _{ATE,\mathbf{G}}^{2}\left( P\right)$ are similar.
\end{proof}

The preceding theorem provides an intuitive decomposition for the gain in
efficiency of using adjustment set $\mathbf{G}$ as opposed to set $\mathbf{B}$. The difference $\sigma _{\mathbf{a},\mathbf{B}}^{2}-\sigma _{\mathbf{a},\mathbf{B\cup }\left( \mathbf{%
G\backslash B}\right) }^{2}$ represents the gain due to supplementing $%
\mathbf{B}$ with the precision component $\mathbf{G\backslash B}$ and $%
\sigma _{\mathbf{a},\mathbf{G\cup }\left( \mathbf{B\backslash G}\right) }^{2}-\sigma _{\mathbf{a},
\mathbf{G}}^{2}$ represents the gain from removing from $\mathbf{G\cup B}$
the overadjustment component $\mathbf{B\backslash G.}$

Theorem \ref{theo:compare_adj} is analogous to 
Theorem 3.10 from \cite{perkovic}, except that it is valid for arbitrary causal graphical models, instead of causal linear models, and for NP-$\mathbf{Z}$ estimators of treatment effects instead of ordinary least squares estimators. Likewise, Lemmas \ref{lemma:supplementation} and \ref{lemma:deletion} are analogous to Henckel et al's Corollaries 3.4 and 3.5.
Building on their Corollary 3.5, \cite{perkovic} provided a simple
procedure that, for a valid adjustment set, returns a
pruned valid adjustment set that yields OLS estimators of treatment effects
with smaller asymptotic variance. Because the validity of their pruning procedure relies only on the ordering of the asymptotic variances corresponding to two adjustment sets implied by the d-separation assumptions of their Corollary 3.5, and because the same ordering of the adjustment sets is valid for the variances of the corresponding NP-$\mathbf{Z}$ estimators, then we conclude that the pruning algorithm of \cite{perkovic} also returns a pruned valid adjustment set that yields an NP-$\mathbf{Z}$ estimator of treatment effect with smaller asymptotic variance.

As noted by \cite{perkovic}, not all pairs of valid time independent adjustment sets can be ordered using the d-separation conditions in Theorem \ref{theo:compare_adj}. In fact, there exist DAGs $%
\mathcal{G}$ with time independent  adjustment sets $\mathbf{Z}$ and $\widetilde{\mathbf{Z}}$
for which $\sigma _{a,\mathbf{Z}}^{2}\left( P\right) >\sigma _{a,\widetilde{%
\mathbf{Z}}}^{2}\left( P\right) $ for some $P\in \mathcal{M}\left( \mathcal{G%
}\right) $ and $\sigma _{a,\widetilde{\mathbf{Z}}}^{2}\left( P^{\prime
}\right) >\sigma _{a,\mathbf{Z}}^{2}\left( P^{\prime }\right) $ for some
other $P^{\prime }\in \mathcal{M}\left( \mathcal{G}\right) $ as the
following example illustrates.

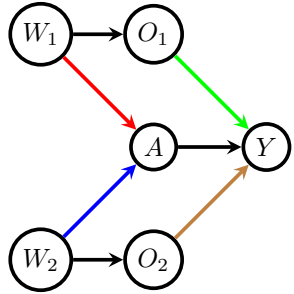
\begin{figure}[ht]
\begin{center}
\begin{tikzpicture}[>=stealth, node distance=1.5cm,
pre/.style={->,>=stealth,ultra thick,line width = 1.4pt}]
  \begin{scope}
    \tikzstyle{format} = [circle, inner sep=2.5pt,draw, thick, circle, line width=1.4pt, minimum size=6mm]
    \node[format] (W1) {$W_{1}$};
        \node[below of=W1] (0) {};
    \node[format, below of=0] (W2) {$W_{2}$};
    \node[format, right of=0] (A) {$A$};
        \node[format, right of=A] (Y) {$Y$};
            \node[format, above of=A] (O1) {$O_{1}$};
      \node[format, below of=A] (O2) {$O_{2}$};
                 \draw (A) edge[pre, black] (Y);
                 \draw (W2) edge[pre, blue] (A);
                 \draw (W2) edge[pre, black] (O2);
                    \draw (W1) edge[pre, black] (O1);
                  \draw (W1) edge[pre, red] (A);
                  \draw (O2) edge[pre, brown] (Y);
             \draw (O1) edge[pre, green] (Y);
  \end{scope} 
  \end{tikzpicture}
\end{center}
\caption{A DAG with two time independent adjustment sets, $\mathbf{Z=}\left\lbrace
O_{1},W_{2}\right\rbrace $ and $\widetilde{\mathbf{Z}}\mathbf{=}\left\lbrace
O_{2},W_{1}\right \rbrace$, that cannot be compared. Note that $\mathbf{Z}$ and $\widetilde{\mathbf{Z}}$ are minimal time independent adjustment sets.}
\label{fig:no_min_opt}
\end{figure}

\begin{example}
In the DAG in Figure \ref{fig:no_min_opt}, $\mathbf{Z=}\left\lbrace
O_{1},W_{2}\right\rbrace $ and $\widetilde{\mathbf{Z}}\mathbf{=}\left\lbrace
O_{2},W_{1}\right \rbrace$ are time independent adjustment sets relative to $\left( A,Y\right) .$
The adjustment set $\mathbf{Z}$ yields a smaller asymptotic variance than
the adjustment set $\widetilde{\mathbf{Z}}$ if the association encoded in
the green edge is stronger than that in the brown edge and the one encoded
in the blue edge is weaker than the one in the red edge. By symmetry, the
adjustment set $\widetilde{\mathbf{Z}}$ is more efficient than $\mathbf{Z}$
if the words stronger and weaker are interchanged in the preceding sentence. \cite{perkovic} illustrated the impossibility of ordering all time independent adjustment sets by the asymptotic variances of the corresponding adjusted linear estimators with a diagram different from the one in Figure \ref{fig:no_min_opt}, in which the treatment was unconfounded.
\end{example}

Following \cite{perkovic} we let $\cn(\A,Y,\mathcal{G})$ be the set of all
nodes that lie on a causal path between a node in $\mathbf{A}$ and $Y$ and are not equal to any node in $\mathbf{A}$ and we define the forbidden set as
\begin{equation*}
\forb(\A,Y,\mathcal{G})\equiv \de_{\mathcal{G}}\left( \cn(\A,Y,\mathcal{G}%
)\right) \cup \left\{ \A\right\}.
\end{equation*}%
Also, 
\begin{equation*}
\mathbf{O}(\A,Y,\mathcal{G})\equiv \pa_{\mathcal{G}}\left( \cn(\A,Y,\mathcal{G}%
) \right) \backslash \forb(\A,Y,\mathcal{G}).
\end{equation*}
\cite{perkovic} showed that, if a time independent adjustment set relative to $(\mathbf{A},Y)$ in $\mathcal{G}$ exists, then $\mathbf{O}(\A,Y,\mathcal{G})$ satisfies the graphical necessary and sufficient conditions of \cite{shpitser-adjustment}
to be an adjustment set.
Furthermore, Lemmas E.4 and E.5 of \cite{perkovic} showed that the conditions \eqref{eq:cond_indep_A} and \eqref{eq:cond_indep_Y} hold 
for $\mathbf{G}=\mathbf{O}(\mathbf{A},Y,%
\mathcal{G})$ and $\mathbf{B}$ any adjustment set. Consequently, we have
the following important corollary to Theorem \ref{theo:compare_adj}.

\begin{theorem}
\label{thep:opt_adj} 
 Let $\mathcal{G} $ be a DAG with vertex set $%
\mathbf{V},$ let $\mathbf{A}\subset \mathbf{V}$ and $Y \in \mathbf{V}\setminus \mathbf{A}$ with $\mathbf{A}$ a random vector taking values on a finite set. If a valid time independent adjustment set $%
\mathbf{Z}$ relative to $\left( \A,Y\right) $ in $\mathcal{G}$ exists then  $\mathbf{%
O=O}(\A,Y,\mathcal{G})$ is a time independent adjustment set and 
\begin{eqnarray*}
\sigma _{\mathbf{a},\mathbf{Z}}^{2}\left( P\right) -\sigma _{\mathbf{a},\mathbf{O}}^{2}\left(
P\right) &=&E_{P}\left[ \left\{ \frac{1}{\pi _{\mathbf{a}}\left( \mathbf{Z};P\right) }%
-1\right\} var_{P}\left[ \left. b_{\mathbf{a}}(\mathbf{O,Z};P)\right\vert \mathbf{Z}%
\right] \right] \\
&&+E_{P}\left[ \pi _{\mathbf{a}}\left( \mathbf{O};P\right) var_{P}\left( Y|A=a,%
\mathbf{O}\right) var_{P}\left( \left. \frac{1}{\pi _{\mathbf{a}}\left( \mathbf{O,Z}%
;P\right) }\right\vert \A=\mathbf{a},\mathbf{O}\right) \right]
\end{eqnarray*}%
and the corresponding formulae for $\Delta$ and ATE hold.
\end{theorem}

\begin{corollary}\label{corollary:point}
If $\mathbf{A}=A$ is point intervention then $\mathbf{O}(A,Y, \mathcal{G})$ is an optimal valid time independent adjustment set.
\end{corollary}
Corollary \ref{corollary:point} follows immediately from Theorem \ref{thep:opt_adj} and the fact that $\pa_{\mathcal{G}}(A)$ is always a valid time independent adjustment set relative to $(A,Y)$ in $\mathcal{G}$.

As an example, in the DAG in Figure \ref{fig:no_min_opt}, $\mathbf{O}(A,Y,\mathcal{G})=\left( O_{1},O_{2}\right) $ is the optimal adjustment
set. 

\cite{vanfinding} proposed an algorithm that, given a DAG $\mathcal{G}=(%
\mathbf{V,E}),$ computes $\mathbf{O}(\A,Y,\mathcal{G})$ with worst-case
complexity $\mathcal{O}(\left\vert \mathbf{V}\right\vert +\left\vert \mathbf{%
E}\right\vert )$ , where $\left\vert \mathbf{V}\right\vert $ is the number
of nodes in $\mathcal{G}$ and $\left\vert \mathbf{E}\right\vert $ is the
number of edges in $\mathcal{G}.$

For simplicity, from now on when no confusion can arise, we abbreviate $\mathbf{O}\equiv \mathbf{O}(\A,Y,\mathcal{G%
})$. 

An interesting question is whether one can find an optimal adjustment set
among the minimal adjustment sets. In the next theorem we show that such
adjustment exists for point interventions. Specifically, let $\mathbf{A}=A$ be a point intervention and let $\mathbf{O}_{\min }\subset \mathbf{O}$
be the subset of $\mathbf{O}$ with the smallest number of vertices such that 
\begin{equation*}
A\perp \!\!\!\perp _{\mathcal{G}}\left[ \mathbf{O\backslash O}_{\min }\right]
|\mathbf{O}_{\min }.
\end{equation*}%
The graphoid properties of d-separation \citep{lauritzen-book} imply that $%
\mathbf{O}_{\min }$ is unique. For completeness we provide a proof of this
result in Lemma \ref{lemma:unique_minimal} in the Appendix. Note that $%
\mathbf{O}_{\min }$ is empty when the empty set is a valid time independent adjustment set. The next theorem
establishes that $\mathbf{O}_{\min }$ is a minimal adjustment set relative
to $(A,Y)$ in $\mathcal{G}$. Furthermore, it establishes that it is optimal
among all minimal adjustment sets.
\begin{theorem}
\label{theo:optimal_adj_min} Let $\mathcal{G}$ be a DAG with vertex set $%
\mathbf{V},$ let $A$ and $Y$ be two distinct vertices in $\mathbf{V}$ with $A
$ corresponding to a point intervention taking values on a finite set.

\begin{enumerate}
\item $\mathbf{O}_{\min }$ as defined above is a minimal adjustment set
relative to $(A,Y)$ in $\mathcal{G}$.

\item If $\mathbf{Z}_{\min }$ is another minimal adjustment set relative to $%
(A,Y)$ in $\mathcal{G}$ then, 
\[
A\text{ }\mathbf{\perp \!\!\!\perp }_{\mathcal{G}}\text{ }\left[ \mathbf{O}%
_{\min }\mathbf{\backslash Z}_{\min }\right] \text{ }\mathbf{|}\text{ }%
\mathbf{Z}_{\min }\quad \text{and}\quad Y\text{ }\mathbf{\perp \!\!\!\perp }%
_{\mathcal{G}}\text{ }\left[ \mathbf{Z}_{\min }\mathbf{\backslash O}_{\min }%
\right] \text{ }\mathbf{|}\text{ }\mathbf{O}_{\min }\mathbf{,}A.
\]
Consequently, 
\begin{align*}
&\sigma _{a,\mathbf{Z}_{\min }}^{2}\left( P\right) -\sigma _{a,\mathbf{O}%
_{\min }}^{2}\left( P\right)  
\\&=E_{P}\left[ \left\{ \frac{1}{\pi _{a}\left( 
\mathbf{Z}_{\min };P\right) }-1\right\} var_{P}\left[ \left. b_{a}(\mathbf{O}%
_{\min }\mathbf{,Z}_{\min };P)\right\vert \mathbf{Z}_{\min }\right] \right] 
\\
&+E_{P}\left[ \pi _{a}\left( \mathbf{O}_{\min };P\right) var_{P}\left(
Y|A=a,\mathbf{O}_{\min }\right) var_{P}\left( \left. \frac{1}{\pi _{a}\left( 
\mathbf{O}_{\min }\mathbf{,Z}_{\min };P\right) }\right\vert A=a,\mathbf{O}%
_{\min }\right) \right].
\end{align*}%
and the corresponding formulae hold for $\Delta $ and ATE.

\item For any minimal adjustment set $\mathbf{Z}_{\min },$ $\mathbf{Z}%
_{\min }\cap \left[ \mathbf{O\backslash O}_{\min }\right] =\emptyset .$
\end{enumerate}
\end{theorem}

\begin{remark}\label{remark:Omin_valid}
Note that the conclusions one and two of Theorem \ref{theo:optimal_adj_min} are purely graphical. Therefore, invoking Theorem 3.1 of \cite{perkovic}, we conclude that $\mathbf{O}_{min}$ is also the minimal adjustment set that yields the adjusted OLS estimators of treatment effects with smallest variance among all adjusted OLS estimators of treatment effects that adjust for minimal adjustment sets.
\end{remark}

\subsection{Time dependent adjustment sets}

\label{sec:optimal_adj_time}

Suppose that $\mathbf{A}=\left( A_{0},\dots,A_{p}\right) $
is a joint intervention and for a given  time dependent adjustment set
$\mathbf{Z}=(\mathbf{Z}_{0},\dots,\mathbf{Z}_{p})$ in order to estimate a given contrast
$$
\Delta(P;\mathcal{G}) \equiv
\sum_{ \mathbf{a} \in \mathcal{A}}c_{\mathbf{a}}E[Y_{\mathbf{a}}]
$$
one estimates each
interventional mean 
\begin{align*}
E\left[ Y_{\mathbf{a}}\right] &= E_{P}\left\lbrace \left\{ E_{P}\left\{ E_{P}\left[ E_{P}\left[ Y|\mathbf{A}=\mathbf{a},%
\mathbf{Z}\right] |\overline{\mathbf{A}}_{p-1}=\overline{\mathbf{a}}_{p-1},\overline{%
\mathbf{Z}}_{p-1}\right] |\overline{\mathbf{A}}_{p-2}=\overline{\mathbf{a}}_{p-2},%
\overline{\mathbf{Z}}_{p-2}\right\} \cdots \mid {\mathbf{A}_{0}}=\mathbf{a}_{0}, \mathbf{Z}_{0}\right\}\right\rbrace
\\
& \equiv \chi_{\mathbf{a}}(P;\mathcal{G}),
\end{align*}
ignoring  the conditional
indepedencies encoded in the causal graphical model, and making at most
smoothness or complexity assumptions on the iterated conditional means
$$
b_{\overline{\mathbf{a}}_{j}}(\overline{\mathbf{Z}}_{j}; P)\equiv E_{P}\left\{ E_{P}\left\{ E_{P}\left[ E_{P}\left[ Y|\mathbf{A}=\mathbf{a},%
\mathbf{Z}\right] |\overline{\mathbf{A}}_{p-1}=\overline{\mathbf{a}}_{p-1},\overline{%
\mathbf{Z}}_{p-1}\right] |\overline{\mathbf{A}}_{p-2}=\overline{\mathbf{a}}_{p-2},%
\overline{\mathbf{Z}}_{p-2}\right\} \cdots \mid \overline{\mathbf{A}}_{j}=\overline{\mathbf{a}}_{j},%
\overline{\mathbf{Z}}_{j}\right\}.
$$
and/or on the conditional treatment probabilities.
$$
\pi_{a_{j}}(\overline{\mathbf{Z}}_{j};P)\equiv P\left( {A}_{j}={a}_{j}
\mid \overline{\mathbf{A}}_{j-1}=\overline{\mathbf{a}}_{j-1}, \overline{\mathbf{Z}}_{j}\right).
$$
It is well known \citep{robins1995} that estimators $\widehat{%
\chi }_{\mathbf{a},\mathbf{Z}}$ of $\chi _{\mathbf{a}}\left( P;\mathcal{G}%
\right) $ that are regular and asymptotically linear under a model $\mathcal{M}$ that imposes at most smoothness or complexity assumptions on $b_{\overline{\mathbf{a}}_{j}}$ and/or $\pi_{a_{j}}$ have a unique influence
function equal to 
\begin{align}
    \psi_{P,\mathbf{a}}(\mathbf{Z};P)\equiv \frac{I_{\mathbf{a}}(\mathbf{A})}{\lambda_{\overline{\mathbf{a}}_{p}}(\mathbf{Z};P)} \lbrace Y -  \chi_{\mathbf{a}}(P;\mathcal{G})\rbrace - \sum\limits_{k=0}^{p} g_{k}(\overline{\mathbf{A}}_{k}, \overline{\mathbf{Z}}_{k};P),
    \label{eq:inf_func_time_dep}
\end{align}
where
$$
g_{k}(\overline{\mathbf{A}}_{k}, \overline{\mathbf{Z}}_{k};P) = \frac{I_{\overline{\mathbf{a}}_{k-1}}(\overline{\mathbf{A}}_{k-1})}{\lambda_{\overline{\mathbf{a}}_{k-1}}(\overline{\mathbf{Z}}_{k-1};P)} \left\lbrace \frac{I_{a_k}(A_k)}{\pi_{a_k}(\overline{\mathbf{Z}}_{k};P)}-1 \right\rbrace \left\lbrace b_{\overline{\mathbf{a}}_{k}}(\overline{\mathbf{Z}}_{k};P)-\chi_{\mathbf{a}}(P;\mathcal{G})\right\rbrace
$$
with
$$
\lambda_{\overline{\mathbf{a}}_{k-1}}(\overline{\mathbf{Z}}_{k-1};P) \equiv \prod\limits_{j=0}^{k-1} \pi_{a_{j}}(\overline{\mathbf{Z}}_{j};P) \quad \text{and} \quad \frac{I_{\overline{\mathbf{a}}_{-1}}(\overline{\mathbf{A}}_{-1})}{\lambda_{\overline{\mathbf{a}}_{-1}}(\overline{\mathbf{Z}}_{-1};P)}\equiv 1.
$$

Consequently, regular and asymptotically linear 
estimators $%
\widehat{\Delta }_{\mathbf{Z}}\equiv \sum_{\mathbf{a}\in \mathcal{A}}c_{\mathbf{a}}\widehat{\chi }_{\mathbf{a},\mathbf{Z}}$ 
of $\Delta(P;\mathcal{G})$
have a unique influence function
equal to 
$$
\psi _{P,\Delta}\left( \mathbf{Z};\mathcal{G}\right) =\sum_{\mathbf{a}\in\mathcal{A}}c_{\mathbf{a}}\psi
_{P,\mathbf{a}}\left( \mathbf{Z};\mathcal{G}\right).
$$ 
Therefore
\begin{equation*}
\sqrt{n}\left\{ \widehat{\chi }_{\mathbf{a},\mathbf{Z}}-\chi _{\mathbf{a}}\left( P;\mathcal{G}%
\right) \right\} \cw  N\left( 0,\sigma _{\mathbf{a},\mathbf{Z}}^{2}\left(
P\right) \right)
\end{equation*}%
where $\sigma _{\mathbf{a},\mathbf{Z}}^{2}\left( P\right) \equiv var_{P}\left[ \psi
_{P,\mathbf{a}}\left( \mathbf{Z};\mathcal{G}\right) \right] .$
Likewise, $\sqrt{n}%
\left\{ \widehat{\Delta}_{\mathbf{Z}}-\Delta(P;\mathcal{G}) \right\rbrace \cw
N\left( 0,\sigma _{\Delta,\mathbf{Z}}^{2}\right) $ where 
\begin{equation*}
\sigma _{\Delta,\mathbf{Z}}^{2}\left( P\right) \equiv var_{P}\left[ \psi
_{P,\Delta}\left( \mathbf{Z};\mathcal{G}\right) \right] .
\end{equation*}

The following lemmas extend Lemmas \ref{lemma:supplementation} and \ref{lemma:deletion} from time independent adjustment sets to time dependent adjustment sets. Throughout we let
$$
b_{\overline{\mathbf{a}}_{-1}}\left( \overline{\mathbf{G}}%
_{-1},\overline{\mathbf{B}}_{-1};P\right)\equiv \chi_{\mathbf{a}}(P;\mathcal{G}).
$$
\begin{lemma}[Supplementation with time dependent precision variables]
\label{lemma:supplementation_dep} Let $\mathcal{G}$ be a DAG with vertex set $%
\mathbf{V},$ let $\mathbf{A=}\left( A_{0},\dots,A_{p}\right) $ be a
topologically ordered vertex set in $\mathbf{V}$ disjoint with $Y\in \mathbf{%
V.}$ Assume $A_{j},j=0,\dots,p,$ correspond to finite valued random variables.
Suppose 
\[
\mathbf{B=}\left( \mathbf{B}_{0},\dots,\mathbf{B}_{p}\right) \mathbf{\subset
V\backslash }\left\{ \mathbf{A},Y\right\} 
\]%
is a time dependent adjustment set relative to $\left( \mathbf{A},Y\right) $
in $\mathcal{G}$ and suppose $\mathbf{G=}\left( \mathbf{G}_{0},\dots,\mathbf{G}%
_{p}\right) $ is a set disjoint with $\mathbf{B}$ that satisfies 
\begin{equation}
A_{j}\text{ }\mathbf{\perp \!\!\!\perp }_{\mathcal{G}}\text{ }\overline{%
\mathbf{G}}_{j}\text{ }\mathbf{|}\text{ }\overline{\mathbf{B}}_{j},\overline{%
\mathbf{A}}_{j-1}\text{ for }j=0,\dots,p,  \label{eq:ind_good}
\end{equation}%
where $\overline{\mathbf{A}}_{-1}=\emptyset$.
Then $\left( \mathbf{G,B}\right) =\left[ \left( \mathbf{G}_{0},\mathbf{B}%
_{0}\right) ,\left( \mathbf{G}_{1},\mathbf{B}_{1}\right) ,\dots,\left( \mathbf{%
G}_{p},\mathbf{B}_{p}\right) \right] $ is also an time dependent adjustment
set relative to $\left( \mathbf{A},Y\right) $ in $\mathcal{G}$ and for all $%
P\in \mathcal{M}\left( \mathcal{G}\right) $
\begin{align*}
\sigma _{\mathbf{a},\mathbf{B}}^{2}\left( P\right) -\sigma _{\mathbf{a},%
\mathbf{G},\mathbf{B}}^{2}\left( P\right) &=\sum_{k=0}^{p}E_{P}\left[ \frac{%
I_{\overline{\mathbf{a}}_{k-1}}\left( \overline{\mathbf{A}}_{k-1}\right) }{%
\lambda _{\overline{\mathbf{a}}_{k-1}}\left( \overline{\mathbf{B}}%
_{k-1};P\right) ^{2}}\left\{ \frac{1}{\pi _{a_{k}}\left( \overline{\mathbf{B}%
}_{k};P\right) }-1\right\} var_{P}\left[ b_{\overline{\mathbf{a}}_{k}}\left( 
\overline{\mathbf{G}}_{k},\overline{\mathbf{B}}_{k};P\right) |\overline{%
\mathbf{A}}_{k-1}=\overline{\mathbf{a}}_{k-1},\overline{\mathbf{B}}_{k}%
\right] \right] \\
&\geq 0.
\end{align*}
Furthermore, 
\[
\sigma _{\mathbf{\Delta },\mathbf{B}}^{2}\left( P\right) -\sigma _{\mathbf{%
\Delta },\mathbf{G},\mathbf{B}}^{2}\left( P\right) =\sum_{k=0}^{p}var_{P}%
\left[ t_{k}\left( \overline{\mathbf{A}}_{k},\overline{\mathbf{G}}_{k},%
\overline{\mathbf{B}}_{k},P\right) \right] \geq 0,
\]%
where%
\[
t_{k}\left( \overline{\mathbf{A}}_{k},\overline{\mathbf{G}}_{k},\overline{%
\mathbf{B}}_{k},P\right) \equiv \sum_{\mathbf{a\in }\mathcal{A}}c_{\mathbf{a}%
}\frac{I_{\overline{\mathbf{a}}_{k-1}}\left( \overline{\mathbf{A}}%
_{k-1}\right) }{\lambda _{\overline{\mathbf{a}}_{k-1}}\left( \overline{%
\mathbf{B}}_{k-1};P\right) }\left\{ \frac{I_{a_{k}}\left( A_{k}\right) }{\pi
_{a_{k}}\left( \overline{\mathbf{B}}_{k};P\right) }-1\right\} \left\{ b_{%
\overline{\mathbf{a}}_{k}}\left( \overline{\mathbf{G}}_{k},\overline{\mathbf{%
B}}_{k};P\right) -b_{\overline{\mathbf{a}}_{k-1}}\left( \overline{\mathbf{G}}%
_{k-1},\overline{\mathbf{B}}_{k-1};P\right) \right\}.
\]
\end{lemma}

\begin{lemma}[Deletion of time dependent overadjustment variables]
\label{lemma:deletion_dep} Let $\mathcal{G}$ be a DAG with vertex set $%
\mathbf{V},$ let $\mathbf{A=}\left( A_{0},\dots,A_{p}\right) $ be a
topologically ordered vertex set in $\mathbf{V}$ disjoint with $Y\in \mathbf{%
V.}$ Assume $A_{j},j=0,\dots,p,$ correspond to finite valued random variables.
Suppose $\left( \mathbf{G,B}\right) \equiv \left[ \left( \mathbf{G}_{0},%
\mathbf{B}_{0}\right) ,\left( \mathbf{G}_{1},\mathbf{B}_{1}\right)
,\dots,\left( \mathbf{G}_{p},\mathbf{B}_{p}\right) \right] \mathbf{\subset
V\backslash }\left\{ \mathbf{A},Y\right\} $ is a time dependent adjustment set
relative to $\left( \mathbf{A},Y\right) $ in $\mathcal{G}$ with $\mathbf{G}$
and $\mathbf{B}$\ disjoint and suppose that%
\begin{equation}
Y\text{ }\mathbf{\perp \!\!\!\perp }_{\mathcal{G}}\text{ }\mathbf{B}\text{ }%
\mathbf{|}\text{ }\mathbf{G,A}.  \label{eq:ind_Y}
\end{equation}%
and%
\begin{equation}
\mathbf{G}_{j}\mathbf{\perp \!\!\!\perp }_{\mathcal{G}}\text{ }\overline{%
\mathbf{B}}_{j-1}\text{ }\mathbf{|}\text{ }\overline{\mathbf{G}}_{j-1},%
\overline{\mathbf{A}}_{j-1}\text{ for }j=1,\dots,p.  \label{eq:ind_G}
\end{equation}%
Then $\mathbf{G=}\left( \mathbf{G}_{0},\mathbf{G}_{1},\dots,\mathbf{G}%
_{p}\right) $ is also a time dependent adjustment set relative to $\left( 
\mathbf{A},Y\right) $ in $\mathcal{G}$ and for all $P\in \mathcal{M}\left( 
\mathcal{G}\right) $%
\begin{align*}
&\sigma _{a,\mathbf{G},\mathbf{B}}^{2}\left( P\right) -\sigma _{a,\mathbf{G}%
}^{2}\left( P\right)  
\\&=E_{P}\left[ var_{P}\left[ \left. \frac{I_{\mathbf{a}%
}\left( \mathbf{A}\right) }{\lambda _{\overline{\mathbf{a}}_{p}}\left( 
\mathbf{G,B;}P\right) }\left\{ Y-b_{\overline{\mathbf{a}}_{p}}\left( \mathbf{%
G,B;}P\right) \right\} \right\vert Y,\overline{\mathbf{G}}_{p},\overline{%
\mathbf{A}}_{p}\right] \right]  \\
&+\sum_{k=0}^{p}E_{P}\left[ var_{P}\left[ \left. \frac{I_{\overline{\mathbf{%
a}}_{k-1}}\left( \overline{\mathbf{A}}_{k-1}\right) }{\lambda _{\overline{%
\mathbf{a}}_{k-1}}\left( \overline{\mathbf{G}}_{k-1},\overline{\mathbf{B}}%
_{k-1}\mathbf{;}P\right) }\left\{ b_{\overline{\mathbf{a}}_{k}}\left( 
\overline{\mathbf{G}}_{k},\overline{\mathbf{B}}_{k}\mathbf{;}P\right) -b_{%
\overline{\mathbf{a}}_{k-1}}\left( \overline{\mathbf{G}}_{k-1},\overline{%
\mathbf{B}}_{k-1}\mathbf{;}P\right) \right\} \right\vert \overline{\mathbf{G}%
}_{k},\overline{\mathbf{A}}_{k-1}\right] \right] \geq 0.
\end{align*}
Furthermore, 
\begin{align*}
&\sigma _{\Delta ,\mathbf{G},\mathbf{B}}^{2}\left( P\right) -\sigma _{\Delta ,%
\mathbf{G}}^{2}\left( P\right)  \\&=E_{P}\left[ var_{P}\left[ \sum_{a\in 
\mathcal{A}}c_{\mathbf{a}}\left. \frac{I_{\mathbf{a}}\left( \mathbf{A}%
\right) }{\lambda _{\overline{\mathbf{a}}_{p}}\left( \mathbf{G,B;}P\right) }%
\left\{ Y-b_{\overline{\mathbf{a}}_{p}}\left( \mathbf{G,B;}P\right) \right\}
\right\vert Y,\overline{\mathbf{G}}_{p},\overline{\mathbf{A}}_{p}\right] %
\right]  \\
&+\sum_{k=0}^{p}E_{P}\left[ var_{P}\left[ \sum_{a\in \mathcal{A}}c_{\mathbf{%
a}}\left. \frac{I_{\overline{\mathbf{a}}_{k-1}}\left( \overline{\mathbf{A}}%
_{k-1}\right) }{\lambda _{\overline{\mathbf{a}}_{k-1}}\left( \overline{%
\mathbf{G}}_{k-1},\overline{\mathbf{B}}_{k-1}\mathbf{;}P\right) }\left\{ b_{%
\overline{\mathbf{a}}_{k}}\left( \overline{\mathbf{G}}_{k},\overline{\mathbf{%
B}}_{k}\mathbf{;}P\right) -b_{\overline{\mathbf{a}}_{k-1}}\left( \overline{%
\mathbf{G}}_{k-1},\overline{\mathbf{B}}_{k-1}\mathbf{;}P\right) \right\}
\right\vert \overline{\mathbf{G}}_{k},\overline{\mathbf{A}}_{k-1}\right] %
\right]\geq 0.
\end{align*}
\end{lemma}

We now have the following corollary to Lemmas \ref{lemma:supplementation_dep} and \ref{lemma:deletion_dep}.

\begin{theorem}
\label{theo:compare_adj_time} Let $\mathcal{G}$ be a DAG with vertex set $%
\mathbf{V},$ let $\mathbf{A=}\left( A_{0},\dots ,A_{p}\right) $ be a
topologically ordered vertex set in $\mathbf{V}$ disjoint with $Y\in \mathbf{%
V.}$ Assume $A_{j},j=0,\dots ,p,$ correspond to finite valued random
variables. Suppose 
\[
\mathbf{B=}\left( \mathbf{B}_{0},\dots ,\mathbf{B}_{p}\right) \mathbf{%
\subset V\backslash }\left\{ \mathbf{A},Y\right\} 
\]%
and 
\[
\mathbf{G=}\left( \mathbf{G}_{0},\dots ,\mathbf{G}_{p}\right) \mathbf{%
\subset V\backslash }\left\{ \mathbf{A},Y\right\} 
\]%
are two time dependent adjustment sets relative to $\left( \mathbf{A}%
,Y\right) $ in $\mathcal{G}$. Suppose that 
\[
A_{j}\text{ }\mathbf{\perp \!\!\!\perp }_{\mathcal{G}}\text{ }\left[ 
\overline{\mathbf{G}}_{j}\mathbf{\backslash }\overline{\mathbf{B}}_{j}\right]
\text{ }\mathbf{|}\text{ }\overline{\mathbf{B}}_{j},\overline{\mathbf{A}}%
_{j-1}\text{ for }j=0,\dots ,p
\]%
\[
Y\text{ }\mathbf{\perp \!\!\!\perp }_{\mathcal{G}}\text{ }\left[ \mathbf{%
B\backslash G}\right] \text{ }\mathbf{|}\text{ }\mathbf{G,A}
\]%
and%
\[
\mathbf{G}_{j}\mathbf{\perp \!\!\!\perp }_{\mathcal{G}}\text{ }\left[ 
\overline{\mathbf{B}}_{j-1}\mathbf{\backslash }\overline{\mathbf{G}}_{j-1}%
\right] \text{ }\mathbf{|}\text{ }\overline{\mathbf{G}}_{j-1},\overline{%
\mathbf{A}}_{j-1}\text{ for }j=1,\dots ,p.
\]%
Then, 
\[
\sigma _{\mathbf{a,B}}^{2}\left( P\right) -\sigma _{\mathbf{a,G}}^{2}\left(
P\right) \geq 0\text{ and } \sigma _{{\Delta ,B}}^{2}\left( P\right)
-\sigma _{{\Delta ,G}}^{2}\left( P\right) \geq 0.
\]%
Specifically 
\begin{align*}
& \sigma _{\mathbf{a,B}}^{2}\left( P\right) -\sigma _{\mathbf{a,G}%
}^{2}\left( P\right) = \\
& \sum_{k=0}^{p}E_{P}\left[ \frac{I_{\overline{\mathbf{a}}_{k-1}}\left( 
\overline{\mathbf{A}}_{k-1}\right) }{\lambda _{\overline{\mathbf{a}}%
_{k-1}}\left( \overline{\mathbf{B}}_{k-1};P\right) ^{2}}\left\{ \frac{1}{\pi
_{a_{k}}\left( \overline{\mathbf{B}}_{k};P\right) }-1\right\} var_{P}\left[
b_{\overline{\mathbf{a}}_{k}}\left( \overline{\mathbf{G}}_{k},\overline{%
\mathbf{B}}_{k};P\right) |\overline{\mathbf{A}}_{k-1}=\overline{\mathbf{a}}%
_{k-1},\overline{\mathbf{B}}_{k}\right] \right] + \\
& E_{P}\left[ var_{P}\left[ \left. \frac{I_{\mathbf{a}}\left( \mathbf{A}%
\right) }{\lambda _{\overline{\mathbf{a}}_{p}}\left( \mathbf{G,B;}P\right) }%
\left\{ Y-b_{\overline{\mathbf{a}}_{p}}\left( \mathbf{G,B;}P\right) \right\}
\right\vert Y,\overline{\mathbf{G}}_{p},\overline{\mathbf{A}}_{p}\right] %
\right] + \\
& \sum_{k=0}^{p}E_{P}\left[ var_{P}\left[ \left. \frac{I_{\overline{\mathbf{a%
}}_{k-1}}\left( \overline{\mathbf{A}}_{k-1}\right) }{\lambda _{\overline{%
\mathbf{a}}_{k-1}}\left( \overline{\mathbf{G}}_{k-1},\overline{\mathbf{B}}%
_{k-1}\mathbf{;}P\right) }\left\{ b_{\overline{\mathbf{a}}_{k}}\left( 
\overline{\mathbf{G}}_{k},\overline{\mathbf{B}}_{k}\mathbf{;}P\right) -b_{%
\overline{\mathbf{a}}_{k-1}}\left( \overline{\mathbf{G}}_{k-1},\overline{%
\mathbf{B}}_{k-1}\mathbf{;}P\right) \right\} \right\vert \overline{\mathbf{G}%
}_{k},\overline{\mathbf{A}}_{k-1}\right] \right]  \\
& \geq 0.
\end{align*}%
and%
\begin{align*}
& \left. \text{ }\sigma _{{\Delta ,B}}^{2}\left( P\right) -\sigma _{%
{\Delta ,G}}^{2}\left( P\right) =\right.  \\
& \sum_{k=0}^{p}var_{P}\left[ t_{k}\left( \overline{\mathbf{A}}_{k},%
\overline{\mathbf{G}}_{k},\overline{\mathbf{B}}_{k},P\right) \right] +E_{P}%
\left[ var_{P}\left[ \sum_{a\in \mathcal{A}}c_{\mathbf{a}}\left. \frac{I_{%
\mathbf{a}}\left( \mathbf{A}\right) }{\lambda _{\overline{\mathbf{a}}%
_{p}}\left( \mathbf{G,B;}P\right) }\left\{ Y-b_{\overline{\mathbf{a}}%
_{p}}\left( \mathbf{G,B;}P\right) \right\} \right\vert Y,\overline{\mathbf{G}%
}_{p},\overline{\mathbf{A}}_{p}\right] \right] + \\
& \sum_{k=0}^{p}E_{P}\left[ var_{P}\left[ \sum_{a\in \mathcal{A}}c_{\mathbf{a%
}}\left. \frac{I_{\overline{\mathbf{a}}_{k-1}}\left( \overline{\mathbf{A}}%
_{k-1}\right) }{\lambda _{\overline{\mathbf{a}}_{k-1}}\left( \overline{%
\mathbf{G}}_{k-1},\overline{\mathbf{B}}_{k-1}\mathbf{;}P\right) }\left\{ b_{%
\overline{\mathbf{a}}_{k}}\left( \overline{\mathbf{G}}_{k},\overline{\mathbf{%
B}}_{k}\mathbf{;}P\right) -b_{\overline{\mathbf{a}}_{k-1}}\left( \overline{%
\mathbf{G}}_{k-1},\overline{\mathbf{B}}_{k-1}\mathbf{;}P\right) \right\}
\right\vert \overline{\mathbf{G}}_{k},\overline{\mathbf{A}}_{k-1}\right] %
\right]  \\
& \geq 0,
\end{align*}%
where $t_{k}\left( \overline{\mathbf{A}}_{k},\overline{\mathbf{G}}_{k},%
\overline{\mathbf{B}}_{k},P\right) $ is defined as in Lemma \ref{lemma:supplementation_dep}.
\end{theorem}

As indicated earlier, optimal adjustment sets always exist for point interventions. In contrast, for joint interventions, even though time dependent adjustment sets always exist, there exist DAGs with  no optimal time dependent adjustment set. Moreover, even when an optimal time independent adjustment set exists, this set is not necessarily uniformly optimal among all adjustment sets. The following example illustrates these two points, as well as the application of Theorem \ref{theo:compare_adj_time}.

\begin{example}
\label{ex:time} For the DAG in Figure \ref{fig:time_dep}, Table \ref{tab:adj}
lists all valid time dependent adjustments sets relative to $\left( \mathbf{A,}%
Y\right) $ for $\mathbf{A}=\left( A_{0},A_{1}\right)$. These can be found applying the criteria in \cite{prob-eval}. The last column of Table \ref{tab:adj} indicates, for every adjustment set, another dominating adjustment set, in the sense that the dominating one results in an NP-$\mathbf{Z}$ estimator that has smaller asymptotic variance for all $P\in\mathcal{M(G)}$. These dominating adjustment sets are found applying Theorem \ref{theo:compare_adj_time}. We note however that  $\mathbf{Z}^{\ast }$ in row 1 is superior to $\mathbf{Z}^{\ast\ast }$ in row 8 for some $P\in\mathcal{M(G)}$ but $\mathbf{Z}^{\ast\ast }$ is superior to $\mathbf{Z}^{\ast}$ for another $P^{\prime}\in\mathcal{M(G)}$. Intuitively, when the association encoded in the red arrow is weak but the associations encoded in the blue arrows are strong, then  $\mathbf{Z}^{\ast\ast }$ in row 8 is preferable to  $\mathbf{Z}^{\ast}$ in row 1. In contrast, when the association encoded in the red arrow is strong but the associations encoded in the blue arrows are weak, then  $\mathbf{Z}^{\ast}$ if preferable to  $\mathbf{Z}^{\ast\ast}$.
For instance, when all variables are binary there exists $P\in\mathcal{M(G)}$ such
that $\sigma _{\mathbf{a,Z}^{\ast }}^{2}\left( P\right) /\sigma _{\mathbf{a,Z%
}^{\ast \ast }}^{2}\left( P\right) =0.675$ and another law $P^{\prime }\in\mathcal{M(G)}$
such that $\sigma _{\mathbf{a,Z}^{\ast }}^{2}\left( P^{\prime}\right) /\sigma _{%
\mathbf{a,Z}^{\ast \ast }}^{2}\left( P^{\prime}\right)=1.08$ for $\mathbf{a}=(1,1)$. See the \texttt{R} scripts available at \url{https://github.com/esmucler/optimal_adjustment}.
We note also that the adjustment set in row 11, namely $\mathbf{Z}_0=\lbrace H \rbrace$ and  $\mathbf{Z}_1=\emptyset$ is the unique time independent adjustment set, and hence optimal among time independent adjustment sets. Nevertheless, it is dominated by the time dependent adjustment set in row 8, thus proving that optimal time independent adjustment sets need not be optimal in the class of all adjustment sets.
\end{example}

\begin{figure}[ht]
\begin{center}
\begin{tikzpicture}[>=stealth, node distance=1.5cm,
pre/.style={->,>=stealth,ultra thick,line width = 1.4pt}]
  \begin{scope}
    \tikzstyle{format} = [circle, inner sep=2.5pt,draw, thick, circle, line width=1.4pt, minimum size=6mm]
    \node[format] (A0) {$A_{0}$};
        \node[format, right of=A0] (R) {$R$};
          \node[right of=R] (0) {};
            \node[format, right of=0] (A1) {$A_1$};
         \node[format, right of=A1] (Y) {$Y$};
          \node[format, above of=0] (H) {$H$};
         \node[format, below of=0] (Q) {$Q$};
                 \draw (A0) edge[pre, black] (R);
               \draw (H) edge[pre, blue] (R);
             \draw (H) edge[pre, red] (A1);
             \draw (A1) edge[pre, black] (Y);
           \draw (Q) edge[pre, black] (Y);
            \draw (R) edge[pre, blue] (Q);
  \end{scope} 
  \end{tikzpicture}
\end{center}
\caption{An example in which no optimal  time dependent adjustment set exists.}
\label{fig:time_dep}
\end{figure}
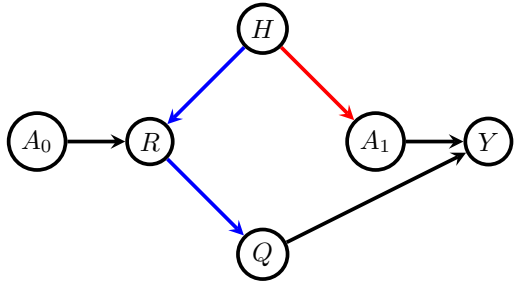

\begin{table}[ht!]
\begin{center}
\begin{tabular}{cccc}
Adjustment set &$\mathbf{Z}_{0}$ & $\mathbf{Z}_{1}$ & Dominating adjustment set\\ \hline
1&$\emptyset$ & $Q$ & - \\ 
2&$\emptyset$ & $R$ & 1\\ 
3&$\emptyset$ & $H$ & 1\\ 
4&$\emptyset$ & $\lbrace Q,R \rbrace$& 1 \\ 
5&$\emptyset$ & $\lbrace Q,H \rbrace$& 1 \\ 
6&$\emptyset$ & $\lbrace R,H \rbrace$& 1 \\ 
7&$\emptyset$ & $\lbrace Q,R,H \rbrace$ & 1 \\ 
8&$H$ & $Q$ & -\\ 
9&$H$ & $R$ & 8 \\ 
10&$H$ & $\lbrace R,Q \rbrace$ & 8 \\ 
11&$H$ & $\emptyset$ & 8 \\ 
& 
\end{tabular}%
\end{center}
\caption{List of all possible time dependent adjustment sets for the DAG in
Figure \protect\ref{fig:time_dep}.}
\label{tab:adj}
\end{table}

\FloatBarrier

An interesting open question is to characterize the class of DAGs for which
there exists an optimal time dependent adjustment set, and for DAGs like the
one in Example \ref{ex:time} for which no optimal adjustment set exists, to
characterize the class of adjustment sets such that those not in the class
are inferior to at least one member of the class.

\subsection{Nonexistence of uniformly optimal covariate adjustment sets in non-parametric causal graphical models with latent variables}
\label{sec:latent}

Consider now the situation in which some vertices of the DAG are not
observable, but some adjustment sets are observable. A
natural question is whether one can find an optimal adjustment set among the
observable ones. Without restricting the topology of the DAG, the answer is
negative as the following example illustrates. For linear causal graphical models and treatments effects estimated by OLS, \cite{perkovic} showed that it is possible that no uniformly optimal adjustment set exists among observable adjustment sets. In the following example we show the same negative result holds for non-linear causal graphical models and NP-$\mathbf{O}$ estimators.

An interesting open problem
is the characterization of settings in which, $\mathbf{O}(\mathbf{A},Y,\mathcal{G})$
is not observed but an optimal observable adjustment set exists.

\begin{example}
Suppose that in the DAG in Figure \ref{fig:no_opt_lat}, $U$ is the only
unobserved variable. Then, $\mathbf{Z}^{\ast} =\emptyset $, $%
{\mathbf{Z}^{\ast\ast}}\mathbf{=}\left\{ Z_{1},Z_{2}\right\} $  and  $\mathbf{Z}^{\ast\ast\ast}=\left\{ Z_{1}\right\} $ are all observable 
adjustment sets for $(A,Y)$ relative to the DAG. Using Lemma \ref{lemma:deletion}, it is easy to show that $\mathbf{Z}^{\ast}$ is uniformly better than $\mathbf{Z}^{\ast\ast\ast}$.
However, $\mathbf{Z}^{\ast}$ is better than $\mathbf{Z}^{\ast\ast}$ if the associations encoded in the blue edges are strong and the
associations encoded in the red edges are weak, but $\mathbf{Z}^{\ast\ast}$ is better than $\mathbf{Z}^{\ast}$ if the blue edges are weak and red ones are strong. In fact, when all variables are binary there exists $P\in\mathcal{M(G)}$ such that $\sigma^{2}_{1, \mathbf{Z}^{\ast}}(P)/\sigma^{2}_{1, \mathbf{Z}^{\ast\ast}}(P)= 0.04$
and another $P^{\prime}\in\mathcal{M(G)}$ such that $\sigma^{2}_{1, \mathbf{Z}^{\ast}}(P^{\prime})/\sigma^{2}_{1, \mathbf{Z}^{\ast\ast}}(P^{\prime})= 1.44$. See the \texttt{R} scripts available at \url{https://github.com/esmucler/optimal_adjustment}.
\end{example}

\begin{figure}[ht]
\begin{center}
\begin{tikzpicture}[>=stealth, node distance=1.5cm,
pre/.style={->,>=stealth,ultra thick,line width = 1.4pt}]
  \begin{scope}
    \tikzstyle{format} = [circle, inner sep=2.5pt,draw, thick, circle, line width=1.4pt, minimum size=6mm]
    \node[format] (A) {$A$};
        \node[right of=A] (O) {};
        \node[format, right of=O] (Y) {$Y$};
               \node[format, above of=A] (Z1) {$Z_1$};
             \node[format ,above of=Y] (U) {$U$};
            \node[format, right of=Z1] (Z2) {$Z_2$};
                 \draw (Z1) edge[pre, blue] (A);
                \draw (Z1) edge[pre, blue] (Z2);
                        \draw (U) edge[pre, red] (Y);
               \draw (U) edge[pre, red] (Z2);
                              \draw (A) edge[pre, black] (Y);
  \end{scope} 
  \end{tikzpicture}
\end{center}
\caption{An example with a latent variable $U$ and observable covariate adjustment sets but with no optimal adjustment set.}
\label{fig:no_opt_lat}
\end{figure}
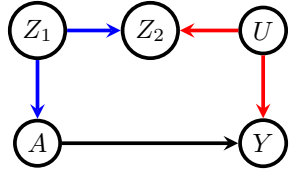

\section{Estimation of point intervention causal effects exploiting the assumptions of the Bayesian Network}
\label{sec:efficient_est}

For a point intervention $\mathbf{A}=A$, NP-$\mathbf{O}$ estimators
of the individual interventional means and their contrasts, even though efficient among NP-$\mathbf{Z}$ estimators, ignore the conditional
independence assumptions encoded in the causal graphical model about the data
generating law $P$. These assumptions may carry information about the
parameters of interest. For instance, consider the Bayesian Network
represented by the DAG $\mathcal{G}$ of Figure \ref{fig:adj}. Under model $%
\mathcal{M}\left( \mathcal{G}\right) $, the components $O_{1}$ and $O_{2}$
of the adjustment set $\mathbf{O}=\left\lbrace  O_{1},O_{2}\right\rbrace $ are marginally
independent. This independence carries information about $\chi _{a}\left( P;%
\mathcal{G}\right) =E_{P}\left[ E_{p}\left[ Y|A=a,O_{1},O_{2}\right] \right] 
$ because the joint distribution of $O_{1}, O_{2}$ is not ancillary for $%
\chi _{a}\left( P;\mathcal{G}\right) .$ Specifically, 
\begin{eqnarray*}
E_{P}\left[ E_{p}\left[ Y|A=a,O_{1},O_{2}\right] \right] &=&\int \int \int
yp\left( y|a,o_{1},o_{2}\right) p\left( o_{1},o_{2}\right) dydo_{1}do_{2} \\
&=&\int \int \int y p\left( y|a,o_{1},o_{2}\right) p\left( o_{1}\right)
p\left( o_{2}\right) dy do_{1}do_{2}
\end{eqnarray*}%
and the last equality is true only under $\mathcal{M}\left( \mathcal{G}%
\right) .$ 

\begin{figure}[ht]
\begin{center}
\begin{tikzpicture}[>=stealth, node distance=1.5cm,
pre/.style={->,>=stealth,ultra thick,line width = 1.4pt}]
  \begin{scope}
    \tikzstyle{format} = [circle, inner sep=2.5pt,draw, thick, circle, line width=1.4pt, minimum size=6mm]
    \node[format] (1) {$O_{1}$};
        \node[below of=1] (0) {};
    \node[format, below of=0] (2) {$O_{2}$};
    \node[format, right of=0] (3) {$A$};
        \node[format, right of=3] (4) {$Y$};
                 \draw (1) edge[pre, black] (3);
                 \draw (2) edge[pre, black] (3);
                 \draw (3) edge[pre, black] (4);
                  \draw (1) edge[pre, black] (4);
                 \draw (2) edge[pre, black] (4);
  \end{scope} 
  \end{tikzpicture}
\end{center}
\caption{A DAG where the NP-$\mathbf{O}$ estimator is inefficient.}
\label{fig:adj}
\end{figure}
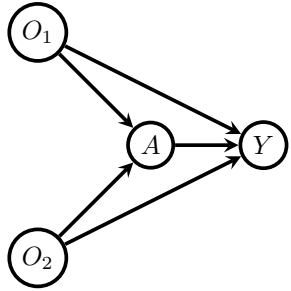

Applying Algorithm \ref{algo:main} of Section \ref{sec:semiparam_effi}, it is easy show that the semiparametric Cramer-Rao bound, defined in Section \ref{sec:semiparam}, under the Bayesian
Network of Figure \ref{fig:adj} is equal to the variance of the random variable
\begin{equation*}
\chi _{P,a,eff}^{1}\left( A,Y,\mathbf{O};\mathcal{G}\right) =\psi
_{P,a}\left( \mathbf{Z};\mathcal{G}\right) -\Delta _{P}\left( \mathbf{O}%
\right) \text{ }
\end{equation*}%
where $\mathbf{O}=\left( O_{1},O_{2}\right)$, $\psi _{P,a}\left( \mathbf{O};%
\mathcal{G}\right) $ is the influence function of the NP-$\mathbf{O}$ estimator and 
\begin{equation*}
\Delta _{P}\left( \mathbf{O}\right) \equiv b_{a}\left( \mathbf{O};P\right)
-E_{P}\left[ b_{a}\left( \mathbf{O};P\right) |O_{1}\right] -E_{P}\left[
b_{a}\left( \mathbf{O};P\right) |O_{2}\right] +E_{P}\left[ b_{a}\left( \mathbf{O}%
;P\right) \right]
\end{equation*}%
with \medskip $b_{a}\left( \mathbf{O};P\right) \equiv E_{P}\left[ Y|A=a,%
\mathbf{O}\right] .$ Furthermore, we show in Lemma \ref{lemma:ARE} in the Appendix that if $P_{\alpha}\in\mathcal{M(G)}$ is such that the following hold
\begin{enumerate}
    \item $b_{a}\left( \mathbf{O};P_{\alpha}\right)
=O_{1}+O_{2}+\alpha O_{1}O_{2}$,
    \item  $%
E_{P_{\alpha}}\left( O_{1}\right) =E_{P_{\alpha}}\left( O_{2}\right) =0$,
\item $%
E_{P_{\alpha}}\left( O_{1}^{2}\right) =E_{P_{\alpha}}\left( O_{2}^{2}\right) =1$,
\item There exists a fixed $C>0$ independent of $\alpha$ such that $
var_{P_{\alpha}}\left(Y \mid A=a, \mathbf{O} \right)\leq C$ and $\pi_{a}(\mathbf{O}_{min};P_{\alpha}) \geq 1/C$,
\end{enumerate}
then
$$
\Delta _{P_{\alpha}}\left( \mathbf{O}\right)  = \alpha O_{1}O_{2}
$$
and
\begin{equation*}
\frac{var_{P_{\alpha}}\left[\psi _{P_{\alpha},a}\left( \mathbf{O};\mathcal{G}\right) \right] 
}{var_{P_{\alpha}}\left[ \chi _{P,a,eff}^{1}\left( \mathbf{V};\mathcal{G}\right)
\right] }\underset{\left\vert \alpha \right\vert \rightarrow \infty }{%
\rightarrow }\infty.
\end{equation*}
This illustrates the point that the NP-$\mathbf{O}$ estimator may
ignore a substantial fraction of the information about $\chi _{a}\left( P;%
\mathcal{G}\right) $ encoded in the Bayesian Network.

Independencies among variables in the adjustment set are not the only
carriers of information about $\chi _{a}\left( P;\mathcal{G}\right) $ in a
Bayesian Network. For instance, consider the model represented by the DAG in
Figure \ref{fig:chain} in which $A$ is randomized but a variable $M$ that
mediates all the effect of $A$ on $Y$ is measured. Then%
\begin{eqnarray*}
\chi _{a}\left( P;\mathcal{G}\right) &=&E_{p}\left[ Y|A=a\right] \\
&=&\int \int yp\left( y,m|a\right) dydm \\
&=&\int \int yp\left( y|m\right) p\left( m|a\right) dydm
\end{eqnarray*}%
and the last equality holds due to the Markov chain structure encoded in the
model. In this example, the empirical mean of $Y$ given $A=a$, can be viewed as the NP-$\mathbf{O}$ estimators where $\mathbf{O}=\emptyset$. However this estimator does not attain the semiparametric Cramer-Rao bound, because it does not exploit the Markov chain structure encoded in the graph.
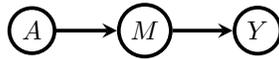
\begin{figure}[ht]
\begin{center}
\begin{tikzpicture}[>=stealth, node distance=1.5cm,
pre/.style={->,>=stealth,ultra thick,line width = 1.4pt}]
  \begin{scope}
    \tikzstyle{format} = [circle, inner sep=2.5pt,draw, thick, circle, line width=1.4pt, minimum size=6mm]
    \node[format] (1) {$A$};
    \node[format, right of=1] (2) {$M$};
    \node[format, right of=2] (3) {$Y$};
                 \draw (1) edge[pre, black] (2);
                 \draw (2) edge[pre, black] (3);
  \end{scope} 
  \end{tikzpicture}
\end{center}
\caption{A DAG where the NP-$\mathbf{O}$ estimator is inefficient.}
\label{fig:chain}
\end{figure}

As a third example, consider the Bayesian Network represented by the DAG in
Figure \ref{fig:front}. 
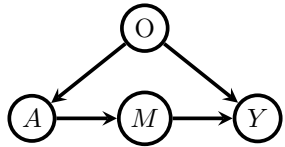
\begin{figure}[ht]
\begin{center}
\begin{tikzpicture}[>=stealth, node distance=1.5cm,
pre/.style={->,>=stealth,ultra thick,line width = 1.4pt}]
  \begin{scope}
    \tikzstyle{format} = [circle, inner sep=2.5pt,draw, thick, circle, line width=1.4pt, minimum size=6mm]
    \node[format] (1) {$A$};
    \node[format, right of=1] (2) {$M$};
    \node[format, right of=2] (3) {$Y$};
    \node[format, above of=2, yshift=-3mm] (4) {O};
                 \draw (1) edge[pre, black] (2);
                 \draw (2) edge[pre, black] (3);
                 \draw (4) edge[pre, black] (1);
                  \draw (4) edge[pre, black] (3);
  \end{scope} 
  \end{tikzpicture}
\end{center}
\caption{The front-door graph.}
\label{fig:front}
\end{figure}
Under this model $O$ is the unique, and hence optimal, covariate adjustment set. Nevertheless, under the model, $
\chi _{a}\left( P;\mathcal{G}\right) =E_{P}\left[ E_{p}\left[ Y|A=a,O\right] %
\right] $ is also equal to the so-called front-door functional 
\begin{equation}
\beta \left( P\right)
\equiv \int y\left\{ \int p\left( m|a\right) \left[ \sum_{a^{\prime
}}p\left( y|m,a^{\prime }\right) p\left( a^{\prime }\right) \right]
dm\right\} dy.
\label{eq:frontdoor}
\end{equation}
See \cite{pearl-causality}. Under regularity conditions,
the non-parametric estimator of $\beta \left( P\right) $, based on estimating the right hand side of \eqref{eq:frontdoor} replacing all densities by smooth non-parametric estimators of them, provides, under regularity conditions, yet
another regular and asymptotically linear estimator of $\chi _{a}\left( P;%
\mathcal{G}\right)$. In fact, invoking Algorithm \ref{algo:main}, 
in Example \ref{ex:front} we argue that neither estimator attains the semiparametric Cramer-Rao bound under the model. The one-step estimation technique described in Section \ref{sec:semiparam} can be used to obtain a regular and asymptotically linear estimator of $\chi _{a}\left( P;\mathcal{G}\right)$ which, under regularity conditions, attains the semiparametric Cramer-Rao Bound under $\mathcal{M}(G)$.

In Section \ref{sec:semiparam} we provide a sound and complete graphical
algorithm that, given a DAG $\mathcal{G}$, decides whether or not under all
laws $P$ of $\mathcal{M}\left( \mathcal{G}\right) $, the NP-$\mathbf{O}$
estimator is semiparametric efficient under the Bayesian Network $\mathcal{M}%
\left( \mathcal{G}\right) $. Furthermore, when the answer is negative, the
algorithm returns a formula for computing an estimator that under regularity
conditions attains the semiparametric Cramer-Rao bound.

Because estimation of causal effects under a DAG $\mathcal{G}$ is only meaningful 
when in $\mathcal{G}$ there exists at least one causal
path between $A$ and $Y$, from now on we will consider only inference about $\chi _{a}\left( P;\mathcal{G}\right) $ under Bayesian Networks $\mathcal{M}\left( \mathcal{G}%
\right) $ represented by such DAGs. 

\subsection{Review of semiparametric efficiency theory}
\label{sec:semiparam}

The problem we are concerned with in this section is formalized as follows. We are
interested in finding an estimator of the functionals $\chi _{a}\left( P;%
\mathcal{G}\right) \equiv E_{P}\left[ E_{p}\left[ Y|A=a,\pa_{\mathcal{G}%
}\left( A\right) \right] \right] $ and 
$\Delta(P;\mathcal{G}) \equiv
\sum_{ \mathbf{a} \in \mathcal{A}}c_{\mathbf{a}}E[Y_{\mathbf{a}}]$,
with the smallest possible variance among
all estimators that are regular and asymptotically linear under any $P\in 
\mathcal{M}\left( \mathcal{G}\right) .$ When not all the variables in $%
\mathcal{G}$ are discrete, model $\mathcal{M}\left( \mathcal{G}\right) $ is
a semiparametric model in the sense that it cannot be parameterized by a
Euclidean parameter. The theory of semiparametric efficient estimation %
\citep{van1998asymptotic, bickel} studies the generic problem of efficient
estimation of a finite dimensional functional, such as $\chi_{a}\left(P;\mathcal{G}\right) $, under a semiparametric model, such as $\mathcal{M(G)}$.
In what follows we review the key elements of this theory, as they apply to inference about an arbitrary  parameter $\gamma(P)$ under model $\mathcal{M(G)}$ for the law of a vector $\mathbf{V}=(V_{1},\dots,V_{s})$. In the next section we apply this theory to $\chi_{a}\left( P;\mathcal{G}\right)$ and $\Delta(P;\mathcal{G})$.

Influence functions of regular and asymptotically linear estimators of a (smooth) functional $\gamma(P)$ can be derived from well known results in
semiparametric theory. Specifically, for any law $P$ in $\mathcal{M}\left( 
\mathcal{G}\right) $ define the tangent space $\Lambda \equiv $ $\Lambda
\left( P\right) $ at $P$ of model $\mathcal{M}\left( \mathcal{G}\right) $ as
the $L_{2}\left( P\right) -$closed linear span of scores at $t=0$ for
regular one-dimensional parametric submodels $t\in \lbrack 0,\varepsilon
)\rightarrow P_{t}$ with $P_{t=0}=P$ \citep{van1998asymptotic}. In Lemma \ref{lemma:orth_decomp} of the Appendix we show that $\Lambda \equiv
\oplus _{j=1}^{s}\Lambda _{j}$ where 
\begin{equation*}
\Lambda _{j}\equiv \left\{ G\equiv g\left( V_{j},\text{pa}_{\mathcal{G}%
}\left( V_{j}\right) \right) \in L_{2}\left( P\right) :E_{P}\left[ G|\text{pa%
}_{\mathcal{G}}\left( V_{j}\right) \right] =0\right\} .
\end{equation*}%
and $\oplus $ stands for the sum of $L_{2}(P)-$orthogonal spaces. Thus, unless $\mathcal{G}$ is a complete DAG, $\Lambda $ is a
strict subset of $L_{2}^{0}\left( P\right) $, where 
\begin{equation*}
L_{2}^{0}\left( P\right) \equiv \left\{ g\in L_{2}\left( P\right) :\int
g dP=0\right\}.
\end{equation*}%

A result from semiparametric theory connects the influence functions of regular asymptotically linear estimators of certain parameters $\gamma(P)$ with the so-called influence functions of the parameters. A parameter $\gamma(P)$, more precisely the map $P^{\prime }\in \mathcal{M}\left( 
\mathcal{G}\right) \rightarrow \gamma(P^{\prime})$, is pathwise differentiable at $P$ if there exists a random variable $\varphi_{P}(\mathbf{V})$ such that 
$E_{P}\left[ \varphi _{P}\left( \mathbf{V};%
\mathcal{G}\right) ^{2}\right] <\infty,\: E_{P}\left[ \varphi _{P}\left( 
\mathbf{V}\right) \right] =0$ and such that for any regular
one-dimensional parametric submodel $t\in \lbrack 0,\varepsilon )\rightarrow
P_{t}$ with $P_{t=0}=P$ and score at $t=0$ denoted as $S,$ it holds that $%
d\gamma(P_{t}) /dt|_{t=0}=E_{P}\left[ \varphi
_{P}\left( \mathbf{V}\right) S\right]$. The random variable $\varphi_{P}(\mathbf{V})$ is called an influence function of the parameter $\gamma(P)$.
Unless $\mathcal{G}$
is complete there exists infinitely many influence functions, because if $\varphi _{P}$ is an influence function so is $\varphi _{P}+T$ for any mean zero $T$ uncorrelated with the elements of $\Lambda$. 
The aforementioned result connecting influence functions of estimators with influence functions of parameters establishes that if $\widehat{\gamma}$ is an
asymptotically linear estimator of $\gamma(P) $
at $P$ with influence function $\varphi_{P}$, then $\widehat{\gamma}$  is regular at $P$ in model $\mathcal{M}\left( 
\mathcal{G}\right) $ if and only if $\gamma(P)$ is pathwise differentiable at $P$ and  $\varphi_{P}$ is an influence function of $\gamma(P)$. See Theorem 2.2 of \cite{newey1990}.

The projection $\Pi \left[ B|\Lambda \right] $ of any $B\in $ $L_{2}\left(
P\right) $ into the tangent space $\Lambda $ at $P$ is defined as the unique
element of $\Lambda $ such that $B-\Pi \left[ B|\Lambda \right] $ is
uncorrelated under $P$ with any element of $\Lambda .$ The projection $\varphi
_{P,eff}\equiv \Pi \left[ \varphi _{P}(\mathbf{V})|\Lambda \right] $ of any
influence function $\varphi _{P}$ of $\gamma(P) $
is itself an influence function. $\varphi
_{P,eff}$ is called \textit{%
the efficient influence function} of $\gamma(P) $
at $P$ in model $\mathcal{M}\left( \mathcal{G}\right) .$ It follows from
Pythagoran Theorem, that the variance $\Omega _{eff}\equiv E_{P}\left[
\left( \varphi_{P,eff}\right) ^{2}\right]$ of $\varphi _{P,eff}(\mathbf{V})$ is less than or equal to the variance $%
E_{P}\left[ \varphi^{2}_{P}(\mathbf{V})\right]$ of any influence function  $\varphi _{P}(\mathbf{V})$.  Consequently, $%
\Omega _{eff}$ is a lower bound for the variance of the limiting mean zero
normal distribution of regular asymptotically linear estimators of $\gamma(P)$. $\Omega _{eff}$ is called the
semiparametric variance bound (also called the semiparametric Cramer-Rao bound) for $\gamma(P)$ at 
$P$ in model $\mathcal{M}\left( \mathcal{G}\right)$. 

Notice that by the linearity of the differentiation operation, if  $\chi _{P,a}^{1}(\mathbf{V};\mathcal{G})$ denotes an influence function for $\chi_{a}\left( P;\mathcal{G}\right)$ then 
$$
\Delta^{1}_{P}(\mathbf{V};\mathcal{G})= \sum\limits_{a\in\mathcal{A}} c_{a}\chi _{P,a}^{1}(\mathbf{V};\mathcal{G})
$$
is an influence function for $\Delta(P;\mathcal{G})$. Consequently, if  $\Delta^{1}_{P,eff}(\mathbf{V};\mathcal{G})$ and $\chi _{P,a,eff}^{1}(\mathbf{V};\mathcal{G})$ denote the efficient influence functions of $\Delta(P;\mathcal{G})$ and $\chi_{a}\left( P;\mathcal{G}\right)$, we have
$$
\Delta^{1}_{P,eff}(\mathbf{V};\mathcal{G})= \sum\limits_{a\in\mathcal{A}} c_{a}\chi _{P,a,eff}^{1}(\mathbf{V};\mathcal{G}).
$$

In the next section we will derive an expression for $\chi _{P,a,eff}^{1}(\mathbf{V};\mathcal{G})$ and indicate how it can be used to construct an efficient estimator of $\chi_{a}\left( P;\mathcal{G}\right)$. These efficient estimators can then be combined to obtain an efficient estimator of  $\Delta(P;\mathcal{G})$.

\subsection{Semiparametric efficient estimation of $\chi_{a}\left( P;\mathcal{G}\right)$}
\label{sec:semiparam_effi}

The next theorem provides an expression for $\chi _{P,a,eff}^{1}$.
Let 
\begin{equation*}
J_{P,a,\mathcal{G}}\equiv \frac{I_{a}\left( A\right) Y}{P\left( A=a|\text{pa}%
_{\mathcal{G}}\left( A\right) \right) },
\end{equation*}%
\begin{equation*}
\indir\left( A,Y,\mathcal{G}\right) \equiv \left\{ V_{j}\in \mathbf{V}%
:V_{j}\in \an_{\mathcal{G}}\left( A\right) \backslash \left\{ A\right\} 
\text{ and all causal paths between }V_{j}\text{ and }Y\text{ intersect }%
A\right\}
\end{equation*}%
and%
\begin{equation*}
\irrel\left( A,Y,\mathcal{G}\right) \equiv \indir\left( A,Y,\mathcal{G}%
\right) \cup \an_{\mathcal{G}}\left( Y\right) ^{c}.
\end{equation*}
Note that $\indir\left( A,Y,\mathcal{G}\right)$ is comprised by the nodes in $\mathbf{V}$ that, conditional on their parents, are instrumental variables for the causal effect of $A$ on $Y$.

\begin{theorem}
\label{lemma:eff_if_expression} Let $\mathcal{M}\left( \mathcal{G}\right) $
be the Bayesian Network represented by DAG $\mathcal{G}$ with vertex set $%
\mathbf{V.}$ Assume $Y$ and $A$ are single disjoint vertices. Then, the
efficient influence function of $\chi _{a}\left( P;\mathcal{G}\right) $ at $%
P $ under $\mathcal{M}\left( \mathcal{G}\right) $ is equal to 
\begin{equation}
\chi _{P,a,eff}^{1}\left( \mathbf{V};\mathcal{G}\right) =\sum_{j:V_{j}\notin %
\left[ \irrel\left( A,Y,\mathcal{G}\right) \cup \left\{ A\right\} \right]}\left\{ E_{P}%
\left[ J_{P,a,\mathcal{G}}|V_{j},\pa_{\mathcal{G}}\left( V_{j}\right) \right]
-E_{P}\left[ J_{P,a,\mathcal{G}}|\pa_{\mathcal{G}}\left( V_{j}\right) \right]
\right\}  \label{eq:effic_score}.
\end{equation}%
Furthermore, $\chi _{P,a,eff}^{1}\left( \mathbf{V};\mathcal{G}\right) $ 
depends on $\mathbf{V}$ only through $\mathbf{V}_{\text{marg}}\equiv 
\mathbf{V\backslash }\irrel\left( A,Y,\mathcal{G}\right) .$
\end{theorem}

Theorem \ref{lemma:eff_if_expression} establishes that the efficient influence function of
$\chi _{P,a,eff}^{1}\left( \mathbf{V};\mathcal{G}\right) $ does not depend on the variables in $\irrel(A,Y,\mathcal{G})$. Our next results will establish that the variables in $\irrel(A,Y,\mathcal{G})$ can be marginalized from the DAG without incurring in any loss of information about the parameter. Recall that for any DAG $\mathcal{G}$ with vertex set $\mathbf{V}$ and a subset of nodes $\mathbf{V}_{marg}$, $\mathcal{M}\left( \mathcal{G},\mathbf{V}_{\text{marg}}\right)$ denotes the marginal DAG model. See Section \ref{sec:def}.

\begin{definition}
Let $\mathcal{G}$ be a DAG with vertex set $\mathbf{V}$ and $\mathbf{V}_{\text{%
marg}}\subset \mathbf{V.}$ For any $P\in \mathcal{M}\left( \mathcal{G}%
\right) $ let $P_{\text{marg}}$ denote the marginal distribution of $ 
\mathbf{V}_{\text{marg}}$ under $P$. Let $\mathcal{G}^{\prime }$ be
a DAG with vertex set $\mathbf{V}_{\text{marg}}$. Let $A$, $Y$ be two distinct nodes such that $\lbrace A,Y\rbrace\subset\mathbf{V}_{marg}$.  We say that $\left( 
\mathbf{V}_{\text{marg }},\mathcal{G}^{\prime }\right) $ is sufficient for
efficient estimation of $\chi _{a}\left( P;\mathcal{G}\right) $ relative to $%
\left( \mathbf{V},\mathcal{G}\right) $ if for all $P\in \mathcal{M}\left( 
\mathcal{G}\right) ,$ the following conditions hold

\begin{enumerate}
\item $\mathcal{M}\left( \mathcal{G},\mathbf{V}_{\text{marg}}\right) =%
\mathcal{M}\left( \mathcal{G}^{\prime }\right) ,$

\item $\chi _{a}\left( P;\mathcal{G}\right) =\chi _{a}\left( P_{\text{marg}};%
\mathcal{G}^{\prime }\right) $,

\item $\mathbf{O}\left( A,Y,\mathcal{G}\right) =\mathbf{O}\left( A,Y,%
\mathcal{G}^{\prime }\right) $,

\item $\psi _{P,a}\left[ \mathbf{O}\left( A,Y,\mathcal{G}\right) ;\mathcal{G}%
\right] =\psi _{P_{\text{marg}},a}\left[ \mathbf{O}\left( A,Y,\mathcal{G}%
^{\prime }\right) ;\mathcal{G}^{\prime }\right] $ and

\item $\chi _{P,a,eff}^{1}\left( \mathbf{V};\mathcal{G}\right) =\chi _{P_{%
\text{marg}},a,eff}^{1}\left( \mathbf{V}_{\text{marg}};\mathcal{G}^{\prime
}\right) $
\end{enumerate}
\end{definition}

If we find $\left( \mathbf{V}_{\text{marg }},\mathcal{G}^{\prime }\right) $
that is sufficient for estimation of $\chi _{a}\left( P;\mathcal{G}\right) $
relative to $\left( \mathbf{V},\mathcal{G}\right) $, then we do not incur in
any loss of information about $\chi _{a}\left( P;\mathcal{G}\right) $ if we
ignore the variables in $\mathbf{V\backslash V}_{\text{marg}}$ and assume
that $\mathbf{V}_{\text{marg}}$ follows a Bayesian Network $\mathcal{M}%
\left( \mathcal{G}^{\prime }\right) $ for the DAG $\mathcal{G}^{\prime }$.
Furthermore, since $\mathcal{G}^{\prime }$ preserves the optimal adjustment
set then the NP-$\mathbf{O}$ estimator of $\chi _{a}\left( P;\mathcal{G}%
\right) $ is the same as the NP-$\mathbf{O}$ estimator of $\chi _{a}\left(
P_{\text{marg}};\mathcal{G}^{\prime }\right) .$ Since by condition 5) of the
preceding definition the efficiency bound for $\chi _{a}\left( P;\mathcal{G}%
\right) $ under $\mathcal{M}\left( \mathcal{G}\right) $ is the same as the
efficiency bound for $\chi _{a}\left( P_{\text{marg}};\mathcal{G}^{\prime
}\right) $ under $\mathcal{M}\left( \mathcal{G}^{\prime }\right) $, then for
studying the loss of efficiency incurred by using the NP-$\mathbf{O}$
estimator we can pretend that the available variables are $\mathbf{V}_{\text{%
marg}}$ and that the problem is to estimate $\chi _{a}\left( P_{\text{marg}};%
\mathcal{G}^{\prime }\right) $ under the Bayesian Network $\mathcal{M}\left( 
\mathcal{G}^{\prime }\right) .$

The next lemma implies that $\mathbf{V}_{\text{marg}}=\mathbf{V}\backslash $ 
$\irrel\left( A,Y,\mathcal{G}\right) $ and $\mathcal{G}^{\prime }$ equal to
the output of Algorithm \ref{algo:prun_indir} below satisfy the preceding
definition. 

\begin{lemma}
\label{lemma:correct_algo_prun_indir} Let $\mathcal{G}$ and $\mathcal{G}%
^{\prime }$ be the input and output DAGs of Algorithm \ref{algo:prun_indir}.
Let $\mathbf{V}$ and $\mathbf{V}_{\text{marg}}$ be the vertex sets of $%
\mathcal{G}$ and $\mathcal{G}^{\prime }$ respectively. Then $\left( \mathbf{V%
}_{\text{marg}},\mathcal{G}^{\prime }\right) $ is sufficient for efficient
estimation of $\chi _{a}\left( P;\mathcal{G}\right) $ relative to $\left( 
\mathbf{V},\mathcal{G}\right) .$
\end{lemma}

\begin{algorithm}[ht!]
	\label{algo:prun_indir}
	\SetKwInOut{Input}{input}\SetKwInOut{Output}{output}
  \SetAlgoLined\DontPrintSemicolon
    \SetKwFunction{prune}{prune}
      \SetKwProg{proc}{procedure}{}{}
	\Input{DAG $\mathcal{G}$ with nodes $\mathbf{V}$ and two distinct nodes $A,Y \in \mathbf{V}$}
	\Output{A new DAG $\mathcal{G}^{\prime}$  with vertex set $\mathbf{V}_{marg} =\mathbf{V} \setminus  \text{irrel}(A,Y, \mathcal{G})$ such that $(\mathbf{V}_{marg},\mathcal{G}^{\prime})$ is sufficient for efficient estimation of $\chi_{a}(P;\mathcal{G})$ relative to $(\mathbf{V},\mathcal{G})$.}
			\proc{\prune{$A, Y,\mathcal{G}$}}
	{
				$\mathcal{G}^{\prime}=\mathcal{G}_{\mathbf{V}\setminus\an_{\mathcal{G}}^{c}(Y)}$ \\
								$I_{1},\dots,I_{L}= \texttt{topological\_sort}\left(\text{indir}(A,Y,\mathcal{G}^{\prime}), \mathcal{G}^{\prime}\right)$\\
		\For{ $j=L,L-1,\dots,1$}{
		$\mathcal{G}^{\prime}=\tau(\mathcal{G}^{\prime}, I_{j})	$}	
				\Return{$\mathcal{G}^{\prime}$;}
	} 

		\caption{{\bf DAG pruning procedure to remove irrelevant nodes} }
	
\end{algorithm}

\FloatBarrier

The output $\mathcal{G}^{\prime }$ of Algorithm \ref%
{algo:prun_indir} is obtained as the result of first deleting the edges and
vertices in $\an_{\mathcal{G}}\left( Y\right) ^{c}$ and subequently
removing, sequentially by a latent projection operation, each node in $\indir%
\left( A,Y,\mathcal{G}\right)$.
For the definition of the latent projection operation  $\tau(\mathcal{G}, V)$ see Section \ref{sec:def}.  Algorithm \ref{algo:prun_indir} assumes the availability of a subroutine \texttt{topological\_sort} to topologically sort a set of nodes relative to a DAG $\mathcal{G}$. One such subroutine is Kahn's algorithm \citep{kahntopological}, which is known to have worst case complexity $\mathcal{O}(\vert \mathbf{V}\vert + \vert \mathbf{E} \vert)$.  

Lemma \ref{lemma:correct_algo_prun_indir} is proven in the Appendix by invoking the following important result.

\begin{proposition}
\label{lemma:marg-eff} Let $\mathcal{M}$ be a semiparametric model for the law
of a random vector $\mathbf{V.}$ Let $\mathbf{V}^{\prime }$ be a subvector
of $\ \mathbf{V.}$ Let $\mathcal{M}^{\prime }$ be the model for the law of $%
\ \mathbf{V}^{\prime }$ induced by model $\mathcal{M}$, that is, $\mathcal{M}%
^{\prime }$ is the collection of laws for $\mathbf{V}^{\prime }$ such that
for every $P^{\prime }\in \mathcal{M}^{\prime }$ there exists a law $P$ for $%
\mathbf{V}$ with $P^{\prime }$ being the marginal of $P$ over $\mathbf{V}%
^{\prime }.$ Let $\chi \left( P\right) $ be a regular parameter in model $%
\mathcal{M}$ with efficient influence function at $P\in \mathcal{M}$ equal
to $\chi _{P,eff}^{1}.$ Suppose $\chi _{P,eff}^{1}$ depends on $\mathbf{V}$
only through $\mathbf{V}^{\prime }.$ Let $P^{\prime }$ be the marginal law
of $P$ over$\ \mathbf{V}^{\prime }$. Suppose $\chi \left( P\right) $ depends
on $P$ only through $P^{\prime }.$ Define $\nu \left( P^{\prime }\right)
\equiv \chi \left( P\right) .$ Let $\nu _{P^{\prime},eff}^{1}$ be the
efficient influence function of $\nu \left( P^{\prime }\right) $ in model $%
\mathcal{M}^{\prime }$ at $P^{\prime }\in \mathcal{M}^{\prime }.$ Then,
given $P^{\prime }\in \mathcal{M}^{\prime }$ it holds that $\chi
_{P,eff}^{1}=\nu _{P^{\prime},eff}^{1}$ for every $P\in \mathcal{M}$ with
marginal law $P^{\prime }$.
\end{proposition}

In light of the Lemma \ref{lemma:marg-eff}, from now on without loss of generality we
will assume that $\irrel\left( A,Y,\mathcal{G}\right) =\emptyset .$ This
assumption implies that we can partition the node set $\mathbf{V}$\textbf{\ }%
of $\mathcal{G}$ as $\mathbf{M}\cup \mathbf{W\cup }\left\{ A,Y\right\} $
where the vertices in $\mathbf{M}$ intersect at least one causal path
between $A$ and $Y$, that is, $\mathbf{M}$ is the set of mediators in the causal
pathways between $A$ and $Y,$ and $\mathbf{W}$ are non-descendants of $A.$
We can therefore sort topologically $\mathbf{V}$ as $\left(
W_{1},\dots,W_{J},A,M_{1},\dots,M_{K},Y\right) .$  The set $\mathbf{O}\left( A,Y,\mathcal{G}\right) \equiv \mathbf{O}\equiv \left(
O_{1},\dots,O_{T}\right)$, where $\left(
O_{1},\dots,O_{T}\right)$ is sorted topologically,  is included in $\mathbf{W}$. Throughout $T=0$ if $\mathbf{O}\left( A,Y,\mathcal{G}\right) =\emptyset$.

The following lemma establishes further identities that are invoked in Theorem \ref{theo:eff_simple} below to derive yet another expression for $\chi _{P,a,eff}^{1}\left( \mathbf{V};\mathcal{G}\right)$.
Let
$$
T_{P,a,\mathcal{G}}\equiv \frac{I_{a}\left( A\right) Y}{\pi_{a}(\mathbf{O};P)}.
$$
\begin{lemma}\label{lemma:eff_simple}
Let $\mathcal{M}\left( \mathcal{G}\right) $
be the Bayesian Network represented by DAG $\mathcal{G}$ with vertex set $%
\mathbf{V.}$ Assume $Y$ and $A$ are single disjoint vertices. 
Assume $\irrel(A,Y,\mathcal{G})=\emptyset$. Then
\begin{enumerate}
    \item If $J\geq 1$ then for all $j\in\lbrace 1,\dots, J\rbrace$
    $$
    E_{P}\left[ J_{P,a,\mathcal{G}}|W_{j},\pa_{\mathcal{%
G}}\left( W_{j}\right) \right] =  E_{P}\left[ b_{a}(\mathbf{O};P)|W_{j},\pa_{\mathcal{%
G}}\left( W_{j}\right) \right].
    $$
    \item If $K\geq 1$ then for all $k\in\lbrace 1,\dots, K\rbrace$
    $$E_{P}\left[ J_{P,a,\mathcal{G}}|M_{k},\pa_{\mathcal{%
G}}\left( M_{k}\right) \right] =  E_{P}\left[ T_{P,a,\mathcal{G}}|M_{k},\pa_{\mathcal{%
G}}\left( M_{k}\right) \right].
$$
  \item 
    $$E_{P}\left[ J_{P,a,\mathcal{G}}|Y,\pa_{\mathcal{%
G}}\left( Y\right) \right] =  E_{P}\left[ T_{P,a,\mathcal{G}}|Y,\pa_{\mathcal{%
G}}\left( Y\right) \right].
$$
\end{enumerate}
\end{lemma}

In what follows we use the conventions
$$
\sum\limits_{k=1}^{0} \cdot \equiv 0, \quad \sum\limits_{j=1}^{0} \cdot \equiv 0 \quad ,\sum\limits_{j=2}^{1} \cdot \equiv 0.
$$

\begin{theorem}\label{theo:eff_simple}
Under the assumptions of Lemma \ref{lemma:eff_simple} the
efficient influence function of $\chi _{a}\left( P;\mathcal{G}\right) $ at $%
P $ under $\mathcal{M}\left( \mathcal{G}\right) $ is equal to 
\begin{eqnarray}
\chi _{P,a,eff}^{1}\left( \mathbf{V};\mathcal{G}\right) &=&E_{P}\left[
T_{P,a,\mathcal{G}}|Y,\pa_{\mathcal{G}}\left( Y\right) \right] -E_{P}\left[
T_{P,a,\mathcal{G}}|\pa_{\mathcal{G}}\left( Y\right) \right]
\label{eq:formula_eff} \\
&&+\sum_{k=1}^{K}\left\{ E_{P}\left[ T_{P,a,\mathcal{G}}|M_{k},\pa_{\mathcal{%
G}}\left( M_{k}\right) \right] -E_{P}\left[ T_{P,a,\mathcal{G}}|\pa_{%
\mathcal{G}}\left( M_{k}\right) \right] \right\}  \notag \\
&&+\sum_{j=1}^{J}\left\{ E_{P}\left[ b_{a}(\mathbf{O};P)|W_{j},\pa_{\mathcal{%
G}}\left( W_{j}\right) \right] -E_{P}\left[ b_{a}(\mathbf{O};P)|\pa_{%
\mathcal{G}}\left( W_{j}\right) \right] \right\}  \notag 
\end{eqnarray}
where 
$$
E_{P}\left[ b_{a}(\mathbf{O};P)|\pa_{\mathcal{G}}\left(
W_{1}\right) \right] = \chi _{a}\left( P;\mathcal{G}\right).
$$
\end{theorem}

The expression for $\chi _{P,a,eff}^{1}\left( \mathbf{V};\mathcal{G}\right)$ in Theorem \ref{theo:eff_simple} can be used to
compute the following one-step estimator $\widehat{\chi }_{one-step}$ \citep{van1998asymptotic},
$$
\widehat{\chi }_{one-step}\equiv \widehat{\chi} _{a}\left( {P};\mathcal{G}%
\right)  + \mathbb{P}_{n}\left[ \widehat{\chi}_{P,a,eff}^{1}\left( \mathbf{V};%
\mathcal{G}\right) \right]
$$
where if $J=0$
$$
\widehat{\chi} _{a}\left( {P};\mathcal{G}%
\right)= \frac{\mathbb{P}_{n}\left[ Y I_{a}(A) \right]]}{\mathbb{P}_{n}\left[ I_{a}(A) \right]]}
$$
and
$$
\mathbb{P}_{n}\left[ \widehat{\chi}_{P,a,eff}^{1}\left( \mathbf{V};%
\mathcal{G}\right)  \right] = \sum_{k=1}^{K}\mathbb{P}_{n}\left\{ \widehat{E}\left[ T_{P,a,\mathcal{%
G}}|M_{k},\pa_{\mathcal{G}}\left( M_{k}\right) \right] -\widehat{E}\left[
T_{P,a,\mathcal{G}}|\pa_{\mathcal{G}}\left( M_{k}\right) \right] \right\} 
$$
and if $J\geq 1$
\begin{eqnarray*}
\widehat{\chi }_{one-step} &\equiv & \widehat{\chi} _{a}\left( {P};\mathcal{G}%
\right)+\mathbb{P}_{n}\left[ \widehat{\chi}_{P,a,eff}^{1}\left( \mathbf{V};%
\mathcal{G}\right) \right] \\
&=&\mathbb{P}_{n}\left\{ \widehat{E}\left[ T_{P,a,\mathcal{G}}|Y,\pa_{%
\mathcal{G}}\left( Y\right) \right] -\widehat{E}\left[ T_{P,a,\mathcal{G}%
}|\pa_{\mathcal{G}}\left( Y\right) \right] \right\} \\
&&+\sum_{k=1}^{K}\mathbb{P}_{n}\left\{ \widehat{E}\left[ T_{P,a,\mathcal{%
G}}|M_{k},\pa_{\mathcal{G}}\left( M_{k}\right) \right] -\widehat{E}\left[
T_{P,a,\mathcal{G}}|\pa_{\mathcal{G}}\left( M_{k}\right) \right] \right\} \\
&&+\sum_{j=2}^{J}\mathbb{P}_{n}\left[ \left\{ \widehat{E}\left[ T_{P,a,%
\mathcal{G}}|W_{j},\pa_{\mathcal{G}}\left( W_{j}\right) \right] -\widehat{E}%
\left[ b_{a}(\mathbf{O};P)|\pa_{\mathcal{G}}\left( W_{j}\right) \right]
\right\} \right] \\
&&+\mathbb{P}_{n}\left[ \widehat{E}\left[ b_{a}(\mathbf{O};P)|W_{1},\pa_{%
\mathcal{G}}\left( W_{1}\right) \right] \right]
\end{eqnarray*}%
and where $\widehat{E}\left( \cdot |\cdot \right) $ are non-parametric
regression estimators of the relevant conditional expectations and $\mathbb{P%
}_{n}$ is the empirical mean operator. Under regularity conditions, which
include restrictions on some measure (for example, the metric entropy) of the complexity of the ambient function space of the conditional expectations
appearing in the expression for $\chi _{P,a,eff}^{1}\left( \mathbf{V};%
\mathcal{G}\right) ,$ and for particular choices of the non-parametric
estimators $\widehat{E}\left( \cdot |\cdot \right) $ of these conditional
expectations, the one-step estimator $\widehat{\chi }_{one-step}$ is regular
and asymptotically linear with influence function equal to
$\chi _{P,a,eff}^{1}\left( \mathbf{V};\mathcal{G}\right) $ \citep{van1998asymptotic}
and therefore it attains the semiparametric variance bound
for $\chi _{a}\left( P;\mathcal{G}\right) $ under model $\mathcal{M}\left( 
\mathcal{G}\right)$. 

It turns out that for special configurations of $\mathcal{G}$, the formula $%
\left( \ref{eq:formula_eff}\right) $ simplifies in that either
\begin{itemize}
    \item[(i)] \begin{equation}
\chi _{P,a,eff}^{1}\left( \mathbf{V};\mathcal{G}\right) =\psi _{P,a}\left[ 
\mathbf{O}\left( A,Y,\mathcal{G}\right) ;\mathcal{G}\right] \text{ for all }%
P\in \mathcal{M}\left( \mathcal{G}\right)  \label{eq:main_id}
\end{equation}
or
\item[(ii)]  some of the terms in the formula vanish for all $P\in \mathcal{M}%
\left( \mathcal{G}\right) $.
\end{itemize}
Case (i) implies that the NP-$\mathbf{O}$
estimator of $\chi _{a}\left( P;\mathcal{G}\right) $ attains the
semiparametric variance bound for $\chi _{a}\left( P;\mathcal{G}\right) $
under $\mathcal{M}\left( \mathcal{G}\right) .$ For such DAG configurations
there is no loss of efficiency in ignoring the observations on the variables 
$\mathbf{V}\backslash \left[ \mathbf{O}\cup \left\{ Y,A\right\} \right] $.
Case (ii) is important even if $\left( \ref{eq:main_id}\right) $ fails
because when case (ii) holds not only is the calculation of the one-step
estimator simplified but also such estimator attains the semiparametric
variance bound under weaker regularity conditions, in that complexity
restrictions are required only on the conditional expectations that appear
in the non-vanishing terms. Algorithm \ref{algo:main} below is sound and complete for the inquiry of whether or
not case (i) holds. In addition, when case (i) does not hold, the algorithm returns a
simplified formula for $\chi _{P,a,eff}^{1}\left( \mathbf{V};\mathcal{G}%
\right) $ for certain DAG configurations. Under some DAG configurations such simplifications imply that some variables in the DAG do not appear in the expression for   $\chi _{P,a,eff}^{1}$. This is important, because such variables are neither needed for consistent nor for efficient estimation of $\chi _{a}\left( P;\mathcal{G}\right)$.

\begin{algorithm}[ht!]
\footnotesize
	\label{algo:main}
	\SetKwInOut{Input}{input}\SetKwInOut{Output}{output}
  \SetAlgoLined\DontPrintSemicolon\LinesNumbered
    \SetKwFunction{checkEfficient}{checkEfficient}
      \SetKwProg{proc}{procedure}{}{}
	\Input{DAG $\mathcal{G}$ with vertex set $\mathbf{V}$ and two distinct nodes $A,Y \in \mathbf{V}$ such that $A\in \an_{\mathcal{G}}(Y)$}
	\Output{An answer to the inquiry of whether $
\chi _{P,a,eff}^{1}\left( \mathbf{V};%
\mathcal{G}\right) =\psi _{P,a}\left( \mathbf{O};\mathcal{G}\right)  \text{ for all } P\in\mathcal{M}(\mathcal{G})
$ and a, possibly, simplified
formula for $\chi _{P,a,eff}^{1}$ if the answer to the inquiry is negative.}
			\proc{\checkEfficient{$A, Y,\mathcal{G}$}}
	{
	\tcc{Operations performed to compute the formula for the efficient influence function should be understood as symbolic operations. The symbol \textup{\&\&} stands for the short-circuit AND operator.}
				$\mathcal{G}$=\texttt{prune}$(A,Y, \mathcal{G})$ \\
				$\left(\mathbf{W}, A, \mathbf{M},Y\right)=\left(W_{1},\dots,W_{J},A,M_{1},\dots,M_{K},Y\right)=$\texttt{topological\_sort}$(\mathbf{V},\mathcal{G})$\\
				 \tcc{$J=0$ if $\mathbf{W}=\emptyset$ and $K=0$ if $\mathbf{M}=\emptyset$}
				$M_{K+1}=Y$\\
				$\mathbf{O}=\mathbf{O}(A,Y,\mathcal{G})$\\
				$O_{1},\dots,O_{T}=$\texttt{topological\_sort}$(\mathbf{O})$\\
\texttt{efficient\_nondesc}$=$\texttt{False}\\
\texttt{efficient\_desc}$=$\texttt{False}\\
				\uIf{$\mathbf{O}\setminus \lbrace O_{T}\rbrace\subset \pa_{\mathcal{G}}(O_{T})$ and $J> 1$}{$j=J-1$\\				
	\lWhile{$\pa_{\mathcal{G}}(W_{j+1})\setminus \lbrace W_{j}\rbrace \subset \pa_{\mathcal{G}}(W_{j})$ and $j\geq 2$ }{$j=j-1$}
\uIf{$j\geq 2$}{
\texttt{offenders\_nondesc}=$\lbrace j \rbrace \cup$\texttt{get\_offenders\_nondesc}$(\mathcal{G}, \mathbf{W}, \mathbf{O}, j-1)$\\
$
\chi _{P,a,eff}^{1,non-desc}  =b_{a}(\mathbf{O};P)-\chi_{a}(P;\mathcal{G})
+
\sum_{h\in\texttt{offenders\_nondesc}}\left\{ E_{P}\left[ b_{a}\left( 
\mathbf{O;}P\right) |\pa_{\mathcal{G}}\left( W_{h}\right) ,W_{h}\right]
-E_{P}\left[ b_{a}\left( \mathbf{O;}P\right) |\pa_{\mathcal{G}}\left(
W_{h+1}\right) \right] \right\}
$
}
\Else{$
\chi _{P,a,eff}^{1,non-desc}=b_{a}(\mathbf{O};P)-\chi_{a}(P;\mathcal{G})$\\
\texttt{efficient\_nondesc}$=$\texttt{True}}
	
	}
				\ElseIf{$J > 1$}{
					\texttt{offenders\_nondesc}=\texttt{get\_offenders\_nondesc}$(\mathcal{G}, \mathbf{W}, \mathbf{O}, J)$\\			
			$
\chi _{P,a,eff}^{1,non-desc}  =E_{P}\left[b_{a}(\mathbf{O};P) \mid W_{J}, \pa_{\mathcal{G}}(W_{J})\right]-\chi_{a}(P;\mathcal{G})
+
\sum_{h\in\texttt{offenders\_nondesc}}\left\{ E_{P}\left[ b_{a}\left( 
\mathbf{O;}P\right) |\pa_{\mathcal{G}}\left( W_{h}\right) ,W_{h}\right]
-E_{P}\left[ b_{a}\left( \mathbf{O;}P\right) |\pa_{\mathcal{G}}\left(
W_{h+1}\right) \right] \right\}
$	}
\ElseIf{$J=1$}{
$
\chi _{P,a,eff}^{1,non-desc}=b_{a}(\mathbf{O};P)-\chi_{a}(P;\mathcal{G})$\\
\texttt{efficient\_nondesc}$=$\texttt{True}
}
\ElseIf{$J=0$}{
$
\chi _{P,a,eff}^{1,non-desc}=0$\\
\texttt{efficient\_nondesc}$=$\texttt{True}
}
\uIf{$A\cup \mathbf{O}_{min}\subset \pa_{\mathcal{G}}(Y)$ and $K\geq 1$}{
$\chi_{P,a,eff}^{1,desc}=I_{a}(A)Y\pi^{-1}_{a}(\mathbf{O}_{min};P)$\\
$k=K+1$\\
\lWhile{$k\geq 2$ \textup{\&}\textup{\&} $\pa_{\mathcal{G}}(M_{k})\subset \pa_{\mathcal{G}}(M_{k-1})\cup \lbrace M_{k-1}\rbrace$ }{$k=k-1$}
\uIf{$k\geq 2$}{
\texttt{offenders\_desc}=$\lbrace k  \rbrace \cup$\texttt{get\_offenders\_desc}$(\mathcal{G}, \mathbf{M}, \mathbf{O},\mathbf{O}_{min}, k-1)$\\	
$\chi_{P,a,eff}^{1,desc}=\chi_{P,a,eff}^{1,desc}  + \sum_{h\in\texttt{offenders\_desc}} \lbrace E_{P}\left[T_{P,a,\mathcal{G}} \mid \pa_{\mathcal{G}}(M_{h-1}), M_{h-1} \right]
-E_{P}\left[T_{P,a,\mathcal{G}} \mid \pa_{\mathcal{G}}(M_{h})\right]\rbrace 
$\\
\uIf{$\lbrace A \rbrace \cup \mathbf{O}=\pa_{\mathcal{G}}(M_{1})$}{
$\chi_{P,a,eff}^{1,desc}= \chi_{P,a,eff}^{1,desc}- I_{a}(A) b_{a}(\mathbf{O};P) \pi_{a}^{-1}(\mathbf{O}_{min};P)$}
\Else{$\chi_{P,a,eff}^{1,desc}= \chi_{P,a,eff}^{1,desc}- E_{P}\left[T_{P,a,\mathcal{G}}\mid \pa_{\mathcal{G}}(M_{1}) \right]$}
}\Else{$\chi_{P,a,eff}^{1,desc}= \chi_{P,a,eff}^{1,desc}- I_{a}(A) b_{a}(\mathbf{O};P) \pi_{a}^{-1}(\mathbf{O}_{min};P)$\\
\texttt{efficient\_desc=True}
}

}
\ElseIf{$K\geq 1$}{
\texttt{offenders\_desc}=$\lbrace K+1\rbrace \cup $\texttt{get\_offenders\_desc}$(\mathcal{G}, \mathbf{M}, \mathbf{O},\mathbf{O}_{min}, K)$\\
$\chi_{P,a,eff}^{1,desc}= \sum_{h\in\texttt{offenders\_desc}} \lbrace E_{P}\left[T_{P,a,\mathcal{G}} \mid \pa_{\mathcal{G}}(M_{h-1}), M_{h-1} \right]
-E_{P}\left[T_{P,a,\mathcal{G}} \mid \pa_{\mathcal{G}}(M_{h})\right]\rbrace 
$\\
\uIf{$\lbrace A \rbrace \cup \mathbf{O}=\pa_{\mathcal{G}}(M_{1})$}{
$\chi_{P,a,eff}^{1,desc}= \chi_{P,a,eff}^{1,desc}- I_{a}(A) b_{a}(\mathbf{O};P) \pi_{a}^{-1}(\mathbf{O}_{min};P)$}
\Else{$\chi_{P,a,eff}^{1,desc}= \chi_{P,a,eff}^{1,desc}- E_{P}\left[T_{P,a,\mathcal{G}}\mid \pa_{\mathcal{G}}(M_{1}) \right]$}
}		
\Else{
$\chi_{P,a,eff}^{1,desc}= I_{a}(A) \pi_{a}^{-1}(\mathbf{O}_{min};P) (Y-b_{a}(\mathbf{O};P))$\\
\texttt{efficient\_desc=True}
}

		$\chi_{P,a,eff}^{1}=\chi_{P,a,eff}^{1,non-desc}+\chi_{P,a,eff}^{1,desc}$\\
	
	\texttt{efficient} = 	\texttt{efficient\_desc} \textup{\&} \texttt{efficient\_nondesc}\\
		\Return{\textup{\texttt{efficient}}, $\chi_{P,a,eff}^{1}$}
		\caption{{\bf An algorithm that is sound and complete for checking if $
\chi _{P,a,eff}^{1}\left( \mathbf{V};%
\mathcal{G}\right) =\psi _{P,a}\left( \mathbf{O};\mathcal{G}\right)  \text{ for all} \quad P\in\mathcal{M}(\mathcal{G})$ and sound for finding a simplified formula for $\chi^{1}_{P,a,eff}$.} }}	
\end{algorithm}

\begin{algorithm}[ht!]
	\label{algo:offenders_desc}
	\SetKwInOut{Input}{input}\SetKwInOut{Output}{output}
  \SetAlgoLined\DontPrintSemicolon
    \SetKwFunction{offdesc}{get\_offenders\_desc}
      \SetKwProg{proc}{procedure}{}{}
	\Input{DAG $\mathcal{G}$, mediator nodes $\mathbf{M}$, optimal adjustment set $\mathbf{O}$, optimal minimal adjustment set $\mathbf{O}_{min}$ and integer init.}
	\Output{The set of nodes in $\mathbf{M}$ that don't satisfy at least one of \eqref{eq:inclusion_1}, \eqref{eq:inclusion_2} or \eqref{eq:d_sep_mediators}}
			\proc{\offdesc{$\mathcal{G},\mathbf{M},\mathbf{O}, \mathbf{O}_{min}, \textup{init}$}}
	{
	  \texttt{offender\_desc}$=\emptyset$\\
	  \For{$i=\textup{init},\dots, 2$}
	  {\If{ $\lbrace A\rbrace \cup \mathbf{O}_{min} \not\subset \pa_{\mathcal{G}}(M_{1})$ or $\pa_{\mathcal{G}}(M_{i})\not\subset \pa_{\mathcal{G}}(M_{i-1})\cup \lbrace M_{i-1}\rbrace$ or $Y \not\ort_{\mathcal{G}} \pa_{\mathcal{G}}(M_{i-1})\cup \lbrace M_{i-1} \setminus \pa_{\mathcal{G}}(M_{i}) \mid \pa_{\mathcal{G}}(M_{i}) $}
	  {	  \texttt{offender\_desc}$=$\texttt{offender\_desc}$\cup\lbrace i \rbrace$ }}
	  \Return{\textup{\texttt{offender\_desc}}}
	} 
		\caption{{\bf Subroutine to find all mediator nodes that don't satisfy at least one of \eqref{eq:inclusion_1}, \eqref{eq:inclusion_2} or \eqref{eq:d_sep_mediators}.} }
\end{algorithm}

\begin{algorithm}[ht!]
	\label{algo:offenders_nondesc}
	\SetKwInOut{Input}{input}\SetKwInOut{Output}{output}
  \SetAlgoLined\DontPrintSemicolon
    \SetKwFunction{offnondesc}{get\_offenders\_nondesc}
      \SetKwProg{proc}{procedure}{}{}
	\Input{DAG $\mathcal{G}$, non-mediator nodes $\mathbf{W}$, optimal adjustment set $\mathbf{O}$ and integer init.}
	\Output{The set of nodes in $\mathbf{W}$ that don't satisfy \eqref{eq_main_independence}}
			\proc{\offnondesc{$\mathcal{G},\mathbf{W},\mathbf{O}, \textup{init}$}}
	{
	  \texttt{offender\_nondesc}$=\emptyset$\\
	  \For{$i=\textup{init},\dots, 1$}
	  {\If{ $\mathbf{O}\setminus \mathbf{I}_{j} \not\ort_{\mathcal{G}} \left[\pa_{\mathcal{G}}(W_{i})\cup \lbrace W_{i} \rbrace\right] \bigtriangleup \pa_{\mathcal{G}}(W_{i+1})\mid \mathbf{I}_{j} $}
	  {	  \texttt{offender\_nondesc}$=$\texttt{offender\_nondesc}$\cup\lbrace i \rbrace$ }}
	  \Return{\textup{\texttt{offender\_nondesc}}}
	} 
		\caption{{\bf Subroutine to find all non-mediator nodes that don't satisfy \eqref{eq_main_independence}}.} 
\end{algorithm}

\FloatBarrier

We will now describe the rationale behind the steps of the algorithm. 
If $J=0$ let 
\begin{align*}
\chi _{P,a,eff}^{1,non-desc}\left( \mathbf{V};\mathcal{G}\right) \equiv 0
\end{align*}
and if $J\geq 1$ let
\begin{align*}
\chi _{P,a,eff}^{1,non-desc}\left( \mathbf{V};\mathcal{G}\right) &\equiv \sum_{j=2}^{J}\left\{ E_{P}\left[ b_{a}(\mathbf{O};P)|W_{j},\pa_{\mathcal{%
G}}\left( W_{j}\right) \right] -E_{P}\left[ b_{a}(\mathbf{O};P)|\pa_{%
\mathcal{G}}\left( W_{j}\right) \right] \right\}  \\
&+E_{P}\left[ b_{a}(\mathbf{O};P)|W_{1},\pa_{\mathcal{G}}\left(
W_{1}\right) \right] -\chi _{a}\left( P;\mathcal{G}\right).
\end{align*}
Furthermore let
\begin{align*}
\chi _{P,a,eff}^{1,desc}\left( \mathbf{V};\mathcal{G}\right) &\equiv E_{P}\left[
T_{P,a,\mathcal{G}}|Y,\pa_{\mathcal{G}}\left( Y\right) \right] -E_{P}\left[
T_{P,a,\mathcal{G}}|\pa_{\mathcal{G}}\left( Y\right) \right]
\\
&+\sum_{k=1}^{K}\left\{ E_{P}\left[ T_{P,a,\mathcal{G}}|M_{k},\pa_{\mathcal{%
G}}\left( M_{k}\right) \right] -E_{P}\left[ T_{P,a,\mathcal{G}}|\pa_{%
\mathcal{G}}\left( M_{k}\right) \right] \right\} .
\end{align*}
By Theorem \ref{theo:eff_simple}, 
$$
\chi _{P,a,eff}^{1}\left( \mathbf{V};\mathcal{G}\right) = \chi _{P,a,eff}^{1,non-desc}\left( \mathbf{V};\mathcal{G}\right) + \chi _{P,a,eff}^{1,desc}\left( \mathbf{V};\mathcal{G}\right).
$$

The algorithm starts by searching for possible deletions and/or simplifications of the terms in the expression for $\chi^{1,non-desc}_{P,a,eff}$ when $J\geq 1$. 
If $J=1$ then $W_{1}$ is necessarily equal to $O_{T}=O_{1}$, because as explained below $W_{J}$ is always equal to $O_{T}$. Then, since $\mathbf{O}=\lbrace O_{1}\rbrace$, 
$$
E_{P}\left[ b_{a}(\mathbf{O};P)|W_{1},\pa_{\mathcal{G}}\left(
W_{1}\right) \right] = b_{a}(\mathbf{O};P)
$$
and consequently
$$
\chi _{P,a,eff}^{1,non-desc}\left( \mathbf{V};\mathcal{G}\right) =  b_{a}(\mathbf{O};P) - \chi_{a}(P;\mathcal{G}).
$$
For $J> 1$ define for each $j\in \left\{ 1,\dots,J-1\right\} $%
\begin{equation*}
\mathbf{I}_{j}\equiv\left[ \pa_{\mathcal{G}}\left( W_{j}\right) \cup \left\{
W_{j}\right\} \right] \cap \pa_{\mathcal{G}}\left( W_{j+1}\right) .
\end{equation*}%
If 
\begin{equation}
\mathbf{O}\backslash \mathbf{I}_{j}\perp \!\!\!\perp _{\mathcal{G}}\left. %
\left[ \left[ \pa_{\mathcal{G}}\left( W_{j}\right) \cup \left\{
W_{j}\right\} \right] \Delta \pa_{\mathcal{G}}\left( W_{j+1}\right) \right]
\right\vert \mathbf{I}_{j}  \label{eq_main_independence}
\end{equation}%
then%
\begin{eqnarray*}
E_{P}\left[ b_{a}\left( \mathbf{O;}P\right) |\pa_{\mathcal{G}}\left(
W_{j+1}\right) \right]  &=&E_{P}\left[ b_{a}\left( \mathbf{O;}P\right) |\pa_{%
\mathcal{G}}\left( W_{j+1}\right) \backslash \left[ \pa_{\mathcal{G}}\left(
W_{j}\right) \cup \left\{ W_{j}\right\} \right] ,\mathbf{I}_{j}\right]  \\
&=&E_{P}\left[ b_{a}\left( \mathbf{O;}P\right) |\mathbf{I}_{j}\right] 
\end{eqnarray*}%
and 
\begin{eqnarray*}
E_{P}\left[ b_{a}\left( \mathbf{O;}P\right) |\pa_{\mathcal{G}}\left(
W_{j}\right) ,W_{j}\right]  &=&E_{P}\left[ b_{a}\left( \mathbf{O;}P\right) |%
\left[ \pa_{\mathcal{G}}\left( W_{j}\right) \cup \left\{ W_{j}\right\} %
\right] \backslash \pa_{\mathcal{G}}\left( W_{j+1}\right) ,\mathbf{I}_{j}%
\right]  \\
&=&E_{P}\left[ b_{a}\left( \mathbf{O;}P\right) |\mathbf{I}_{j}\right].
\end{eqnarray*}%
Thus \eqref{eq_main_independence} is a graphical criterion for checking if the differences
\begin{equation}
E_{P}\left[ b_{a}\left( \mathbf{O;}P\right) |\pa_{\mathcal{G}}\left(
W_{j}\right) ,W_{j}\right] -E_{P}\left[ b_{a}\left( \mathbf{O;}P\right) |\pa%
_{\mathcal{G}}\left( W_{j+1}\right) \right]   \label{eq_differences}
\end{equation}%
cancel out from the expression for $\chi _{P,a,eff}^{1}\left( \mathbf{V};%
\mathcal{G}\right) $ for all $P\in \mathcal{M}\left( \mathcal{G}\right)$.

There is one important instance in which
the graphical criterion \eqref{eq_main_independence} can be significantly simplified. Specifically,
first note that $W_{J}=O_{T}$. This holds because, since $\irrel(A,Y,\mathcal{G})=\emptyset
$, there exists a
directed path from $W_{J}$ to $Y$ that does not intersect $A.$ Let $W$
be a child of $W_{J}$ in that path. Then $W$ cannot be in the set $\left\{
W_{1},\dots,W_{J}\right\} $ because $W_{J}$ is the last element in the
topolocally ordered sequence $W_{1},\dots,W_{J}$ of non-descendants of $A.$
Then $W\in \mathbf{M}\cup \left\{ Y\right\} $ which implies that $%
W_{J}\in \mathbf{O}$ and, since $\left( O_{1},\dots,O_{T}\right) $ is ordered
topologically, we conclude that $W_{J}=O_{T}.$ 
Suppose now that 
\begin{equation}
\mathbf{O\backslash }O_{T}\subset \pa_{\mathcal{G}}\left( W_{J}\right) .
\label{eq:padres_O_T}
\end{equation}%
Lemma \ref{lemma:parents_nondesc} in the Appendix establishes that, under $\left( \ref%
{eq:padres_O_T}\right) ,$ the criterion $\left( \ref{eq_main_independence}%
\right) $ holds for $j=J-1$ if and only if $\pa_{\mathcal{G}}\left(
W_{J}\right) \backslash \left\{ W_{J-1}\right\} \subset \pa_{\mathcal{G}%
}\left( W_{J-1}\right) .$ Furthermore, the lemma also establishes that
if for some $1<j^{\ast }\leq J-1$ 
\begin{equation}
\pa_{\mathcal{G}}\left( W_{j+1}\right)
\subset \pa_{\mathcal{G}}\left( W_{j}\right)  \cup  \left\{ W_{j}\right\} \label{eq_inclusion_padres}
\end{equation}
is valid for $j\in \left\{ j^{\ast },\dots,J-1\right\} ,$ then $\left( \ref%
{eq_main_independence}\right) $ holds for $j\in \left\{ j^{\ast
},\dots,J-1\right\} $, and in addition, $\left( \ref{eq_main_independence}%
\right) $ and $\left( \ref{eq_inclusion_padres}\right) $ are equivalent for $%
j=j^{\ast }-1.$ Note that whereas $\left( \ref{eq_main_independence}\right) $
requires checking d-separations, $\left( \ref{eq_inclusion_padres}\right) $
requires simply checking the inclusion of sets. Interestingly, we show in
the proof of Theorem \ref{theo:main} that the validity of $\left( \ref{eq:padres_O_T}%
\right) $ and of $\left( \ref{eq_inclusion_padres}\right) $ for all $j\in
\left\{ 1,\dots,J-1\right\} $ is a necessary condition for \eqref{eq:main_id}. 

Aside from the implications for term
cancellations, note that when $\left( \ref{eq:padres_O_T}\right) $ holds
\begin{equation*}
E_{P}\left[ b_{a}\left( \mathbf{O;}P\right) |W_{J},\pa_{\mathcal{G}}\left( W_{J}\right) %
\right] =b_{a}\left( \mathbf{O;}P\right).
\end{equation*}

Steps 9-26 of Algorithm \ref{algo:main} implement the preceding checks. Specifically,
step 9 inquires if both $J> 1$ and $\left( \ref{eq:padres_O_T}\right) $ hold. If $J> 1$ but \eqref{eq:padres_O_T} does not hold, then the algorithm goes on to inquire for each $j\in \left\{
1,\dots,J-1\right\} $, if $\left( \ref{eq_main_independence}\right) $ holds (see Algorithm \ref{algo:offenders_nondesc})
and it stores the formula 
\begin{eqnarray*}
\chi _{P,a,eff}^{1,non-desc}\left( \mathbf{V};\mathcal{G}\right)  &=&E_{P}%
\left[ b_{a}\left( \mathbf{O;}P\right) |W_{J},\pa_{\mathcal{G}}\left(
W_{J}\right) \right] -\chi _{a}\left( P;\mathcal{G}\right)  \\
&&+\underset{\left( \ref{eq_main_independence}\right) \text{ does not hold}}{%
\sum_{j\in \left\{ 1,2,\dots,J-1\right\} :}}\left\{ E_{P}\left[ b_{a}\left( 
\mathbf{O;}P\right) |\pa_{\mathcal{G}}\left( W_{j}\right) ,W_{j}\right]
-E_{P}\left[ b_{a}\left( \mathbf{O;}P\right) |\pa_{\mathcal{G}}\left(
W_{j+1}\right) \right] \right\} .
\end{eqnarray*}%
If both $J>  1$ and $\left( \ref{eq:padres_O_T}\right) $ hold, then iteratively in reverse
order from $j=J-1,$ the algorithm inquires if $\left( \ref{eq_inclusion_padres}\right) $
holds until the first $j$, if any, such that the inclusion \eqref{eq_inclusion_padres} fails. If such $j,
$ say $j=j^{\ast }$ exists, $j^{\ast }$ is necessarily greater than $1$
because of the topological order of $\mathbf{W}$ and the fact that $\irrel(A,Y,\mathcal{G})=\emptyset$. Then the algorithm
inquires for each $j\in \left\{ 1,\dots,j^{\ast }-1\right\} $ if $\left( \ref%
{eq_main_independence}\right) $ holds and it stores the formula%
\begin{eqnarray}
\chi _{P,a,eff}^{1,non-desc}\left( \mathbf{V};\mathcal{G}\right) 
&=&b_{a}\left( \mathbf{O;}P\right) -\chi _{a}\left( P;\mathcal{G}\right)
+\left\{ E_{P}\left[ b_{a}\left( \mathbf{O;}P\right) |\pa_{\mathcal{G}%
}\left( W_{j^{\ast }}\right) ,W_{j^{\ast }}\right] -E_{P}\left[ b_{a}\left( 
\mathbf{O;}P\right) |\pa_{\mathcal{G}}\left( W_{j^{\ast }+1}\right) \right]
\right\} \nonumber \\
&&+\underset{\left( \ref{eq_main_independence}\right) \text{ does not hold}}{%
\sum_{j\in \left\{ 1,2,\dots,j^{\ast }-1\right\} :}}\left\{ E_{P}\left[
b_{a}\left( \mathbf{O;}P\right) |\pa_{\mathcal{G}}\left( W_{j}\right) ,W_{j}%
\right] -E_{P}\left[ b_{a}\left( \mathbf{O;}P\right) |\pa_{\mathcal{G}%
}\left( W_{j+1}\right) \right] \right\} .
\label{eq:chi1_non}
\end{eqnarray}
If $\left( \ref{eq_inclusion_padres}\right) $ holds for all $j\in\left\{
1,\dots,J-1\right\} $ for $J>1$ or if $J=1$ then the algorithm stores the formula%
\begin{equation}
\chi _{P,a,eff}^{1,non-desc}\left( \mathbf{V};\mathcal{G}\right)
=b_{a}\left( \mathbf{O;}P\right) -\chi _{a}\left( P;\mathcal{G}\right) .
\label{eq:formula_arriba_reducida}
\end{equation}
Otherwise if $J=0$ it stores $\chi _{P,a,eff}^{1,non-desc}\left( \mathbf{V};\mathcal{G}\right)=0$.

Importantly the expression \eqref{eq:chi1_non} for $\chi _{P,a,eff}^{1,non-desc}\left( \mathbf{V};\mathcal{G}\right)$ does not depend on the variables $\lbrace W_{j^{\ast}+1},\dots, W_{J} \rbrace \setminus \mathbf{O}$. Since the expression for $\chi _{P,a,eff}^{1,desc}\left( \mathbf{V};\mathcal{G}\right)$ does not depend on these variables then we conclude that  $\lbrace W_{j^{\ast}+1},\dots, W_{J} \rbrace \setminus \mathbf{O}$ do not enter into the formula for $\chi _{P,a,eff}^{1}\left( \mathbf{V};\mathcal{G}\right)$ and consequently do not provide information about the parameter $\chi_{a}(P;\mathcal{G})$. We emphasize that this is important from a practical standpoint because if the algorithm returns expression \eqref{eq:chi1_non}, then the investigator does not need to measure these variables. A similar comment applies if the algorithm returns expression \eqref{eq:formula_arriba_reducida}. 

\begin{remark}\label{remark:ex}
In Lemma \ref{lemma:parents_nondesc} of the Appendix we show that for $J>1$, whenever \eqref{eq:padres_O_T} holds and \eqref{eq_inclusion_padres} holds for all $j\in\lbrace 1, \dots, J-1 \rbrace$ then 
\begin{equation}
W_{j}\in \pa_{\mathcal{G}}(W_{j+1}) \text{ for all } j\in\lbrace 1, \dots, J-1 \rbrace.
\label{eq:chain_padres}
\end{equation}
Consequently, for $J>1$, \eqref{eq:chain_padres} is necessary for \eqref{eq:formula_arriba_reducida} to hold.
\end{remark}

\begin{figure}[ht]
\begin{center}
\begin{tikzpicture}[>=stealth, node distance=1.5cm,
pre/.style={->,>=stealth,ultra thick,line width = 1.4pt}]
  \begin{scope}
    \tikzstyle{format} = [circle, inner sep=2.5pt,draw, thick, circle, line width=1.4pt, minimum size=6mm]
    \node[format] (W1) {$W_1$};
    \node[format, right of=W1] (W2) {$W_2$};
    \node[format, right of=W2] (W3) {$W_3$};
    \node[format, right of=W3] (W4) {$W_4$};
    \node[format, right of=W4] (O) {$O$};
    \node[format, below of=W3] (A) {$A$};
    \node[format, below of=A] (Y) {$Y$};
                 \draw (W1) edge[pre, black] (W2);
                 \draw (W2) edge[pre, black] (W3);
                 \draw (W3) edge[pre, black] (W4);
                 \draw (W4) edge[pre, black] (O);
                 \draw (W1) edge[pre, orange] (A);
                 \draw (W2) edge[pre, orange] (A);
                 \draw (W3) edge[pre, orange] (A);
                 \draw (W4) edge[pre, orange] (A);
                 \draw (O) edge[pre, orange] (A);
                 \draw (O) edge[pre, black] (Y);
                 \draw (A) edge[pre, black] (Y);
                 \draw (W2) edge[pre, black, out=45,in=135] (O);
  \end{scope} 
  \end{tikzpicture}
\end{center}
\caption{A DAG where the NP-$\mathbf{O}$ estimator is inefficient.}
\label{fig:ex_nondesc}
\end{figure}
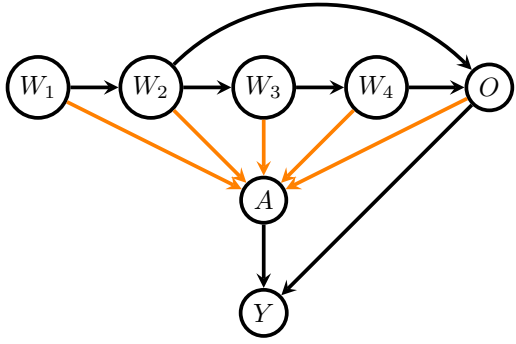

\begin{example}\label{ex:algorithm_nondesc}
Consider the DAG in Figure \ref{fig:ex_nondesc}. In this DAG, $\mathbf{O}=\mathbf{O}_{min}=\lbrace O_{T}\rbrace=\lbrace O\rbrace\equiv\lbrace W_{5}\rbrace$ with $T=1$. Therefore condition \eqref{eq:padres_O_T} holds trivially. However, condition \eqref{eq_inclusion_padres} with $j=4$ fails, because $W_2$ is a parent of $O$ but not of $W_4$. The algorithm now goes on to check condition \eqref{eq_main_independence} for each $j=1,2,3,4$. The following table lists the results.
\begin{table}[h!]
\label{tab:ex:algorithm_nondesc}
\centering
\begin{tabular}{ccccc}
 $j$ & $\mathbf{I}_{j} $ & $\left[\pa_{\mathcal{G}}(W_{j})\cup \lbrace W_{j}\rbrace \bigtriangleup \pa_{\mathcal{G}}(W_{j+1})\right]$ & $\mathbf{O} \setminus \mathbf{I}_{j}$ & \eqref{eq_main_independence}  \\ \hline 
 $1$& $W_1$ & $\emptyset$  & $O$  & holds \\
 $2$& $W_{2}$ & $W_{1}$ & $O$ & holds\\
 $3$&  $W_{3}$&  $W_{2}$& $O$ & fails \\
 $4$&  $W_{4}$&  $\lbrace W_{2}, W_{3}\rbrace$& $O$ & fails \\
\end{tabular}
\end{table}
\FloatBarrier
The algorithm then stores the formula
\begin{align*}
\chi^{1,non-desc}_{P,a,eff}(\mathbf{V};\mathcal{G})&= b_{a}(O;P)-\chi_{a}(P;\mathcal{G}) +
E_{P}\left[b_{a}(O;P)\mid W_{4}, \pa_{\mathcal{G}}(W_{4}) \right]-E_{P}\left[b_{a}(O;P)\mid \pa_{\mathcal{G}}(O) \right]  \\
&+E_{P}\left[b_{a}(O;P)\mid W_{3}, \pa_{\mathcal{G}}(W_{3}) \right] - E_{P}\left[b_{a}(O;P)\mid W_{3} \right] 
\\
&=b_{a}(O;P)-\chi_{a}(P;\mathcal{G})+E_{P}\left[b_{a}(O;P)\mid W_{3}, W_{4} \right]-E_{P}\left[b_{a}(O;P)\mid W_{2}, W_{4} \right] 
\\
&+ E_{P}\left[b_{a}(O;P)\mid W_{2}, W_{3} \right] - E_{P}\left[b_{a}(O;P)\mid W_{3}\right].
\end{align*}
This example illustrates the following interesting points. 

For $j=2$ the d-separation \eqref{eq_main_independence} holds and consequently the term \eqref{eq_differences} vanishes from the expression for $\chi^{1,non-desc}_{P,a,eff}$. However, $W_{2}$ appears in the expression for $\chi^{1,non-desc}_{P,a,eff}$ and therefore it appears also in the expression for the efficient influence  function  $\chi^{1}_{P,a,eff}$. Thus, $W_{2}$ provides information about $\chi_{a}(P;\mathcal{G})$ even though the term \eqref{eq_differences} vanishes for $j=2$. In contrast, for $j=1$ term  \eqref{eq_differences} vanishes and $W_{1}$ does not enter into the expression for $\chi^{1,non-desc}_{P,a,eff}$. This illustrates the point that once condition \eqref{eq_inclusion_padres} fails, the check of the d-separation condition \eqref{eq_main_independence} is useful for detecting term cancellations but not for deciding if the corresponding node is informative about the parameter $\chi_{a}(P;\mathcal{G})$. On the other hand, the next example illustrates the point made earlier that whenever condition \eqref{eq_inclusion_padres} holds for a given $j$, say $j=j^{\ast}$, and for all subsequent $j$, that is, for all $j=j^{\ast}+1,\dots, J-1$, then $W_{j}$ does not appear in the expression for $\chi^{1,non-desc}_{P,a,eff}$ and therefore is not informative. 

Another interesting point illustrated by this example is that the composition of the set $\pa_{\mathcal{G}}(A)$ does not affect the expression for $\chi^{1,non-desc}_{P,a,eff}$. That is, all or a subset of the orange edges could have been absent in the DAG and nevertheless the expression for $\chi^{1,non-desc}_{P,a,eff}$ would have remained the same. However, which elements of $\mathbf{W}$ are members of the set $\pa_{\mathcal{G}}(A)$ does affect the composition of the minimal optimal adjustment set $\mathbf{O}_{min}$. For instance in the DAG of Figure \ref{fig:ex_nondesc}, $\mathbf{O}_{min}=\mathbf{O}$. Instead, if all the orange arrows had been absent, then $\mathbf{O}_{min}$ would have been empty.

We will analyse the expression for $\chi^{1,desc}_{P,a,eff}$ in Example \ref{ex:nondesc2_cont}.
\end{example}

\begin{figure}[ht!]
\begin{center}
\begin{tikzpicture}[>=stealth, node distance=1.5cm,
pre/.style={->,>=stealth,ultra thick,line width = 1.4pt}]
  \begin{scope}
    \tikzstyle{format} = [circle, inner sep=2.5pt,draw, thick, circle, line width=1.4pt, minimum size=6mm]
    \node[format] (O1) {$O_1$};
    \node[format, right of= O1] (B1) {$B_1$};
    \node[format, right of= B1] (B2) {$B_2$};
    \node[format, right of= B2] (B3) {$B_3$};
    \node[format, right of= B3] (B4) {$B_4$};
    \node[format, right of= B4] (O2) {$O_2$};
    \node[format, right of= O2] (B5) {$B_5$};
    \node[format, right of= B5] (O3) {$O_3$};
      \coordinate (C) at ($(B3)!0.5!(B4)$);
    \node (A) [format, below of= C] {$A$};
    \node[format, below of= A] (M1) {$M_1$};
    \node[format, below of= M1] (M2) {$M_2$};
    \node[format, below of= M2] (M3) {$M_3$};
    \node[format, below of= M3] (Y) {$Y$};
    
    \draw (O1) edge[pre, black] (B1);
    \draw (B1) edge[pre, black] (B2);
    \draw (B2) edge[pre, black] (B3);
    \draw (B3) edge[pre, black] (B4);
    \draw (B4) edge[pre, black] (O2);
    \draw (O2) edge[pre, black] (B5);
    \draw (B5) edge[pre, black] (O3);
    
    \draw (O1) edge[pre, blue, out=25, in=130] (B2);
    \draw (O1) edge[pre, blue, out=30, in=130] (B3);
    \draw (O1) edge[pre, blue, out=35, in=130] (B4);
    \draw (O1) edge[pre, blue, out=40, in=130] (O2);
    \draw (O1) edge[pre, blue, out=45, in=130] (B5);
    \draw (O1) edge[pre, blue, out=50, in=130] (O3);
    \draw (O2) edge[pre, blue, out=35, in=145] (O3);
    
    \draw (B1) edge[pre, red, out=330, in=220] (B3);
    \draw (B1) edge[pre, red, out=330, in=210] (B4);
    
    \draw (B2) edge[pre, orange] (A);
    \draw (B4) edge[pre, orange] (A);
    
    \draw (O2) edge[pre, black] (M1);
    \draw (O2) edge[pre, black] (M2);
    \draw (O2) edge[pre, black] (M3);
    \draw (O2) edge[pre, black] (Y);
    
    \draw (O1) edge[pre, black] (M1);
    \draw (O1) edge[pre, black] (M2);
    \draw (O1) edge[pre, black] (M3);
    \draw (O1) edge[pre, black] (Y);
    
    \draw (O3) edge[pre, black] (M1);
    \draw (O3) edge[pre, magenta] (M2);
    
    \draw (A) edge[pre, black] (M1);
    \draw (A) edge[pre, black, out=200, in=120] (M2);
    \draw (A) edge[pre, black, out=200, in=110] (M3);
    \draw (A) edge[pre, black, out=200, in=150] (Y);
    
    \draw (M1) edge[pre, black] (M2);
    \draw (M2) edge[pre, black] (M3);
    \draw (M3) edge[pre, black] (Y);
    
    \draw (M1) edge[pre, green, out=360, in=360] (M3);
  \end{scope} 
  \end{tikzpicture}
\end{center}
\caption{A DAG where the NP-$\mathbf{O}$ estimator is efficient.}
\label{fig:ex_nondesc_large}
\end{figure}
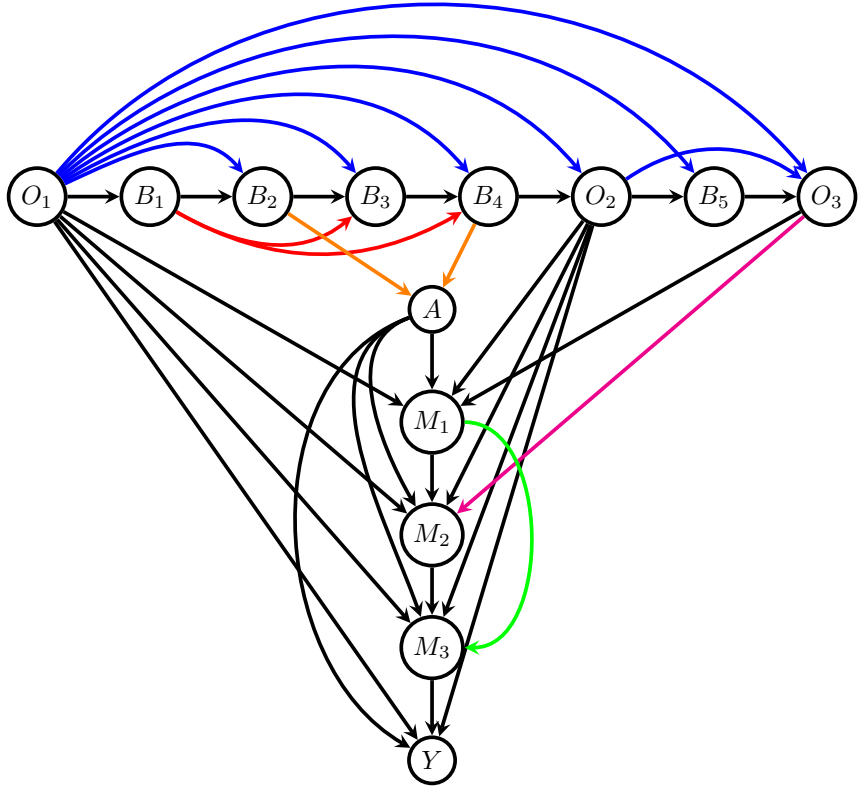
\FloatBarrier
\begin{example}\label{ex:algorithm_nondesc2}
Consider the DAG in Figure \ref{fig:ex_nondesc_large}. In this DAG, $J=8$, $T=3$, 
$$
\mathbf{W}=\left( W_{1},W_{2},W_{3},W_{4},W_{5},W_{6},W_{7}, W_{8}\right)=\left( O_{1}, B_{1}, B_{2}, B_{3},B_{4}, O_{2}, B_{5}, O_{3}\right),
$$
$\mathbf{O}=\lbrace O_1, O_2, O_3\rbrace$ and  $\mathbf{O}_{min}=\lbrace O_1, O_2\rbrace$. Condition \eqref{eq:padres_O_T} holds and \eqref{eq_inclusion_padres} holds for $j=1,\dots, J-1$. Therefore the algorithm stores the formula $\chi _{P,a,eff}^{1,non-desc}\left( \mathbf{V};\mathcal{G}\right)
=b_{a}\left( \mathbf{O;}P\right) -\chi _{a}\left( P;\mathcal{G}\right)
$. Note that the blue and black arrows are necessary for the condition \eqref{eq_inclusion_padres} to hold for all $j$. 
Specifically, because \eqref{eq:padres_O_T} holds, then for \eqref{eq_inclusion_padres} to hold for all $j$, each $O$ in $\mathbf{O}_{min}$ must be a parent of all the $W_{j}$s that follow it in the topological order. On the other hand, also because \eqref{eq:padres_O_T} holds, for \eqref{eq_main_independence} to hold for all $j$, each $W_j \in \mathbf{W}\setminus \mathbf{O}$ must satisfy that if it is a parent of a node $W_{j^{\prime}}$ then it must also be a parent of all nodes $W_{j+1},\dots, W_{j^{\prime}-1}$. For instance $B_{1} \in\mathbf{W}\setminus \mathbf{O} $ is a parent of $B_{4}$ and is also a parent of $B_{2}, B_{3}$.  Note also that while the requirement that each $O$ is a parent of all the subsequent $W_{j}$s in the topological order is necessary for \eqref{eq_inclusion_padres} to hold it is not necessary that each $W_{j}\in\mathbf{W}\setminus\mathbf{O}$ be a parent of \textit{all} the subsequent nodes in the topological order. For instance, $B_{1}$ is not a parent of $O_{2}$.

One again we emphasize that edges from $\mathbf{W}$ to $A$, in orange in the DAG of Figure \ref{fig:ex_nondesc_large}, are irrelevant for finding simplifications for $\chi^{1,non-desc}_{P,a,eff}$. However, they are relevant for determining which $O_{j}$s are members of $\mathbf{O}_{min}$. As we will see next the composition of $\mathbf{O}_{min}$ is important for determining possible simplifications of $\chi^{1,desc}_{P,a,eff}$.
\end{example}

Having checked for possible simplifications of the expression of $\chi^{1,non-desc}_{P,a,eff}$, Algorithm  \ref{algo:main} goes on to check for possible simplifications of $\chi _{P,a,eff}^{1,desc}\left( \mathbf{V};\mathcal{G}\right)$.
For simplicity in what follows we define $M_{K+1}\equiv Y$.

Suppose first that $K=0$. Then, the definition of $\mathbf{O}$ and the assumption that $A\in\an_{\mathcal{G}}(Y)$ imply that
$\lbrace A\rbrace\cup \mathbf{O} = \pa_{\mathcal{G}}(Y)$. Then $E_{P}\left[T_{P,a,\mathcal{G}} \mid Y, \pa_{\mathcal{G}}(Y)\right]=T_{P,a,\mathcal{G}}$ and 
$E_{P}\left[T_{P,a,\mathcal{G}}\mid\pa_{\mathcal{G}}(Y)\right]=I_{a}(A)b_{a}(\mathbf{O};P)\pi^{-1}_{a}(\mathbf{O};P)$. Consequently
$$
\chi^{1,desc}_{P,a,eff}(\mathbf{V};\mathcal{G})=\frac{I_{a}(A)}{\pi_{a}(\mathbf{O}_{min};P) }\left\{ Y-b_{a}\left( \mathbf{O;}%
P\right) \right\}.
$$

Suppose next that $K\geq 1$. If for some $k\in \left\{
2,\dots,K+1\right\} ,$ it holds that 
\begin{equation}
\left\{ A\right\} \cup \mathbf{O}_{\min }\subset \pa_{\mathcal{G}}\left(
M_{k}\right) \label{eq:inclusion_1}
\end{equation}%
then
\begin{align}
E_{P}\left[ T_{P,a,\mathcal{G}}|M_{k},\text{pa}_{\mathcal{G}}\left(
M_{k}\right) \right] =
\frac{I_{a}(A)}{\pi_{a}(\mathbf{O}_{min};P)} E_{P}\left[ Y|M_{k},\text{pa}_{\mathcal{G}}\left(
M_{k}\right) \right].
\label{eq:IPW}
\end{align}
Note that for $k=K+1,$ $\left( \ref{eq:IPW}\right) $ is equal to
\begin{equation}
E_{P}\left[ T_{P,a,\mathcal{G}}|Y,\text{pa}_{\mathcal{G}}\left( Y\right) %
\right] =\frac{I_{a}(A)Y}{\pi_{a}(\mathbf{O}_{min};P)}.
\label{eq:ipw_1}
\end{equation}
Note that if \eqref{eq:inclusion_1} does not hold for $k=K+1$ then there exists $P^{\ast}\in\mathcal{M(G)}$ such that \eqref{eq:ipw_1} does not hold, because the definition of $\mathbf{O}_{min}$ implies that there exists $P^{\ast}\in\mathcal{M(G)}$ such that the right hand side of \eqref{eq:ipw_1} is a non-trivial function of $A$ and $\mathbf{O}_{min}$. Note also that the influence function of the NP-$\mathbf{O}$ estimator includes the term on the right hand side of \eqref{eq:ipw_1}. Because such term cannot appear in the expression for $\chi^{1}_{P^{\ast},a,eff}$ since $Y$ does not appear in any of the remaining terms in the expression for $\chi^{1}_{P^{\ast},a,eff}$ then for such $P^{\ast}$, $\psi_{P^{\ast},a}\left( \mathbf{O};\mathcal{G} \right) \neq \chi_{P^{\ast},a,eff}^{1}(\mathbf{V};P^{\ast})$ and consequently the 
asymptotic variance of the NP-$\mathbf{O}$ estimator does not achieve the semiparametric Cramer-Rao bound at $P^{\ast}$.

Now, suppose that, for some $k\in \left\{ 2,\dots,K+1\right\} ,$ in addition
to $\left( \ref{eq:inclusion_1}\right) $ it holds that
\begin{equation}
\pa_{\mathcal{G}}\left( M_{k}\right) \subset \pa_{\mathcal{G}}\left(
M_{k-1}\right) \cup \left\{ M_{k-1}\right\}  \label{eq:inclusion_2}
\end{equation}
and 
\begin{equation}
Y\ort_{\mathcal{G}}\left[ M_{k-1},\text{pa}_{\mathcal{G}}\left(
M_{k-1}\right) \right] \backslash \text{pa}_{\mathcal{G}}\left( M_{k}\right)
|\text{pa}_{\mathcal{G}}\left( M_{k}\right)  . \label{eq:d_sep_mediators}
\end{equation}%
Then, for such $k$ 
\begin{eqnarray*}
E_{P}\left[ T_{P,a,\mathcal{G}}|M_{k-1},\text{pa}_{\mathcal{G}}\left(
M_{k-1}\right) \right]  &=&\frac{I_{a}(A)}{\pi_{a}(\mathbf{O}_{min};P)} E_{P}\left[ Y|M_{k-1},\text{pa}_{\mathcal{G}}\left(
M_{k-1}\right) \right]  \\
&=&\frac{I_{a}(A)}{\pi_{a}(\mathbf{O}_{min};P)}E_{P}\left[ \left.
Y\right\vert \text{pa}_{\mathcal{G}}\left( M_{k}\right) ,\left[ M_{k-1},%
\text{pa}_{\mathcal{G}}\left( M_{k-1}\right) \right] \backslash \text{pa}_{%
\mathcal{G}}\left( M_{k}\right) \right]  \\
&=&\frac{I_{a}(A)}{\pi_{a}(\mathbf{O}_{min};P)}E_{P}\left[ \left.
Y\right\vert \text{pa}_{\mathcal{G}}\left( M_{k}\right) \right]  \\
&=&E_{P}\left[ T_{P,a,\mathcal{G}}|\text{pa}_{\mathcal{G}}\left(
M_{k}\right) \right] 
\end{eqnarray*}%
where the first equality follows  from \eqref{eq:IPW}, the second from \eqref{eq:inclusion_2} , the third from \eqref{eq:d_sep_mediators} and the fourth from \eqref{eq:inclusion_1}. We
therefore arrive at the conclusion that $\left( \ref{eq:inclusion_1}\right)
,\left( \ref{eq:inclusion_2}\right) $ and $\left( \ref{eq:d_sep_mediators}%
\right) $ imply that the difference 
\begin{equation}
E_{P}\left[ T_{P,a,\mathcal{G}}|M_{k-1},\text{pa}_{\mathcal{G}}\left(
M_{k-1}\right) \right] -E_{P}\left[ T_{P,a,\mathcal{G}}|\text{pa}_{\mathcal{G%
}}\left( M_{k}\right) \right] 
\label{eq:term_med}
\end{equation}%
vanishes from the expression for $\chi _{P,a,eff}^{1,desc}\left(\mathbf{V};\mathcal{G}\right)$ for all $P\in\mathcal{M(G)}$.

In analogy to the examination of the terms in the expression for  $\chi _{P,a,eff}^{1,non-desc}\left(\mathbf{V};\mathcal{G}\right)$
there exists an instance in which the d-separation criterion $\left( \ref%
{eq:d_sep_mediators}\right) $ can be simplified to a condition just involving
set inclusions. Specifically, in Lemma \ref{lemma:indep_med} of the Appendix we show
that if there exists
some $k^{\ast }\in \left\{ 2,\dots,K+1\right\} $ such that $\left( \ref{eq:inclusion_2}%
\right) $ holds for all $k\in \left\{ k^{\ast },k^{\ast }+1,\dots,K+1\right\} $
then, $\left( \ref{eq:d_sep_mediators}\right) $ holds for $k\in \left\{
k^{\ast },k^{\ast }+1,\dots,K+1\right\}$. In particular, this implies that if 
\eqref{eq:inclusion_1} holds for $k=K+1$ and \eqref{eq:inclusion_2} holds for all $k\in\lbrace k^{\ast},\dots,K+1\rbrace$ then the term \eqref{eq:term_med} vanishes for all $k\in\lbrace k^{\ast},\dots,K+1\rbrace$ and for all $P\in\mathcal{M(G)}$.
Furthermore, if such $k^{\ast}$ is strictly greater than 2, condition \eqref{eq:inclusion_1} holds for $k=K+1$, and condition \eqref{eq:inclusion_2} fails for $k=k^{\ast}-1$, then in parts 2) and 3) of Lemma \ref{lemma:main_Z} in the Appendix we show that \eqref{eq:term_med} not only does not vanish for all $P\in\mathcal{M(G)}$ but also there exists $P^{\ast}\in\mathcal{M(G)}$ such that \eqref{eq:term_med} is a non-constant function of $M_{k-1}$. Since $M_{k-1}$ does not appear in any of the remaining non-vanishing terms of $\chi_{P^{\ast},a,eff}^{1,desc}$ nor it appears in the expression for $\chi_{P^{\ast},a,eff}^{1,non-desc}$, then we conclude that for such $P^{\ast}\in\mathcal{M(G)}$, $\chi_{P^{\ast},a,eff}^{1}$ depends on $M_{k-1}$.
Because the influence function $\psi_{P^{\ast},a}\left( \mathbf{O};\mathcal{G} \right)$ of the NP-$\mathbf{O}$ estimator does not depend on $M_{k-1}$ this immediately implies that $\psi_{P^{\ast},a}\left( \mathbf{O};\mathcal{G} \right) \neq \chi_{P^{\ast},a,eff}^{1}(\mathbf{V};P^{\ast})$ and consequently the 
asymptotic variance of the NP-$\mathbf{O}$ estimator does not achieve the semiparametric Cramer-Rao bound at $P^{\ast}$. We therefore have the following important result.

\begin{proposition}\label{prop:pa_M}
If $K\geq 1$, condition \eqref{eq:inclusion_1} for $k=K+1$ and condition \eqref{eq:inclusion_2} for $k\in\lbrace 2,\dots,K+1 \rbrace$ are necessary for \eqref{eq:main_id} to hold for all $P\in\mathcal{M(G)}$.
\end{proposition}

\begin{remark}\label{remark:chain_med}
In Lemma \ref{lemma:indep_med} we show that if $K\geq 1$, whenever condition \eqref{eq:inclusion_2}
holds for $k\in\lbrace 2,\dots,K+1\rbrace$ then $A\in\pa_{\mathcal{G}}(M_1)$ and $M_{k}\in \pa_{\mathcal{G}}(M_{k+1})$ for $k\in\lbrace 2,\dots,K+1\rbrace$. Consequently, by Proposition \ref{prop:pa_M}, if $K\geq 1$, $A\in\pa_{\mathcal{G}}(M_1)$ and $M_{k}\in \pa_{\mathcal{G}}(M_{k+1})$ for $k\in\lbrace 2,\dots,K+1\rbrace$ are necessary for \eqref{eq:main_id} to hold.
\end{remark}

Aside from the examination of term cancellations, we note that if 
\begin{equation}
\pa_{\mathcal{G}}\left( M_{1}\right) =\left\{ A\right\} \cup \mathbf{O}\text{
}  \label{eq:padres_M1}
\end{equation}%
holds, then 
\begin{equation*}
E_{P}\left[ T_{P,a,\mathcal{G}}|\text{pa}_{\mathcal{G}}\left( M_{1}\right) %
\right] =\frac{I_{a}(A)}{\pi_{a}(\mathbf{O}_{min};P) }b_{a}\left( \mathbf{%
O;}P\right) .
\end{equation*}

Steps 27-50 of Algorithm \ref{algo:main} implement the preceding checks. 
Specifically, step 27 inquires if both $K\geq1$ and \eqref{eq:inclusion_1} hold for $k=K+1$.
If $K\geq1$ but \eqref{eq:inclusion_1} does not hold for $k=K+1$, the algorithm goes on to inquire for each $k\in\lbrace 2,\dots, K\rbrace$ if $\left( \ref{eq:inclusion_1}\right)
,\left( \ref{eq:inclusion_2}\right) $ and $\left( \ref{eq:d_sep_mediators}%
\right) $ hold and subsequently if \eqref{eq:padres_M1} holds. It then stores the formula 
\begin{eqnarray*}
\chi _{P,a,eff}^{1,desc}\left( \mathbf{V};\mathcal{G}\right)  &=&E_{P}\left[
T_{P,a,\mathcal{G}}|Y,\pa_{\mathcal{G}}\left( Y\right) \right] -E_{P}\left[
T_{P,a,\mathcal{G}}|\pa_{\mathcal{G}}\left( Y\right) \right] +E_{P}\left[
T_{P,a,\mathcal{G}}|\pa_{\mathcal{G}}\left( M_{K}\right) ,M_{K}\right] -\chi
_{P,a,eff}^{1,M_{1}}\left( \mathbf{V};\mathcal{G}\right)  \\
&&+\sum\limits_{k\in \texttt{offenders\_desc}(K)}\left\{ E_{P}\left[ T_{P,a,\mathcal{G}}|\pa_{%
\mathcal{G}}\left( M_{k-1}\right) ,M_{k-1}\right] -E_{P}\left[ T_{P,a,%
\mathcal{G}}|\pa_{\mathcal{G}}\left( M_{k}\right) \right] \right\} 
\end{eqnarray*}%
where
$$
\chi _{P,a,eff}^{1,M_{1}}\left( \mathbf{V};\mathcal{G}\right) \equiv
\begin{cases}\frac{I_{a}(A)}{%
\pi_{a}\left( \mathbf{O}_{\min ;P}\right) }b_{a}\left( \mathbf{O;}P\right) &\text{if \eqref{eq:padres_M1} holds} \\
E_{P}\left[
T_{P,a,\mathcal{G}}|\text{pa}_{\mathcal{G}}\left( M_{1}\right) \right] &\text{if \eqref{eq:padres_M1} does not hold}
\end{cases}
$$
and for any $h \in \lbrace 2,\dots, K \rbrace$
$$
\texttt{offenders\_desc}(h)\equiv\left\lbrace k \in \lbrace 2,\dots,h\rbrace: \text{at least one of }\eqref{eq:inclusion_1}, \: \eqref{eq:inclusion_2} \text{ or } \eqref{eq:d_sep_mediators} \text{ does not hold} \right\rbrace.
$$
See Algorithm \ref{algo:offenders_desc}.
If $K\geq 1$ and \eqref{eq:inclusion_1} holds for $k=K+1$ then iteratively in reverse order from $k=K+1$ the algorithm inquires if $\eqref{eq:inclusion_2}$ holds until the first $k\geq 2$, if any, in which the condition fails. If such $k$, say $k=k^{\ast }$ exists and $k^{\ast}>2$, then it inquires for each $k\in \left\{ 2,\dots,k^{\ast
}-1\right\} $ if $\left( \ref{eq:inclusion_1}\right) ,\left( \ref%
{eq:inclusion_2}\right) $ and $\left( \ref{eq:d_sep_mediators}\right) $ hold, 
and subsequently if \eqref{eq:padres_M1} holds. It then stores the formula%
\begin{eqnarray}
\chi _{P,a,eff}^{1,desc}\left( \mathbf{V};\mathcal{G}\right)  &=&\frac{I_{a}(A)Y}{%
\pi_{a}(\mathbf{O}_{min};P) }-\chi _{P,a,eff}^{1,M_{1}}\left( 
\mathbf{V};\mathcal{G}\right) \nonumber\\
&&+ E_{P}\left[ T_{P,a,\mathcal{G}}|\pa_{%
\mathcal{G}}\left( M_{k^{\ast }-1}\right) ,M_{k^{\ast }-1}\right] -E_{P}%
\left[ T_{P,a,\mathcal{G}}|\pa_{\mathcal{G}}\left( M_{k^{\ast }}\right) %
\right] 
\nonumber
\\
&&+\sum\limits_{k\in \texttt{offenders\_desc}(k^{\ast}-1)}\left\{ E_{P}\left[ T_{P,a,\mathcal{%
G}}|\pa_{\mathcal{G}}\left( M_{k-1}\right) ,M_{k-1}\right] -E_{P}\left[
T_{P,a,\mathcal{G}}|\pa_{\mathcal{G}}\left( M_{k}\right) \right] \right\}.
\label{eq:chi1_desc}
\end{eqnarray}%
Notice that in a similar fashion as for the expression \eqref{eq:chi1_non} for $\chi _{P,a,eff}^{1,non-desc}\left( \mathbf{V};\mathcal{G}\right)$, the expression \eqref{eq:chi1_desc} does not depend on the variables $M_{k^{\ast}},\dots, M_{K}$. Since the expression for $\chi _{P,a,eff}^{1,non-desc}$ does not depend on these variables, we conclude that $M_{k^{\ast}},\dots, M_{K}$ do not enter into the formula for $\chi _{P,a,eff}^{1}\left( \mathbf{V};\mathcal{G}\right)$ and consequently do not provide information about the parameter $\chi_{a}(P;\mathcal{G})$.

If $k^{\ast }=2$, then it stores 
\begin{equation}
\chi _{P,a,eff}^{1,desc}\left( \mathbf{V};\mathcal{G}\right) =\frac{I_{a}(A)Y}{%
\pi_{a}(\mathbf{O}_{min};P) }-\chi _{P,a,eff}^{1,M_{1}}\left( 
\mathbf{V};\mathcal{G}\right) +\left\{ E_{P}\left[ T_{P,a,\mathcal{G}}|\pa_{%
\mathcal{G}}\left( M_{k^{\ast }-1}\right) ,M_{k^{\ast }-1}\right] -E_{P}%
\left[ T_{P,a,\mathcal{G}}|\pa_{\mathcal{G}}\left( M_{k^{\ast }}\right) %
\right] \right\} .
\label{eq:chi2_desc}
\end{equation}%
If no such $k^{\ast}$ exists condition \eqref{eq:padres_M1} automatically holds. Then the algorithm stores the formula
\begin{equation}
\chi _{P,a,eff}^{1,desc}\left( \mathbf{V};\mathcal{G}\right) =\frac{I_{a}(A)}{%
\pi_{a}(\mathbf{O}_{min};P) }\left\{ Y-b_{a}\left( \mathbf{O;}%
P\right) \right\}.
\label{eq:chi3_desc}
\end{equation}%
If $K=0$ then the algorithm also stores the formula in \eqref{eq:chi3_desc}.

Finally, the algorithm exits returning the formula 
\begin{equation*}
\chi _{P,a,eff}^{1}\left( \mathbf{V};\mathcal{G}\right) =\chi
_{P,a,eff}^{1,non-desc}\left( \mathbf{V};\mathcal{G}\right) +\chi
_{P,a,eff}^{1,desc}\left( \mathbf{V};\mathcal{G}\right) 
\end{equation*}
and an answer to the inquiry of whether \eqref{eq:main_id} holds.

Notice that for $J>1$ and $K>0$, the answer to such inquiry is positive when the following holds
\begin{itemize}
    \item[(i)] Equation $\left( \ref{eq:padres_O_T}\right) $,
    \item[(ii)] Equation $\left( \ref{eq_inclusion_padres}\right)$ for all $\left\{ 1,\dots,J-1\right\}$,
    \item[(iii)] Equation $\left( \ref{eq:inclusion_1}\right)$ for $k=K+1$,
    \item[(iv)] Equation $\left( \ref{eq:inclusion_2}\right)$ for $k\in \left\{
2,\dots,K+1\right\} $.
\end{itemize}
This is because under (i) and (ii) the algorithm stores $
\chi _{P,a,eff}^{1,non-desc}\left( \mathbf{V};\mathcal{G}\right)
=b_{a}\left( \mathbf{O;}P\right) -\chi _{a}\left( P;\mathcal{G}\right)
$ for all $P\in\mathcal{M(G)}$ and under (iii) and (iv), \eqref{eq:chi3_desc} holds for all $P\in\mathcal{M(G)}$.
Likewise, the answer is also positive in the following situations: (a) if $J=0$ or $J=1$, $K>0$, and (iii) and (iv) hold, (b) if $J>1$, $K=0$ and (i) and (ii) hold and (c) if $J\in\lbrace 0,1 \rbrace$ and $K=0$.

In the Appendix we show that these conditions are not only sufficient but also necessary for \eqref{eq:main_id} to hold as the following theorem establishes.
\begin{theorem}[Soundness and completeness of Algorithm \ref{algo:main}]
\label{theo:main}
Algorithm \ref{algo:main} exits returning \textup{\texttt{efficient=True}} if and only if 
\begin{equation}
\chi _{P,a,eff}^{1}\left( \mathbf{V};%
\mathcal{G}\right) =\psi _{P,a}\left( \mathbf{O};\mathcal{G}\right) \quad \text{for all} \quad P\in\mathcal{M}(\mathcal{G}).
\label{eq:adj_eff_cond}
\end{equation}
\end{theorem}

\begin{example}[Continuation of Example \eqref{ex:algorithm_nondesc}]
\label{ex:nondesc_cont}

In the DAG of Figure \ref{fig:ex_nondesc}, $\mathbf{M}=\emptyset$ and hence $K=0$. The algorithm then stores the formula in \eqref{eq:chi3_desc}
and finally returns
\begin{align*}
\chi _{P,a,eff}^{1}\left( \mathbf{V};\mathcal{G}\right)&=
b_{a}(O;P)-\chi_{a}(P;\mathcal{G})+E_{P}\left[b_{a}(O;P)\mid W_{3}, W_{4} \right]-E_{P}\left[b_{a}(O;P)\mid W_{2}, W_{4} \right] 
\\
&+ E_{P}\left[b_{a}(O;P)\mid W_{2}, W_{3} \right] - E_{P}\left[b_{a}(O;P)\mid W_{3}\right]
+ \frac{I_{a}(A)}{%
\pi_{a}(\mathbf{O}_{min};P) }\left\{ Y-b_{a}\left( \mathbf{O;}%
P\right) \right\}
\end{align*}
and a negative answer to the inquiry of whether \eqref{eq:adj_eff_cond} holds.
\end{example}

\begin{example}[Continuation of Example \eqref{ex:algorithm_nondesc2}]
\label{ex:nondesc2_cont}
In the DAG of Figure \ref{fig:ex_nondesc_large}, $\mathbf{M}=\lbrace M_{1},M_{2},M_{3}\rbrace$ and $K=3$. Condition \eqref{eq:inclusion_1} holds for $M_{4}=Y$ and \eqref{eq:inclusion_2} holds for $k=2,3,4$. Consequently the algorithm stores the formula in \eqref{eq:chi3_desc} and finally returns 
$$
\chi _{P,a,eff}^{1}\left( \mathbf{V};\mathcal{G}\right)=b_{a}\left( \mathbf{O;}P\right) -\chi _{a}\left( P;\mathcal{G}\right) + \frac{I_{a}(A)}{%
\pi_{a}(\mathbf{O}_{min};P) }\left\{ Y-b_{a}\left( \mathbf{O;}%
P\right) \right\}
$$
and a positive answer to the inquiry of whether \eqref{eq:adj_eff_cond} holds. Note that the black edges connecting $O_{1}$ and $O_{2}$ to the mediators and $Y$ and the black edges from $A$ to the mediators and $Y$ are necessary for the conditions \eqref{eq:inclusion_1} and \eqref{eq:inclusion_2} to hold for all $k\in\lbrace 2,\dots, K+1\rbrace$. The edges connecting each mediator to the next element in the topological order and connecting $A$ to $M_1$ are also needed as indicated in Remark \ref{remark:chain_med}. The edge connecting $O_{3}$ with $M_1$ is also necessary. The grey edge between $O_{3}$ and $M_{2}$ may or may not be present, without affecting the validity of expression \eqref{eq:chi3_desc} for all $P\in\mathcal{M(G)}$. However, had an edge between $O_{3}$ and $M_{3}$ existed then the presence of the purple edge between $O_{3}$ and $M_{2}$ would have been  necessary for the validity of expression \eqref{eq:chi3_desc} for all $P\in\mathcal{M(G)}$. In fact this illustrates the point that for expression \eqref{eq:chi3_desc}  to be valid for all $P\in\mathcal{M(G)}$ it is necessary that whenever a node in $\mathbf{O}\setminus \mathbf{O}_{min}$ is a parent of $M_{k}$ then it must be a parent of $M_{k^{\prime}}$ for every $k^{\prime}<k$. Likewise, the presence of the green edge connecting $M_{1}$ with $M_{3}$ is not necessary for  the validity of expression \eqref{eq:chi3_desc} for all $P\in\mathcal{M(G)}$. We note that a necessary condition for \eqref{eq:chi3_desc}  to be valid for all $P\in\mathcal{M(G)}$ is that whenever $M_{k}$ is a parent of $M_{k^{\prime}}$ then it is also a parent of $M_{k^{\prime\prime}}$ for all $k^{\prime\prime}\in\lbrace k^{\prime}+1,\dots, K  \rbrace$.
\end{example}

\begin{figure}[ht]
\begin{center}
\begin{tikzpicture}[>=stealth, node distance=1.5cm,
pre/.style={->,>=stealth,ultra thick,line width = 1.4pt}]
  \begin{scope}
    \tikzstyle{format} = [circle, inner sep=2.5pt,draw, thick, circle, line width=1.4pt, minimum size=6mm]
    \node[format] (A) {$A$};
    \node[format, right of=A] (M1) {$M_1$};
    \node[format, right of=M1] (M2) {$M_2$};
    \node[format, right of=M2] (M3) {$M_3$};
    \node[format, right of=M3] (Y) {$Y$};
    \node[format, above of=M1] (O) {$O$};
                 \draw (A) edge[pre, black] (M1);
                 \draw (M1) edge[pre, black] (M2);
                 \draw (M2) edge[pre, black] (M3);
                 \draw (M3) edge[pre, black] (Y);
                 \draw (O) edge[pre, black] (A);
                 \draw (O) edge[pre, black] (M1);
                 \draw (O) edge[pre, black] (M2);
                 \draw (M1) edge[pre, red, out=330, in=220] (Y);
  \end{scope} 
  \end{tikzpicture}
\end{center}
\caption{A DAG where the NP-$\mathbf{O}$ estimator is inefficient.}
\label{fig:ex_desc}
\end{figure}
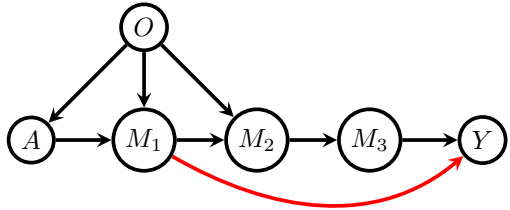

\begin{example}\label{ex:desc}
Consider the DAG in Figure \ref{fig:ex_desc}. In this case $\mathbf{O}=\mathbf{O}_{min}=\lbrace O \rbrace\equiv\lbrace W_{1} \rbrace$, $J=T=1$, $\mathbf{M}=\lbrace M_{1},M_{2},M_{3}\rbrace$ and $K=3$. 
Because $J=1$ the algorithm stores the formula \eqref{eq:formula_arriba_reducida}.
In addition, it is easy to check that, conditions \eqref{eq:inclusion_1}, \eqref{eq:inclusion_2} and \eqref{eq:d_sep_mediators}  hold for $k=2$, hence
\begin{equation}
E_{P}\left[T_{P,a,\mathcal{G}}\mid M_1, \pa_{\mathcal{G}}(M_{1}) \right]- E_{P}\left[T_{P,a,\mathcal{G}}\mid \pa_{\mathcal{G}}(M_{2}) \right]
\label{eq:term_vanish}
\end{equation}
vanishes for all $P\in\mathcal{M(G)}$. 
On the other hand, condition \eqref{eq:inclusion_1} fails for $k=4$, condition \eqref{eq:d_sep_mediators} fails for $k=3$ and $\lbrace A,\mathbf{O}\rbrace=\pa_{\mathcal{G}}(M_{1})$. Hence  the algorithm stores
\begin{align*}
\chi^{1,non-desc}_{P,a,eff}(\mathbf{V};\mathcal{G})&= E_{P}\left[T_{P,a,\mathcal{G}}\mid Y, \pa_{\mathcal{G}}(Y) \right]- E_{P}\left[T_{P,a,\mathcal{G}}\mid \pa_{\mathcal{G}}(Y) \right] + E_{P}\left[T_{P,a,\mathcal{G}}\mid M_3, \pa_{\mathcal{G}}(M_{3}) \right] 
\\
&-  E_{P}\left[T_{P,a,\mathcal{G}}\mid \pa_{\mathcal{G}}(M_{3}) \right] +  E_{P}\left[T_{P,a,\mathcal{G}}\mid M_2, \pa_{\mathcal{G}}(M_2) \right] - \frac{I_{a}(A)b_{a}(O;P)}{\pi_{a}(O;P)}.
\end{align*}
Notice that even though \eqref{eq:term_vanish} vanishes, $\chi^{1,non-desc}_{P,a,eff}(\mathbf{V};\mathcal{G})$ depends on $M_{1}$ because
$$
E_{P}\left[T_{P,a,\mathcal{G}}\mid Y, \pa_{\mathcal{G}}(Y) \right]= Y E_{P}\left[\frac{I_{a}(A)b_{a}(O;P)}{\pi_{a}(O;P)} \mid  M_{1}, M_{3}\right].
$$
\end{example}

\begin{example}\label{ex:front}
Consider now the DAG in Figure \ref{fig:front}. In this example, $\mathbf{O}=\mathbf{O}_{min}=\lbrace O \rbrace\equiv \lbrace W_{1} \rbrace$, $J=T=1$, $\mathbf{M}=\lbrace M\rbrace$ and $K=1$. Condition \eqref{eq:inclusion_1} fails, because $A$ is not a parent of $Y$. It is easy to check that the algorithm returns the following formula
\begin{align*}
\chi^{1}_{P,a,eff}(\mathbf{V};\mathcal{G}) &= b_{a}(O;P) - \chi_{a}(\mathcal{G};P)+ E_{P}\left[ T_{P,a,\mathcal{G}}\mid Y, M\right] - E_{P}\left[ T_{P,a,\mathcal{G}}\mid M\right] +  E_{P}\left[ T_{P,a,\mathcal{G}}\mid A, M\right] - E_{P}\left[ T_{P,a,\mathcal{G}}\mid A\right]
\\
&=  b_{a}(O;P) - \chi_{a}(\mathcal{G};P) + Y E_{P}\left[ \frac{I_{a}(A)}{\pi_{a}(O;P)}\mid Y, M\right] - E_{P}\left[ \frac{I_{a}(A)Y}{\pi_{a}(O;P)}\mid M\right]
\\
&+ I_{a}(A) E_{P}\left[ \frac{Y}{\pi_{a}(O;P)}\mid A, M\right] -I_{a}(A) E_{P}\left[ \frac{Y}{\pi_{a}(O;P)}\mid A\right].
\end{align*}
Note that this expression depends on both $O$ and $M$. 
This shows that the  NP-$\mathbf{O}$ estimator cannot be globally efficient in model $\mathcal{M(G)}$ because the NP-$\mathbf{O}$ estimator does not depend on $M$. It also demonstrates the  point announced earlier, that the non-parametric estimator of the front-door formula \eqref{eq:frontdoor} is also not globally efficient, because this estimator does not depend on the variable $O$.
\end{example}

\subsection{A connection between identification and efficient NP-$\mathbf{O}$ estimation}

Theorem \ref{theo:main} has the following interesting corollary. 
\begin{theorem}\label{theo:id}
Suppose that for a given DAG $\mathcal{G}$, $\mathbf{O}_{min}$ is not empty. 
Let $\mathbf{M}=\cn(A,Y,\mathcal{G})\setminus\lbrace Y\rbrace$.
If there exists an identifying formula for $\chi_{a}(P;\mathcal{G})$ that depends only on  $A, Y$ and the mediators $\mathbf{M}$ then the NP-$\mathbf{O}$ estimator of $\chi_{a}(P;\mathcal{G})$ is not globally efficient under the Bayesian Network $\mathcal{M(G)}$.
\end{theorem}
\begin{proof}
We prove the result by contradiction. 
By Lemma \ref{lemma:correct_algo_prun_indir}, without loss of generality, we can assume that $\irrel(A,Y,\mathbf{G})=\emptyset$.
Suppose that the NP-$\mathbf{O}$ estimator of $\chi_{a}(P;\mathcal{G})$ is globally efficient under $\mathcal{M(G)}$. Then Theorem \ref{theo:main} implies that every vertex $\mathbf{O}_{min}$ must be a parent of $Y$ and of every vertex in $\mathbf{M}$. Furthermore, $A$ must be a parent of $M_{1}$ and by Remark \ref{remark:chain_med}, each $M_{k}$ must be a parent of $M_{k+1}$. 
Let $\mathcal{G}\left[\lbrace A,Y\rbrace \cup \mathbf{M} \right]$ be the latent projection of  $\mathcal{G}$ \citep{admg} onto the vertex set $\lbrace A, Y\rbrace \cup \mathbf{M}$. Because $\mathbf{O}_{min}$ is not empty, then, in $\mathcal{G}\left[\lbrace A,Y\rbrace \cup \mathbf{M} \right]$, the nodes $A, M_{1},M_{2},\dots, M_{K},Y$ are all in the same district. The completeness of the ID algorithm, see \cite{id-orig} and \cite{shpitser-id}, implies that $\chi_{a}(P;\mathcal{G})$ is not identified when only $\lbrace A, Y\rbrace \cup \mathbf{M}$ are observed.
\end{proof}

Interestingly, Theorem \ref{theo:id} implies that in the front-door DAG in Figure \ref{fig:front}, the NP-$\mathbf{O}$ estimator is not efficient.
\section{Discussion}\label{sec:disc}

The results in this paper raise a number of open problems, several of which we are currently investigating.
\begin{enumerate}
    \item The derivation of a graphical criterion to characterize the class of all time dependent adjustment sets, like adjustment sets in row 1 and 8 in Example \ref{ex:time}, that dominate the rest even if they don't dominate each other.
    \item The characterization of DAGs under which an optimal time dependent adjustment set exists for joint interventions. 
    \item  The characterization of the subset of DAGs such that an optimal time dependent adjustment set exists for joint interventions, and for which the optimal time dependent adjustment set is time independent. 
    \item The characterization of DAGs such that among the adjustment sets of minimal size, there exists an optimal one.
    \item For DAGs for which an optimal time dependent adjustment set exists, the derivation of a sound and complete algorithm to answer the inquiry of whether the non-parametric optimally adjusted estimator is globally efficient under the Bayesian Network.
    \item For DAGs with latent variables such that observable adjustment sets exist, the characterization of the subset of DAGs for which an optimal adjustment set exists among the observable adjustment sets. 
    \item For DAGs with latent variables, the derivation of a general expression for the efficient influence function of $\chi_{a}(P;\mathcal{G})$ and a non-parametric globally efficient estimator.
    \item For DAGs with latent variables for which an optimal observable adjustment set exists, the derivation of a sound and complete algorithm to answer the inquiry of whether the non-parametric optimally adjusted estimator is globally efficient under the marginal of Bayesian Network for the observable variables.
\end{enumerate}

\section*{Acknowledgments}

Ezequiel Smucler was partially supported by Grant 20020170100330BA from
Universidad de Buenos Aires and PICT-201-0377 from ANPYCT, Argentina.

\newpage

\section{Appendix}

\subsection{Main proofs}

\subsubsection{Proofs of results in Section \protect\ref{sec:optimal_adj}}

\begin{proof}[Proof of Lemma \protect\ref{lemma:supplementation}]
Let 
$$
\chi_{\mathbf{a}}(P;\mathcal{G}) = E_{P}\left[\frac{I_{\mathbf{a}}(\mathbf{A})Y}{\pi_{\mathbf{a}}(\pa_{\mathcal{G}}(\mathbf{A};\mathcal{G})} \right].
$$
We first show that $(\mathbf{G,B})$ is an adjustment set. Since $\A \perp
\!\!\!\perp_{\mathcal{G}} \mathbf{G} \mid \mathbf{B}$ we have 
\begin{equation}
\pi_{\mathbf{a}}\left( \mathbf{G,B};P\right) = \pi_{\mathbf{a}}(\mathbf{B};P).
\label{eq:pi_eq}
\end{equation}
Then, for all $P\in\mathcal{M}(\mathcal{G})$ 
\begin{align*}
E_{P}\left\lbrace E_{P}\left[ Y \mid \A=\mathbf{a}, \mathbf{G},\mathbf{B} \right]
\right\rbrace&= E_{P}\left[ \frac{I_{\mathbf{a}}(\A)Y}{\pi_{\mathbf{a}}(\mathbf{G,B};P)}\right]
\\&= E_{P}\left[ \frac{I_{\mathbf{a}}(\mathbf{a})Y}{\pi_{\mathbf{a}}(\mathbf{B};P)}\right] \\
& =E_{P}\left\lbrace E_{P}\left[ Y \mid \A=\mathbf{a}, \mathbf{B} \right] \right\rbrace
\\
&= \chi_{\mathbf{a}}(P;\mathcal{G}),
\end{align*}
where the last equality holds because $\mathbf{B}$ is by assumption an
adjustment set. This shows that $(\mathbf{G,B})$ is an adjustment set.

Now, 
\begin{eqnarray*}
\psi _{P,\mathbf{a}}\left( \mathbf{B};\mathcal{G}\right)  &=&\frac{I_{%
\mathbf{a}}(\mathbf{A})Y}{\pi _{\mathbf{a}}(\mathbf{B};P)}-\left[ \frac{I_{%
\mathbf{a}}(\mathbf{A})}{\pi _{\mathbf{a}}(\mathbf{B};P)}-1\right] b_{%
\mathbf{a}}(\mathbf{B};P)-\chi_{\mathbf{a}} (P;\mathcal{G}) \\
&=&\frac{I_{\mathbf{a}}(\mathbf{A})Y}{\pi _{\mathbf{a}}(\mathbf{G},\mathbf{B%
};P)}-\left[ \frac{I_{\mathbf{a}}(\mathbf{A})}{\pi _{\mathbf{a}}(\mathbf{G},%
\mathbf{B};P)}-1\right] b_{\mathbf{a}}(\mathbf{G},\mathbf{B};P)-\chi_{\mathbf{a}} (P;%
\mathcal{G})\\
&+&\left[ \frac{I_{\mathbf{a}}(\mathbf{A})}{\pi _{\mathbf{a}}(%
\mathbf{G},\mathbf{B};P)}-1\right] \left\{ b_{\mathbf{a}}(\mathbf{G},\mathbf{%
B};P)-b_{\mathbf{a}}(\mathbf{B};P)\right\}  \\
&=&\psi _{P,\mathbf{a}}\left( \mathbf{G},\mathbf{B};\mathcal{G}\right) +%
\left[ \frac{I_{\mathbf{a}}(\mathbf{A})}{\pi _{\mathbf{a}}(\mathbf{G},%
\mathbf{B};P)}-1\right] \left[ b_{\mathbf{a}}(\mathbf{G},\mathbf{B};P)-b_{%
\mathbf{a}}(\mathbf{B};P)\right] 
\end{eqnarray*}%
where the second equality follows from $\left( \ref{eq:pi_eq}\right) .$
Next, noting that 
\begin{equation}
E_{P}\left\{ \psi _{P,\mathbf{a}}\left[ \mathbf{G,B};\mathcal{G}\right] g(%
\mathbf{A},\mathbf{G},\mathbf{B})\right\} =0\text{ for any }g\text{ such
that }E_{P}\left[ g(\mathbf{A},\mathbf{G},\mathbf{B})|\mathbf{G},\mathbf{B}%
\right] =0  
\label{eq:uncorr}
\end{equation}%
and that 
$$
E_{P}\left\{ \left. \left[ \frac{I_{\mathbf{a}}(\mathbf{A})}{\pi _{%
\mathbf{a}}(\mathbf{G},\mathbf{B};P)}-1\right] \left[ b_{\mathbf{a}}(\mathbf{%
G},\mathbf{B};P)-b_{\mathbf{a}}(\mathbf{B};P)\right] \right\vert \mathbf{G},%
\mathbf{B}\right\} =0
$$
we conclude that 
\begin{eqnarray*}
\sigma _{\mathbf{a},\mathbf{B}}^{2}\left( P\right)  &\equiv &var_{P}\left[ \psi _{P,%
\mathbf{a}}\left( \mathbf{B};\mathcal{G}\right) \right]  \\
&=&var_{P}\left[ \psi _{P,\mathbf{a}}\left( \mathbf{G},\mathbf{B};\mathcal{G}%
\right) \right] +var_{P}\left[ \left\{ \frac{I_{\mathbf{a}}(\mathbf{A})}{\pi
_{\mathbf{a}}(\mathbf{G},\mathbf{B};P)}-1\right\} \left\{ b_{\mathbf{a}}(%
\mathbf{G},\mathbf{B};P)-b_{\mathbf{a}}(\mathbf{B};P)\right\} \right]  \\
&\equiv &\sigma _{\mathbf{a},\mathbf{G},\mathbf{B}}^{2}\left( P\right) +var_{P}\left[
\left\{ \frac{I_{\mathbf{a}}(\mathbf{A})}{\pi _{\mathbf{a}}(\mathbf{G},%
\mathbf{B};P)}-1\right\} \left\{ b_{\mathbf{a}}(\mathbf{G},\mathbf{B};P)-b_{%
\mathbf{a}}(\mathbf{B};P)\right\} \right] .
\end{eqnarray*}%
Now 
\begin{align}
& \left. var_{P}\left[ \left\{ \frac{I_{\mathbf{a}}(\mathbf{A})}{\pi _{%
\mathbf{a}}(\mathbf{G},\mathbf{B};P)}-1\right\} \left\{ b_{\mathbf{a}}(%
\mathbf{G},\mathbf{B};P)-b_{\mathbf{a}}(\mathbf{B};P)\right\} \right]
=\right.   \nonumber \\
& E_{P}\left\{ \left[ b_{\mathbf{a}}(\mathbf{G,B};P)-b_{\mathbf{a}}(\mathbf{B%
};P)\right] ^{2}var_{P}\left[ \frac{I_{\mathbf{a}}(\mathbf{A})}{\pi _{%
\mathbf{a}}(\mathbf{G,B};P)}-1\mid \mathbf{G,B}\right] \right\} =  \nonumber
\\
& E_{P}\left\{ \left[ b_{\mathbf{a}}(\mathbf{G,B};P)-b_{\mathbf{a}}(\mathbf{B%
};P)\right] ^{2}\left[ \frac{1}{\pi _{\mathbf{a}}(\mathbf{G,B};P)}-1\right]
\right\} =  \nonumber \\
& E_{P}\left\{ \left[ b_{\mathbf{a}}(\mathbf{G,B};P)-b_{\mathbf{a}}(\mathbf{B%
};P)\right] ^{2}\left[ \frac{1}{\pi _{\mathbf{a}}(\mathbf{B};P)}-1\right]
\right\} =  \nonumber \\
& E_{P}\left\{ var_{P}(b_{\mathbf{a}}(\mathbf{G,B};P)\mid \mathbf{B})\left[ 
\frac{1}{\pi _{\mathbf{a}}(\mathbf{B};P)}-1\right] \right\} ,
\label{eq:var_Qa}
\end{align}%
where the last equality follows from 
\begin{align*}
b_{\mathbf{a}}(\mathbf{B},P)&=E_{P}\left( Y\mid \mathbf{A}=\mathbf{a},\mathbf{%
B}\right) \\
&=E_{P}\left[ E_{P}\left( Y\mid \mathbf{A}=\mathbf{a},\mathbf{G},%
\mathbf{B}\right) \mid \mathbf{A}=\mathbf{a},\mathbf{B}\right] 
\\
&=E_{P}\left[
b_{\mathbf{a}}(\mathbf{G,B})\mid \mathbf{A}=\mathbf{a},\mathbf{B}\right]
\\
&=E_{P}\left[ b_{\mathbf{a}}(\mathbf{G,B})\mid \mathbf{B}\right] ,
\end{align*}
since $\mathbf{A}\perp \!\!\!\perp _{\mathcal{G}}\mathbf{G}\mid \mathbf{B}$.
Next, recall that $\mathbf{c}\equiv \left( c_{\mathbf{a}}\right) _{\mathbf{a}%
\in \mathcal{\mathbf{A}}}$ $,\mathbf{Q\equiv }\left[ Q_{a}\right] _{\mathbf{a%
}\in \mathcal{\mathbf{A}}}$ where 
$$
Q_{\mathbf{a}}\equiv \left\{ \frac{I_{%
\mathbf{a}}(\mathbf{A})}{\pi _{\mathbf{a}}(\mathbf{G},\mathbf{B};P)}%
-1\right\} \left\{ b_{\mathbf{a}}(\mathbf{G},\mathbf{B};P)-b_{\mathbf{a}}(%
\mathbf{B};P)\right\} .
$$ 
For any $\mathbf{Z,}$ define $\mathbf{\psi }%
_{P}\left( \mathbf{Z};\mathcal{G}\right) \equiv \left( \psi _{P,\mathbf{a}%
}\left( \mathbf{Z};\mathcal{G}\right) \right) _{\mathbf{a}\in \mathcal{%
\mathbf{A}}}$. Then, writing $\sum_{\mathbf{a}\in \mathcal{\mathbf{A}}}c_{%
\mathbf{a}}\psi _{P,\mathbf{a}}\left( \mathbf{Z};\mathcal{G}\right) =\mathbf{%
c}^{T}\mathbf{\psi }_{P}\left( \mathbf{Z};\mathcal{G}\right) $ and noticing
that $E_{P}\left[ \mathbf{Q|G},\mathbf{B}\right] =\mathbf{0\,,}$ it follows
from $\left( \ref{eq:uncorr}\right) $ that 
\begin{eqnarray*}
\sigma _{\Delta ,\mathbf{B}}^{2}\left( P\right)  &=&var_{P}\left[ \mathbf{c}%
^{T}\mathbf{\psi }_{P}\left( \mathbf{B};\mathcal{G}\right) \right]  \\
&=&var_{P}\left[ \mathbf{c}^{T}\mathbf{\psi }_{P}\left( \mathbf{G,B};%
\mathcal{G}\right) \right] +var_{P}\left[ \mathbf{c}^{T}\mathbf{Q}%
\right]  \\
&=&\sigma _{\Delta ,\mathbf{G},\mathbf{B}}^{2}\left( P\right) +\mathbf{c}%
^{T}var_{P}\left( \mathbf{Q}\right) \mathbf{c}
\end{eqnarray*}%
The expression for $var_{P}\left( Q_{\mathbf{a}}\right) $ was derived in $%
\left( \ref{eq:var_Qa}\right) .$ On the other hand if $\mathbf{a}\not=%
\mathbf{a}\prime $%
\begin{eqnarray*}
cov_{P}\left( Q_{\mathbf{a}},Q_{\mathbf{a}^{\prime }}\right)  &=&E_{P}\left[
\left\{ \frac{I_{\mathbf{a}}(\mathbf{A})}{\pi _{\mathbf{a}}(\mathbf{B};P)}%
-1\right\} \left\{ \frac{I_{\mathbf{a}^{\prime }}(\mathbf{A})}{\pi _{\mathbf{%
a}^{\prime }}(\mathbf{B};P)}-1\right\} \left\{ b_{\mathbf{a}}(\mathbf{G},%
\mathbf{B};P)-b_{\mathbf{a}}(\mathbf{B};P)\right\} \left\{ b_{\mathbf{a}%
^{\prime }}(\mathbf{G},\mathbf{B};P)-b_{\mathbf{a}^{\prime }}(\mathbf{B}%
;P)\right\} \right]  \\
&=&E_{P}\left[ \left\{ \frac{I_{\mathbf{a}}(\mathbf{A})}{\pi _{\mathbf{a}}(%
\mathbf{B};P)}-1\right\} \left\{ \frac{I_{\mathbf{a}^{\prime }}(\mathbf{A})}{%
\pi _{\mathbf{a}^{\prime }}(\mathbf{B};P)}-1\right\} cov_{P}\left[ b_{%
\mathbf{a}}(\mathbf{G},\mathbf{B};P),b_{\mathbf{a}^{\prime }}(\mathbf{G},%
\mathbf{B};P)|\mathbf{B,A}\right] \right]  \\
&=&E_{P}\left[ \left\{ \frac{I_{\mathbf{a}}(\mathbf{A})}{\pi _{\mathbf{a}}(%
\mathbf{B};P)}-1\right\} \left\{ \frac{I_{\mathbf{a}^{\prime }}(\mathbf{A})}{%
\pi _{\mathbf{a}^{\prime }}(\mathbf{B};P)}-1\right\} cov_{P}\left[ b_{%
\mathbf{a}}(\mathbf{G},\mathbf{B};P),b_{\mathbf{a}^{\prime }}(\mathbf{G},%
\mathbf{B};P)|\mathbf{B}\right] \right]  \\
&=&E_{P}\left[ \frac{cov_{P}\left[ I_{\mathbf{a}}(\mathbf{A}),I_{\mathbf{a}%
^{\prime }}(\mathbf{A})|\mathbf{B}\right] }{\pi _{\mathbf{a}}(\mathbf{B}%
;P)\pi _{\mathbf{a}^{\prime }}(\mathbf{B};P)}cov_{P}\left[ b_{\mathbf{a}}(%
\mathbf{G},\mathbf{B};P),b_{\mathbf{a}^{\prime }}(\mathbf{G},\mathbf{B};P)|%
\mathbf{B}\right] \right]  \\
&=&-E_{P}\left[ cov_{P}\left[ b_{\mathbf{a}}(\mathbf{G},\mathbf{B};P),b_{%
\mathbf{a}^{\prime }}(\mathbf{G},\mathbf{B};P)|\mathbf{B}\right] \right].
\end{eqnarray*}%
This concludes the proof Lemma \ref{lemma:supplementation}
\end{proof}

\medskip

\begin{proof}[Proof of Lemma \protect\ref{lemma:deletion}]
We first show that $\mathbf{G}$ is an adjustment set. For any $P\in \mathcal{%
M}(\mathcal{G})$ the assumption $Y\perp \!\!\!\perp _{\mathcal{G}}\mathbf{B}%
\mid \mathbf{A},\mathbf{G}$ implies 
\begin{eqnarray}
b_{\mathbf{a}}(\mathbf{G,B};P) &\equiv &E_{P}\left( Y\mid \mathbf{A}=\mathbf{%
a},\mathbf{G},\mathbf{B}\right)   \nonumber \\
&=&E_{P}\left( Y\mid \mathbf{A}=\mathbf{a},\mathbf{G}\right)   \nonumber \\
&\equiv &b_{\mathbf{a}}(\mathbf{G};P)  \label{eq:b_a}
\end{eqnarray}%
and consequently that  
\[
E_{P}\left[ b_{\mathbf{a}}(\mathbf{G};P)\right] =E_{P}\left[ b_{\mathbf{a}}(%
\mathbf{G,B};P)\right] =\chi _{\mathbf{a}}(P;\mathcal{G}),
\]%
where the second equality follows from the assumption that $\left( \mathbf{%
G,B}\right) $ is an adjustment set. This shows that $\mathbf{G}$ is an
adjustment set.

Next, write 
\[
var_{P}\left[ \psi _{P,\mathbf{a}}(\mathbf{G,B};\mathcal{G})\right] =E_{P}%
\left[ var_{P}(\psi _{P,\mathbf{a}}(\mathbf{G,B};\mathcal{G})\mid \mathbf{A}%
,Y,\mathbf{G})\right] +var_{P}\left[ E_{P}(\psi _{P,\mathbf{a}}(\mathbf{G,B};%
\mathcal{G})\mid \mathbf{A},Y,\mathbf{G})\right] .
\]%
Now 
\begin{align}
E_{P}\left[ \psi _{P,\mathbf{a}}(\mathbf{G,B};\mathcal{G})\mid \mathbf{A},Y,%
\mathbf{G}\right] & =E_{P}\left[ \left. \frac{I_{\mathbf{a}}(\mathbf{A})}{%
\pi _{\mathbf{a}}(\mathbf{G,B};P)}\left\{ Y-b_{\mathbf{a}}(\mathbf{G}%
;P)\right\} +\left\{ b_{\mathbf{a}}(\mathbf{G};P)-\chi _{\mathbf{a}}(P;%
\mathcal{G})\right\} \right\vert \mathbf{A},Y,\mathbf{G}\right]   \nonumber
\\
& =I_{\mathbf{a}}(\mathbf{A})\left\{ Y-b_{\mathbf{a}}(\mathbf{G};P)\right\}
E_{P}\left[ \left. \frac{1}{\pi _{\mathbf{a}}(\mathbf{G,B};P)}\right\vert 
\mathbf{A}=\mathbf{a},Y,\mathbf{G}\right] +\left\{ b_{\mathbf{a}}(\mathbf{G}%
;P)-\chi _{\mathbf{a}}(P;\mathcal{G})\right\}   \nonumber \\
& =I_{\mathbf{a}}(\mathbf{A})\left\{ Y-b_{\mathbf{a}}(\mathbf{G};P)\right\}
E_{P}\left[ \left. \frac{1}{\pi _{\mathbf{a}}(\mathbf{G,B};P)}\right\vert 
\mathbf{A}=\mathbf{a},\mathbf{G}\right] +\left\{ b_{\mathbf{a}}(\mathbf{G}%
;P)-\chi _{\mathbf{a}}(P;\mathcal{G})\right\}  \nonumber \\
& =\frac{I_{\mathbf{a}}(\mathbf{A})}{\pi _{\mathbf{a}}\left( \mathbf{G}%
;P\right) }\left\{ Y-b_{\mathbf{a}}(\mathbf{G};P)\right\} +\left\{ b_{%
\mathbf{a}}(\mathbf{G};P)-\chi _{\mathbf{a}}(P;\mathcal{G})\right\}  \nonumber\\
& =\psi _{P,\mathbf{a}}(\mathbf{G};\mathcal{G}).  \nonumber
\end{align}%
where the first equality follows from $\left( \ref{eq:b_a}\right) ,$ the
third follows from $Y\perp \!\!\!\perp _{\mathcal{G}}\mathbf{B}\mid \mathbf{A%
},\mathbf{G}$ and the fourth by invoking Lemma \ref{lemma:inv_pi} in Section \ref{sec:aux_res}. On the
other hand, 
\begin{align*}
var_{P}\left[ \psi _{P,\mathbf{a}}(\mathbf{G,B};\mathcal{G})\mid \mathbf{A}%
,Y,\mathbf{G}\right] & =var_{P}\left\{ \frac{I_{\mathbf{a}}(\mathbf{A})}{\pi
_{\mathbf{a}}(\mathbf{G,B};P)}\left\{ Y-b_{\mathbf{a}}(\mathbf{G};P)\right\}
+\left\{ b_{\mathbf{a}}(\mathbf{G};P)-\chi _{\mathbf{a}}(P;\mathcal{G}%
)\right\} \mid \mathbf{A},Y,\mathbf{G}\right\}  \\
& =I_{\mathbf{a}}(\mathbf{A})\left( Y-b_{\mathbf{a}}(\mathbf{G};P)\right)
^{2}var_{P}\left\{ \frac{1}{\pi _{\mathbf{a}}(\mathbf{G,B};P)}\mid \mathbf{A}%
=\mathbf{a},\mathbf{G}\right\} 
\end{align*}%
where the first equality follows from $\left( \ref{eq:b_a}\right) $\ and the
second follows from $Y\perp \!\!\!\perp _{\mathcal{G}}\mathbf{B}\mid \mathbf{%
A},\mathbf{G.}$ Thus 
\[
E_{P}\left\{ var_{P}\left[ \psi _{P,\mathbf{a}}(\mathbf{G,B};\mathcal{G}%
)\mid \mathbf{A},Y,\mathbf{G}\right] \right\} =E_{P}\left\{ \pi _{\mathbf{a}%
}(\mathbf{G};P)var_{P}(Y\mid \mathbf{A}=\mathbf{a},\mathbf{G})var_{P}\left[ 
\frac{1}{\pi _{\mathbf{a}}(\mathbf{G,B};P)}\mid \mathbf{A=a},\mathbf{G}%
\right] \right\} 
\]%
We therefore have 
\begin{eqnarray*}
\sigma _{\mathbf{a},\mathbf{G,B}}^{2}\left( P\right)  &\equiv &var_{P}\left[ \psi _{P,%
\mathbf{a}}(\mathbf{G,B};\mathcal{G})\right]  \\
&=&var_{P}\left[ \psi _{P,\mathbf{a}}(\mathbf{G};\mathcal{G})\right]
+E_{P}\left\{ \pi _{\mathbf{a}}(\mathbf{G};P)var_{P}(Y\mid \mathbf{A}=%
\mathbf{a},\mathbf{G})var_{P}\left[ \frac{1}{\pi _{\mathbf{a}}(\mathbf{G,B}%
;P)}\mid \mathbf{A},\mathbf{G}\right] \right\}  \\
&=&\sigma _{\mathbf{a},\mathbf{G}}^{2}\left( P\right) +E_{P}\left\{ \pi _{\mathbf{a}}(%
\mathbf{G};P)var_{P}(Y\mid \mathbf{A}=\mathbf{a},\mathbf{G})var_{P}\left[ 
\frac{1}{\pi _{\mathbf{a}}(\mathbf{G,B};P)}\mid \mathbf{A},\mathbf{G}\right]
\right\} .
\end{eqnarray*}%
Next, 
\begin{eqnarray*}
\sigma _{\Delta ,\mathbf{G,B}}^{2}\left( P\right)  &=&var_{P}\left[ E_{P}%
\left[ \mathbf{c}^{T}\mathbf{\psi }_{P}\left( \mathbf{G,B};\mathcal{G}%
\right) |\mathbf{A},Y,\mathbf{G}\right] \right] +E_{P}\left[ var_{P}\left[ 
\mathbf{c}^{T}\mathbf{\psi }_{P}\left( \mathbf{G,B};\mathcal{G}\right) |%
\mathbf{A},Y,\mathbf{G}\right] \right]  \\
&=&var_{P}\left[ \mathbf{c}^{T}\mathbf{\psi }_{P}\left( \mathbf{G};\mathcal{G%
}\right) \right] +\mathbf{c}^{T}E_{P}\left[ var_{P}\left[ \mathbf{\psi }%
_{P}\left( \mathbf{G,B};\mathcal{G}\right) |\mathbf{A},Y,\mathbf{G}\right] %
\right] \mathbf{c} \\
&=&\sigma _{\Delta ,\mathbf{G}}^{2}\left( P\right) +\mathbf{c}^{T}E_{P}\left[
var_{P}\left[ \mathbf{\psi }_{P}\left( \mathbf{G,B};\mathcal{G}\right) |%
\mathbf{A},Y,\mathbf{G}\right] \right] \mathbf{c}.
\end{eqnarray*}%
But by $\left( \ref{eq:b_a}\right) $ we have%
\begin{eqnarray*}
&&cov_{P}\left[ \psi _{\mathbf{a},P}\left( \mathbf{G,B};\mathcal{G}\right)
,\psi _{\mathbf{a}^{\prime },P}\left( \mathbf{G,B};\mathcal{G}\right) |%
\mathbf{A},Y,\mathbf{G}\right]  \\
&=&cov_{P}\left[ \frac{I_{\mathbf{a}}(\mathbf{A})}{\pi _{\mathbf{a}}(\mathbf{%
G,B};P)}\left\{ Y-b_{\mathbf{a}}(\mathbf{G};P)\right\} ,\frac{I_{\mathbf{a}%
^{\prime }}(\mathbf{A})}{\pi _{\mathbf{a}^{\prime }}(\mathbf{G,B};P)}\left\{
Y-b_{\mathbf{a}^{\prime }}(\mathbf{G};P)\right\} |\mathbf{A},Y,\mathbf{G}%
\right]  \\
&=&I_{\mathbf{a}}(\mathbf{A})I_{\mathbf{a}^{\prime }}(\mathbf{A})\left\{
Y-b_{\mathbf{a}}(\mathbf{G};P)\right\} \left\{ Y-b_{\mathbf{a}^{\prime }}(%
\mathbf{G};P)\right\} cov_{P}\left[ \frac{1}{\pi _{\mathbf{a}}(\mathbf{G,B}%
;P)},\frac{1}{\pi _{\mathbf{a}^{\prime }}(\mathbf{G,B};P)}|\mathbf{A},Y,%
\mathbf{G}\right]  \\
&=&0
\end{eqnarray*}%
because $I_{\mathbf{a}}(\mathbf{A})I_{\mathbf{a}^{\prime }}(\mathbf{A})=0.$
Consequently, 
\[
\mathbf{c}^{T}E_{P}\left[ var_{P}\left[ \mathbf{\psi }_{P}\left( \mathbf{G,B}%
;\mathcal{G}\right) |\mathbf{A},Y,\mathbf{G}\right] \right] \mathbf{c=}%
\sum_{\mathbf{a}\in \mathcal{A}}c_{\mathbf{a}}^{2}E_{P}\left\{ \pi _{\mathbf{a}}(%
\mathbf{G};P)var_{P}(Y\mid \mathbf{A}=\mathbf{a},\mathbf{G})var_{P}\left[ 
\frac{1}{\pi _{\mathbf{a}}(\mathbf{G,B};P)}\mid \mathbf{A=a},\mathbf{G}%
\right] \right\}.
\]%
In particular, for $\mathbf{a}=a=1$ and $\mathbf{a}=a=0$ we have%
\begin{eqnarray*}
\sigma _{ATE,\mathbf{G,B}}^{2}\left( P\right) -\sigma _{ATE,\mathbf{G}%
}^{2}\left( P\right)  &=&E_{P}\left\{ \pi _{a=0}(\mathbf{G};P)var_{P}(Y\mid
A=0,\mathbf{G})var_{P}\left[ \frac{1}{\pi _{a=0}(\mathbf{G,B};P)}\mid A=0,%
\mathbf{G}\right] \right\}  \\
&&+E_{P}\left\{ \pi _{a=1}(\mathbf{G};P)var_{P}(Y\mid A=1,\mathbf{G})var_{P}%
\left[ \frac{1}{\pi _{a=0}(\mathbf{G,B};P)}\mid A=1,\mathbf{G}\right]
\right\}.
\end{eqnarray*}%
This concludes the proof of Lemma \ref{lemma:deletion}.
\end{proof}
\medskip
\begin{lemma}
\label{lemma:unique_minimal} Given a DAG\ $\mathcal{G}$ and disjoint
vertex sets $\mathbf{A}$ and $\mathbf{B}$ there exists a unique subset
$\mathbf{C}$ of $\mathbf{B}$ such that $\mathbf{A\perp \!\!\!\perp }_{%
\mathcal{G}}\mathbf{B}\setminus \mathbf{C}\mid \mathbf{C}$ and such that no strict subset $\mathbf{C}^{\prime}$ of $\mathbf{C}$ satisfies $\mathbf{A\perp \!\!\!\perp }_{%
\mathcal{G}} \mathbf{B}\setminus\mathbf{C}^{\prime}\mid \mathbf{C}^{\prime}$.
\end{lemma}

\begin{proof}
The result is a consequence of the fact that d-separation
is a graphoid. See \cite{graphoid}. Suppose there were two distinct minimal
sets, say $\mathbf{C}_{1}$ and $\mathbf{C}_{2}.$ Let $\mathbf{I}=%
\mathbf{C}_{1}\cap $ $\mathbf{C}_{2},$ $\mathbf{W}_{1}=\mathbf{C}_{1}\mathbf{%
\backslash I}$ and $\mathbf{W}_{2}\mathbf{=C}_{2}\mathbf{\backslash I}$ and $%
\mathbf{R=B\backslash }\left( \mathbf{C}_{1}\mathbf{\cup C}_{2}\right) 
\mathbf{.}$ Then $\mathbf{A\perp \!\!\!\perp }_{\mathcal{G}}\mathbf{%
B\backslash C}_{1}\mathbf{|C}_{1}$ is equivalent to 
\begin{equation*}
\mathbf{A\perp \!\!\!\perp }_{\mathcal{G}}\left( \mathbf{R},\mathbf{\mathbf{W%
}_{2}}\right) \mathbf{|\mathbf{W}_{1},I}
\end{equation*}%
and $\mathbf{A\perp \!\!\!\perp }_{\mathcal{G}}\mathbf{B\backslash C}_{2}%
\mathbf{|C}_{2}$ is equivalent to 
\begin{equation*}
\mathbf{A\perp \!\!\!\perp }_{\mathcal{G}}\left( \mathbf{R,\mathbf{W}_{1}}%
\right) \mathbf{|\mathbf{W}_{2},I}.
\end{equation*}%
The weak union axiom implies that 
\begin{equation}
\mathbf{A\perp \!\!\!\perp }_{\mathcal{G}}\mathbf{R|}\left( \mathbf{\mathbf{W%
}_{1},\mathbf{W}_{2}}\right) \mathbf{,I}  \label{eq:union}
\end{equation}
The decomposition axiom implies that 
\begin{equation}
\mathbf{A\perp \!\!\!\perp }_{\mathcal{G}}\mathbf{\mathbf{W}_{2}|\mathbf{W}%
_{1},I}\text{ and }\mathbf{A\perp \!\!\!\perp }_{\mathcal{G}}\mathbf{\mathbf{%
W}_{1}|\mathbf{W}_{2},I}.  \label{eq:intersection}
\end{equation}%
Next, it follows from $\left( \ref{eq:intersection}\right) $ and the
intersection axiom that 
\begin{equation}
\mathbf{A\perp \!\!\!\perp }_{\mathcal{G}}\left( \mathbf{\mathbf{W}_{1},%
\mathbf{W}_{2}}\right) \mathbf{|I}  \label{eq:marginal}
\end{equation}
Finally, from $\left( \ref{eq:union}\right) $ and $\left( \ref{eq:marginal}%
\right) $, the contraction axiom implies that 
\begin{equation*}
\mathbf{A\perp \!\!\!\perp }_{\mathcal{G}}\left( \mathbf{R,\mathbf{W}_{1},%
\mathbf{W}_{2}}\right) \mathbf{|I}\text{ }
\end{equation*}%
or equivalently, 
\begin{equation*}
\mathbf{A\perp \!\!\!\perp }_{\mathcal{G}}\mathbf{B\backslash I|I}.
\end{equation*}%
If $\mathbf{C}_{1}$ and $\mathbf{C}_{2}$ are distinct then $\mathbf{I}$ is
a strict subset of $\mathbf{C}_{1}$ and $\mathbf{C}_{2}$ which cannot happen
because $\mathbf{C}_{1}$ and $\mathbf{C}_{2}$ were minimal sets $\mathbf{C}%
^{\prime }$ with the property that $\mathbf{A\perp \!\!\!\perp }_{\mathcal{G}%
}\mathbf{B\backslash C}^{\prime }\mathbf{|C}^{\prime }.$
\end{proof}

\medskip
\begin{proof}[Proof of Theorem \ref{theo:optimal_adj_min}]

\textbf{Proof of part (1).} We will prove that $\mathbf{O}_{min}$ is an
adjustment set. Note that $A\perp \!\!\!\perp _{\mathcal{G}}\mathbf{O}%
\setminus \mathbf{O}_{min}\mid \mathbf{O}_{min}$, implies that $\pi _{a}(%
\mathbf{O},P)=\pi _{a}(\mathbf{O}_{min},P)$. Then, for all $P\in \mathcal{%
M(G)}$ 
\[
E_{P}\left\{ E_{P}\left[ Y\mid A=a,\mathbf{O}_{min}\right] \right\} =E_{P}%
\left[ \frac{I_{a}(A)Y}{\pi _{a}(\mathbf{O}_{min},P)}\right] =E_{P}\left[ 
\frac{I_{a}(A)Y}{\pi _{a}(\mathbf{O},P)}\right] =\chi _{a}(P;\mathcal{G})
\]%
where the last equality follows because, since $A$ is a point intervention, $%
\mathbf{O}$ is an adjustment set. 

\textbf{Proof of part (2).} In our proof of part (2) we will invoke at
several places the following property.

\textbf{Property (O):} For any $O\in \mathbf{O}$ there exists a directed path
from $O$ to $Y\,\ $such that for any adjustment set $\mathbf{Z}$ relative to 
$\left( A,Y\right) $ in $\mathcal{G}\mathbf{,}$ the path does not intersect
the nodes in $\mathbf{Z}$ other than, at most, at the node $O.$ 

The proof of Property (O) is immediate because by definition of $\mathbf{O,}$ $%
O$ is the parent of a node in cn$\left( A,Y;\mathcal{G}\right) $.
If such node is $Y$ then the assertion holds trivially for the path $%
O\rightarrow Y.$ Otherwise, for any node $M$ in cn$\left( A,Y;\mathcal{G}%
\right) \backslash \left\{ Y\right\} $ there exists a directed path from $M$
to $Y$ that intersects solely nodes in cn$\left( A,Y;\mathcal{G}\right) .$
The assertion then holds for such path because for any adjustment set $%
\mathbf{Z}$ it holds that $\mathbf{Z\cap }$cn$\left( A,Y;\mathcal{G}\right)
=\emptyset .$

Turn now to the proof of  
\[
A\text{ }\mathbf{\perp \!\!\!\perp }_{\mathcal{G}}\text{ }\left[ \mathbf{O}%
_{\min }\mathbf{\backslash Z}_{\min }\right] \text{ }\mathbf{|}\text{ }%
\mathbf{Z}_{\min }.
\]%
Suppose there existed $O\in \mathbf{O}_{\min }\mathbf{\backslash Z}_{\min }$
such that $O$ is not d-separated from $A$ given $\mathbf{Z}_{\min }$ in $%
\mathcal{G}$. Let $\alpha $ denote the path between $A$ and $O$ that is open
given $\mathbf{Z}_{\min }.$ By Property (O) there exists a directed path, say $%
\varphi ,$ between $O$ and $Y$ that is open given $\mathbf{Z}_{\min }.$
Then, the path obtained by concatenating $\alpha $ with $\varphi $ is a
non-causal path between $A$ and $Y$ that is open given $\mathbf{Z}_{\min },$
which is impossible because $\mathbf{Z}_{\min }$ is an adjustment set.

Turn now to the proof of 
\begin{equation}
Y\perp \!\!\!\perp _{\mathcal{G}}[\mathbf{Z}_{min}\setminus \mathbf{O}%
_{min}]\mid \mathbf{O}_{min},A.  \label{eq:indep_O_min}
\end{equation}%
We will show it by contradiction.

Assume there exists $X^{\ast }\in \mathbf{Z}_{min}\backslash \mathbf{O}_{min}
$ such that $Y\not\perp \!\!\!\perp _{\mathcal{G}}X^{\ast }\mid \mathbf{O}%
_{min},A$. By $X^{\ast }\in \mathbf{Z}_{min},$ Lemma \ref{lemma:a3} in Section \ref{sec:aux_opt}, implies
that 
\begin{equation}
X^{\ast }\not\in \text{de}_{\mathcal{G}}(A). \label{eq:X_no_desc}
\end{equation}%
Also, by \cite{shpitser-adjustment}, we have 
\begin{equation}
X^{\ast }\not\in \text{forb}\left( A,Y;\mathcal{G}\right) .
\label{eq:no_forbid}
\end{equation}

Let $\eta ^{\ast }$ be the path between $X^{\ast }$ and $Y$ that is open
when we condition on $\left( \mathbf{O}_{min},A\right) $. We will first show
that $\eta ^{\ast }$ must intersect a vertex in $\mathbf{O\backslash O}_{min}
$. So, if $\mathbf{O\backslash O}_{min}=\emptyset ,$ this result already
shows $\left( \ref{eq:indep_O_min}\right) .$

To show that $\eta ^{\ast }$ must intersect a vertex in $\mathbf{O\backslash
O}_{min}$ we first note that $\eta ^{\ast }$ must be of the form 
\[
X^{\ast }-\circ ...\circ \rightarrow Y.
\]%
The justification for why the last edge in $\eta ^{\ast }$ must point into $Y
$ is as follows. Suppose the edge pointed out of $Y$. Then, since by $\left( %
\ref{eq:no_forbid}\right) $ $X^{\ast }$ cannot be a descendant of $Y,$ the
path $\eta ^{\ast }$ would have to intersect a vertex that would be both a
descendant of $Y$ and a collider in $\eta ^{\ast },$ and either such vertex
or any of its descendants would have to be in the conditioning set $\mathbf{O%
}_{min}\cup \left\{ A\right\} $ so as to yield the path $\eta ^{\ast }$
open. But this is impossible because neither $A$ can be a descendant of $Y$
nor can any element of $\mathbf{O}_{min}$, by the very definition of $%
\mathbf{O}_{min}.$

Next we note that, by definition of $\mathbf{O,}$ in the edge $\circ
\rightarrow Y$ the vertex $\circ $ is in the set $\mathbf{O\cup M}$ where 
$$
\mathbf{M}\mathbf{\equiv }\cn\left( A,Y;\mathcal{G}\right) \backslash
\left\{ Y\right\}.
$$
If the vertex is in $\mathbf{O}$ then it must
be in $\mathbf{O\backslash O}_{min}$ because the path $\eta ^{\ast }$ is
open when conditioning on $\left( \mathbf{O}_{min},A\right) ,$ therefore
proving the assertion that $\eta ^{\ast }$ intersects $\mathbf{O\backslash O}%
_{min}.$
If the vertex is in $\mathbf{M,}$ then the next edge in the path
must be of the form 
\[
X^{\ast }-\circ ...\circ \rightarrow M_{k}\rightarrow Y
\]%
for some $M_{k}\in \mathbf{M.}$ The justification for why the edge $\circ
\rightarrow M_{k}$ points into $M_{k}$ is along the same lines as before.
Specifically, if the edge pointed out of $M_{k}$ then, by virtue of $X^{\ast }$
not being a descendant of $M_{k},$ then the path $\eta ^{\ast }$ would have
to intersect a vertex that would be both a descendant of $M_{k}$ and a
collider in $\eta ^{\ast },$ and either such vertex or any of its
descendants would have to be in the conditioning set $\mathbf{O}_{min}\cup
\left\{ A\right\} $. But this is impossible because neither $A$ can be a
descendant of $M_{k}$ nor can any element of $\mathbf{O}_{min}$, by the very
definition of $\mathbf{O}_{min}.$

By the same argument as above, in the edge $\circ \rightarrow M_{k}$ the
vertex $\circ $ is in the set $\mathbf{O\cup M.}$ If the vertex is in $%
\mathbf{O}$ then it must be in $\mathbf{O\backslash O}_{min}$ because the
path $\eta ^{\ast }$ is open when conditioning on $\left( \mathbf{O}%
_{min},A\right) ,$ therefore proving the assertion that $\eta ^{\ast }$
intersects $\mathbf{O\backslash O}_{min}.$ If the vertex is a, say $M_{j},$
in $\mathbf{M}$ then reasoning as above, the path $\eta ^{\ast }$ must be of
the form $X^{\ast }-\circ ...\circ \rightarrow M_{j}\rightarrow
M_{k}\rightarrow Y.$ Continuing in the same fashion, we arrive at the
conclusion that either any of the vertices $\circ $ are in $\mathbf{%
O\backslash O}_{min}$ or otherwise, the path is of the form $X^{\ast
}\rightarrow M_{r}\rightarrow M_{l}\rightarrow ...\rightarrow
M_{j}\rightarrow M_{k}\rightarrow Y.$ In the latter case, $X^{\ast }\in 
\mathbf{O\backslash O}_{min}$ which therefore concludes the proof that the
path $\eta ^{\ast }$ intersects $\mathbf{O\backslash O}_{min}.$

Let $O^{\ast }\in \mathbf{O}\setminus \mathbf{O}_{min}$ be the element of $%
\mathbf{O}\setminus \mathbf{O}_{min}$ that is closest to $Y$ in the path $%
\eta ^{\ast },$ that is, such that the subpath of $\eta ^{\ast }$ between $%
O^{\ast }$ and $Y$ does not intersect any other vertex of $\mathbf{O}%
\setminus \mathbf{O}_{min}$.

Let $D_{1}^{\ast },\dots ,D_{k}^{\ast }$ be the colliders on $\eta ^{\ast }$%
, with $D_{1}^{\ast }$ the one closest to $X^{\ast }$ in $\eta ^{\ast }$, $%
D_{2}^{\ast }$ the one second closest to $X^{\ast }$ and so on. For each $j$
there exists a descendant of $D_{j}^{\ast }$ that is an element of $\left( 
\mathbf{O}_{min},A\right) $. Furthermore, if there exists a directed path
between $D_{j}^{\ast }$ and $A$, this path necessarily has to intersect an
element of $\mathbf{O}_{min}$ for suppose this was not the case. Then, take $%
j^{\ast }$ to be the largest $j$ such that there exists a directed path
between $D_{j}^{\ast }$ and $A$ that does not intersect $\mathbf{O}_{min}.$
Then the path $A\leftarrow ...\leftarrow D_{j^{\ast }}^{\ast }\leftarrow
...-O^{\ast }$ is open given $\mathbf{O}_{min}$, which contradicts $O^{\ast
}\in \mathbf{O\backslash O}_{min}$. We therefore conclude that $\eta ^{\ast }
$ is open by conditioning just on $\mathbf{O}_{min}$.

From the nodes in $\mathbf{Z}_{min}\backslash \mathbf{O}_{min}$ that
intersect $\eta ^{\ast },~$let $W^{\ast }$ be the closest one to $O^{\ast }$
in the path $\eta ^{\ast },$ possibly $W^{\ast }=O^{\ast }$. Consider now
the subpath $\alpha ^{\ast }$ of $\eta ^{\ast }$ between $W^{\ast }$ and $Y.$
Because $\eta ^{\ast }$ is open by conditioning on $\mathbf{O}_{min}$, so is 
$\alpha ^{\ast }.$ The path $\alpha ^{\ast }$ has one of the following two
forms 
\begin{equation}
W^{\ast }\rightarrow \circ -...\rightarrow \underset{\equiv C_{1}^{\ast }}{%
\underbrace{\circ }}\leftarrow ...\rightarrow \underset{\equiv C_{2}^{\ast }}%
{\underbrace{\circ }}\leftarrow ....\rightarrow \underset{\equiv C_{r}^{\ast
}}{\underbrace{\circ }}\leftarrow ...-O^{\ast }\rightarrow
M_{u_{1}}\rightarrow M_{u_{2}}...\rightarrow M_{u_{t}}\rightarrow Y
\label{path:eta1}
\end{equation}%
or 
\begin{equation}
W^{\ast }\leftarrow \circ -...\rightarrow \underset{\equiv C_{1}^{\ast }}{%
\underbrace{\circ }}\leftarrow ...\rightarrow \underset{\equiv C_{2}^{\ast }}%
{\underbrace{\circ }}\leftarrow ....\rightarrow \underset{\equiv C_{r}^{\ast
}}{\underbrace{\circ }}\leftarrow ... -O^{\ast }\rightarrow
M_{u_{1}}\rightarrow M_{u_{2}}...\rightarrow M_{u_{t}}\rightarrow Y.
\label{path:eta2}
\end{equation}%
where the set of colliders $\left\{ C_{1}^{\ast },...,C_{r}^{\ast }\right\}
\ $is included in $\left\{ D_{1}^{\ast },...,D_{k}^{\ast }\right\} $ and can
possibly be empty, and the set 
\[
\left\{ M_{u_{1}},M_{u_{2}},...,M_{u_{t}}\right\} \ 
\]%
is included in $\mathbf{M}$ and can also possibly be empty.

Next, let 
\[
\Delta \equiv \left\{ \delta :\delta \text{ is a path between }W^{\ast }%
\text{ and }A\text{ that is open given }\mathbf{Z}_{min}\backslash \{W^{\ast
}\}\right\} .
\]%
Lemma \ref{lemma:a2} in Section \ref{sec:aux_opt} implies that $\Delta $ is not empty. Any path $\delta $
in $\Delta $ has one of the following forms:

a) $\delta $ is a directed path from $W^{\ast }$ to $A:$ 
\[
W^{\ast }\rightarrow \circ \rightarrow \circ ...\circ \rightarrow A
\]

b) $\delta $ has one and only one fork: 
\[
W^{\ast }\leftarrow \circ ...\circ \leftarrow \circ \rightarrow \circ
...\circ \rightarrow A.
\]

c) $\delta $ has at least one collider and the first edge points out of $%
W^{\ast }:$ 
\[
W^{\ast }\rightarrow \circ -...\rightarrow \underset{\equiv H_{1}^{\ast }}{%
\underbrace{\circ }}\leftarrow ...\rightarrow \underset{\equiv H_{2}^{\ast }}%
{\underbrace{\circ }}\leftarrow ....\rightarrow \underset{\equiv H_{s}^{\ast
}}{\underbrace{\circ }}\leftarrow ... -A.
\]

d) $\delta $ has at least one collider and the first edge points into $%
W^{\ast }:$ 
\[
W^{\ast }\leftarrow \circ -...\rightarrow \underset{\equiv H_{1}^{\ast }}{%
\underbrace{\circ }}\leftarrow ...\rightarrow \underset{\equiv H_{2}^{\ast }}%
{\underbrace{\circ }}\leftarrow ....\rightarrow \underset{\equiv H_{s}^{\ast
}}{\underbrace{\circ }}\leftarrow ... -A.
\]%
Moreover, we can assume without loss of generality that $W^{\ast}$ appears only once in the path $\delta$.
Note that $\delta \in \Delta $ cannot be a directed path from $A\,\ $to $%
W^{\ast }$ because $W^{\ast }\in \mathbf{Z}_{min}$ and by Lemma \ref%
{lemma:a3} in Section \ref{sec:aux_opt} we have that $W^{\ast }\not\in $de$_{\mathcal{G}}(A)$.

We will show that neither of the forms $\left( \ref{path:eta1}\right) $ or $%
\left( \ref{path:eta2}\right) $ for the path $\alpha ^{\ast }$ are possible
by showing that if $\alpha ^{\ast }$ was of one such form then it would
imply that $\Delta $ is empty.

Henceforth, assume that $\alpha ^{\ast }$ takes one of the forms $\left( \ref%
{path:eta1}\right) $ or $\left( \ref{path:eta2}\right) .$ Below we will show
the following claims.

\textbf{Claim (i)}. $\forall $ $\delta \in \Delta $ with form (a) or (b), $%
\delta $ is open given $\mathbf{O}_{min}.$

\textbf{Claim (ii).} If $\exists \delta \in \Delta $ with form (b) or (d)
then the path $\alpha ^{\ast }$ cannot be of the form $\left( \ref{path:eta2}%
\right) $.

\textbf{Claim (iii).} Every $\delta \in \Delta $ of the form (c) or (d) is
blocked given $\mathbf{O}_{min}.$

\textbf{Proof of Claim (i)}. Let $\delta $ have form (a) or (b). Then no
node in $\mathbf{O}_{min}\cap \mathbf{Z}_{min}$ intersects $\delta $, for if
it did, the path would be blocked by $\mathbf{Z}_{min}\backslash \{W^{\ast
}\}$. On the other hand, suppose the path $\delta $ intersected a node $%
O^{\ast \ast }$ in $\mathbf{O}_{min}\backslash \mathbf{Z}_{min}$. Let $\xi $
be the subpath of $\delta $ between $O^{\ast \ast }$ and $A.$ The path $\xi $
is open given $\mathbf{Z}_{min}.$ By Property (O) there exists a directed path, say $\varphi$, from $O^{\ast\ast}$ to $Y$ that does not intersect $\mathbf{Z}_{min}$. Then, the path obtained by concatenating $\xi$ with $\varphi$ is a non-causal path between $A$ and $Y$ that is open given $\mathbf{Z}_{min}$. This contradicts the assumption that $\mathbf{Z}_{min}$ is an adjustment set. This concludes
the proof of Claim (i).

\textbf{Proof of Claim (ii).} Suppose that there exists a path $\delta \in
\Delta $ with form (b) or (d). We will prove by contradiction that there
cannot be any path $\alpha ^{\ast }$ of the form $\left( \ref{path:eta2}%
\right) $ that is open when by conditioning on $\mathbf{O}_{min}$. Suppose
there existed one such path $\alpha ^{\ast }.$ Suppose first that there are
no colliders in $\alpha ^{\ast }$, that is, there exist no nodes $%
C_{j}^{\ast }$. The path $\alpha ^{\ast }$ does not intersect any element of 
$\mathbf{Z}_{min}\backslash \mathbf{O}_{min}$ other than at the node $%
W^{\ast }$ because, by definition, $W^{\ast }$ was chosen to be the closest
element in $\mathbf{Z}_{min}\backslash \mathbf{O}_{min}$ to $O^{\ast }.$ On
the other hand, since the path $\alpha ^{\ast }$ is open by conditioning on $%
\mathbf{O}_{min},$ then $\alpha ^{\ast }$ cannot intersect any element of $%
\mathbf{O}_{min}.$ Then, $\alpha ^{\ast }$ is open given $\mathbf{Z}%
_{min}\backslash W^{\ast }.$ Take now the path $\delta \in \Delta $ with
form (b) or (d) and concatenate it with the path $\alpha ^{\ast }.$ The
concatenated path is a non-causal path between $A$ and $Y$ which is open
given $\mathbf{Z}_{min}$ because $W^{\ast }$ is a collider in the path and $%
W^{\ast }\in \mathbf{Z}_{min}$. This is impossible because $\mathbf{Z}_{min}$
is an adjustment set. We therefore conclude if a path $\alpha ^{\ast }$
exists, then the set of colliders $\left\{ C_{1}^{\ast },...,C_{k}^{\ast
}\right\} $ is not empty. Furthermore, at least one of the colliders is not
an ancestor of any node in $\mathbf{Z}_{min},$ for if all $C_{1}^{\ast
},...,C_{k}^{\ast }$ were ancestors of some node in $\mathbf{Z}_{min},$ then
again the path $\alpha ^{\ast }$ would be open given $\mathbf{Z}%
_{min}\backslash W^{\ast }$ and consequently, the concatenated path between
a path $\delta $ of the form (b) or (d) with the path $\alpha ^{\ast }$
would be a non-causal path between $A$ and $Y$ that is open given $\mathbf{Z}%
_{min},$ contradicting the assumption that $\mathbf{Z}_{min}$ is an
adjustment set. Take the smallest $j,$ say $j^{\prime },$ such that the
collider $C_{j}^{\ast }$ is not an ancestor of $\mathbf{Z}_{min}.$ Because
the path $\alpha ^{\ast }$ is open by conditioning on $\mathbf{O}_{min}$,
then there exists $O^{\ast \ast }\in \mathbf{O}_{min}\backslash \mathbf{Z}%
_{min}$ such that $C_{j^{\prime }}^{\ast }$ is an ancestor of $O^{\ast \ast }
$ so that either there exists a directed path, say $\lambda ,$ from $%
C_{j^{\prime }}^{\ast }$ to $O^{\ast \ast }$ or $C_{j^{\prime }}^{\ast
}=O^{\ast \ast }$. Now, by Property (O)
there exists a directed path, say $\varphi ,$ from $O^{\ast \ast }$
to $Y$ that is open given $\mathbf{Z}_{min}.$ Now, consider the path that
concatenates a path $\delta \in \Delta $ with form (b) or (d), with the
subpath of $\alpha ^{\ast }$ between $W^{\ast }$ and $C_{j^{\prime }}^{\ast }
$, next concatenates with $\lambda $ if $C_{j^{\prime }}^{\ast }\not=O^{\ast
\ast }$ and finally concatenates with $\varphi $. Such path is a non-causal
path between $A$ and $Y$ that is open given $\mathbf{Z}_{min}$ which is
impossible because $\mathbf{Z}_{min}$ is an adjustment set. This concludes
the proof of the Claim (ii)

\textbf{Proof of Claim (iii).} Suppose that $\delta \in \Delta $ is of the
form (c) and that $\delta $ is open given $\mathbf{O}_{min}$. Then
concatenating $\delta $ with $\alpha ^{\ast }$ we obtain a non-causal path
between $A$ and $Y$ that is open given $\mathbf{O}_{min}$ because $W^{\ast }$
is not a collider on this path. This contradicts the fact that $\mathbf{O}_{min}$ is an adjustment set.

Suppose now that $\delta \in \Delta $ is of the form (d) and is open given $%
\mathbf{O}_{min}.$ By Claim (ii), the path $\alpha ^{\ast }$ has to be of
the form $\left( \ref{path:eta1}\right) $. Then concatenating $\delta $ with 
$\alpha ^{\ast }$ we once again obtain a path between $A$
and $Y$ that is open given $\mathbf{O}_{min}$ arriving at a contradiction.
This concludes the proof of Claim (iii). 

We will now argue that $\Delta $ must be empty by showing that Claims (i),
(ii) and (iii) imply that if $\delta \in \Delta ,$ then $\delta $ cannot
take any of the forms (a), (b), (c) or (d).

\textbf{(I) Proof that }$\delta \in \Delta $\textbf{\ cannot take the form
(a). }Suppose there exists $\delta \in \Delta $ with the form (a). Then,
invoking Claim (i), we conclude that the path $\gamma $ between $O^{\ast }$
and $A$ formed by concatenating the path $\delta $ between $W^{\ast }$ and $A$ and the
subpath of $\alpha^{\ast }$ between $O^{\ast }$ and $W^{\ast }$ is a path
between $O^{\ast }$ and $A$ that is open given $\mathbf{O}_{min}.$ This is
impossible because the existence of such path $\gamma $ contradicts the
assertion that $O^{\ast }\in \mathbf{O}\setminus \mathbf{O}_{min}.$

\textbf{(II) Proof that }$\delta \in \Delta $\textbf{\ cannot take the form
(b).} Suppose there exists $\delta \in \Delta $ with the form (b). Then
invoking Claim (ii), the path $\alpha ^{\ast }$ has to be of the form $%
\left( \ref{path:eta1}\right) $. By Claim (i), $\delta $ is open given $%
\mathbf{O}_{min}$. On the other hand, $\alpha ^{\ast }$ is open given $%
\mathbf{O}_{min}.$ Concatenating $\delta $ with $\alpha ^{\ast }$ we form a path,
say $\pi $, that is open given $\mathbf{O}_{min}$, since $W^{\ast }$ is not
a collider on $\pi $. This is impossible because the existence of such path $%
\pi $ contradicts the fact that $O^{\ast }\in \mathbf{O}\setminus \mathbf{O}%
_{min}.$

\textbf{(III) Proof that }$\delta \in \Delta $\textbf{\ can take neither the
form (c) nor the form (d). }Suppose that there exists a $\delta \in \Delta $
of the form (c) or (d). By Claim (iii), $\delta $ is blocked by conditioning
on $\mathbf{O}_{min}.$ Furthermore, by definition of $\Delta ,$ the path is
open when conditioning on $\mathbf{Z}_{min}\backslash W^{\ast }.$ Then, one
of the following happens:

(III.a) the path $\delta $ intersects a node $O^{\ast \ast }\in \mathbf{O}%
_{min}\backslash \mathbf{Z}_{min}$ that is not a collider in the path, or

(III.b) the property (III.a) does not hold and there exists a non-empty
subset, say $\mathcal{H}\equiv \left\{ H_{j_{1}}^{\ast },\dots,H_{j_{l}}^{\ast
}\right\} ,$ of the collider set $\left\{ H_{1}^{\ast },...,H_{s}^{\ast
}\right\} $ such that each $H_{j_{u}}^{\ast }$ is an ancestor in $\mathcal{G}
$ of a node in $\mathbf{Z}_{min}\backslash W^{\ast }$ but is not an ancestor
of a node in $\mathbf{O}_{min}.$

We will show by contradiction that both (III.a) and (III.b) are impossible. 

Suppose first that (III.a) holds. Let $\phi $ be the subpath in $\delta $
between $O^{\ast \ast }$ and $A.$ The path $\phi $ is open given $\mathbf{Z}%
_{min}$ because the path $\delta $ is open given $\mathbf{Z}_{min}\backslash
W^{\ast }.$ Let $\nu $ a directed path between $O^{\ast \ast }$ and $Y$ that
does not intersect $\mathbf{Z}_{min}$, which exists by Property
(O). The path between $A$ and $Y$ obtained by concatenating $\nu $ with $\phi 
$ is a non-causal path between $A$ and $Y$ that is open given $\mathbf{Z}_{min}$. This is
impossible because $\mathbf{Z}_{min}$ is an adjustment set.

Suppose next that (III.b) holds. Let $\mathcal{H}=\left\{ H_{j_{1}}^{\ast
},...,H_{j_{l}}^{\ast }\right\} $ be the maximal subset of the collider set $%
\left\{ H_{1}^{\ast },...,H_{s}^{\ast }\right\} $ such that each $%
H_{j_{u}}^{\ast }$ is an ancestor in $\mathcal{G}$ of a node in $\mathbf{Z}%
_{min}\backslash W^{\ast }$ but is not an ancestor of a node in $\mathbf{O}%
_{min}.$ Assume without loss of generality that $j_{1}<j_{2}<\dots<j_{l}$ so
that $H_{j_{l}}^{\ast }$ is the closest node in $\mathcal{H}$ to $A$ in the
path $\delta .$ Then, the subpath of $\delta ,$ say $\zeta _{j_{l}},$
between $H_{j_{l}}^{\ast }$ and $A$ is open given $\mathbf{O}_{min}.$ 

We will show next that if (III.b) holds then 
\begin{equation}
\text{for each }u\text{ in }\left\{ 1,...,l\right\} \text{ there exists }%
O_{j_{u}}^{\ast }\in \mathbf{O\backslash O}_{min}\text{ such that }%
H_{j_{u}}^{\ast }\not\ort_{\mathcal{G}}O_{j_{u}}^{\ast }|\mathbf{O}%
_{min}.  \label{claim:dificil}
\end{equation}%
However, $\left( \ref{claim:dificil}\right) $ leads to a contradiction. To
see this, let $\nu _{j_{l}}^{\ast }$ be a directed path between $%
O_{j_{l}}^{\ast }$ and $Y$ that does not intersect $\mathbf{O}_{min}$, which
exists by Property (O). Let $\gamma _{j_{l}}^{\ast }\,\ $denote the path
between $H_{j_{l}}^{\ast }$ and $O_{j_{l}}^{\ast }$ which is open by
conditioning on $\mathbf{O}_{min}$ (which exists by $\left( \ref%
{claim:dificil}\right) ).$ Then, the path obtained by concatenating the
paths $\nu _{j_{l}}^{\ast }$ with $\gamma _{j_{l}}^{\ast }\,$\ and with $%
\zeta _{j_{l}}$ is a non-causal path between $A$ and $Y$ that is open by
conditioning on $\mathbf{O}_{min}.$ This is impossible because $\mathbf{O}%
_{min}$ is an adjustment set. The proof of part (2) of the theorem is then
finished if we show that (III.b) implies \eqref{claim:dificil}. We will show this by induction in $u.$ Suppose first that $u=1.$
By definition of the set $\mathcal{H}$, either $H_{j_{1}}^{\ast }\equiv
Z_{u=1,1}^{\ast }\in \mathbf{Z}_{min}\backslash \left\{ \mathbf{O}%
_{min},W^{\ast }\right\} $ or there exists a node $Z_{u=1,1}^{\ast }\in 
\mathbf{Z}_{min}\backslash \left\{ \mathbf{O}_{min},W^{\ast }\right\} $ such
that there exists a directed path from $H_{j_{1}}^{\ast }$ to $%
Z_{u=1,1}^{\ast }$ that does not intersect any other element of $\mathbf{Z}%
_{min}\backslash \left\{ \mathbf{O}_{min},W^{\ast }\right\} .$ Now, because $%
Z_{u=1,1}^{\ast }\in \mathbf{Z}_{min}$ and $\mathbf{Z}_{min}$ is a minimal
adjustment set, then 
\[
\text{there exists a non-causal path }\theta _{1}\text{ between }A\text{ and }%
Y\text{ such that }\theta _{1}\text{ is open by conditioning on }\mathbf{Z}%
_{min}\backslash Z_{u=1,1}^{\ast }
\]%
and 
\[
\text{the path }\theta _{1}\text{ is closed by conditioning on }\mathbf{Z}%
_{min}.
\]%
The path $\theta _{1}$ must then intersect $Z_{u=1,1}^{\ast }$ and $%
Z_{u=1,1}^{\ast }$ must be a non-collider vertex in the path. Now, define 
\[
\tau _{1}=\text{ subpath of }\theta _{1}\text{ between }A\text{ and }%
Z_{u=1,1}^{\ast }
\]%
and 
\[
\kappa _{1}=\text{ subpath of }\theta _{1}\text{ between }Z_{u=1,1}^{\ast }%
\text{ and }Y.
\]%
Because $Z_{u=1,1}^{\ast }$ is a non-collider in the path $\theta _{1},$
then in at least one of the subpaths $\tau _{1}$ or $\kappa _{1},$ the edge
with vertex $Z_{u=1,1}^{\ast }$ is pointing out of $Z_{u=1,1}^{\ast }.$
Furthermore, because $\theta _{1}$ is open by conditioning on $\mathbf{Z}%
_{min}\backslash Z_{u=1,1}^{\ast },$ so are $\tau _{1}$ and $\kappa _{1}.$
We will show next that either 
\begin{equation}
\left( \ref{claim:dificil}\right) \text{ holds for }u=1\text{ or }\exists 
\text{ a vertex }Z_{u=1,2}^{\ast }\text{ in }\mathbf{Z}_{min}\backslash %
\left[ \mathbf{O}_{min}\cup \left\{ Z_{u=1,1}^{\ast }\right\} \right] \text{
such that }Z_{u=1,2}^{\ast }\text{ is a descendant of }Z_{u=1,1}^{\ast }
\label{claim:2}
\end{equation}%

Suppose first that the edge with vertex $Z_{u=1,1}^{\ast }$ in $\tau _{1}$
points out of $Z_{u=1,1}^{\ast }$. We will now show that $\tau _{1}$ cannot
be a directed path from $Z_{u=1,1}^{\ast }$ to $A.$ Suppose $\tau _{1}$ was
a directed path. Then $\tau _{1}$ cannot intersect any vertex of $\mathbf{O}%
_{min}$, because $H_{j_{1}}^{\ast })$, and hence $Z_{u=1,1}^{\ast }$, is not ancestor of any vertex in $ \mathbf{O}_{min}.$ We therefore conclude that if $\tau _{1}$ is a directed path
between $Z_{u=1,1}^{\ast }$ and $A,$ then it must be open by conditioning on 
$\mathbf{O}_{min}.$ Now, let $\lambda $ be the subpath of $\delta ^{\ast }$
between $W^{\ast }$ and $H_{j_{1}}^{\ast }.$ By definition of $%
H_{j_{1}}^{\ast },\lambda$ is open given $\mathbf{O}_{min}.$ Let $\rho 
$ be the directed path between $H_{j_{1}}^{\ast }$ and $Z_{u=1,1}^{\ast }$
if $H_{j_{1}}^{\ast }\not=$ $Z_{u=1,1}^{\ast },$ otherwise let $\rho $
denote the degenerate path consisting of just the vertex $H_{j_{1}}^{\ast }.$
Let 
\[
\beta =\text{ the path between }A\text{ and }Y\text{ obtained by
concatenating }\tau _{1}\text{ with }\rho \text{ with }\lambda \text{ with }%
\alpha^{\ast}.
\]%
Because all the paths $\tau _{1},$ $\rho ,$ $\lambda $ and $\alpha^{\ast} $ are
open given $\mathbf{O}_{min}$ and because none of the vertices $W^{\ast
},H_{j_{1}}^{\ast }$ and $Z_{u=1,1}^{\ast }$ are in $\mathbf{O}_{min},$ and
none are colliders in the path $\beta ,$ then the path $\beta $ is open
given $\mathbf{O}_{min}.$ This is impossible because $\mathbf{O}_{min}$ is
an adjustment set. We therefore conclude that $\tau _{1}$ cannot be a
directed path between $Z_{u=1,1}^{\ast }$ and $A.\,$Therefore, $\tau _{1}$
must intersect a collider. Any collider in the path $\tau _{1}$ must be an
ancestor of a node in the set $\mathbf{Z}_{min}\backslash \left\{
Z_{u=1,1}^{\ast }\right\} $ because $\tau _{1}$ is open given $\mathbf{Z}%
_{min}\backslash \left\{ Z_{u=1,1}^{\ast }\right\} .$ Furthermore, the
collider in $\tau _{1}$ that is closest to $Z_{u=1,1}^{\ast }$ cannot be an
ancestor of any element of $\mathbf{O}_{min},$ because if it was, then $%
Z_{u=1,1}^{\ast }$ and consequently $H_{j_{1}}^{\ast }$ would be an ancestor
of a vertex in $\mathbf{O}_{min},$ which is not possible by the definition
of the set $\mathcal{H}.$ We therefore conclude that there exists a vertex,
say $Z_{u=1,2}^{\ast },$ in $\mathbf{Z}_{min}\backslash \left[ \mathbf{O}%
_{min}\cup \left\{ Z_{u=1,1}^{\ast }\right\} \right] $ such that $%
Z_{u=1,2}^{\ast }$ is a descendant of $Z_{u=1,1}^{\ast },$ thus showing $%
\left( \ref{claim:2}\right) .$ 

Next suppose that the edge with vertex $Z_{u=1,1}^{\ast }$ in $\kappa _{1}$
points out of $Z_{u=1,1}^{\ast }.$ If there exists a directed path between $%
Z_{u=1,1}^{\ast }$ and $Y,$ then this path necessarily has to intersect an
element $O_{j_{1}}^{\ast }\in \mathbf{O.}$ The vertex $O_{j_{1}}^{\ast }$%
\textbf{\ }$\ $cannot be in $\mathbf{O}_{min}$ because if it were, then $%
H_{j_{1}}^{\ast }$ would be an ancestor of an element of $\mathbf{O}_{min},$
which is impossible by the definition of the set $\mathcal{H}$. Then, if
there exists a directed path between $Z_{u=1,1}^{\ast }$ and $Y,$ the
assertion $\left( \ref{claim:2}\right) $ holds. Now, suppose that there
exists no directed path between $Z_{u=1,1}^{\ast }$ and $Y.$ Then, the path $%
\kappa _{1}$ must intersect a collider. Because $\kappa _{1}$ is open given $%
\mathbf{Z}_{min}\backslash \left\{ Z_{u=1,1}^{\ast }\right\} $ and because $%
Z_{u=1,1}^{\ast }$ cannot be the ancestor of any vertex in $\mathbf{O}_{min},
$ then we reason exactly as before, and conclude that there exists a $%
Z_{u=1,2}^{\ast },$ in $\mathbf{Z}_{min}\backslash \left[ \mathbf{O}%
_{min}\cup \left\{ Z_{u=1,1}^{\ast }\right\} \right] $ such that $%
Z_{u=1,2}^{\ast }$ is a descendant of $%
Z_{u=1,1}^{\ast }$, thus proving $%
\left( \ref{claim:2}\right) .$ 

Next, suppose $%
\left( \ref{claim:2}\right)$ holds because there exists  
a vertex $Z_{u=1,2}^{\ast }\text{ in }\mathbf{Z}_{min}\backslash %
\left[ \mathbf{O}_{min}\cup \left\{ Z_{u=1,1}^{\ast }\right\} \right]$ 
such that $Z_{u=1,2}^{\ast }$ is a descendant of $Z_{u=1,1}^{\ast }$. We can now reason exactly as we did for $%
Z_{u=1,1}^{\ast }$ and conclude that 
\begin{align}
&\left( \ref{claim:dificil}\right) \text{ holds for }u=1\text{ or }\exists 
\text{ a vertex }Z_{u=1,3}^{\ast }\text{ in }\mathbf{Z}_{min}\backslash %
\left[ \mathbf{O}_{min}\cup \left\{ Z_{u=1,1}^{\ast },Z_{u=1,2}^{\ast
}\right\} \right] \text{ such that }Z_{u=1,3}^{\ast }
\nonumber
\\
&\text{ is a descendant
of }Z_{u=1,2}^{\ast }  \label{claim:3}
\end{align}%
Continuing in this fashion until depleting the set of vertices in $\mathbf{Z}%
_{min}$ we arrive at the conclussion that $\left( \ref{claim:dificil}\right) 
$ holds for $u=1.$

Suppose now that $\left( \ref{claim:dificil}\right) $ holds for $u=1,...,t-1$ with $t\leq l.$ We will show that it holds for $u=t.$ Let $%
Z_{u=t,1}^{\ast }\in \mathbf{Z}_{min}\backslash \mathbf{O}_{min}$ be a
descendant of $H_{j_{t}}^{\ast }$ which exists by the definition of $%
\mathcal{H}$. Let $\theta _{t}$ be a path that is open given $\mathbf{Z}%
_{min}\backslash Z_{u=t,1}^{\ast }$ but closed given $\mathbf{Z}_{min}.$
Reasoning as before, the path $\theta _{t}$ must intersect $Z_{u=t,1}^{\ast }
$ and $Z_{u=t,1}^{\ast }$ cannot be a collider in the path. Then,
partitioning $\theta _{t}$ as $\left( \tau _{t},\kappa _{t}\right) $ where 
\[
\tau _{t}=\text{ subpath of }\theta _{t}\text{ between }A\text{ and }%
Z_{u=t,1}^{\ast }
\]%
and 
\[
\kappa _{t}=\text{ subpath of }\theta _{t}\text{ between }Z_{u=t,1}^{\ast }%
\text{ and }Y
\]%
we know that in at least one of $\tau _{t}$ or $\kappa _{t}$ the edge with
one endpoint equal to $Z_{u=t,1}^{\ast }$ must point out of $Z_{u=t,1}^{\ast
}.$ Furthermore, both $\tau _{t}$ and $\kappa _{t}$ are open given $\mathbf{Z%
}_{min}\backslash Z_{u=t,1}^{\ast }.$ We will show that 
\begin{equation}
\left( \ref{claim:dificil}\right) \text{ holds for }u=t\text{ or }\exists 
\text{ a vertex }Z_{u=t,2}^{\ast }\text{ in }\mathbf{Z}_{min}\backslash %
\left[ \mathbf{O}_{min}\cup \left\{ Z_{u=t,1}^{\ast }\right\} \right] \text{
such that }Z_{u=t,2}^{\ast }\text{ is a descendant of }Z_{u=t,1}^{\ast }
\label{claim:4}
\end{equation}

Suppose the edge with one endpoint equal to $Z_{u=t,1}^{\ast }$ in $\tau _{t}
$ points out of $Z_{u=t,1}^{\ast }.$ We will show that $\tau _{t}$ cannot be
a directed path from $Z_{u=t,1}^{\ast }$ to $A$. As we reasoned for $\tau
_{1}$ above, if $\tau _{t}$ was directed it could not intersect any element
of $\mathbf{O}_{min}$, for if it did, then such element of  $\mathbf{O}_{min}$ would be a descendant of $H^{\ast}_{j_t}$ which is impossible by the definition of the set $\mathcal{H}$. So, if a directed
path between $Z_{u=t,1}^{\ast }$ and $A$ exists, then it must be open given $%
\mathbf{O}_{min}.$ Now, by the inductive hypothesis, we know that there
exists $O_{j_{t-1}}^{\ast }\in \mathbf{O}\backslash \mathbf{O}_{min}$ such
that $O_{j_{t-1}}^{\ast }$ is a descendant of $H_{j_{t-1}}^{\ast }.$
Because, by definition of $\mathcal{H},$ $H_{j_{t-1}}^{\ast }$ cannot be an
ancestor of any vertex in $\mathbf{O}_{min},$ then we conclude that there
exists a directed path, say $\sigma ,$ from $H_{j_{t-1}}^{\ast }$ and $%
O_{j_{t-1}}^{\ast }$ that is open by conditioning on $\mathbf{O}_{min}.$ Let 
$\lambda _{j_{t-1}}$ be the subpath of $\delta$ between $%
H_{j_{t-1}}^{\ast }$ and $H_{j_{t}}^{\ast }.$ The path $\lambda _{j_{t-1}}$
is open by conditioning on $\mathbf{O}_{min}$ because we have assumed that $%
\delta$ does not intersect any node of $\mathbf{O}_{min}$ that is a
non-collider in the path, and by the definition of $H_{j_{t-1}}^{\ast }$ and 
$H_{j_{t}}^{\ast },$ if in the path $\lambda _{j_{t-1}}$ there are
colliders, each of these colliders must be ancestors of $\mathbf{O}_{min}.$
Let $\rho _{j_{t-1}}$ be the directed path between $H_{j_{t}}^{\ast }$ and $%
Z_{u=t,1}^{\ast }$ if $H_{j_{t}}^{\ast }\not=$ $Z_{u=t,1}^{\ast },$
otherwise let $\rho _{j_{t-1}}$ denote the degenerate path consisting of
just the vertex $H_{j_{t}}^{\ast }.$ Note that $\rho _{j_{t-1}}$ is open
given $\mathbf{O}_{min}$ because $H_{j_{t}}^{\ast }$ is not an ancestor of
any vertex in $\mathbf{O}_{min}.$ Let 
\[
\beta _{j_{t}}=\text{ the path between }A\text{ and }O_{j_{t-1}}^{\ast }%
\text{ obtained by concatenating }\tau _{t}\text{ with }\rho _{j_{t-1}}\text{
with }\lambda _{j_{t-1}}\text{ with }\sigma 
\]%
Because all the paths $\tau _{t},$ $\rho _{j_{t-1}},\lambda _{j_{t-1}}$ and $%
\sigma $ are open given $\mathbf{O}_{min}$ and because none of the vertices $%
H_{j_{t-1}}^{\ast },H_{j_{t}}^{\ast }$ and $Z_{u=t,1}^{\ast }$ are in $%
\mathbf{O}_{min},$ and none are colliders in the path $\beta _{j_{t}},$ then
the path $\beta _{j_{t}}$ is open given $\mathbf{O}_{min}.$ This is
impossible because by definition of $\mathbf{O}_{min},$ $O_{j_{t-1}}^{\ast }$
is d-separated from $A$ given $\mathbf{O}_{min}.$ We therefore conclude that 
$\tau _{t}$ cannot be a directed path between $Z_{u=t,1}^{\ast }$ and $A.\,$%
Therefore, $\tau _{t}$ must intersect a collider. Any collider in the path $%
\tau _{t}$ must be an ancestor of a node in the set $\mathbf{Z}%
_{min}\backslash \left\{ Z_{u=t,1}^{\ast }\right\} $ because $\tau _{t}$ is
open given $\mathbf{Z}_{min}\backslash \left\{ Z_{u=t,1}^{\ast }\right\} .$
Furthermore, the collider in $\tau _{t}$ that is closest to $Z_{u=t,1}^{\ast
}$ cannot be an ancestor of any element of $\mathbf{O}_{min},$ because if it
was, then $Z_{u=t,1}^{\ast }$ and consequently $H_{j_{t}}^{\ast }$ would be
an ancestor of a vertex in $\mathbf{O}_{min},$ which is not possible by the
definition of the set $\mathcal{H}.$ We therefore conclude that there exists
a vertex, say $Z_{u=t,2}^{\ast },$ in $\mathbf{Z}_{min}\backslash \left[ 
\mathbf{O}_{min}\cup \left\{ Z_{u=t,1}^{\ast }\right\} \right] $ such that $%
Z_{u=t,2}^{\ast }$ is a descendant of $Z_{u=t,1}^{\ast },$ thus showing $%
\left( \ref{claim:4}\right) $ holds if the edge with one endpoint equal to $%
Z_{u=t,1}^{\ast }$ in $\tau _{t}$ points out of $Z_{u=t,1}^{\ast }.$

Suppose next that the edge with one endpoint equal to $Z_{u=t,1}^{\ast }$ in 
$\kappa _{t}$ points out of $Z_{u=t,1}^{\ast }$. If there exists a directed
path between $Z_{u=t,1}^{\ast }$ and $Y,$ then this path necessarily has to
intersect an element $O_{j_{t}}^{\ast }\in \mathbf{O.}$ The vertex $%
O_{j_{t}}^{\ast }$ cannot be in $\mathbf{O}_{min}$ because if
it were, then $H_{j_{t}}^{\ast }$ would be an ancestor of an element of $%
\mathbf{O}_{min},$ which is impossible by the definition of the set $%
\mathcal{H}$. Then, if there exists a directed path between $Z_{u=t,1}^{\ast
}$ and $Y,$ the assertion $\left( \ref{claim:4}\right) $ holds. Now, suppose
that there exists no directed path between $Z_{u=t,1}^{\ast }$ and $Y.$
Then, the path $\kappa _{t}$ must intersect a collider. Because $\kappa _{t}$
is open given $\mathbf{Z}_{min}\backslash \left\{ Z_{u=t,1}^{\ast }\right\} $
and because $Z_{u=t,1}^{\ast }$ cannot be the ancestor of any vertex in $%
\mathbf{O}_{min},$ then we reason exactly as before, and conclude that there
exists a $Z_{u=t,2}^{\ast },$ in $\mathbf{Z}_{min}\backslash \left[ \mathbf{O%
}_{min}\cup \left\{ Z_{u=t,1}^{\ast }\right\} \right] $ such that $%
Z_{u=t,2}^{\ast }$ is a descendant of $%
Z_{u=t,1}^{\ast }$.

Next, because $Z_{u=t,2}^{\ast }$ is in $\mathbf{Z}_{min}\backslash \left[ 
\mathbf{O}_{min}\cup \left\{ Z_{u=t,1}^{\ast }\right\} \right] $ and is a
descendant of $Z_{u=t,1}^{\ast },$ we can reason exactly as we did for $%
Z_{u=t,1}^{\ast }$ and conclude that 
\begin{align*}
&\left( \ref{claim:dificil}\right) \text{ holds for }u=t\text{ or }\exists 
\text{ a vertex }Z_{u=t,3}^{\ast }\text{ in }\mathbf{Z}_{min}\backslash %
\left[ \mathbf{O}_{min}\cup \left\{ Z_{u=t,1}^{\ast },Z_{u=t,2}^{\ast
}\right\} \right] \text{ such that }Z_{u=t,3}^{\ast }
\\&\text{ is a descendant
of }Z_{u=t,2}^{\ast }.
\end{align*}
Continuing in this fashion until depleting the set of vertices in $\mathbf{Z}%
_{min}$ we arrive at the conclussion that $\left( \ref{claim:dificil}\right) 
$ holds for $u=t.${} This concludes the proof of the part (2).

\medskip

\textbf{Proof of part (3). }Suppose there existed a minimal adjustment set $%
\mathbf{Z}_{min}$ that contained a vertex $O\in \mathbf{O}\backslash \mathbf{%
O}_{min}.$ Then $O\in \mathbf{O}$ and $O\in \mathbf{Z}_{min}\backslash 
\mathbf{O}_{min}.$ Part (2) of this Theorem then implies $Y\perp \!\!\!\perp
_{\mathcal{G}}O\mid \mathbf{O}_{min},A.$ This is impossible because by
Property (O) there exists a directed path from $O$ to $Y$ that does not
intersect $\mathbf{O}_{min}.$ The path also does not intersect $A.$
Consequently, by virtue of being a directed path, the path is open given $%
\mathbf{O}_{min}$ and $A.$ This concludes the proof of \ the Theorem.
\end{proof}

\medskip

\begin{proof}[Proof of Lemma \ref{lemma:supplementation_dep}]
First note that for $k\in \left\{ 0,\dots,p\right\} ,$ 
\begin{eqnarray}
\pi _{a_{k}}\left( \overline{\mathbf{G}}_{k},\overline{\mathbf{B}}%
_{k};P\right)  &\equiv &P\left( A_{k}=a_{k}|\overline{\mathbf{A}}_{k-1}=%
\overline{\mathbf{a}}_{k-1},\overline{\mathbf{G}}_{k},\overline{\mathbf{B}}%
_{k}\right)   \label{eq:pis} \\
&=&P\left( A_{k}=a_{k}|\overline{\mathbf{A}}_{k-1}=\overline{\mathbf{a}}%
_{k-1},\overline{\mathbf{B}}_{k}\right)   \nonumber \\
&\equiv &\pi _{a_{k}}\left( \overline{\mathbf{B}}_{k};P\right)   \nonumber
\end{eqnarray}%
where the second equality follows from $\left( \ref{eq:ind_good}\right)$.
Consequently, 
\begin{eqnarray*}
\chi _{\mathbf{a}}\left( P;\mathcal{G}\right)  &=&E_{P}\left[ \frac{I_{
\mathbf{a}}\left( \mathbf{A}\right) Y}{\prod\limits_{k=0}^{p}\pi
_{a_{k}}\left( \overline{\mathbf{B}}_{k};P\right) }\right]  \\
&=&E_{P}\left[ \frac{I_{\mathbf{a}}\left( \mathbf{A}\right) Y}{%
\prod\limits_{k=0}^{p}\pi _{a_{k}}\left( \overline{\mathbf{G}}_{k},\overline{%
\mathbf{B}}_{k};P\right) }\right].
\end{eqnarray*}%
The first equality is true because $\mathbf{B}$ is a time dependent
adjustment set. The second equality, which follows from $\left( \ref{eq:pis}%
\right) $, proves that $\left( \mathbf{G},\mathbf{B}\right) $ is also an
adjustment set. 

Next, for $k=0,\dots,p$, let 
\[
\Lambda _{k}\left( P\right) \equiv \left\{ q_{k}\left( \overline{\mathbf{A}}%
_{k},\overline{\mathbf{G}}_{k},\overline{\mathbf{B}}_{k}\right) :E_{P}\left[
q_{k}\left( \overline{\mathbf{A}}_{k},\overline{\mathbf{G}}_{k},\overline{%
\mathbf{B}}_{k}\right) |\overline{\mathbf{A}}_{k-1},\overline{\mathbf{G}}%
_{k},\overline{\mathbf{B}}_{k}\right] =0\right\}.
\]%
Note that for any $0\leq k\not=k^{\prime }\leq p,$ the elements of $\Lambda
_{k}\left( P\right) $ are uncorrelated under $P$ with those of $\Lambda _{k^{\prime
}}\left( P\right) .$ Note also that for any function $s_{k}\left( \overline{%
\mathbf{G}}_{k},\overline{\mathbf{B}}_{k}\right) $ and any $P\in \mathcal{M}%
\left( \mathcal{G}\right) ,$ the function 
\[
r_{k}\left( \overline{\mathbf{A}}_{k},\overline{\mathbf{G}}_{k},\overline{%
\mathbf{B}}_{k};s_{k},P\right) \equiv \frac{I_{\overline{\mathbf{a}}%
_{k-1}}\left( \overline{\mathbf{A}}_{k-1}\right) }{\lambda _{\overline{%
\mathbf{a}}_{k-1}}\left( \overline{\mathbf{B}}_{k-1};P\right) }\left\{ \frac{%
I_{a_{k}}\left( A_{k}\right) }{\pi _{a_{k}}\left( \overline{\mathbf{B}}%
_{k};P\right) }-1\right\} s_{k}\left( \overline{\mathbf{G}}_{k},\overline{%
\mathbf{B}}_{k}\right) 
\]%
belongs to $\Lambda _{k}\left( P\right) $ because 
\begin{align*}
&E_{P}\left[ r_{k}\left( \overline{\mathbf{A}}_{k},\overline{\mathbf{G}}_{k},%
\overline{\mathbf{B}}_{k};s_{k},P\right) |\overline{\mathbf{A}}_{k-1},%
\overline{\mathbf{G}}_{k},\overline{\mathbf{B}}_{k}\right]  
\\
&=\frac{I_{%
\overline{\mathbf{a}}_{k-1}}\left( \overline{\mathbf{A}}_{k-1}\right) }{%
\lambda _{\overline{\mathbf{a}}_{k-1}}\left( \overline{\mathbf{B}}%
_{k-1};P\right) }s_{k}\left( \overline{\mathbf{G}}_{k},\overline{\mathbf{B}}%
_{k}\right) E_{P}\left[ \left. \frac{I_{a_{k}}\left( A_{k}\right) }{\pi
_{a_{k}}\left( \overline{\mathbf{B}}_{k};P\right) }-1\right\vert \overline{%
\mathbf{A}}_{k-1},\overline{\mathbf{G}}_{k},\overline{\mathbf{B}}_{k}\right] 
\\
&=\frac{I_{\overline{\mathbf{a}}_{k-1}}\left( \overline{\mathbf{A}}%
_{k-1}\right) }{\lambda _{\overline{\mathbf{a}}_{k-1}}\left( \overline{%
\mathbf{B}}_{k-1};P\right) }s_{k}\left( \overline{\mathbf{G}}_{k},\overline{%
\mathbf{B}}_{k}\right) E_{P}\left[ \frac{E_{P}\left[ I_{a_{k}}\left(
A_{k}\right) |\overline{\mathbf{A}}_{k-1}=\overline{\mathbf{a}}_{k-1},%
\overline{\mathbf{G}}_{k},\overline{\mathbf{B}}_{k}\right] }{\pi
_{a_{k}}\left( \overline{\mathbf{B}}_{k};P\right) }-1\right]  \\
&=0,
\end{align*}%
where the last equality follows by $\left( \ref{eq:ind_good}\right) .$ Next,
write%
\[
\psi _{P,\mathbf{a}}\left( \mathbf{B};\mathcal{G}\right) =\psi _{P,\mathbf{a}%
}\left( \mathbf{G},\mathbf{B};\mathcal{G}\right) +\sum_{k=0}^{p}r_{k}\left( 
\overline{\mathbf{A}}_{k},\overline{\mathbf{G}}_{k},\overline{\mathbf{B}}%
_{k};s_{\mathbf{a,}k}^{\ast },P\right) 
\]%
where%
\[
s_{\mathbf{a,}k}^{\ast }\left( \overline{\mathbf{G}}_{k},\overline{\mathbf{B}%
}_{k}\right) \equiv b_{\overline{\mathbf{a}}_{k}}\left( \overline{\mathbf{G}}%
_{k},\overline{\mathbf{B}}_{k};P\right) -b_{\overline{\mathbf{a}}%
_{k}}\left( \overline{\mathbf{B}}%
_{k};P\right).
\]%
Noting that $\psi _{P,\mathbf{a}}\left( \mathbf{G},\mathbf{B};\mathcal{G}%
\right) $ is uncorrelated under $P$ with the elements of $\Lambda _{k}\left(
P\right) $ for all $0\leq k\leq p$ \citep{AIDS}, we conclude that 
\[
var_{P}\left[ \psi _{P,\mathbf{a}}\left( \mathbf{B};\mathcal{G}\right) %
\right] =var_{P}\left[ \psi _{P,\mathbf{a}}\left( \mathbf{G},\mathbf{B};%
\mathcal{G}\right) \right] +\sum_{k=0}^{p}var_{P}\left[ r_{k}\left( 
\overline{\mathbf{A}}_{k},\overline{\mathbf{G}}_{k},\overline{\mathbf{B}}%
_{k};s_{\mathbf{a,}k}^{\ast },P\right) \right].
\]%
Finally%
\begin{align*}
&var_{P}\left[ r_{k}\left( \overline{\mathbf{A}}_{k},\overline{\mathbf{G}}%
_{k},\overline{\mathbf{B}}_{k};s_{\mathbf{a,}k}^{\ast },P\right) \right] 
\\&=E_{P}\left[ \frac{I_{\overline{\mathbf{a}}_{k-1}}\left( \overline{\mathbf{%
A}}_{k-1}\right) }{\lambda _{\overline{\mathbf{a}}_{k-1}}\left( \overline{%
\mathbf{B}}_{k-1};P\right) ^{2}}\left\{ \frac{I_{a_{k}}\left( A_{k}\right) }{%
\pi _{a_{k}}\left( \overline{\mathbf{B}}_{k};P\right) }-1\right\}
^{2}\left\{ b_{\overline{\mathbf{a}}_{k}}\left( \overline{\mathbf{G}}_{k},%
\overline{\mathbf{B}}_{k};P\right) -b_{\overline{\mathbf{a}}_{k}}\left( 
\overline{\mathbf{B}}_{k};P\right) \right\}^2 \right]  \\
&=E_{P}\left[ \frac{I_{\overline{\mathbf{a}}_{k-1}}\left( \overline{\mathbf{%
A}}_{k-1}\right) }{\lambda _{\overline{\mathbf{a}}_{k-1}}\left( \overline{%
\mathbf{B}}_{k-1};P\right) ^{2}}\left\{ \frac{1}{\pi _{a_{k}}\left( 
\overline{\mathbf{B}}_{k};P\right) }-1\right\} \left\{ b_{\overline{\mathbf{a%
}}_{k}}\left( \overline{\mathbf{G}}_{k},\overline{\mathbf{B}}_{k};P\right)
-b_{\overline{\mathbf{a}}_{k}}\left( \overline{\mathbf{B}}_{k};P\right)
\right\}^2 \right]  \\
&=E_{P}\left[ \frac{I_{\overline{\mathbf{a}}_{k-1}}\left( \overline{\mathbf{%
A}}_{k-1}\right) }{\lambda _{\overline{\mathbf{a}}_{k-1}}\left( \overline{%
\mathbf{B}}_{k-1};P\right) ^{2}}\left\{ \frac{1}{\pi _{a_{k}}\left( 
\overline{\mathbf{B}}_{k};P\right) }-1\right\} \left\{ b_{\overline{\mathbf{a%
}}_{k}}\left( \overline{\mathbf{G}}_{k},\overline{\mathbf{B}}_{k};P\right)
-E_{P}\left[ b_{\overline{\mathbf{a}}_{k}}\left( \overline{\mathbf{G}}_{k},%
\overline{\mathbf{B}}_{k};P\right) |\overline{\mathbf{A}}_{k-1}=\overline{%
\mathbf{a}}_{k-1},\overline{\mathbf{B}}_{k}\right] \right\} ^{2}\right]  \\
&=E_{P}\left[ \frac{I_{\overline{\mathbf{a}}_{k-1}}\left( \overline{\mathbf{%
A}}_{k-1}\right) }{\lambda _{\overline{\mathbf{a}}_{k-1}}\left( \overline{%
\mathbf{B}}_{k-1};P\right) ^{2}}\left\{ \frac{1}{\pi _{a_{k}}\left( 
\overline{\mathbf{B}}_{k};P\right) }-1\right\} var_{P}\left[ b_{\overline{%
\mathbf{a}}_{k}}\left( \overline{\mathbf{G}}_{k},\overline{\mathbf{B}}%
_{k};P\right) |\overline{\mathbf{A}}_{k-1}=\overline{\mathbf{a}}_{k-1},%
\overline{\mathbf{B}}_{k}\right] \right].
\end{align*}
Next, noticing that $\mathbf{c}^{T}\mathbf{\psi }_{P}\left( \mathbf{B};%
\mathcal{G}\right) =\mathbf{c}^{T}\mathbf{\psi }_{P}\left( \mathbf{G},%
\mathbf{B};\mathcal{G}\right) +\sum_{k=0}^{p}t_{k}\left( \overline{\mathbf{A}%
}_{k},\overline{\mathbf{G}}_{k},\overline{\mathbf{B}}_{k},P\right) $ and
that $t_{k}\left( \overline{\mathbf{A}}_{k},\overline{\mathbf{G}}_{k},%
\overline{\mathbf{B}}_{k},P\right) \in \Lambda _{k}\left( P\right) $ we
obtain 
\begin{eqnarray*}
\sigma _{\mathbf{\Delta },\mathbf{B}}^{2}\left( P\right)  &=&var_{P}\left[ 
\mathbf{c}^{T}\mathbf{\psi }_{P}\left( \mathbf{B};\mathcal{G}\right) \right] 
\\
&=&var_{P}\left[ \mathbf{c}^{T}\mathbf{\psi }_{P}\left( \mathbf{G},\mathbf{B}%
;\mathcal{G}\right) \right] +\sum_{k=0}^{p}var_{P}\left[ t_{k}\left( 
\overline{\mathbf{A}}_{k},\overline{\mathbf{G}}_{k},\overline{\mathbf{B}}%
_{k},P\right) \right]  \\
&=&\sigma _{\mathbf{\Delta },\mathbf{G},\mathbf{B}}^{2}\left( P\right)
+\sum_{k=0}^{p}var_{P}\left[ t_{k}\left( \overline{\mathbf{A}}_{k},\overline{%
\mathbf{G}}_{k},\overline{\mathbf{B}}_{k},P\right) \right].
\end{eqnarray*}%
This concludes the proof of Lemma \ref{lemma:supplementation_dep}.
\end{proof}

\medskip

\begin{proof}[Proof of Lemma \protect\ref{lemma:deletion_dep}]
First we show by reverse induction in $k$ that for all $k\in \left\{
0,1,\dots,p\right\} $ it holds that 
\begin{equation}
b_{\overline{\mathbf{a}}_{k}}\left( \overline{\mathbf{B}}_{k},\overline{%
\mathbf{G}}_{k}\mathbf{;}P\right) =b_{\overline{\mathbf{a}}_{k}}\left( 
\overline{\mathbf{G}}_{k}\mathbf{;}P\right)   \label{eq:TD_1}.
\end{equation}%
This result immediately implies that $\mathbf{G=}\left( \mathbf{G}_{0},%
\mathbf{G}_{1},\dots,\mathbf{G}_{p}\right) $ is a time dependent adjustment
set because, 
\begin{eqnarray*}
\chi_{\mathbf{a}} \left( P;\mathcal{G}\right)  &\equiv &E_{P}\left[ b_{\overline{\mathbf{a%
}}_{0}}\left( \overline{\mathbf{B}}_{0},\overline{\mathbf{G}}_{0}\mathbf{;}%
P\right) \right]  \\
&=&E_{P}\left[ b_{\overline{\mathbf{a}}_{0}}\left( \overline{\mathbf{G}}_{0}%
\mathbf{;}P\right) \right],
\end{eqnarray*}%
where the first equality follows from the assumption that $\left( \mathbf{G,B%
}\right) $ is a time dependent adjustment set and the second follows from $%
\left( \ref{eq:TD_1}\right) $ applied to $k=0.$ We show that $\left( \ref%
{eq:TD_1}\right) $ holds for $k\in \left\{ 0,1,\dots,p\right\} $ by reverse
induction in $k.$ First note that 
\begin{eqnarray*}
b_{\overline{\mathbf{a}}_{p}}\left( \overline{\mathbf{B}}_{p},\overline{%
\mathbf{G}}_{p}\mathbf{;}P\right)  &\equiv &E_{P}\left[ Y|\mathbf{B},\mathbf{%
G},\overline{\mathbf{A}}_{p}=\overline{\mathbf{a}}_{p}\right]  \\
&=&E_{P}\left[ Y|\mathbf{G},\overline{\mathbf{A}}_{p}=\overline{\mathbf{a}}%
_{p}\right]  \\
&\equiv &b_{\overline{\mathbf{a}}_{p}}\left( \overline{\mathbf{G}}_{p}%
\mathbf{;}P\right) 
\end{eqnarray*}%
where the second equality follows by $\left( \ref{eq:ind_Y}\right) .$ Then $%
\left( \ref{eq:TD_1}\right) $ holds for $k=p.\,$\ Next, assume that $\left( %
\ref{eq:TD_1}\right) $ holds for $k\in\left\{ k^{\ast }+1,\dots,p\right\} $ for
some $k^{\ast }\geq 0.$ We will show that it holds for $k=k^{\ast }.$ This
follows from%
\begin{eqnarray*}
b_{\overline{\mathbf{a}}_{k^{\ast }}}\left( \overline{\mathbf{B}}_{k^{\ast
}},\overline{\mathbf{G}}_{k^{\ast }}\mathbf{;}P\right)  &\equiv &E_{P}\left[
b_{\overline{\mathbf{a}}_{k^{\ast }}}\left( \overline{\mathbf{B}}_{k^{\ast
}+1},\overline{\mathbf{G}}_{k^{\ast }+1}\mathbf{;}P\right) |\overline{%
\mathbf{B}}_{k^{\ast }},\overline{\mathbf{G}}_{k^{\ast }},\overline{\mathbf{A%
}}_{k^{\ast }}=\overline{\mathbf{a}}_{k^{\ast }}\right]  \\
&=&E_{P}\left[ b_{\overline{\mathbf{a}}_{k^{\ast }}}\left( \overline{\mathbf{%
G}}_{k^{\ast }+1}\mathbf{;}P\right) |\overline{\mathbf{B}}_{k^{\ast }},%
\overline{\mathbf{G}}_{k^{\ast }},\overline{\mathbf{A}}_{k^{\ast }}=%
\overline{\mathbf{a}}_{k^{\ast }}\right]  \\
&=&E_{P}\left[ b_{\overline{\mathbf{a}}_{k^{\ast }}}\left( \overline{\mathbf{%
G}}_{k^{\ast }+1}\mathbf{;}P\right) |\overline{\mathbf{G}}_{k^{\ast }},%
\overline{\mathbf{A}}_{k^{\ast }}=\overline{\mathbf{a}}_{k^{\ast }}\right] 
\\
&\equiv &b_{\overline{\mathbf{a}}_{k^{\ast }}}\left( \overline{\mathbf{G}}%
_{k^{\ast }}\mathbf{;}P\right),
\end{eqnarray*}%
where the second equality is by the inductive hypothesis and the third is by 
$\left( \ref{eq:ind_G}\right) $ applied to $j=k^{\ast }+1.$ Next we show
that for any $k\in \left\{ 0,\dots,p\right\} $%
\begin{equation}
E_{P}\left[ \left. \frac{1}{\lambda _{\overline{\mathbf{a}}_{k}}\left( 
\overline{\mathbf{G}}_{k},\overline{\mathbf{B}}_{k}\mathbf{;}P\right) }%
\right\vert \overline{\mathbf{G}}_{k}\mathbf{,}\overline{\mathbf{A}}_{k}%
\mathbf{=}\overline{\mathbf{a}}_{k}\right] =\frac{1}{\lambda _{\overline{%
\mathbf{a}}_{k}}\left( \overline{\mathbf{G}}_{k}\mathbf{;}P\right) }
\label{eq:reduccion}.
\end{equation}%
To do so we write for $k\in \left\{ 1,\dots,p\right\} $%
\begin{eqnarray*}
&&E_{P}\left[ \left. \frac{1}{\lambda _{\overline{\mathbf{a}}_{k}}\left( 
\overline{\mathbf{G}}_{k},\overline{\mathbf{B}}_{k}\mathbf{;}P\right) }%
\right\vert \overline{\mathbf{G}}_{k}\mathbf{,}\overline{\mathbf{A}}_{k}%
\mathbf{=}\overline{\mathbf{a}}_{k}\right] \pi _{a_{k}}\left( \overline{%
\mathbf{G}}_{k}\mathbf{;}P\right)  \\
&=&E_{P}\left[ \left. \frac{1}{\lambda _{\overline{\mathbf{a}}_{k-1}}\left( 
\overline{\mathbf{G}}_{k-1},\overline{\mathbf{B}}_{k-1}\mathbf{;}P\right) }%
\frac{1}{\pi _{a_{k}}\left( \overline{\mathbf{G}}_{k},\overline{\mathbf{B}}%
_{k}\mathbf{;}P\right) }\right\vert \overline{\mathbf{G}}_{k}\mathbf{,}%
\overline{\mathbf{A}}_{k-1}\mathbf{=}\overline{\mathbf{a}}_{k-1},{A}%
_{k}{=a}_{k}\right] \pi _{a_{k}}\left( \overline{\mathbf{G}}_{k}%
\mathbf{;}P\right)  \\
&=&E_{P}\left[ \left. \frac{1}{\lambda _{\overline{\mathbf{a}}_{k-1}}\left( 
\overline{\mathbf{G}}_{k-1},\overline{\mathbf{B}}_{k-1}\mathbf{;}P\right) }%
\frac{{A}_{k}}{\pi _{a_{k}}\left( \overline{\mathbf{G}}_{k},\overline{%
\mathbf{B}}_{k}\mathbf{;}P\right) }\right\vert \overline{\mathbf{G}}_{k}%
\mathbf{,}\overline{\mathbf{A}}_{k-1}\mathbf{=}\overline{\mathbf{a}}_{k-1}%
\right]  \\
&=&E_{P}\left[ \left. \frac{1}{\lambda _{\overline{\mathbf{a}}_{k-1}}\left( 
\overline{\mathbf{G}}_{k-1},\overline{\mathbf{B}}_{k-1}\mathbf{;}P\right) }%
\frac{E_{P}\left[ {A}_{k}|\overline{\mathbf{G}}_{k},\overline{\mathbf{%
B}}_{k},\overline{\mathbf{A}}_{k-1}\mathbf{=}\overline{\mathbf{a}}_{k-1}%
\right] }{\pi _{a_{k}}\left( \overline{\mathbf{G}}_{k},\overline{\mathbf{B}}%
_{k}\mathbf{;}P\right) }\right\vert \overline{\mathbf{G}}_{k}\mathbf{,}%
\overline{\mathbf{A}}_{k-1}\mathbf{=}\overline{\mathbf{a}}_{k-1}\right]  \\
&=&E_{P}\left[ \left. \frac{1}{\lambda _{\overline{\mathbf{a}}_{k-1}}\left( 
\overline{\mathbf{G}}_{k-1},\overline{\mathbf{B}}_{k-1}\mathbf{;}P\right) }%
\right\vert \overline{\mathbf{G}}_{k}\mathbf{,}\overline{\mathbf{A}}_{k-1}%
\mathbf{=}\overline{\mathbf{a}}_{k-1}\right]  \\
&=&E_{P}\left[ \left. \frac{1}{\lambda _{\overline{\mathbf{a}}_{k-1}}\left( 
\overline{\mathbf{G}}_{k-1},\overline{\mathbf{B}}_{k-1}\mathbf{;}P\right) }%
\right\vert \overline{\mathbf{G}}_{k-1}\mathbf{,}\overline{\mathbf{A}}_{k-1}%
\mathbf{=}\overline{\mathbf{a}}_{k-1}\right] 
\end{eqnarray*}%
where the last equality is by $\left( \ref{eq:ind_G}\right) $ applied to $%
j=k.$ In addition, 
\begin{eqnarray*}
&&E_{P}\left[ \left. \frac{1}{\lambda _{\overline{\mathbf{a}}_{0}}\left( 
\overline{\mathbf{G}}_{0},\overline{\mathbf{B}}_{0}\mathbf{;}P\right) }%
\right\vert \overline{\mathbf{G}}_{0}\mathbf{,}\overline{\mathbf{A}}_{0}%
\mathbf{=}\overline{\mathbf{a}}_{0}\right] \pi _{a_{0}}\left( \overline{%
\mathbf{G}}_{0}\mathbf{;}P\right)  \\
&=&E_{P}\left[ \left. \frac{1}{\pi _{a_{0}}\left( \mathbf{G}_{0},\mathbf{B}%
_{0}\mathbf{;}P\right) }\right\vert \mathbf{G}_{0}\mathbf{,}A_{0}\mathbf{=}%
a_{0}\right] P\left( A_{0}\mathbf{=}a_{0}|\mathbf{G}_{0}\right)  \\
&=&E_{P}\left[ \left. \frac{A_{0}}{\pi _{a_{0}}\left( \mathbf{G}_{0},\mathbf{%
B}_{0}\mathbf{;}P\right) }\right\vert \mathbf{G}_{0}\right]  \\
&=&1
\end{eqnarray*}%
so%
\[
E_{P}\left[ \left. \frac{1}{\lambda _{\overline{\mathbf{a}}_{0}}\left( 
\overline{\mathbf{G}}_{0},\overline{\mathbf{B}}_{0}\mathbf{;}P\right) }%
\right\vert \overline{\mathbf{G}}_{0}\mathbf{,}\overline{\mathbf{A}}_{0}%
\mathbf{=}\overline{\mathbf{a}}_{0}\right] =\frac{1}{\pi _{a_{0}}\left( 
\overline{\mathbf{G}}_{0}\mathbf{;}P\right) }.
\]%
Then, for any $k\in \left\{ 0,\dots,p\right\} $%
\begin{align*}
&E_{P}\left[ \left. \frac{1}{\lambda _{\overline{\mathbf{a}}_{k}}\left( 
\overline{\mathbf{G}}_{k},\overline{\mathbf{B}}_{k}\mathbf{;}P\right) }%
\right\vert \overline{\mathbf{G}}_{k}\mathbf{,}\overline{\mathbf{A}}_{k}%
\mathbf{=}\overline{\mathbf{a}}_{k}\right]  \\&=\frac{1}{\pi _{a_{k}}\left( 
\overline{\mathbf{G}}_{k}\mathbf{;}P\right) }E_{P}\left[ \left. \frac{1}{%
\lambda _{\overline{\mathbf{a}}_{k-1}}\left( \overline{\mathbf{G}}_{k-1},%
\overline{\mathbf{B}}_{k-1}\mathbf{;}P\right) }\right\vert \overline{\mathbf{%
G}}_{k-1}\mathbf{,}\overline{\mathbf{A}}_{k-1}\mathbf{=}\overline{\mathbf{a}}%
_{k-1}\right]  \\
&=\frac{1}{\pi _{a_{k-1}}\left( \overline{\mathbf{G}}_{k-1}\mathbf{;}%
P\right) }\frac{1}{\pi _{a_{k}}\left( \overline{\mathbf{G}}_{k}\mathbf{;}%
P\right) }E_{P}\left[ \left. \frac{1}{\lambda _{\overline{\mathbf{a}}%
_{k-2}}\left( \overline{\mathbf{G}}_{k-2},\overline{\mathbf{B}}_{k-2}\mathbf{%
;}P\right) }\right\vert \overline{\mathbf{G}}_{k-2}\mathbf{,}\overline{%
\mathbf{A}}_{k-2}\mathbf{=}\overline{\mathbf{a}}_{k-2}\right]  \\
&=\dots \\
&=\frac{1}{\pi _{a_{0}}\left( \overline{\mathbf{G}}_{0}\mathbf{;}P\right) }%
\frac{1}{\pi _{a_{k-1}}\left( \overline{\mathbf{G}}_{k-1}\mathbf{;}P\right) }%
\frac{1}{\pi _{a_{k}}\left( \overline{\mathbf{G}}_{k}\mathbf{;}P\right) } \\
&=\frac{1}{\lambda _{\overline{\mathbf{a}}_{k}}\left( \overline{\mathbf{G}}%
_{k}\mathbf{;}P\right) }.
\end{align*}%
Next, we note that re-arranging terms, the influence function \eqref{eq:inf_func_time_dep} can
be re-expressed as 
\[
\psi _{P,\mathbf{a}}\left( \mathbf{Z;}P\right) =\frac{I_{\mathbf{a}}\left( 
\mathbf{A}\right) }{\lambda _{\overline{\mathbf{a}}_{p}}\left( \mathbf{Z;}%
P\right) }\left\{ Y-b_{\overline{\mathbf{a}}_{p}}\left( \mathbf{Z;}P\right)
\right\} +\sum_{k=0}^{p}\frac{I_{\overline{\mathbf{a}}_{k-1}}\left( 
\overline{\mathbf{A}}_{k-1}\right) }{\lambda _{\overline{\mathbf{a}}%
_{k-1}}\left( \overline{\mathbf{Z}}_{k-1}\mathbf{;}P\right) }\left\{ b_{%
\overline{\mathbf{a}}_{k}}\left( \overline{\mathbf{Z}}_{k}\mathbf{;}P\right)
-b_{\overline{\mathbf{a}}_{k-1}}\left( \overline{\mathbf{Z}}_{k-1}\mathbf{;}%
P\right) \right\} 
\]%
where $b_{\overline{\mathbf{a}}_{-1}}\left( \overline{\mathbf{Z}}_{-1}%
\mathbf{;}P\right) \equiv \chi _{\mathbf{a}}\left( P;\mathcal{G}\right) .$
Furthermore, for any $\mathbf{Z,}$ the terms 
$$
\frac{I_{\overline{\mathbf{a}}%
_{k-1}}\left( \overline{\mathbf{A}}_{k-1}\right) }{\lambda _{\overline{%
\mathbf{a}}_{k-1}}\left( \overline{\mathbf{Z}}_{k-1}\mathbf{;}P\right) }%
\left\{ b_{\overline{\mathbf{a}}_{k}}\left( \overline{\mathbf{Z}}_{k}\mathbf{%
;}P\right) -b_{\overline{\mathbf{a}}_{k-1}}\left( \overline{\mathbf{Z}}_{k-1}%
\mathbf{;}P\right) \right\} ,k\in \left\{ 0,\dots,p\right\} 
$$
and 
$$
\frac{I_{%
\mathbf{a}}\left( \mathbf{A}\right) }{\lambda _{\overline{\mathbf{a}}%
_{p}}\left( \mathbf{Z;}P\right) }\left\{ Y-b_{\overline{\mathbf{a}}%
_{p}}\left( \mathbf{Z;}P\right) \right\} 
$$
are mutually uncorrelated under $P$. Then, 
\begin{align*}
var_{P}\left[ \psi _{P,\mathbf{a}}\left( \mathbf{G,B;}P\right) \right]
&=var_{P}\left[ \frac{I_{\mathbf{a}}\left( \mathbf{A}\right) }{\lambda _{%
\overline{\mathbf{a}}_{p}}\left( \mathbf{G,B;}P\right) }\left\{ Y-b_{%
\overline{\mathbf{a}}_{p}}\left( \mathbf{G,B;}P\right) \right\} \right]
\\
&+\sum_{k=0}^{p}var_{P}\left[ \frac{I_{\overline{\mathbf{a}}_{k-1}}\left( 
\overline{\mathbf{A}}_{k-1}\right) }{\lambda _{\overline{\mathbf{a}}%
_{k-1}}\left( \overline{\mathbf{G}}_{k-1},\overline{\mathbf{B}}_{k-1}\mathbf{%
;}P\right) }\left\{ b_{\overline{\mathbf{a}}_{k}}\left( \overline{\mathbf{G}}%
_{k},\overline{\mathbf{B}}_{k}\mathbf{;}P\right) -b_{\overline{\mathbf{a}}%
_{k-1}}\left( \overline{\mathbf{G}}_{k-1},\overline{\mathbf{B}}_{k-1}\mathbf{%
;}P\right) \right\} \right] .
\end{align*}
Now,%
\begin{eqnarray*}
&&E_{P}\left[ \left. \frac{I_{\mathbf{a}}\left( \mathbf{A}\right) }{\lambda
_{\overline{\mathbf{a}}_{p}}\left( \mathbf{G,B;}P\right) }\left\{ Y-b_{%
\overline{\mathbf{a}}_{p}}\left( \mathbf{G,B;}P\right) \right\} \right\vert
Y,\overline{\mathbf{G}}_{p},\overline{\mathbf{A}}_{p}\right]  \\
&=&I_{\mathbf{a}}\left( \mathbf{A}\right) \left\{ Y-b_{\overline{\mathbf{a}}%
_{p}}\left( \mathbf{G;}P\right) \right\} E_{P}\left[ \left. \frac{1}{\lambda
_{\overline{\mathbf{a}}_{p}}\left( \overline{\mathbf{G}}_{p},\overline{%
\mathbf{B}}_{p}\mathbf{;}P\right) }\right\vert Y,\overline{\mathbf{G}}_{p},%
\overline{\mathbf{A}}_{p}=\overline{\mathbf{a}}_{p}\right]  \\
&=&I_{\mathbf{a}}\left( \mathbf{A}\right) \left\{ Y-b_{\overline{\mathbf{a}}%
_{p}}\left( \mathbf{G;}P\right) \right\} E_{P}\left[ \left. \frac{1}{\lambda
_{\overline{\mathbf{a}}_{p}}\left( \overline{\mathbf{G}}_{p},\overline{%
\mathbf{B}}_{p}\mathbf{;}P\right) }\right\vert \overline{\mathbf{G}}_{p},%
\overline{\mathbf{A}}_{p}=\overline{\mathbf{a}}_{p}\right]  \\
&=&\frac{I_{\mathbf{a}}\left( \mathbf{A}\right) }{\lambda _{\overline{%
\mathbf{a}}_{p}}\left( \mathbf{G;}P\right) }\left\{ Y-b_{\overline{\mathbf{a}%
}_{p}}\left( \mathbf{G;}P\right) \right\} 
\end{eqnarray*}%
where the first equality is by $\left( \ref{eq:TD_1}\right) $ applied to $k=p
$, the second is by $\left( \ref{eq:ind_Y}\right) $ and the third is by $%
\left( \ref{eq:reduccion}\right) $ applied to $k=p.$ Also, for any $k\in
\left\{ 0,\dots,p\right\} $%
\begin{eqnarray*}
&&E_{P}\left[ \left. \frac{I_{\overline{\mathbf{a}}_{k-1}}\left( \overline{%
\mathbf{A}}_{k-1}\right) }{\lambda _{\overline{\mathbf{a}}_{k-1}}\left( 
\overline{\mathbf{G}}_{k-1},\overline{\mathbf{B}}_{k-1}\mathbf{;}P\right) }%
\left\{ b_{\overline{\mathbf{a}}_{k}}\left( \overline{\mathbf{G}}_{k},%
\overline{\mathbf{B}}_{k}\mathbf{;}P\right) -b_{\overline{\mathbf{a}}%
_{k-1}}\left( \overline{\mathbf{G}}_{k-1},\overline{\mathbf{B}}_{k-1}\mathbf{%
;}P\right) \right\} \right\vert \overline{\mathbf{G}}_{k},\overline{\mathbf{A%
}}_{k-1}\right]  \\
&=&I_{\overline{\mathbf{a}}_{k-1}}\left( \overline{\mathbf{A}}_{k-1}\right)
\left\{ b_{\overline{\mathbf{a}}_{k}}\left( \overline{\mathbf{G}}_{k}\mathbf{%
;}P\right) -b_{\overline{\mathbf{a}}_{k-1}}\left( \overline{\mathbf{G}}_{k-1}%
\mathbf{;}P\right) \right\} E_{P}\left[ \left. \frac{1}{\lambda _{\overline{%
\mathbf{a}}_{k-1}}\left( \overline{\mathbf{G}}_{k-1},\overline{\mathbf{B}}%
_{k-1}\mathbf{;}P\right) }\right\vert \overline{\mathbf{G}}_{k},\overline{%
\mathbf{A}}_{k-1}=\overline{\mathbf{a}}_{k-1}\right]  \\
&=&I_{\overline{\mathbf{a}}_{k-1}}\left( \overline{\mathbf{A}}_{k-1}\right)
\left\{ b_{\overline{\mathbf{a}}_{k}}\left( \overline{\mathbf{G}}_{k}\mathbf{%
;}P\right) -b_{\overline{\mathbf{a}}_{k-1}}\left( \overline{\mathbf{G}}_{k-1}%
\mathbf{;}P\right) \right\} E_{P}\left[ \left. \frac{1}{\lambda _{\overline{%
\mathbf{a}}_{k-1}}\left( \overline{\mathbf{G}}_{k-1},\overline{\mathbf{B}}%
_{k-1}\mathbf{;}P\right) }\right\vert \overline{\mathbf{G}}_{k-1},\overline{%
\mathbf{A}}_{k-1}=\overline{\mathbf{a}}_{k-1}\right]  \\
&=&\frac{I_{\overline{\mathbf{a}}_{k-1}}\left( \overline{\mathbf{A}}%
_{k-1}\right) }{\lambda _{\overline{\mathbf{a}}_{k-1}}\left( \overline{%
\mathbf{G}}_{k-1}\mathbf{;}P\right) }\left\{ b_{\overline{\mathbf{a}}%
_{k}}\left( \overline{\mathbf{G}}_{k}\mathbf{;}P\right) -b_{\overline{%
\mathbf{a}}_{k-1}}\left( \overline{\mathbf{G}}_{k-1}\mathbf{;}P\right)
\right\} 
\end{eqnarray*}%
the first equality is by $\left( \ref{eq:TD_1}\right) ,$ the second is by $%
\left( \ref{eq:ind_G}\right) $ and the third is by $\left( \ref{eq:reduccion}%
\right) $ and where, recall, for $k=0$, $I_{\overline{\mathbf{a}}%
_{k-1}}\left( \overline{\mathbf{A}}_{k-1}\right) \equiv \lambda _{\overline{%
\mathbf{a}}_{k-1}}\left( \overline{\mathbf{G}}_{k-1}\mathbf{;}P\right)
\equiv 1$ and $b_{\overline{\mathbf{a}}_{k-1}}\left( \overline{\mathbf{G}}%
_{k-1}\mathbf{;}P\right) \equiv \chi_{\mathbf{a}} \left( P;\mathcal{G}\right) .$ 

Then%
\begin{align*}
&var_{P}\left[ \psi _{P,\mathbf{a}}\left( \mathbf{G,B;}P\right) \right]  \\
&=var_{P}\left[ \frac{I_{\mathbf{a}}\left( \mathbf{A}\right) }{\lambda _{%
\overline{\mathbf{a}}_{p}}\left( \mathbf{G,B;}P\right) }\left\{ Y-b_{%
\overline{\mathbf{a}}_{p}}\left( \mathbf{G,B;}P\right) \right\} \right]
\\
&+\sum_{k=0}^{p}var_{P}\left[ \frac{I_{\overline{\mathbf{a}}_{k-1}}\left( 
\overline{\mathbf{A}}_{k-1}\right) }{\lambda _{\overline{\mathbf{a}}%
_{k-1}}\left( \overline{\mathbf{G}}_{k-1},\overline{\mathbf{B}}_{k-1}\mathbf{%
;}P\right) }\left\{ b_{\overline{\mathbf{a}}_{k}}\left( \overline{\mathbf{G}}%
_{k},\overline{\mathbf{B}}_{k}\mathbf{;}P\right) -b_{\overline{\mathbf{a}}%
_{k-1}}\left( \overline{\mathbf{G}}_{k-1},\overline{\mathbf{B}}_{k-1}\mathbf{%
;}P\right) \right\} \right]  
\\
&=var_{P}\left[ E_{P}\left[ \left. \frac{I_{\mathbf{a}}\left( \mathbf{A}%
\right) }{\lambda _{\overline{\mathbf{a}}_{p}}\left( \mathbf{G,B;}P\right) }%
\left\{ Y-b_{\overline{\mathbf{a}}_{p}}\left( \mathbf{G,B;}P\right) \right\}
\right\vert Y,\overline{\mathbf{G}}_{p},\overline{\mathbf{A}}_{p}\right] %
\right] 
\\
&+E_{P}\left[ var_{P}\left[ \left. \frac{I_{\mathbf{a}}\left( \mathbf{%
A}\right) }{\lambda _{\overline{\mathbf{a}}_{p}}\left( \mathbf{G,B;}P\right) 
}\left\{ Y-b_{\overline{\mathbf{a}}_{p}}\left( \mathbf{G,B;}P\right)
\right\} \right\vert Y,\overline{\mathbf{G}}_{p},\overline{\mathbf{A}}_{p}%
\right] \right]  
\\
&+\sum_{k=0}^{p}var_{P}\left[ E_{P}\left[ \left. \frac{I_{\overline{\mathbf{%
a}}_{k-1}}\left( \overline{\mathbf{A}}_{k-1}\right) }{\lambda _{\overline{%
\mathbf{a}}_{k-1}}\left( \overline{\mathbf{G}}_{k-1},\overline{\mathbf{B}}%
_{k-1}\mathbf{;}P\right) }\left\{ b_{\overline{\mathbf{a}}_{k}}\left( 
\overline{\mathbf{G}}_{k},\overline{\mathbf{B}}_{k}\mathbf{;}P\right) -b_{%
\overline{\mathbf{a}}_{k-1}}\left( \overline{\mathbf{G}}_{k-1},\overline{%
\mathbf{B}}_{k-1}\mathbf{;}P\right) \right\} \right\vert \overline{\mathbf{G}%
}_{k},\overline{\mathbf{A}}_{k-1}\right] \right] 
\\
&+\sum_{k=0}^{p}E_{P}\left[ var_{P}\left[ \left. \frac{I_{\overline{\mathbf{%
a}}_{k-1}}\left( \overline{\mathbf{A}}_{k-1}\right) }{\lambda _{\overline{%
\mathbf{a}}_{k-1}}\left( \overline{\mathbf{G}}_{k-1},\overline{\mathbf{B}}%
_{k-1}\mathbf{;}P\right) }\left\{ b_{\overline{\mathbf{a}}_{k}}\left( 
\overline{\mathbf{G}}_{k},\overline{\mathbf{B}}_{k}\mathbf{;}P\right) -b_{%
\overline{\mathbf{a}}_{k-1}}\left( \overline{\mathbf{G}}_{k-1},\overline{%
\mathbf{B}}_{k-1}\mathbf{;}P\right) \right\} \right\vert \overline{\mathbf{G}%
}_{k},\overline{\mathbf{A}}_{k-1}\right] \right]  
\\
&=var_{P}\left[ \frac{I_{\mathbf{a}}\left( \mathbf{A}\right) }{\lambda _{%
\overline{\mathbf{a}}_{p}}\left( \mathbf{G;}P\right) }\left\{ Y-b_{\overline{%
\mathbf{a}}_{p}}\left( \mathbf{G;}P\right) \right\} \right]
+\sum_{k=0}^{p}var_{P}\left[ \frac{I_{\overline{\mathbf{a}}_{k-1}}\left( 
\overline{\mathbf{A}}_{k-1}\right) }{\lambda _{\overline{\mathbf{a}}%
_{k-1}}\left( \overline{\mathbf{G}}_{k-1}\mathbf{;}P\right) }\left\{ b_{%
\overline{\mathbf{a}}_{k}}\left( \overline{\mathbf{G}}_{k}\mathbf{;}P\right)
-b_{\overline{\mathbf{a}}_{k-1}}\left( \overline{\mathbf{G}}_{k-1}\mathbf{;}%
P\right) \right\} \right]  
\\
&+E_{P}\left[ var_{P}\left[ \left. \frac{I_{\mathbf{a}}\left( \mathbf{A}%
\right) }{\lambda _{\overline{\mathbf{a}}_{p}}\left( \mathbf{G,B;}P\right) }%
\left\{ Y-b_{\overline{\mathbf{a}}_{p}}\left( \mathbf{G,B;}P\right) \right\}
\right\vert Y,\overline{\mathbf{G}}_{p},\overline{\mathbf{A}}_{p}\right] %
\right]  \\
&+\sum_{k=0}^{p}E_{P}\left[ var_{P}\left[ \left. \frac{I_{\overline{\mathbf{%
a}}_{k-1}}\left( \overline{\mathbf{A}}_{k-1}\right) }{\lambda _{\overline{%
\mathbf{a}}_{k-1}}\left( \overline{\mathbf{G}}_{k-1},\overline{\mathbf{B}}%
_{k-1}\mathbf{;}P\right) }\left\{ b_{\overline{\mathbf{a}}_{k}}\left( 
\overline{\mathbf{G}}_{k},\overline{\mathbf{B}}_{k}\mathbf{;}P\right) -b_{%
\overline{\mathbf{a}}_{k-1}}\left( \overline{\mathbf{G}}_{k-1},\overline{%
\mathbf{B}}_{k-1}\mathbf{;}P\right) \right\} \right\vert \overline{\mathbf{G}%
}_{k},\overline{\mathbf{A}}_{k-1}\right] \right]  \\
&=var_{P}\left[ \psi _{P,\mathbf{a}}\left( \mathbf{G;}P\right) \right]
+E_{P}\left[ var_{P}\left[ \left. \frac{I_{\mathbf{a}}\left( \mathbf{A}%
\right) }{\lambda _{\overline{\mathbf{a}}_{p}}\left( \mathbf{G,B;}P\right) }%
\left\{ Y-b_{\overline{\mathbf{a}}_{p}}\left( \mathbf{G,B;}P\right) \right\}
\right\vert Y,\overline{\mathbf{G}}_{p},\overline{\mathbf{A}}_{p}\right] %
\right]  \\
&+\sum_{k=0}^{p}E_{P}\left[ var_{P}\left[ \left. \frac{I_{\overline{\mathbf{%
a}}_{k-1}}\left( \overline{\mathbf{A}}_{k-1}\right) }{\lambda _{\overline{%
\mathbf{a}}_{k-1}}\left( \overline{\mathbf{G}}_{k-1},\overline{\mathbf{B}}%
_{k-1}\mathbf{;}P\right) }\left\{ b_{\overline{\mathbf{a}}_{k}}\left( 
\overline{\mathbf{G}}_{k},\overline{\mathbf{B}}_{k}\mathbf{;}P\right) -b_{%
\overline{\mathbf{a}}_{k-1}}\left( \overline{\mathbf{G}}_{k-1},\overline{%
\mathbf{B}}_{k-1}\mathbf{;}P\right) \right\} \right\vert \overline{\mathbf{G}%
}_{k},\overline{\mathbf{A}}_{k-1}\right] \right].
\end{align*}
Next, 
\begin{align*}
&\sigma _{\Delta ,\mathbf{B}}^{2}\left( \mathbf{G},\mathbf{B}\right) \equiv
var_{P}\left[ \mathbf{c}^{T}\mathbf{\psi }_{P}\left( \mathbf{G},\mathbf{B;}%
P\right) \right]  \\
&=var_{P}\left[ \sum_{a\in \mathcal{A}}c_{\mathbf{a}}\frac{I_{\mathbf{a}%
}\left( \mathbf{A}\right) }{\lambda _{\overline{\mathbf{a}}_{p}}\left( 
\mathbf{G,B;}P\right) }\left\{ Y-b_{\overline{\mathbf{a}}_{p}}\left( \mathbf{%
G,B;}P\right) \right\} \right] 
\\
&+\sum_{k=0}^{p}var_{P}\left[ \sum_{a\in 
\mathcal{A}}c_{\mathbf{a}}\frac{I_{\overline{\mathbf{a}}_{k-1}}\left( 
\overline{\mathbf{A}}_{k-1}\right) }{\lambda _{\overline{\mathbf{a}}%
_{k-1}}\left( \overline{\mathbf{G}}_{k-1},\overline{\mathbf{B}}_{k-1}\mathbf{%
;}P\right) }\left\{ b_{\overline{\mathbf{a}}_{k}}\left( \overline{\mathbf{G}}%
_{k},\overline{\mathbf{B}}_{k}\mathbf{;}P\right) -b_{\overline{\mathbf{a}}%
_{k-1}}\left( \overline{\mathbf{G}}_{k-1},\overline{\mathbf{B}}_{k-1}\mathbf{%
;}P\right) \right\} \right]  
\\
&=var_{P}\left[ E_{P}\left[ \left. \sum_{a\in \mathcal{A}}c_{\mathbf{a}}%
\frac{I_{\mathbf{a}}\left( \mathbf{A}\right) }{\lambda _{\overline{\mathbf{a}%
}_{p}}\left( \mathbf{G,B;}P\right) }\left\{ Y-b_{\overline{\mathbf{a}}%
_{p}}\left( \mathbf{G,B;}P\right) \right\} \right\vert Y,\overline{\mathbf{G}%
}_{p},\overline{\mathbf{A}}_{p}\right] \right] 
\\
&+E_{P}\left[ var_{P}\left[
\sum_{a\in \mathcal{A}}c_{\mathbf{a}}\left. \frac{I_{\mathbf{a}}\left( 
\mathbf{A}\right) }{\lambda _{\overline{\mathbf{a}}_{p}}\left( \mathbf{G,B;}%
P\right) }\left\{ Y-b_{\overline{\mathbf{a}}_{p}}\left( \mathbf{G,B;}%
P\right) \right\} \right\vert Y,\overline{\mathbf{G}}_{p},\overline{\mathbf{A%
}}_{p}\right] \right]  
\\
&+\sum_{k=0}^{p}var_{P}\left[ E_{P}\left[ \sum_{a\in \mathcal{A}}c_{\mathbf{%
a}}\left. \frac{I_{\overline{\mathbf{a}}_{k-1}}\left( \overline{\mathbf{A}}%
_{k-1}\right) }{\lambda _{\overline{\mathbf{a}}_{k-1}}\left( \overline{%
\mathbf{G}}_{k-1},\overline{\mathbf{B}}_{k-1}\mathbf{;}P\right) }\left\{ b_{%
\overline{\mathbf{a}}_{k}}\left( \overline{\mathbf{G}}_{k},\overline{\mathbf{%
B}}_{k}\mathbf{;}P\right) -b_{\overline{\mathbf{a}}_{k-1}}\left( \overline{%
\mathbf{G}}_{k-1},\overline{\mathbf{B}}_{k-1}\mathbf{;}P\right) \right\}
\right\vert \overline{\mathbf{G}}_{k},\overline{\mathbf{A}}_{k-1}\right] %
\right]  
\\
&+\sum_{k=0}^{p}E_{P}\left[ var_{P}\left[ \sum_{a\in \mathcal{A}}c_{\mathbf{%
a}}\left. \frac{I_{\overline{\mathbf{a}}_{k-1}}\left( \overline{\mathbf{A}}%
_{k-1}\right) }{\lambda _{\overline{\mathbf{a}}_{k-1}}\left( \overline{%
\mathbf{G}}_{k-1},\overline{\mathbf{B}}_{k-1}\mathbf{;}P\right) }\left\{ b_{%
\overline{\mathbf{a}}_{k}}\left( \overline{\mathbf{G}}_{k},\overline{\mathbf{%
B}}_{k}\mathbf{;}P\right) -b_{\overline{\mathbf{a}}_{k-1}}\left( \overline{%
\mathbf{G}}_{k-1},\overline{\mathbf{B}}_{k-1}\mathbf{;}P\right) \right\}
\right\vert \overline{\mathbf{G}}_{k},\overline{\mathbf{A}}_{k-1}\right] %
\right] 
\\
&=var_{P}\left[ \sum_{a\in \mathcal{A}}c_{\mathbf{a}}\frac{I_{\mathbf{a}%
}\left( \mathbf{A}\right) }{\lambda _{\overline{\mathbf{a}}_{p}}\left( 
\mathbf{G;}P\right) }\left\{ Y-b_{\overline{\mathbf{a}}_{p}}\left( \mathbf{G;%
}P\right) \right\} \right] 
\\
&+\sum_{k=0}^{p}var_{P}\left[ \sum_{a\in \mathcal{A%
}}c_{\mathbf{a}}\frac{I_{\overline{\mathbf{a}}_{k-1}}\left( \overline{%
\mathbf{A}}_{k-1}\right) }{\lambda _{\overline{\mathbf{a}}_{k-1}}\left( 
\overline{\mathbf{G}}_{k-1}\mathbf{;}P\right) }\left\{ b_{\overline{\mathbf{a%
}}_{k}}\left( \overline{\mathbf{G}}_{k}\mathbf{;}P\right) -b_{\overline{%
\mathbf{a}}_{k-1}}\left( \overline{\mathbf{G}}_{k-1}\mathbf{;}P\right)
\right\} \right]  
\\
&+E_{P}\left[ var_{P}\left[ \sum_{a\in \mathcal{A}}c_{\mathbf{a}}\left. 
\frac{I_{\mathbf{a}}\left( \mathbf{A}\right) }{\lambda _{\overline{\mathbf{a}%
}_{p}}\left( \mathbf{G,B;}P\right) }\left\{ Y-b_{\overline{\mathbf{a}}%
_{p}}\left( \mathbf{G,B;}P\right) \right\} \right\vert Y,\overline{\mathbf{G}%
}_{p},\overline{\mathbf{A}}_{p}\right] \right]  
\\
&+\sum_{k=0}^{p}E_{P}\left[ var_{P}\left[ \sum_{a\in \mathcal{A}}c_{\mathbf{%
a}}\left. \frac{I_{\overline{\mathbf{a}}_{k-1}}\left( \overline{\mathbf{A}}%
_{k-1}\right) }{\lambda _{\overline{\mathbf{a}}_{k-1}}\left( \overline{%
\mathbf{G}}_{k-1},\overline{\mathbf{B}}_{k-1}\mathbf{;}P\right) }\left\{ b_{%
\overline{\mathbf{a}}_{k}}\left( \overline{\mathbf{G}}_{k},\overline{\mathbf{%
B}}_{k}\mathbf{;}P\right) -b_{\overline{\mathbf{a}}_{k-1}}\left( \overline{%
\mathbf{G}}_{k-1},\overline{\mathbf{B}}_{k-1}\mathbf{;}P\right) \right\}
\right\vert \overline{\mathbf{G}}_{k},\overline{\mathbf{A}}_{k-1}\right] %
\right]  
\\
&=var_{P}\left[ \mathbf{c}^{T}\mathbf{\psi }_{P}\left( \mathbf{G;}P\right) %
\right] +E_{P}\left[ var_{P}\left[ \sum_{a\in \mathcal{A}}c_{\mathbf{a}%
}\left. \frac{I_{\mathbf{a}}\left( \mathbf{A}\right) }{\lambda _{\overline{%
\mathbf{a}}_{p}}\left( \mathbf{G,B;}P\right) }\left\{ Y-b_{\overline{\mathbf{%
a}}_{p}}\left( \mathbf{G,B;}P\right) \right\} \right\vert Y,\overline{%
\mathbf{G}}_{p},\overline{\mathbf{A}}_{p}\right] \right]  
\\
&+\sum_{k=0}^{p}E_{P}\left[ var_{P}\left[ \sum_{a\in \mathcal{A}}c_{\mathbf{%
a}}\left. \frac{I_{\overline{\mathbf{a}}_{k-1}}\left( \overline{\mathbf{A}}%
_{k-1}\right) }{\lambda _{\overline{\mathbf{a}}_{k-1}}\left( \overline{%
\mathbf{G}}_{k-1},\overline{\mathbf{B}}_{k-1}\mathbf{;}P\right) }\left\{ b_{%
\overline{\mathbf{a}}_{k}}\left( \overline{\mathbf{G}}_{k},\overline{\mathbf{%
B}}_{k}\mathbf{;}P\right) -b_{\overline{\mathbf{a}}_{k-1}}\left( \overline{%
\mathbf{G}}_{k-1},\overline{\mathbf{B}}_{k-1}\mathbf{;}P\right) \right\}
\right\vert \overline{\mathbf{G}}_{k},\overline{\mathbf{A}}_{k-1}\right] %
\right].
\end{align*}%
This concludes the proof of Lemma \ref{lemma:deletion_dep}.
\end{proof}

\subsubsection{Proofs of results in Section \protect\ref{sec:efficient_est}}

\begin{lemma}\label{lemma:ARE}
For $\mathcal{G}$ the DAG in Figure \ref{fig:adj}, let
$P_{\alpha}\in\mathcal{M(G)}$ satisfy
\begin{enumerate}
    \item $b_{a}\left( \mathbf{O};P_{\alpha}\right)
=O_{1}+O_{2}+\alpha O_{1}O_{2}$,
    \item  $%
E_{P_{\alpha}}\left( O_{1}\right) =E_{P_{\alpha}}\left( O_{2}\right) =0$,
\item $%
E_{P_{\alpha}}\left( O_{1}^{2}\right) =E_{P_{\alpha}}\left( O_{2}^{2}\right) =1$,
\item There exists a fixed $C>0$ independent of $\alpha$ such that $
var_{P_{\alpha}}\left(Y \mid A=a, \mathbf{O} \right)\leq C$ and $\pi_{a}(\mathbf{O}_{min};P_{\alpha}) \geq 1/C$.
\end{enumerate}
Then
$$
\Delta _{P_{\alpha}}\left( \mathbf{O}\right)  = b_{a}\left( \mathbf{O};P_{\alpha}\right)
-E_{P_{\alpha}}\left[ b_{a}\left( \mathbf{O};P_{\alpha}\right) |O_{1}\right] -E_{P_{\alpha}}\left[
b_{a}\left( \mathbf{O};P_{\alpha}\right) |O_{2}\right] +E_{P_{\alpha}}\left[ b_{a}\left( \mathbf{O}%
;P_{\alpha}\right) \right]= \alpha O_{1}O_{2}
$$
and
\begin{equation*}
\frac{var_{P_{\alpha}}\left[\psi _{P_{\alpha},a}\left( \mathbf{O};\mathcal{G}\right) \right] 
}{var_{P_{\alpha}}\left[ \chi _{P,a,eff}^{1}\left( \mathbf{V};\mathcal{G}\right)
\right] }\underset{\left\vert \alpha \right\vert \rightarrow \infty }{%
\rightarrow }\infty.
\end{equation*}
\end{lemma}

\begin{proof}

\[
\psi _{,a}(\mathbf{O};\mathcal{G})\equiv \frac{I_{a}(A)}{\pi _{a}(\mathbf{O}%
_{min};P)}\left\{ Y-b_{a}\left( \mathbf{O};P\right)
\right\} +b_{a}\left( \mathbf{O};P\right) -\chi _{a}(P;%
\mathcal{G})
\]%
is an influence function of $\chi _{a}\left( P;\mathcal{G}\right) $ under
the Bayesian Network $\mathcal{M}\left( \mathcal{G}\right) $. This is
because by $\mathbf{O}$ being an adjustment set we know that for all $P\in $ 
$\mathcal{M}\left( \mathcal{G}\right) ,$ $\chi _{a}\left( P;\mathcal{G}%
\right) =E_{P}\left[ E_{p}\left( Y|A=a,\mathbf{O}\right) \right] .$ Then, 
\[
\chi _{P,a,eff}^{1}(\mathbf{V};\mathcal{G})=\Pi \left[ \psi _{P,a}(\mathbf{O}%
;\mathcal{G})|\Lambda \left( P\right) \right] 
\]%
where $\Lambda \left( P\right) $ is the tangent space of model $\mathcal{M}%
\left( \mathcal{G}\right) $ at $P.$ Consequently,
\[
\Delta _{P_{\alpha }}\left( \mathbf{O}\right) =\Pi \left[ \psi _{P_{\alpha
},a}(\mathbf{O};\mathcal{G})|\Lambda \left( P_{\alpha}\right) ^{\perp }\right] 
\]%
and by Pythagoras's Theorem, we have 
\[
var_{P_{\alpha }}\left[ \chi _{P_{\alpha},a,eff}^{1}(\mathbf{V};\mathcal{G})\right]
=var_{P_{\alpha }}\left[ \psi _{P_{\alpha },a}(\mathbf{O};\mathcal{G})\right]
-var_{P_{\alpha }}\left[ \Delta _{P_{\alpha }}\left( \mathbf{O}\right) %
\right] 
\]%
Therefore, 
\[
\frac{var_{P_{\alpha }}\left[ \chi _{P_{\alpha},a,eff}^{1}(\mathbf{V};\mathcal{G})%
\right] }{var_{P_{\alpha }}\left[ \psi _{P_{\alpha },a}(\mathbf{O};\mathcal{G%
})\right] }=1-\frac{var_{P_{\alpha }}\left[ \Delta _{P_{\alpha }}\left( 
\mathbf{O}\right) \right] }{var_{P_{\alpha }}\left[ \psi _{P_{\alpha },a}(%
\mathbf{O};\mathcal{G})\right] }.
\]%
Now, $O_{1}$ and $O_{2}$ are marginally independent under all $P\in \mathcal{%
M(G)}$. Since $E_{P_{\alpha }}\left( O_{1}\right) =E_{P_{\alpha }}\left(
O_{2}\right) =0$, we have that $E_{P_{\alpha }}\left[ b_{a}\left( \mathbf{O}%
;P_{\alpha }\right) |O_{1}\right] =O_{1}$, $E_{P_{\alpha }}\left[
b_{a}\left( \mathbf{O};P_{\alpha }\right) |O_{2}\right] =O_{2}$ and $%
E_{P_{\alpha }}\left[ b_{a}\left( \mathbf{O};P_{\alpha }\right) \right] =0$.
Thus 
\[
\Delta _{P_{\alpha }}\left( \mathbf{O}\right) =b_{a}\left( \mathbf{O}%
;P_{\alpha }\right) -E_{P_{\alpha }}\left[ b_{a}\left( \mathbf{O};P_{\alpha
}\right) |O_{1}\right] -E_{P_{\alpha }}\left[ b_{a}\left( \mathbf{O}%
;P_{\alpha }\right) |O_{2}\right] +E_{P_{\alpha }}\left[ b_{a}\left( \mathbf{%
O};P_{\alpha }\right) \right] =\alpha O_{1}O_{2}
\]%
Consequently, 
\[
var_{P_{\alpha }}\left[ \Delta _{P_{\alpha }}\left( \mathbf{O}\right) \right]
=\alpha ^{2}E_{P_{\alpha }}\left[ O_{1}^{2}O_{2}^{2}\right] =\alpha ^{2}.
\]%
On the other hand,       
\begin{align*}
var_{P_{\alpha }}\left[ \psi _{P,a}(\mathbf{V};\mathcal{G})\right] &
=var_{P_{\alpha }}\left[ \frac{I_{a}(A)}{\pi _{a}(\mathbf{O}_{min};P_{\alpha
})}\left\{ Y-b_{a}\left( \mathbf{O};P_{\alpha }\right) \right\} \right]
+var_{P_{\alpha }}\left[ b_{a}\left( \mathbf{O};P_{\alpha }\right) -\chi
_{a}(P_{\alpha };\mathcal{G})\right]  \\
& =E_{P_{\alpha }}\left[ \frac{I_{a}(A)}{\pi _{a}^{2}(\mathbf{O}%
_{min};P_{\alpha })}\left\{ Y-b_{a}\left( \mathbf{O};P_{\alpha }\right)
\right\} ^{2}\right] +E_{P_{\alpha }}\left[ b_{a}^{2}\left( \mathbf{O}%
;P_{\alpha }\right) \right]  \\
& =E_{P_{\alpha }}\left[ \frac{I_{a}(A)}{\pi _{a}^{2}(\mathbf{O}%
_{min};P_{\alpha })}var_{P_{\alpha }}(Y\mid A=a,\mathbf{O})\right]
+E_{P_{\alpha }}\left[ b_{a}^{2}\left( \mathbf{O};P_{\alpha }\right) \right]
.
\end{align*}%
Moreover, since $O_{1}$ and $O_{2}$ have zero mean, unit variance, and are
uncorrelated under $P_{\alpha }$, 
\[
E_{P_{\alpha }}\left[ b_{a}^{2}\left( \mathbf{O};P_{\alpha }\right) \right]
=E_{P_{\alpha }}\left[ \left\{ O_{1}+O_{2}+\alpha O_{1}O_{2}\right\} ^{2}%
\right] =E_{P_{\alpha }}\left[ O_{1}^{2}+O_{2}^{2}+\alpha
^{2}(O_{1}O_{2})^{2}\right] =2+\alpha ^{2}.
\]%
Thus 
\[
var_{P_{\alpha }}\left[ \psi _{P,a}(\mathbf{V};\mathcal{G})\right]
=E_{P_{\alpha }}\left[ \frac{I_{a}(A)}{\pi _{a}^{2}(\mathbf{O}%
_{min};P_{\alpha })}var_{P_{\alpha }}(Y\mid A=a,\mathbf{O})\right] +2+\alpha
^{2}.
\]
Since by assumption $var_{P_{\alpha }}(Y\mid A=a,\mathbf{O})\leq C$ and $\pi
_{a}(\mathbf{O}_{min};P_{\alpha })\geq 1/C$, we have 
\begin{equation}
E_{P_{\alpha }}\left[ \frac{I_{a}(A)}{\pi _{a}^{2}(\mathbf{O}%
_{min};P_{\alpha })}var_{P_{\alpha }}(Y\mid A=a,\mathbf{O})\right] \leq
CE_{P_{\alpha }}\left[ \frac{I_{a}(A)}{\pi _{a}^{2}(\mathbf{O}%
_{min};P_{\alpha })}\right] \leq C^{3}.  \label{eq:are_bound_var}
\end{equation}%
Consequently,%
\begin{eqnarray*}
\frac{var_{P_{\alpha }}\left[ \Delta _{P_{\alpha }}\left( \mathbf{O}\right) %
\right] }{var_{P_{\alpha }}\left[ \psi _{P_{\alpha },a}(\mathbf{O};\mathcal{G%
})\right] } &=&\frac{\alpha ^{2}}{E_{P_{\alpha }}\left[ {I_{a}(A)}{\pi
_{a}^{-2}(\mathbf{O}_{min};P_{\alpha })}var_{P_{\alpha }}(Y\mid A=a,\mathbf{O}%
)\right] +2+\alpha ^{2}} \\
&\rightarrow &1
\end{eqnarray*}%
and therefore 
$$
\frac{var_{P_{\alpha }}\left[ \chi _{P_{\alpha},a,eff}^{1}(\mathbf{V};%
\mathcal{G})\right] }{var_{P_{\alpha }}\left[ \psi _{P_{\alpha },a}(\mathbf{O%
};\mathcal{G})\right] }=1-\frac{var_{P_{\alpha }}\left[ \Delta _{P_{\alpha
}}\left( \mathbf{O}\right) \right] }{var_{P_{\alpha }}\left[ \psi
_{P_{\alpha },a}(\mathbf{O};\mathcal{G})\right] }\rightarrow 0.
$$ 
\end{proof}

\medskip

\begin{lemma}\label{lemma:orth_decomp}
Let $\mathcal{G}$ be a DAG with vertex set that stands for a
random vector $\mathbf{V=}\left( V_{1},...,V_{s}\right) \mathbf{.}$ Suppose
that the laws in the Bayesian Network $\mathcal{M}\left( \mathcal{G}\right) $
are dominated by some measure $\mu .$ Then the tangent space of model $%
\mathcal{M}\left( \mathcal{G}\right) $ at a law $P$ is given by $\Lambda
\equiv \oplus _{j=1}^{s}\Lambda _{j}$ where 
\begin{equation}
\Lambda _{j}=\left\{ G\equiv g\left( V_{j},\pa_{\mathcal{G}}\left(
V_{j}\right) \right) \in L_{2}\left( P\right) :E_{P}\left[ G|\pa_{%
\mathcal{G}}\left( V_{j}\right) \right] =0\right\} .  \label{eq:tangent_j}
\end{equation}
\end{lemma}
\begin{proof}
For any $P\in \mathcal{M}\left( \mathcal{G}\right) $ let $p$
denote, any version of, the density of $P$ with respect to $\mu .$ For any $%
P\in $ $\mathcal{M}\left( \mathcal{G}\right) ,$ $p\left( \mathbf{V}\right) $ factors
as  
\[
p\left( \mathbf{V}\right) =\prod\limits_{k=1}^{s}p_{k}\left( V_{k}|\text{pa}_{\mathcal{G}%
}\left( V_{k}\right) \right) 
\]%
where $p_{j}$ is, any version of, the conditional density of $V_{j}$ given pa%
$_{\mathcal{G}}\left( V_{j}\right) .$ Lemma 1.6 of \cite{vanderlaan}, implies that the tangent space of model $\mathcal{M}\left( \mathcal{G%
}\right) $ at a law $P$ is given by $\Lambda \equiv \oplus _{j=1}^{s}\Lambda
_{j}$ where $\Lambda _{j}$ is the closed linear span of scores of one
dimensional regular parametric submodels 
$$
t\rightarrow p\left( \mathbf{V};t\right)
=p_{j}\left( V_{j}|\text{pa}_{\mathcal{G}}\left( V_{j}\right) ;t\right) \prod\limits_{k=1,k\not=j}^{s}p_{k}\left( V_{k}|\text{pa}_{\mathcal{G}}\left(
V_{k}\right) \right) .
$$
Such $\Lambda _{j}$ is equal to the set in the right
hand side of $\left( \ref{eq:tangent_j}\right) $ because model $\mathcal{M}%
\left( \mathcal{G}\right) $ does not impose restrictions on the law $%
p_{j}\left( V_{j}|\text{pa}_{\mathcal{G}}\left( V_{j}\right) \right) $
(\citet{tsiatis}, Theorem 4.5). This concludes the proof.
\end{proof}

\medskip

\begin{proof}[Proof of Theorem \protect\ref{lemma:eff_if_expression}]
\begin{equation*}
\psi _{P,a}\left[ \pa_{\mathcal{G}}(A);\mathcal{G} \right]= J_{P,a,\mathcal{G%
}}-\left\{ \frac{I_{a}(A)}{\pi_{a} \left( \text{pa}_{\mathcal{G}}\left(
A\right) ;P\right) }-1\right\} b_{a}\left( \text{pa}_{\mathcal{G}}\left(
A\right) ;P\right) -\chi_{a} \left( P;\mathcal{G}\right)
\end{equation*}
is an influence function for $\chi_{a} \left( P;\mathcal{G}\right) $ in
model $\mathcal{M}\left( \mathcal{G}\right) $ because it is the unique
influence function for $\chi_{a} \left( P;\mathcal{G}\right) $ in the
non-parametric model that does not impose any restrictions on $P$.

Let the vertex set of $\mathcal{G}$ be given by $\mathbf{V}=\left\lbrace
V_{1},\dots, V_{s}\right\rbrace$. In Lemma \ref{lemma:orth_decomp} we showed that for all $P\in 
\mathcal{M(G)}$ the tangent space at $P$ of model $\mathcal{M(G)}$ is given
by 
\begin{equation*}
\Lambda(P) = \oplus _{j=1}^{s}\Lambda_{j}(P),
\end{equation*}
where 
\begin{equation*}
\Lambda_{j}(P)\equiv \left\{ G\equiv g\left( V_{j},\text{pa}_{\mathcal{G}%
}\left( V_{j}\right) \right) \in L_{2}\left( P\right) :E_{P}\left[ G|\text{pa%
}_{\mathcal{G}}\left( V_{j}\right) \right] =0\right\} .
\end{equation*}%
Now, it is easy to show that the projection of any random variable $U$ onto $%
\Lambda_{j}(P)$ is given by 
\begin{equation*}
E_{P}\left[ U \mid V_{j}, \pa_{\mathcal{G}}(V_{j}) \right]-E_{P}\left[ U
\mid \pa_{\mathcal{G}}(V_{j}) \right]
\end{equation*}
and hence the projection of $U$ onto $\Lambda(P)$ is given by 
\begin{equation*}
\sum\limits_{j=1}^{s}\left\lbrace E_{P}\left[ U \mid V_{j}, \pa_{\mathcal{G}%
}(V_{j}) \right]-E_{P}\left[ U \mid \pa_{\mathcal{G}}(V_{j}) \right]
\right\rbrace.
\end{equation*}
Thus 
\begin{equation*}
\chi^{1}_{P,a,eff}(\mathbf{V};\mathcal{G})=\sum\limits_{j=1}^{s}\left\lbrace
E_{P}\left[ \psi _{P,a}\left[ \pa_{\mathcal{G}}(A);\mathcal{G} \right] \mid
V_{j}, \pa_{\mathcal{G}}(V_{j}) \right]-E_{P}\left[ \psi _{P,a}\left[ \pa_{%
\mathcal{G}}(A);\mathcal{G} \right] \mid \pa_{\mathcal{G}}(V_{j}) \right]
\right\rbrace.
\end{equation*}

Because $\chi_{a}\left( P;\mathcal{G}\right) $ does not depend on the law of 
$A$ given pa$_{\mathcal{G}}\left( A\right)$, $\psi _{P,a}\left[ \pa_{%
\mathcal{G}}(A);\mathcal{G} \right]$ is orthogonal to the scores for all
regular parametric submodels for the law $A$ given pa$_{\mathcal{G}}\left(
A\right) .$ Consequently, $E_{P}\left\lbrace \left. \psi _{P,a}\left[ \pa_{%
\mathcal{G}}(A);\mathcal{G} \right]\right\vert A,\text{pa}_{\mathcal{G}%
}\left( A\right) \right\rbrace -E_{P}\left\lbrace \left. \psi _{P,a}\left[ %
\pa_{\mathcal{G}}(A);\mathcal{G} \right]\right\vert \text{pa}_{\mathcal{G}%
}\left( A\right) \right\rbrace =0$ for all $P\in \mathcal{M}\left( \mathcal{G%
}\right)$. This implies that 
\begin{equation*}
\chi^{1}_{P,a,eff}(\mathbf{V};\mathcal{G})=\sum\limits_{j: V_{j}\neq
A}\left\lbrace E_{P}\left[ \psi _{P,a}\left[ \pa_{\mathcal{G}}(A);\mathcal{G}
\right] \mid V_{j}, \pa_{\mathcal{G}}(V_{j}) \right]-E_{P}\left[ \psi _{P,a}%
\left[ \pa_{\mathcal{G}}(A);\mathcal{G} \right] \mid \pa_{\mathcal{G}%
}(V_{j}) \right] \right\rbrace.
\end{equation*}

Now consider any $V_{j}\neq A.$ We will show that 
\begin{align}
&E_{P}\left[ \psi _{P,a}\left[ \pa_{\mathcal{G}}(A);\mathcal{G} \right] \mid
V_{j}, \pa_{\mathcal{G}}(V_{j}) \right]-E_{P}\left[ \psi _{P,a}\left[ \pa_{%
\mathcal{G}}(A);\mathcal{G} \right] \mid \pa_{\mathcal{G}}(V_{j}) \right]
\nonumber
\\
&=E_{P}\left[ \left. J_{P,a,\mathcal{G}}\right\vert V_{j},\text{pa}_{\mathcal{%
G}}\left( V_{j}\right) \right] -E_{P}\left[ \left. J_{P,a \mathcal{G}%
}\right\vert \text{pa}_{\mathcal{G}}\left( V_{j}\right) \right] .
\label{eq:solo_J}
\end{align}%
Suppose first that $V_{j}\in $de$_{\mathcal{G}}^{c}\left( A\right) ,$ then 
\begin{eqnarray*}
&&E_{P}\left[ \left. \left\{ \frac{I_{a}(A)}{\pi_{a} \left( \text{pa}_{%
\mathcal{G}}\left( A\right) ;P\right) }-1\right\} b_{a}\left( \text{pa}_{%
\mathcal{G}}\left( A\right) ;P\right) \right\vert V_{j},\text{pa}_{\mathcal{G%
}}\left( V_{j}\right) \right] \\
&=&E_{P}\left[ \left. \left\{ \frac{E_{P}\left[ I_{a}(A)|\text{pa}_{\mathcal{%
G}}\left( A\right) ,V_{j},\text{pa}_{\mathcal{G}}\left( V_{j}\right) \right] 
}{\pi_{a} \left( \text{pa}_{\mathcal{G}}\left( A\right) ;P\right) }%
-1\right\} b_{a}\left( \text{pa}_{\mathcal{G}}\left( A\right) ;P\right)
\right\vert V_{j},\text{pa}_{\mathcal{G}}\left( V_{j}\right) \right] \\
&=&E_{P}\left[ \left. \left\{ \frac{E_{P}\left[ I_{a}(A)|\text{pa}_{\mathcal{%
G}}\left( A\right) \right] }{\pi_{a} \left( \text{pa}_{\mathcal{G}}\left(
A\right) ;P\right) }-1\right\} b_{a}\left( \text{pa}_{\mathcal{G}}\left(
A\right) ;P\right) \right\vert V_{j},\text{pa}_{\mathcal{G}}\left(
V_{j}\right) \right] \\
&=&0.
\end{eqnarray*}%
where the second equality holds because $A\perp \!\!\!\perp \left[
\left[ V_{j},\text{pa}_{\mathcal{G}}\left( V_{j}\right) \right] \backslash 
\text{pa}_{\mathcal{G}}\left( A\right) \right] |$pa$_{\mathcal{G}}\left(
A\right) $ by the Local Markov property since $\lbrace V_{j}\rbrace \cup \pa_{\mathcal{G}%
}\left( V_{j}\right) \subset \de_{\mathcal{G}}^{c}\left( A\right) .$ The last display implies that
\begin{equation*}
E_{P}\left[ \left. \left\{ \frac{I_{a}(A)}{\pi_{a} \left( \text{pa}_{\mathcal{G}%
}\left( A\right) ;P\right) }-1\right\} b_{a}\left( \text{pa}_{\mathcal{G}%
}\left( A\right) ;P\right) \right\vert \text{pa}_{\mathcal{G}}\left(
V_{j}\right) \right] =0,
\end{equation*}
thus showing \eqref{eq:solo_J}. Next, suppose that $V_{j}\in $de$_{\mathcal{G%
}}\left( A\right) .$ Then, $\lbrace A \rbrace \cup $pa$_{\mathcal{G}}\left( A\right) \subset $%
de$_{\mathcal{G}}^{c}\left( V_{j}\right) .$ Consequently, by the Local
Markov property, $\left[ A,\text{pa}_{\mathcal{G}}\left( A\right) \right]
\perp \!\!\!\perp V_{j}|$pa$_{\mathcal{G}}\left( V_{j}\right)
. $ Then \eqref{eq:solo_J} holds because 
$$
 \left\{ \frac{I_{a}(A)}{\pi_{a} \left( \text{pa}_{\mathcal{G}%
}\left( A\right) ;P\right) }-1\right\} b_{a}\left( \text{pa}_{\mathcal{G}%
}\left( A\right) ;P\right)
$$
is a function of $A$ and $\pa_{\mathcal{G}}\left(A\right) $ only.

We have thus shown that 
\begin{equation}
\chi _{P,a,eff}^{1}\left( \mathbf{V};\mathcal{G}\right) =\sum_{j:V_{j}\neq
A}\left\{ E_{P}\left[ J_{P,a,\mathcal{G}}|V_{j},\pa_{\mathcal{G}}\left(
V_{j}\right) \right] -E_{P}\left[ J_{P,a,\mathcal{G}}|\pa_{\mathcal{G}%
}\left( V_{j}\right) \right] \right\} .  \notag
\end{equation}%
By Proposition \ref{prop:remove_directed_through_A} in Section \ref{sec:aux_eff}, if $V_{j}\in \indir(A,Y,%
\mathcal{G})$ then 
\begin{equation*}
E_{P} \left[ J_{P,a,\mathcal{G}}|V_{j},\pa_{\mathcal{G}}\left( V_{j}\right) %
\right] -E_{P}\left[ J_{P,a,\mathcal{G}}|\pa_{\mathcal{G}}\left(
V_{j}\right) \right]=0.
\end{equation*}
Next, take $V_{j}\in\an^{c}_{\mathcal{G}}(\lbrace A,Y\rbrace)$. Then $\pa_{%
\mathcal{G}}(A), A, Y$ are non-descendants of $V_{j}$ and thus by the Local
Markov Property 
\begin{equation*}
V_{j} \perp \!\!\!\perp \pa_{\mathcal{G}}(A), A, Y \mid \pa_{\mathcal{G}%
}(V_{j}).
\end{equation*}
Therefore, since $J_{P,a,\mathcal{G}}$ is a function of only $\pa_{\mathcal{G%
}}(A), A, Y$ 
\begin{equation*}
E_{P} \left[ J_{P,a,\mathcal{G}}|V_{j},\pa_{\mathcal{G}}\left( V_{j}\right) %
\right] -E_{P}\left[ J_{P,a,\mathcal{G}}|\pa_{\mathcal{G}}\left(
V_{j}\right) \right]=0.
\end{equation*}
Hence 
\begin{equation}
\chi _{P,a,eff}^{1}\left( \mathbf{V};\mathcal{G}\right) =\sum_{j:V_{j}\notin %
\irrel(A,Y,\mathcal{G})\cup \lbrace A\rbrace}\left\{ E_{P}\left[ J_{P,a,%
\mathcal{G}}|V_{j},\pa_{\mathcal{G}}\left( V_{j}\right) \right] -E_{P}\left[
J_{P,a,\mathcal{G}}|\pa_{\mathcal{G}}\left( V_{j}\right) \right] \right\} . 
\notag
\end{equation}%

Turn now to the proof that $\chi _{P,a,eff}^{1}(%
\mathbf{V};\mathcal{G})$ does not depend on any $V\in \irrel(A,Y,\mathcal{G})$. 
Take $V\in \irrel(A,Y,\mathcal{G})$ and $W\in \ch_{\mathcal{G}}(V)\setminus
\lbrace A \rbrace$. We will show next that $W\in \irrel(A,Y,\mathcal{G})$.
This, together with the last display, will imply that $\chi _{P,a,eff}^{1}(%
\mathbf{V};\mathcal{G})$ is a function only of $\mathbf{V}_{marg}=\mathbf{V}%
\backslash \left\{ \an_{\mathcal{G}}^{c}\left( \left\{ A,Y\right\} \right)
\cup \indir(A,Y,\mathcal{G}) \right\}$. This is because the only way in which $V\in\irrel(A,Y,\mathcal{G})$ can appear in  $\chi _{P,a,eff}^{1}(%
\mathbf{V};\mathcal{G})$ is if it belongs to the parent set of a node $W$ that is not in $\irrel(A,Y,\mathcal{G})\cup\lbrace A \rbrace$.

Now, if $W\in \an_{\mathcal{G}%
}(A)\setminus \lbrace A \rbrace$ then $W\in \indir(A,Y,\mathcal{G})$, since $%
W$ is a children of $V$ and $V\in \irrel(A,Y,\mathcal{G})$. If $W\notin \an_{%
\mathcal{G}}(A)$ then if $W$ where an ancestor of $Y$, there would exist a directed path from $W$ to $Y$ that does not intersect $A$. Since $W$ is a child of $V$, this would imply that $V\notin \irrel(A,Y,{\mathcal{G}})$, contradicting the assumption that $V\in \irrel(A,Y,{\mathcal{G}})$. Hence, $W$ is not an ancestor of $Y$ nor of $A$, thus $W\in\irrel(A,Y,\mathcal{G})$.
Thus, in all cases, $W\in \irrel(A,Y,
\mathcal{G})$, which is what we wanted to show.
\end{proof}

\medskip

\begin{proof}[Proof of Proposition \protect\ref{lemma:marg-eff}]
Let $P^{\prime }\in \mathcal{M}^{\prime }$ and $P\in \mathcal{M}$ with
marginal law $P^{\prime }.$ Let $\mathbf{V}^{c}\mathbf{=V\backslash V}%
^{\prime }.$ Let $t\in [0,\varepsilon ]\rightarrow P_{t}$ be a regular
parametric submodel of $\mathcal{M}$ with $P_{t=0}=P$ and score $S.$
Decompose $S$ as $S_{\mathbf{V}^{\prime }}+S_{\mathbf{V}^{c}|\mathbf{V}%
^{\prime }}$ where $S_{\mathbf{V}^{\prime }}$ is the score in the induced
regular parametric submodel $t\in (0,\varepsilon ]\rightarrow P_{t}^{\prime
} $ of $\mathcal{M}^{\prime }$ with $P_{t=0}^{\prime }=P^{\prime }.$ Then%
\begin{eqnarray*}
\left. \frac{d}{dt}\chi \left( P_{t}\right) \right\vert _{t=0} &=&E_{P}\left[
\chi _{P,eff}^{1}S\right] \\
&=&E_{P}\left[ \chi _{P,eff}^{1}S_{\mathbf{V}^{\prime }}\right] +E_{P}\left[
\chi _{P,eff}^{1}S_{\mathbf{V}^{c}|\mathbf{V}^{\prime }}\right] \\
&=&E_{P}\left[ \chi _{P,eff}^{1}S_{\mathbf{V}^{\prime }}\right]
\end{eqnarray*}%
where the last equality follows because $S_{\mathbf{V}^{c}|\mathbf{V}%
^{\prime }}$ is a conditional score for the law of $\mathbf{V}^{c}|\mathbf{V}%
^{\prime }$ and, by assumption, $\chi _{P,eff}^{1}$ is a function of $%
\mathbf{V}^{\prime }$ only. On the other hand, $\left. \frac{d}{dt}\chi
\left( P_{t}\right) \right\vert _{t=0}=\left. \frac{d}{dt}\nu \left(
P_{t}^{\prime }\right) \right\vert _{t=0}$ because by assumption, $\chi
\left( P_{t}\right) =\nu \left( P_{t}^{\prime }\right) .$ Then, $\chi
_{P,eff}^{1}$ is an influence function for $\nu \left( P^{\prime }\right) .$
Now let $\Lambda ^{\prime }$ be the tangent space for model $\mathcal{M}%
^{\prime }$ at $P^{\prime }.$ Then, $\Lambda =\Lambda ^{\prime }\oplus $ $\ $%
the closed linear span of $\left\{ S_{\mathbf{V}^{c}|\mathbf{V}^{\prime
}}:S_{\mathbf{V}^{c}|\mathbf{V}^{\prime }}\text{ is a conditional score under model } \mathcal{M}%
\right\} .$ Since $E_{P}\left[ \chi _{P,eff}^{1}S_{\mathbf{V}^{c}|\mathbf{V}%
^{\prime }}\right] =0$ for all conditional scores $S_{\mathbf{V}^{c}|\mathbf{%
V}^{\prime }}$ we conclude that $\chi _{P,eff}^{1}$ is in $\Lambda ^{\prime
} $ and consequently, it is the efficient influence function $\nu
_{P^{\prime },eff}^{1}.$
\end{proof}

\medskip

\begin{proof}[Proof of Lemma \protect\ref{lemma:correct_algo_prun_indir}]
We will use the following property which can be shown straightforwardly. Let 
$\mathcal{G}^{1},$ $\mathcal{G}^{2}$ and $\mathcal{G}^{3}$ be DAGs with
vertex sets $\mathbf{V}^{1}, \mathbf{V}^{2}$ and $\mathbf{V}^{3}$ such that $%
\mathbf{V}^{1}\supset \mathbf{V}^{2}\supset\mathbf{V}^{3}$. Then, 
\begin{equation}
\mathcal{M}\left( \mathcal{G}^{1},\mathbf{V}^{2}\right) =\mathcal{M}\left( 
\mathcal{G}^{2}\right) \text{ and }\mathcal{M}\left( \mathcal{G}^{2},\mathbf{%
V}^{3}\right) =\mathcal{M}\left( \mathcal{G}^{3}\right) \Rightarrow \mathcal{%
M}\left( \mathcal{G}^{1},\mathbf{V}^{3}\right) =\mathcal{M}\left( \mathcal{G}%
^{3 }\right)  \label{eq:p1}
\end{equation}%
The set $\mathbf{V\backslash }\an_{\mathcal{G}}^{c}\left( \left\{
A,Y\right\} \right) $ is an ancestral set, that is, it contains all its own
ancestors: 
\begin{equation*}
\mathbf{V\backslash }\an_{\mathcal{G}}^{c}\left( \left\{ A,Y\right\} \right)
=\an_{\mathcal{G}}\left( \mathbf{V\backslash }\an_{\mathcal{G}}^{c}\left(
\left\{ A,Y\right\} \right) \right) .
\end{equation*}%
Then, by Proposition 1 (a) of \cite{evans-marginal} 
\begin{equation}
\mathcal{M}\left( \mathcal{G},\mathbf{V}\backslash \an_{\mathcal{G}%
}^{c}\left( \left\{ A,Y\right\} \right) \right) =\mathcal{M}\left( \mathcal{G%
}_{\mathbf{V}\backslash \an_{\mathcal{G}}^{c}\left( \left\{ A,Y\right\}
\right) }\right) .  \label{eq:p2}
\end{equation}%

Now, let $\widetilde{\mathcal{G}}^{l+1}\equiv \mathcal{G}_{\mathbf{V}%
\backslash \an_{\mathcal{G}}^{c}\left( \left\{ A,Y\right\} \right) }$ and 
let $(I_{1},\dots,I_{l})$ be the set of nodes in $\indir(A,Y,\widetilde{\mathcal{G}}^{l+1})$, topologically sorted with respect to $\widetilde{\mathcal{G}}^{l+1}$. Recursively define for $j=l,l-1,\dots,1,$ $\widetilde{\mathcal{G}}^{j}\equiv
\tau \left( \widetilde{\mathcal{G}}^{j+1},I_{j}\right) .$ Noticing that in $%
\widetilde{\mathcal{G}}^{j+1},$ $I_{j}$ has a sole child equal to $A$, then
combining Lemma 1 and Lemma 3 of \cite{evans-marginal}, yields that for $%
j=l,l-1,\dots,1$ , 
\begin{equation}
\mathcal{M}\left( \widetilde{\mathcal{G}}^{j+1},\mathbf{V}\backslash \left\{ %
\an^{c}_{\mathcal{G}}\left( \left\{ A,Y\right\} \right) \cup \left( \cup
_{i=j}^{l}I_{i}\right) \right\} \right) =\mathcal{M}\left( \widetilde{%
\mathcal{G}}^{j}\right) .  \label{eq:p3}
\end{equation}%
Repeatedly invoking \eqref{eq:p1} to the equalities \eqref{eq:p2} and %
\eqref{eq:p3} yields 
\begin{equation*}
\mathcal{M}\left( \mathcal{G},\mathbf{V}\setminus\lbrace \an_{\mathcal{G}%
}^{c}(\lbrace A, Y\rbrace)\cup\indir(A,Y,\mathcal{G}) \rbrace \right) =%
\mathcal{M}\left( \widetilde{\mathcal{G}}^{1}\right).
\end{equation*}
Since $\mathcal{G}^{\prime}=\widetilde{\mathcal{G}}^{1}$ is the output of
Algorithm \ref{algo:prun_indir}, this finishes the proof of the first part
of the Lemma.

Now note that the pruning algorithm prunes neither $A$ nor $Y.$
Furthermore, it neither adds new causal paths nor deletes causal paths
between $A$ and $Y.$ Then, $\cn(A,Y,\mathcal{G})=\cn(A,Y,\mathcal{G}^{\prime
}).$ Also, the pruning algorithm neither adds nor deletes any vertex that is both a
non-descendant of $A$ in $\mathcal{G}$ and parent of a vertex in $\cn(A,Y,%
\mathcal{G})$ in $\mathcal{G}.$ But the set of such vertices is precisely
the set $\mathbf{O}\left( A,Y,\mathcal{G}\right) .$ This shows that $\mathbf{%
O}\left( A,Y,\mathcal{G}\right) =\mathbf{O}\left( A,Y,\mathcal{G}^{\prime
}\right) .$ Then, if $P\in\mathcal{M(G)}$, $b_{a}\left( \mathbf{O}\left( A,Y,%
\mathcal{G}\right) ;P\right) =b_{a}\left( \mathbf{O}\left( A,Y,\mathcal{G}%
^{\prime }\right) ;P_{marg}\right) $ and $\pi_{a} \left( \mathbf{O}\left(
A,Y,\mathcal{G}\right) ;P\right) =\pi_{a} \left( \mathbf{O}\left( A,Y,%
\mathcal{G}^{\prime }\right) ;P_{marg}\right) .$ Consequently, 
\begin{equation*}
\psi _{P,a}\left[ \mathbf{O}\left( A,Y,\mathcal{G}\right); \mathcal{G} %
\right] =\psi _{P_{marg},a}\left[ \mathbf{O}\left( A,Y,\mathcal{G}^{\prime
}\right) ; \mathcal{G}^{\prime}\right] .
\end{equation*}%
But since $\mathbf{O}\left( A,Y,%
\mathcal{G}\right) $ is an adjustment set relative to $A$ and $Y$ in $%
\mathcal{G}$ (and $\mathcal{G}^{\prime }$) we have that 
$$
\chi_{a} \left( P;%
\mathcal{G}\right) =E_{P}\left[ b_{a}\left( \mathbf{O}\left( A,Y,\mathcal{G}%
\right) ;P\right) \right]  \text{ and } \chi_{a} \left( P_{marg};\mathcal{G}%
^{\prime }\right) =E_{P_{marg}}\left[ b_{a}\left( \mathbf{O}\left( A,Y,%
\mathcal{G}^{\prime }\right) ;P_{marg}\right) \right] 
$$
and thus conclude
that $\chi_{a} \left( P;\mathcal{G}\right) =\chi_{a} \left( P_{marg};%
\mathcal{G}^{\prime }\right)$.

We turn next to the proof of $\chi _{P,a,eff}^{1}(\mathbf{V};\mathcal{G}%
)=\chi _{P_{marg},eff}^{1}(\mathbf{V}_{marg};\mathcal{G}^{\prime}).$ 
By Theorem \ref{lemma:eff_if_expression}, $\chi _{P,a,eff}^{1}(\mathbf{V};%
\mathcal{G})$ is a function only of $\mathbf{V}_{marg}=\mathbf{V}\backslash
\left\{ \an_{\mathcal{G}}^{c}\left( \left\{ A,Y\right\} \right) \cup \indir%
(A,Y,\mathcal{G}) \right\}$. Since we have already shown that $\mathcal{M}%
\left( \mathcal{G},\text{ }\mathbf{V}_{marg}\right) =\mathcal{M}\left( 
\mathcal{G}^{\prime }\right)$, that $\chi_{a} \left( P_{marg};\mathcal{G}%
^{\prime }\right) =\chi_{a} \left( P;\mathcal{G}\right) $, Proposition \ref%
{lemma:marg-eff} implies that $\chi _{P,a,eff}^{1}(\mathbf{V};\mathcal{G}%
)=\chi _{P_{marg},a,eff}^{1}(\mathbf{V}_{marg};\mathcal{G}^{\prime}).$
\end{proof}

\medskip

\begin{proof}[Proof of Lemma \ref{lemma:eff_simple}]
We begin with the proof of part 1).
\begin{align*}
    &E_{P}\left[ \frac{I_{a}(A)Y}{\pi_{a}(\pa_{\mathcal{G}}(A);P)}\mid W_{j}, \pa_{\mathcal{G}}(W_{j})\right]
    \\
    &=E_{P}\left[ \frac{I_{a}(A)E_{P}\left[ Y\mid A=a, W_{j},\pa_{\mathcal{G}}(W_{j}), \mathbf{O}, \pa_{\mathcal{G}}(A) \right]}{\pi_{a}(\pa_{\mathcal{G}}(A);P)}\mid W_{j}, \pa_{\mathcal{G}}(W_{j})\right]
    \\
    &=E_{P}\left[ \frac{I_{a}(A)E_{P}\left[ Y\mid A=a, \mathbf{O}\right]}{\pi_{a}(\pa_{\mathcal{G}}(A);P)}\mid W_{j}, \pa_{\mathcal{G}}(W_{j})\right]
    \\
    &=E_{P}\left[ E_{P}\left[ Y\mid A=a, \mathbf{O}\right]\frac{E_{P}\left[ I_{a}(A)\mid  \mathbf{O},W_{j}, \pa_{\mathcal{G}}(W_{j}),\pa_{\mathcal{G}}(A)) \right]}{\pi_{a}(\pa_{\mathcal{G}}(A);P)}\mid W_{j}, \pa_{\mathcal{G}}(W_{j})\right]
    \\
    &=E_{P}\left[ E_{P}\left[ Y\mid A=a, \mathbf{O}\right]\mid W_{j}, \pa_{\mathcal{G}}(W_{j})\right]
    \\
    &=E_{P}\left[b_{a}(\mathbf{O};P)\mid W_{j}, \pa_{\mathcal{G}}(W_{j})\right],
\end{align*}
where the second equality holds because 
$$
Y\ort_{\mathcal{G}}  \left[\lbrace W_{j}\rbrace \cup \pa_{\mathcal{G}}(W_{j}) \cup \pa_{\mathcal{G}}(A)\right]\setminus \mathbf O\mid \mathbf{O}, A
$$
and the third equality holds because the set
$
\left[\lbrace W_{j}\rbrace \cup \pa_{\mathcal{G}}(W_{j}) \cup \mathbf{O}\right]
$
is comprised of non-descendants of $A$ and hence by the Local Markov Property
$$
A \ort_{\mathcal{G}} \left[\lbrace W_{j}\rbrace \cup \pa_{\mathcal{G}}(W_{j}) \cup \mathbf{O}\right]\setminus \pa_{\mathcal{G}}(A) \mid \pa_{\mathcal{G}}(A).
$$

Next, we prove part 2). First note that%
\begin{eqnarray}
E_{P}\left[ J_{P,a,\mathcal{G}}|M_{k},\text{pa}_{\mathcal{G}}\left(
M_{k}\right) \right]  &=&E_{P}\left[ I_{a}(A)Y\left. E_{P}\left[ \left. 
\frac{1}{\pi_{a}(\pa_{\mathcal{G}}(A);P)}%
\right\vert A=a,\mathbf{O},Y,M_{k},\text{pa}_{\mathcal{G}}\left(
M_{k}\right) \right] \right\vert M_{k},\text{pa}_{\mathcal{G}}\left(
M_{k}\right) \right]   \notag\\
&=&E_{P}\left[ I_{a}(A)Y\left. E_{P}\left[ \left. \frac{1}{\pi_{a}(\pa_{\mathcal{G}}(A);P)}\right\vert A=a, \mathbf{O}\right]
\right\vert M_{k},\text{pa}_{\mathcal{G}}\left( M_{k}\right) \right]   \notag
\\
&=&E_{P}\left[ \left. \frac{I_{a}(A)}{\pi_{a}(\mathbf{O}_{min};P)}%
Y\right\vert M_{k},\text{pa}_{\mathcal{G}}\left( M_{k}\right) \right]  
\notag
\\
&=&E_{P}\left[ \left. T_{P,a,\mathcal{G}} \right\vert M_{k},\text{pa}_{\mathcal{G}}\left( M_{k}\right) \right]  
\notag
\end{eqnarray}%
where the second equality follows from 
\begin{equation}
\left( Y,\mathbf{M}\right) \perp \!\!\!\perp _{\mathcal{G}}\text{pa}_{%
\mathcal{G}}\left( A\right) \backslash \mathbf{O}\mid \left[ \mathbf{O}\cup
\left\{ A\right\} \right]   \label{eq:indep_med_Y}
\end{equation}%
and the fact that for any $k,$ pa$_{\mathcal{G}}\left( M_{k}\right) \subset 
\mathbf{M\cup }\left\{ A\right\} \mathbf{\cup O,}$ and the third equality
follows because 
\begin{equation}
E_{P}\left[ \left. \frac{1}{\pi_{a}(\pa_{\mathcal{G}}(A);P) }\right\vert A=a,\mathbf{O}\right] =\frac{1}{\pi_{a}(\mathbf{O}_{min};P)}  \nonumber
\end{equation}%
which is a consequence of Lemma \ref{lemma:inv_pi} in Section \ref{sec:aux_res} and the definition of 
$\mathbf{O}_{min}$. This finishes the proof of part 2).

Turn now to the proof of part 3). 
\begin{eqnarray}
E_{P}\left[ J_{P,a,\mathcal{G}}|Y,\text{pa}_{\mathcal{G}}\left(
Y\right) \right]  &=& E_{P}\left[ I_{a}(A)Y\left. E_{P}\left[ \left. 
\frac{1}{\pi_{a}(\pa_{\mathcal{G}}(A);P)}%
\right\vert A=a,\mathbf{O},Y,\text{pa}_{\mathcal{G}}\left(
Y\right) \right] \right\vert Y,\text{pa}_{\mathcal{G}}\left(
Y\right) \right]   \notag\\
&=&E_{P}\left[ I_{a}(A)Y\left. E_{P}\left[ \left. \frac{1}{\pi_{a}(\pa_{\mathcal{G}}(A);P)}\right\vert A=a, \mathbf{O}\right]
\right\vert Y,\text{pa}_{\mathcal{G}}\left( Y\right) \right]   \notag
\\
&=&E_{P}\left[ \left. \frac{I_{a}(A)}{\pi_{a}(\mathbf{O}_{min};P)}%
Y\right\vert Y,\text{pa}_{\mathcal{G}}\left( Y\right) \right]  
\notag
\\
&=&E_{P}\left[ \left. T_{P,a,\mathcal{G}} \right\vert Y,\text{pa}_{\mathcal{G}}\left(Y\right) \right]  
\notag
\end{eqnarray}%
where the second equality follows again from \eqref{eq:indep_med_Y} and the 
fact that $\pa_{\mathcal{G}}\left( Y\right) \subset 
\mathbf{M\cup }\left\{ A\right\} \mathbf{\cup O,}$
and third equality follows from Lemma \ref{lemma:inv_pi} and the definition of 
$\mathbf{O}_{min}$. This concludes the proof of the theorem.

\end{proof}

\begin{proof}[Proof of Theorem \ref{theo:eff_simple}]

Because $\irrel(A,Y,\mathcal{G})=\emptyset$,  we can partition the node set $\mathbf{V}$\textbf{\ }%
of $\mathcal{G}$ as $\mathbf{M}\cup \mathbf{W\cup }\left\{ A,Y\right\} $
where the vertices in $\mathbf{M}$ intersect at least one causal path
between $A$ and $Y$, that is, $\mathbf{M}$ is the set of mediators in the causal
pathways between $A$ and $Y,$ and $\mathbf{W}$ are non-descendants of $A.$
We can therefore sort topologically $\mathbf{V}$ as $\left(
W_{1},\dots,W_{J},A,M_{1},\dots,M_{K},Y\right)$, where the set $\mathbf{W}=\empty$ if $J=0$ and the set $\mathbf{K}=\empty$ if $K=0$.

By Theorem \ref{lemma:eff_if_expression}, 
\begin{align*}
\chi _{P,a,eff}^{1}\left( \mathbf{V};\mathcal{G}\right) &=\sum_{j:V_{j}\notin %
\left[ \irrel\left( A,Y,\mathcal{G}\right) \cup \left\{ A\right\} \right]}\left\{ E_{P}%
\left[ J_{P,a,\mathcal{G}}|V_{j},\pa_{\mathcal{G}}\left( V_{j}\right) \right]
-E_{P}\left[ J_{P,a,\mathcal{G}}|\pa_{\mathcal{G}}\left( V_{j}\right) \right]
\right\} \\
&=E_{P}\left[
J_{P,a,\mathcal{G}}|Y,\pa_{\mathcal{G}}\left( Y\right) \right] -E_{P}\left[
J_{P,a,\mathcal{G}}|\pa_{\mathcal{G}}\left( Y\right) \right]
\label{eq:formula_eff} \\
&+\sum_{k=1}^{K}\left\{ E_{P}\left[ J_{P,a,\mathcal{G}}|M_{k},\pa_{\mathcal{%
G}}\left( M_{k}\right) \right] -E_{P}\left[ J_{P,a,\mathcal{G}}|\pa_{%
\mathcal{G}}\left( M_{k}\right) \right] \right\}  \notag \\
&+\sum_{j=1}^{J}\left\{ E_{P}\left[ J_{P,a,\mathcal{G}}|W_{j},\pa_{\mathcal{%
G}}\left( W_{j}\right) \right] -E_{P}\left[ J_{P,a,\mathcal{G}}|\pa_{%
\mathcal{G}}\left( W_{j}\right) \right] \right\} 
\end{align*}
where we make the conventions that
$$
\sum\limits_{k=1}^{0} \cdot \equiv 0, \quad \sum\limits_{j=1}^{0} \cdot \equiv 0.
$$
Next, using Lemma \ref{lemma:eff_simple},
\begin{align*}
\chi _{P,a,eff}^{1}\left( \mathbf{V};\mathcal{G}\right) 
&=E_{P}\left[
T_{P,a,\mathcal{G}}|Y,\pa_{\mathcal{G}}\left( Y\right) \right] -E_{P}\left[
T_{P,a,\mathcal{G}}|\pa_{\mathcal{G}}\left( Y\right) \right]
\nonumber \\
&+\sum_{k=1}^{K}\left\{ E_{P}\left[ T_{P,a,\mathcal{G}}|M_{k},\pa_{\mathcal{%
G}}\left( M_{k}\right) \right] -E_{P}\left[ T_{P,a,\mathcal{G}}|\pa_{%
\mathcal{G}}\left( M_{k}\right) \right] \right\}  \notag \\
&+\sum_{j=1}^{J}\left\{ E_{P}\left[ b_{a}(\mathbf{O};P)|W_{j},\pa_{\mathcal{%
G}}\left( W_{j}\right) \right] -E_{P}\left[ b_{a}(\mathbf{O};P)|\pa_{%
\mathcal{G}}\left( W_{j}\right) \right] \right\}.
\end{align*}

This concludes the proof of the theorem.
\end{proof}

In what follows let 
\begin{equation}
q_{\mathcal{G}}\left( \mathbf{W};P\right) \mathbf{\equiv }%
\sum_{j=1}^{J}\left\{ E_{P}\left[b_{a}\left( \mathbf{O};P\right) \mid W_{j},\pa_{\mathcal{G}}(W_j)\right] -E_{P}\left[ b_{a}\left( \mathbf{O};P\right)\mid \pa_{\mathcal{G}}(W_j) \right] \right\}   \label{eq:qz}
\end{equation}
and
\begin{eqnarray}
h_{\mathcal{G}}\left( A,\mathbf{O,M,}Y;P\right) &\mathbf{\equiv }%
&\sum_{j=1}^{K}\left\{ E_{P}\left[ T_{P,a,\mathcal{G}}|M_{j},\text{pa}_{%
\mathcal{G}}\left( M_{j}\right) \right] -E_{P}\left[ T_{P,a,\mathcal{G}}|\text{%
pa}_{\mathcal{G}}\left( M_{j}\right) \right] \right\}  \label{eq:hm} \\
&&+\left\{ E_{P}\left[ T_{P,a,\mathcal{G}}|Y,\text{pa}_{\mathcal{G}}\left(
Y\right) \right] -E_{P}\left[ T_{P,a,\mathcal{G}}|\text{pa}_{\mathcal{G}%
}\left( Y\right) \right] \right\} .  \notag
\end{eqnarray}%
Note that by Theorem \ref{theo:eff_simple}, 
$$
\chi^{1}_{P,a,eff}(\mathbf{V};\mathcal{G})=q_{\mathcal{G}}\left( \mathbf{W};P\right)+h_{\mathcal{G}}\left( A,\mathbf{O,M,}Y;P\right).
$$

\begin{proof}[Proof of Theorem \ref{theo:main}]

The assertion that if Algorithm \ref{algo:main} exits with output \textup{\texttt{efficient=True}} then
$$
\chi _{P,a,eff}^{1}\left( \mathbf{V};%
\mathcal{G}\right) =\psi _{P,a}\left( \mathbf{O};\mathcal{G}\right) \quad \text{for all} \quad P\in\mathcal{M}(\mathcal{G})
$$
was proved in the discussion preceding Theorem \ref{theo:main}. Here, we prove that if Algorithm \ref{algo:main} exits with output \texttt{efficient=False} then there exists $P^{\ast}\in\mathcal{M}(\mathcal{G})$ 
$$
\chi _{P^{\ast},a,eff}^{1}\left( \mathbf{V};%
\mathcal{G}\right) \neq \psi _{P^{\ast},a}\left( \mathbf{O};\mathcal{G}\right) .
$$

Assume first that the algorithm exits with output \texttt{efficient=False} because $\mathbf{O}\setminus O_{T} \not\subset \pa_{\mathcal{G}}(O_{T})$. This can only occur if $J>1$. By Lemma \ref{lemma:non_desc} in Section \ref{sec:aux_eff} we have that $O_{T}=W_{J}$. 
Then, since $W_{J}$  appears only in the
term $E_{P}\left[ b_{a}\left( \mathbf{O};P\right) |W_{J},\text{pa}_{\mathcal{G}%
}\left( W_{J}\right) \right] $ of $q_{\mathcal{G}}\left( \mathbf{W};P\right)
,$ we conclude that $q_{\mathcal{G}}\left( \mathbf{W};P\right) =g_{1}\left[
W_{J},\text{pa}_{\mathcal{G}}\left( W_{J}\right) \right] +g_{2}\left( 
\mathbf{W}\backslash W_{J}\right) $ for some functions $g_{1}$ and $g_{2}.$
This implies that $q_{\mathcal{G}}\left( \mathbf{W};P\right) $ cannot be equal
to, for instance, $b^{\ast }\left( \mathbf{O}\right) +g_{2}\left( \mathbf{W}%
\backslash O_{T}\right) $ for $b^{\ast }\left( \mathbf{O}\right)
=O_{1}\times \dots \times O_{T}.$ By Lemma \ref{lemma:unrestricted} in Section \ref{sec:aux_eff} we can find $P^{\ast}\in\mathcal{M}(\mathcal{G})$ such that $b_{a}\left( \mathbf{O};P^{\ast}\right)=b^{\ast }\left( \mathbf{O}\right)$. For this $P^{\ast}\in\mathcal{M}(\mathcal{G})$ clearly
$$
\chi _{P^{\ast},a,eff}^{1}\left( \mathbf{V};%
\mathcal{G}\right) \neq \psi _{P^{\ast},a}\left( \mathbf{O};\mathcal{G}\right).
$$

Assume now that $\mathbf{O}\setminus O_{T} \subset \pa_{\mathcal{G}}(O_{T}), J>1$ and the algorithm exits with output \texttt{efficient=False} because there exists $j^{\ast}\in\lbrace 2,\dots, J-1\rbrace$ such that
$\pa_{\mathcal{G}}(W_{j^{\ast}+1})\setminus \lbrace W_{j^{\ast}} \rbrace \not\subset \pa_{\mathcal{G}}(W_{j^{\ast}})$. Then, by part 4) of Lemma \ref{lemma:parents_nondesc} of Section \ref{sec:aux_eff} we have that
$$
\mathbf{O\backslash I}_{j^{\ast }}\not\ort_{\mathcal{G}}\left[ \pa_{%
\mathcal{G}}\left( W_{j^{\ast }}\right) \cup W_{j^{\ast }}\right]
\bigtriangleup \pa_{\mathcal{G}}\left( W_{j^{\ast }+1}\right) |\mathbf{I}%
_{j^{\ast }}.
$$
By Lemma \ref{lemma:non_desc} in Section \ref{sec:aux_eff}, there exists $P^{\ast}\in\mathcal{M}$ such that
\begin{equation}
q_{\mathcal{G}}\left( \mathbf{W};P^{\ast}\right) = b_{a}(\mathbf{O};P^{\ast}) - \chi_{a}(P^{\ast};\mathcal{G})+ g(\mathbf{W}),
\label{eq:q_theo10}
\end{equation}
where $g(\mathbf{W})$ is non-constant function of $W_{j^{\ast}}$. 
We argued in the discussion preceding Theorem \ref{theo:main} that if equation $\left( \ref{eq:inclusion_1}\right)$ holds for $k=K+1$ and equation $\left( \ref{eq:inclusion_2}\right)$ holds for all  $k\in \left\{
2,\dots,K+1\right\} $ hold then $h_{\mathcal{G}}(A,\mathbf{O},\mathbf{M},Y;P)$ is equal to
$$
\frac{I_{a}(A)}{\pi_{a}(\mathbf{O}_{min};P)}(Y-b_{a}(\mathbf{O};P)).
$$
Therefore
\begin{align*}
\chi _{P^{\ast},a,eff}^{1}\left( \mathbf{V};%
\mathcal{G}\right) &= q_{\mathcal{G}}\left( \mathbf{W};P^{\ast}\right)+h_{\mathcal{G}}(A,\mathbf{O},\mathbf{M},Y;P^{\ast})
\\
&=b_{a}(\mathbf{O};P^{\ast}) - \chi_{a}(P^{\ast};\mathcal{G})+ g(\mathbf{W})+\frac{I_{a}(A)}{\pi_{a}(\mathbf{O}_{min};P^{\ast})}(Y-b_{a}(\mathbf{O};P^{\ast})).
\end{align*}
cannot be equal to 
$$
\psi _{P^{\ast},a}\left( \mathbf{O};\mathcal{G}\right)=b_{a}(\mathbf{O};P^{\ast}) - \chi_{a}(P^{\ast};\mathcal{G})+\frac{I_{a}(A)}{\pi_{a}(\mathbf{O}_{min};P^{\ast})}(Y-b_{a}(\mathbf{O};P^{\ast})).
$$

On the other hand, by part 1 of Lemma \ref{lemma:main_Z} in Section \ref{sec:aux_eff}, if $\left( \ref{eq:inclusion_1}\right)$ fails for $k=K+1$, then there exists $P^{\ast}\in\mathcal{M(G)}$ such that the term
\begin{equation}
\frac{I_{a}(A)}{\pi_{a}(\mathbf{O}_{min};P^{\ast})}Y
\label{eq:pi_Ay}
\end{equation}
does not appear in the expression for $h_{\mathcal{G}}(A,\mathbf{O},\mathbf{M},Y;P^{\ast})$.
Since the term \eqref{eq:pi_Ay} appears in the expression for $\psi _{P^{\ast},a}\left( \mathbf{O};\mathcal{G}\right)$ this shows that
$$
\chi _{P^{\ast},a,eff}^{1}\left( \mathbf{V};%
\mathcal{G}\right) \neq \psi _{P^{\ast},a}\left( \mathbf{O};\mathcal{G}\right).
$$
Next, if $\left( \ref{eq:inclusion_1}\right)$ holds for $k=K+1$ but  $\left( \ref{eq:inclusion_2}\right)$ fails for $k=K+1$ then by part 2 of Lemma \ref{lemma:main_Z}
there exists $P^{\ast}\in\mathcal{M(G)}$ such that
$h_{\mathcal{G}}(A,\mathbf{O},\mathbf{M},Y;P^{\ast})$ depends on $M_{K}$. Then
$$
\chi _{P^{\ast},a,eff}^{1}\left( \mathbf{V};%
\mathcal{G}\right) =q_{\mathcal{G}}\left( \mathbf{W};P^{\ast}\right)+h_{\mathcal{G}}(A,\mathbf{O},\mathbf{M},Y;P^{\ast})
$$
cannot be equal to $\psi _{P^{\ast},a}\left( \mathbf{O};\mathcal{G}\right)$, since 
$\psi _{P^{\ast},a}\left( \mathbf{O};\mathcal{G}\right)$ is not a function of $M_{K}$.

Finally, if  $\left( \ref{eq:inclusion_1}\right)$ holds for $k=K+1$ and \eqref{eq:inclusion_2} 
fails for some 
$k\in \left\{
2,\dots,K\right\} $ but holds for all $j\in\lbrace k+1,\dots, K+1\rbrace$ then by part 3) of Lemma \ref{lemma:main_Z}, there exists $P^{\ast}\in\mathcal{M(G)}$ such that $h_{\mathcal{G}}(A,\mathbf{O},\mathbf{M},Y;P^{\ast})$ depends on $M_{k}$.  Then again
$
\chi _{P^{\ast},a,eff}^{1}\left( \mathbf{V};%
\mathcal{G}\right)
$
cannot be equal to $\psi _{P^{\ast},a}\left( \mathbf{O};\mathcal{G}\right)$, since 
$\psi _{P^{\ast},a}\left( \mathbf{O};\mathcal{G}\right)$ is not a function of $M_{k}$.

Next, assume that either $J\in\lbrace 0, 1\rbrace$ or $J>1$ and $\mathbf{O}\setminus O_{T} \subset \pa_{\mathcal{G}}(O_{T})$,
$
\pa_{\mathcal{G}}(W_{j+1})\setminus \lbrace W_{j} \rbrace \subset \pa_{\mathcal{G}}(W_{j})  
$
for all $j\in\lbrace 2,\dots, J-1\rbrace$,
and that the algorithm exits with output \texttt{efficient=False} because $\lbrace A\rbrace \cup \mathbf{O}_{min} \not\subset \pa_{\mathcal{G}}(Y)$. Assume for the sake of contradiction that 
\begin{equation}
\psi _{P,a}\left( \mathbf{O}\left( A,Y;%
\mathcal{G}\right);\mathcal{G} \right) =\chi _{P,a,eff}^{1}(\mathbf{V};%
\mathcal{G})=q_{\mathcal{G}}\left( \mathbf{W}%
;P\right) +h_{\mathcal{G}}\left( A,\mathbf{O,M,}Y;P\right)  \text{ for all } P\in\mathcal{M(G)}.
\label{eq:first_med_theo10}
\end{equation}
The only term in the expression for $h_{\mathcal{G}}\left( A,\mathbf{O,M,}%
Y;P\right) $ in $\left( \ref{eq:hm}\right) $ that could possibly be a
non-constant function of $Y$ is $E_{P}\left[ T_{P,\mathcal{G}}|Y,\text{pa}_{%
\mathcal{G}}\left( Y\right) \right] $ because ch$_{\mathcal{G}}\left(
Y\right) =\emptyset .$ Then, since \eqref{eq:first_med_theo10} holds, the following equality must also hold 
\begin{equation}
E_{P}\left[ T_{P,\mathcal{G}}|Y,\text{pa}_{\mathcal{G}}\left( Y\right) %
\right] =\frac{I_{a}(A)}{\pi_{a} \left( \mathbf{O}_{\min };P\right) }Y+g\left( A,\mathbf{O},\mathbf{M};P\right) ,  \label{eq:ultimo}
\end{equation}%
for some $g\left( A,\mathbf{O}, \mathbf{M};P\right) $ that does not depend on $Y$. 
This implies $\lbrace  A \rbrace \cup \mathbf{O}_{\min
} \subset \pa_{\mathcal{G}}\left( Y\right)$. We have arrived at a contradiction. It must therefore be that there exists $P^{\ast}\in\mathcal{M}(\mathcal{G})$ such that
$$
\chi _{P^{\ast},a,eff}^{1}\left( \mathbf{V};%
\mathcal{G}\right) \neq \psi _{P^{\ast},a}\left( \mathbf{O};\mathcal{G}\right) .
$$

Finally assume that
\begin{enumerate}
    \item $J\in\lbrace 0, 1\rbrace$ or,
    \item $J>1$ and $\mathbf{O}\setminus O_{T} \subset \pa_{\mathcal{G}}(O_{T})$,
$
\pa_{\mathcal{G}}(W_{j+1})\setminus \lbrace W_{j} \rbrace \subset \pa_{\mathcal{G}}(W_{j})  
$
for all $j\in\lbrace 2,\dots, J-1\rbrace$
\end{enumerate}
and that $\lbrace A\rbrace \cup \mathbf{O}_{min} \subset \pa_{\mathcal{G}}(Y)$ but the algorithm exits with output \texttt{efficient=False} because there exists $k\in \lbrace 2,\dots, K+1\rbrace$ such that
$$
\pa_{\mathcal{G}}(M_{k})\not\subset \pa_{\mathcal{G}}(M_{k-1})\cup \lbrace M_{k-1}\rbrace
$$
and
$$
\pa_{\mathcal{G}}(M_{j})\subset \pa_{\mathcal{G}}(M_{j-1})\cup \lbrace M_{j-1}\rbrace
$$
for $j \in \lbrace k+1,\dots,K+1 \rbrace$, where the last statement is nil if $k=K+1$.
Then, parts 2) and 3) of Lemma \ref{lemma:main_Z} imply
that there exists $P^{\ast }\in \mathcal{M}\left( \mathcal{G}\right) $ such
that $E_{P^{\ast }}\left[ T_{P^{\ast },a,\mathcal{G}}|M_{k-1},\text{pa}_{%
\mathcal{G}}\left( M_{k-1}\right) \right] -E_{P^{\ast }}\left[ T_{P^{\ast },a,
\mathcal{G}}|\text{pa}_{\mathcal{G}}\left( M_{k}\right) \right] $ is a
non-constant function of $M_{k-1}$. Moreover, as argued in Section \ref{sec:semiparam_effi}, if $k<K+1$ then for all $j\in\lbrace k+1,\dots, K+1\rbrace$ 
$$
E_{P^{\ast }}\left[
T_{P^{\ast },a,\mathcal{G}}|M_{j-1},\text{pa}_{\mathcal{G}}\left( M_{j-1}\right) %
\right] -E_{P^{\ast }}\left[ T_{P^{\ast },a,\mathcal{G}}|\text{pa}_{\mathcal{G}%
}\left( M_{j}\right) \right] =0.
$$
Also, by part 1) of Lemma \ref{lemma:main_Z},
$$
E_{P^{\ast }}\left[T_{P^{\ast },a,\mathcal{G}} \mid Y, \pa_{\mathcal{G}}(Y) \right]=\frac{I_{a}(A)Y}{\pi_{a}(\mathbf{O}_{min};P^{\ast })}.
$$
Then, with the
convention that $\sum_{j=2}^{k-1}\left( \cdot \right) \equiv 0$ if $k=2,$ we
have 
\begin{eqnarray*}
h_{\mathcal{G}}\left( A,\mathbf{O,M,}Y;P^{\ast }\right) &=& \frac{I_{a}(A)Y}{\pi_{a}(\mathbf{O}_{min};P^{\ast })}
+\left\{ E_{P^{\ast }}\left[ T_{P^{\ast },\mathcal{G}}|M_{k-1},\text{pa}_{%
\mathcal{G}}\left( M_{k-1}\right) \right] -E_{P^{\ast }}\left[ T_{P^{\ast },a,
\mathcal{G}}|\text{pa}_{\mathcal{G}}\left( M_{k}\right) \right] \right\} \\
&&+\sum_{j=2}^{k-1}\left\{ E_{P^{\ast }}\left[ T_{P^{\ast },a,\mathcal{G}%
}|M_{j-1},\text{pa}_{\mathcal{G}}\left( M_{j-1}\right) \right] -E_{P^{\ast }}%
\left[ T_{P^{\ast },a,\mathcal{G}}|\text{pa}_{\mathcal{G}}\left(
M_{j}\right) \right] \right\} 
\\
&&-E_{P^{\ast }}\left[ T_{P^{\ast },a,\mathcal{G}%
}|\text{pa}_{\mathcal{G}}\left( M_{1}\right) \right].
\end{eqnarray*}%
Now, by the topological order of $\left( M_{1},\dots,M_{K+1}\right) ,$ $M_{k-1}$
does not belong to $\pa_{\mathcal{G}}\left( M_{j}\right) $ for any $j\leq k-1$
and consequently none of the terms $E_{P^{\ast }}\left[ T_{P^{\ast },%
\mathcal{G}}|M_{j-1},\text{pa}_{\mathcal{G}}\left( M_{j-1}\right) \right]
-E_{P^{\ast }}\left[ T_{P^{\ast },\mathcal{G}}|\text{pa}_{\mathcal{G}}\left(
M_{j}\right) \right] $ for $j<k-1$ in the last display depend on $M_{k-1}$.
This then shows that $h_{\mathcal{G}%
}\left( A,\mathbf{O,M,}Y;P^{\ast }\right) $ is a non-constant function of $M_{k-1}$
thus implying that $\psi _{P^{\ast},a}\left[ \mathbf{O}\left( A,Y;\mathcal{G}%
\right);\mathcal{G} \right] \neq \chi _{P^{\ast},a,eff}^{1}(\mathbf{V};%
\mathcal{G})$ since $\psi _{P^{\ast},a}\left[ \mathbf{O}\left( A,Y;\mathcal{G}%
\right);\mathcal{G} \right] $ does not depend on $M_{k-1}.$

This finishes the proof of the theorem.
\end{proof}

\subsection{Auxiliary results}\label{sec:aux_res}

In the proof of several of the assertions in the paper we invoke the
following lemma.

\begin{lemma}
\label{lemma:inv_pi} If $\A\perp \!\!\!\perp _{\mathcal{G}}\mathbf{Z}%
_{1}\backslash \mathbf{Z}_{2}\mid \mathbf{Z}_{2}$ then for all $P\in 
\mathcal{M}\left( \mathcal{G}\right) $%
\begin{equation*}
E_{P}\left[ \left. \frac{1}{\pi_{\mathbf{a}} (\mathbf{Z}_{2};P)}\right\vert \A=\mathbf{a},%
\mathbf{Z}_{1}\right] =\frac{1}{\pi_{\mathbf{a}} (\mathbf{Z}_{1};P)},
\end{equation*}
\end{lemma}

\begin{proof}[Proof of Lemma \protect\ref{lemma:inv_pi}]
\begin{eqnarray*}
E_{P}\left[ \left. \frac{1}{\pi_{\mathbf{a}} (\mathbf{Z}_{2};P)}\right\vert \A=\mathbf{a},%
\mathbf{Z}_{1}\right] \pi_{\mathbf{a}} \left( \mathbf{Z}_{1};P\right) &\equiv &E_{P}%
\left[ \left. \frac{1}{\pi_{\mathbf{a}} \left( \mathbf{Z}_{2};P\right) }\right\vert
\A=\mathbf{a},\mathbf{Z}_{1}\right] P\left( \A=\mathbf{a}|\mathbf{Z}_{1}\right) \\
&=&E_{P}\left[ \left. \frac{I_{\mathbf{a}}(\A)}{\pi_{\mathbf{a}} \left( \mathbf{Z}_{2};P\right) 
}\right\vert \mathbf{Z}_{1}\right] \\
&=&E_{P}\left[ \left. \frac{E_{P}\left( I_{\mathbf{a}}(\A)|\mathbf{Z}_{2},\mathbf{Z}%
_{1}\right) }{\pi_{\mathbf{a}} \left( \mathbf{Z}_{2};P\right) }\right\vert \mathbf{Z}%
_{1}\right] \\
&=&1
\end{eqnarray*}%
where the last equality follows because by the fact that, $%
\A\perp \!\!\!\perp \mathbf{Z}_{1}\backslash \mathbf{Z}_{2}\mid \mathbf{Z}%
_{2} $ $\left[ P\right] $ so 
$$
E_{P}\left( I_{\mathbf{a}}(\A)|\mathbf{Z}_{2},\mathbf{Z}%
_{1}\right) =E_{P}\left( I_{\mathbf{a}}(\A)|\mathbf{Z}_{2}\right) =\pi_{\mathbf{a}} \left( 
\mathbf{Z}_{2};P\right) .
$$
\end{proof}

\subsubsection{Auxiliary results for Section \ref{sec:optimal_adj}}
\label{sec:aux_opt}

\begin{lemma}\label{lemma:a2}
If $\mathbf{Z}$ is a minimal adjustment set relative to $(A,Y)$ in DAG $\mathcal{G}$, then for all $W$ in $%
\mathbf{Z}$ there exists a path $\delta $ between $W$ and $A$ that is open given $\mathbf{%
Z}\backslash W$.
\end{lemma}

\begin{proof}[Proof of Lemma \ref{lemma:a2}]
Since $\mathbf{Z}$ is a minimal adjustment set, we know (see \cite{shpitser-adjustment}) that there exists a non-causal  $\gamma$ path between $A$ and $Y$ that is open when we condition on $\mathbf{Z}
\backslash W$ but is blocked when we condition on $\mathbf{Z}$. The path $%
\gamma $ must intersect $W$ because if it did not,
since the path is open when we condition on $\mathbf{Z}\backslash W$ 
it would also be open when we condition on $\mathbf{Z}$.
Let $\delta$ be the subpath of $\gamma$ that goes from $A$ to the first ocurrence of $W$ in $\gamma$. $\delta$ is open given $\mathbf{Z}\backslash W$, since $\gamma$ is open given $\mathbf{Z}\backslash W$.
\end{proof}

\begin{lemma}\label{lemma:a3}
If $\mathbf{Z}$ is a minimal adjustment set relative to $(A,Y)$ in DAG $\mathcal{G}$, then $\mathbf{Z}\subset \de_{\mathcal{G}}^{c}\left( A\right) .$
\end{lemma}
\begin{proof}[Proof of Lemma \ref{lemma:a3}]
This is an immediate corollary of Theorem 5 from \cite{shpitser-adjustment}.
\end{proof}

\subsubsection{Auxiliary results for Section \protect\ref{sec:efficient_est}}
\label{sec:aux_eff}

\begin{definition}
\label{def:dir} 
\begin{align*}
& \mathbf{F}\left( A,Y,\mathcal{G}\right) \equiv \{V_{j}\in \mathbf{V}:\text{
there exists a path between A and Y in } \mathcal{G} \text{ that has }V_{j}\text{ as its only fork}%
\}, \\
& \dir\left( A,Y,\mathcal{G}\right) \equiv \{Y\}\cup \left\{ V_{j}\in 
\mathbf{V}:V_{j}\text{ has a directed path to } Y \text{ in } \mathcal{G}\text{ that does not
intersect }A\right\} \setminus \mathbf{F}(A,Y,\mathcal{G}).
\end{align*}
\end{definition}

\begin{lemma}
\label{lemma:indep_dir_indir} Let $V\in \dir\left( A,Y,\mathcal{G}\right) $
and $W\in \indir\left( A,Y,\mathcal{G}\right) $. Then 
\begin{equation*}
V\perp \!\!\!\perp _{\mathcal{G}}W\mid A,\mathbf{F}\left( A,Y,\mathcal{G}%
\right)
\end{equation*}
\end{lemma}

\begin{proof}[Proof of Lemma \protect\ref{lemma:indep_dir_indir}]
Let $\mathbf{F}\equiv \mathbf{F}\left( A,Y,\mathcal{G}\right)$. We will show
that no path between $V$ and $W$ can be open given $A,\mathbf{F} $. We
analyze separately paths that (i) are directed, (ii) are not directed and have exactly
one fork and (iii) are not directed and have at least one
collider. We use the notation $T\rightrightarrows S$ to represent a directed
path between $T$ and $S$.

\textbf{(i) Directed}

Assume that there is a directed path between $V$ and $W$ and call it $\delta$%
. Assume first that $\delta$ leaves $V$ through the front-door. If $V=Y$,
since $W$ is an ancestor of $A$, this implies that $Y$ is an ancestor of $A$%
, which is a contradiction. If $V\neq Y$, since $V$ has a directed path to $%
Y $ that does not intersect $A$, we deduce that $V\in \mathbf{F}$, a
contradiction. Assume now that $\delta$ leaves $V$ through the backdoor.
This implies that there is a directed path betweeen $W$ and $Y$ that does
not intersect $A$, which is a contradiction.

Hence, there are no directed paths between $V$ and $W$ that are open given  $(A,\mathbf{F})$.

\textbf{(ii) Not directed, exactly one fork}

Assume there is a path between $V$ and $W$ that has at exactly one fork, and consequently no colliders, and is open given $(A,\mathbf{F})$. Call the path $\delta$ and
call the fork, H. Recall that $W$ is an ancestor of $A$. Since $V$ is
either equal to $Y$ or has a directed path to $Y$ that does not intersect $A$%
, the path $V \leftleftarrows H \rightrightarrows
W\rightrightarrows A$ shows that $H\in \mathbf{F}$ and hence $%
\delta$ is blocked by $\mathbf{F}$, a contradiction.

\textbf{(iii) Not directed, with at least one collider}

Assume there is a path between $V$ and $W$ that has at least one collider
and is open given $(A,\mathbf{F})$. Call the path $\delta$. All colliders in 
$\delta$ must be either in $(A,\mathbf{F})$ or have a descendant in $(A,%
\mathbf{F})$. Hence, all colliders are ancestors of $A$.

Assume first that $\delta$ leaves $V$ through the frontdoor. Consider the
collider in $\delta$ that is closest to $V$ and call it $C$. If $V=Y$, then
the directed path $Y\rightrightarrows C \rightrightarrows A$ shows
that $A$ is a descendant of $Y$, a contradiction. If $V\neq Y$, since $V$
has a directed path to $Y$ that does not intersect $A$, the path $Y
\leftleftarrows V \rightrightarrows C \rightrightarrows A$ shows
that $V\in \mathbf{F}$, which is a contradiction.

Assume now that $\delta$ leaves $V$ throught the backdoor. Consider the
collider in $\delta$ that is closest to $V$ and call it $D$. Because in the subpath of $\delta$ between $V$ and $D$ the edge with endpoint $V$ points into $V$ and the edge with endpoint $D$ points to $D$ then in that subpath there has to be a fork, say $K$. Such $K$ belongs to $\mathbf{F}$, because $K$ has directed path to $D$ and $D$ is an ancestor of $A$ and also $K$ has a directed path to $V$ that does not intersect $A$ and $V$ is either equal to $Y$ or has directed path to $Y$ that does not intersect $A$. Hence $\delta$ is blocked by $K$, which is a contradiction.

This concludes the proof of the lemma.
\end{proof}

\begin{proposition}
\label{prop:remove_directed_through_A} For any node $V_{j}\in \indir(A,Y,%
\mathcal{G})$ 
\begin{equation*}
E_{P}\left[ J_{P,\mathcal{G}}\mid V_{j},\pa_{\mathcal{G}}(V_{j})\right]
-E_{P}\left[ J_{P,\mathcal{G}}\mid \pa_{\mathcal{G}}(V_{j})\right] =0.
\end{equation*}
\end{proposition}

\begin{proof}[Proof of Proposition \protect\ref%
{prop:remove_directed_through_A}]
It suffices to show that 
\begin{equation*}
E_{P}\left[ J_{P,\mathcal{G}}\mid V_{j},\pa_{\mathcal{G}}(V_{j})\right]
\end{equation*}%
does not depend on $V_{j}$. Let $\mathbf{F}\equiv \mathbf{F}\left( A,Y,%
\mathcal{G}\right)$. We begin by noting the following: $ \mathbf{F}\cup \lbrace
V_{j}\rbrace \cup \pa_{\mathcal{G}}(V_{j})$ is comprised of non-descendants of $A. $ This
is because $V_{j}$ is a non-descendant of $A$ by assumption, since $A $ is a
descendant of $V_{j}$. This implies that $\pa_{\mathcal{G}}(V_{j})$ is a
non-descendant of $A$. Also, any node in $\mathbf{F}$ is, by definition, an
ancestor of a parent of $A$, therefore it cannot be a descendant of $A$.
Then, by the Local Markov property, 
\begin{equation*}
E_{P}\left[ I_{a}(A)|\pa_{\mathcal{G}}(A),\mathbf{F},V_{j},\pa_{\mathcal{G}}(V_{j})%
\right] =E_{P}\left[ I_{a}(A)\mid \pa_{\mathcal{G}}(A)\right] =\pi \left( \pa_{%
\mathcal{G}}(A);P\right) .
\end{equation*}%
Thus, 
\begin{equation*}
E_{P}\left[ J_{P,\mathcal{G}}\mid V_{j},\pa_{\mathcal{G}}(V_{j})\right]
=E_{P}\left[ E_{P}\left[ Y|A=a,\pa_{\mathcal{G}}(A),\mathbf{F},V_{j},\pa_{%
\mathcal{G}}(V_{j})\right] \mid V_{j},\pa_{\mathcal{G}}(V_{j})\right] .
\end{equation*}

We will show next that 
\begin{equation*}
E_{P}\left[ Y|A=a,\pa_{\mathcal{G}}(A),\mathbf{F},V_{j},\pa_{\mathcal{G}%
}(V_{j})\right] =E_{P}\left[ Y|A=a,\mathbf{F}\right] .
\end{equation*}%
To do so, it suffices to show that 
\begin{equation}
\left. Y\perp \!\!\!\perp _{\mathcal{G}} \left[\lbrace V_{j}\rbrace \cup \pa_{\mathcal{G}%
}(V_{j}) \cup\pa_{\mathcal{G}}(A)\right] \backslash \mathbf{F}%
\right\vert A,\mathbf{F}.  \label{eq:remove_directed_through_A_indep}
\end{equation}%
Note that 
\begin{equation*}
\left[\lbrace V_{j}\rbrace \cup \pa_{\mathcal{G}%
}(V_{j}) \cup\pa_{\mathcal{G}}(A)\right]\subset \text{indir}(A,Y,\mathcal{G}).
\end{equation*}%
Then by Lemma \ref{lemma:indep_dir_indir} equation %
\eqref{eq:remove_directed_through_A_indep} holds. Hence 
\begin{equation*}
E_{P}\left[ J_{P,\mathcal{G}}\mid V_{j},\pa_{\mathcal{G}}(V_{j})\right]
=E_{P}\left[ E_{P}\left[ Y|A=a,\mathbf{F}\right] \mid V_{j},\pa_{\mathcal{G}%
}(V_{j})\right].
\end{equation*}

Now note that vertices in $\mathbf{F}$ cannot be descendants of $V_{j}$,
since, if $V\in \mathbf{F}$ were a descendant of $V_{j}$, then there would
be a directed path from $V_{j}$ to $Y$ that does not intersect $A$, a
contradiction. Hence by the Local Markov Property 
\begin{equation*}
V_{j}\perp \!\!\!\perp \mathbf{F}\mid \pa_{\mathcal{G}}(V_{j}).
\end{equation*}%
Thus 
\begin{equation*}
E_{P}\left[ J_{P,\mathcal{G}}\mid V_{j},\pa_{\mathcal{G}}(V_{j})\right]
=E_{P}\left[ E_{P}\left[ Y|A=a,\mathbf{F}\right] \mid V_{j},\pa_{\mathcal{G}%
}(V_{j})\right] =E_{P}\left[ E_{P}\left[ Y|A=a,\mathbf{F}\right] \mid \pa_{%
\mathcal{G}}(V_{j})\right].
\end{equation*}%
which does not depend on $V_{j}$. This finishes the proof of the proposition.
\end{proof}

\begin{lemma}\label{lemma:parents_nondesc}
Assume that $\mathcal{G}$ is a DAG and $A$ and $Y$ are two distinct vertices in $%
\mathcal{G}$ such that $A\in \an_{\mathcal{G}}\left( Y\right)$.  Let $%
\mathbf{W}\equiv \de_{\mathcal{G}}^{c}\left( A\right) $ and $\mathbf{O\equiv
O}\left( A,Y,\mathcal{G}\right) $. Assume that $\irrel(A,Y,\mathcal{G})=\emptyset$. Write $\mathbf{W}\equiv \left(
W_{1},\dots,W_{J}\right) $, where we assume $J\geq 1$ and write $\mathbf{O\equiv }\left( O_{1},\dots,O_{T}\right) 
$ in topological order relative to $\mathcal{G}.$ Assume 
$$
\mathbf{%
O\backslash }O_{T}\mathbf{\subset }\pa_{\mathcal{G}}(O_{T}).
$$ Let 
$$
\mathbf{I}_{j} \equiv \left[ \pa_{\mathcal{G}}\left( W_{j}\right)
\cup \left\{ W_{j}\right\} \right] \cap \pa_{\mathcal{G}}\left(
W_{j+1}\right) \quad \text{for} \quad  j\in \lbrace 1,\dots,J-1\rbrace.
$$
Then,
\begin{enumerate}
\item  $W_{J}=O_{T}$
\end{enumerate}
Moreover, if $J\geq 2$,
\begin{enumerate}
\item[2.] $W_{J-1}\in \pa_{\mathcal{G}}\left( W_{J}\right) $

\item[3.] If for some $1<j^{\ast }\leq J-1$ it holds that for $j\in \left\{
j^{\ast },\dots,J-1\right\} ,$%
\begin{equation}
\pa_{\mathcal{G}}\left( W_{j+1}\right) \backslash \left\{ W_{j}\right\}
\subset \pa_{\mathcal{G}}\left( W_{j}\right) , \label{eq_new_assertion}
\end{equation}%
then
\begin{equation}
W_{j}\in \pa_{\mathcal{G}}\left( W_{j+1}\right) \text{ for }j\in \left\{
j^{\ast }-1,j^{\ast },\dots,J-1\right\}   \label{eq:padres_W_2}
\end{equation}
and
\begin{equation}
\mathbf{O\backslash I}_{j}\perp \!\!\!\perp _{\mathcal{G}}\left[ \pa_{%
\mathcal{G}}\left( W_{j}\right) \cup W_{j}\right] \bigtriangleup \pa_{%
\mathcal{G}}\left( W_{j+1}\right) |\mathbf{I}_{j} 
\text{ for }j\in \left\{
j^{\ast },\dots,J-1\right\}    \label{eq:vale}
\end{equation}

\item[4.]  Suppose that for some $j^{\ast }\in \left\{ 2,\dots ,J-1\right\} $ it
holds that 
\begin{equation}
\pa_{\mathcal{G}}\left( W_{j^{\ast }+1}\right) \backslash \left\{ W_{j^{\ast
}}\right\} \not\subset \pa_{\mathcal{G}}\left( W_{j^{\ast }}\right) 
\label{eq:ultima_3}
\end{equation}%
and that $\left( \ref{eq_new_assertion}\right) $ holds for all $j\in \left\{
j^{\ast }+1,\dots,J-1\right\} $\ if $j^{\ast }\,<J-1.$ Then,
\begin{equation}
\mathbf{O\backslash I}_{j^{\ast }}\not\ort_{\mathcal{G}}\left[ \pa_{%
\mathcal{G}}\left( W_{j^{\ast }}\right) \cup W_{j^{\ast }}\right]
\bigtriangleup \pa_{\mathcal{G}}\left( W_{j^{\ast }+1}\right) |\mathbf{I}%
_{j^{\ast }}  \label{eq:no_vale_2}.
\end{equation}%
\end{enumerate}
\end{lemma}

\begin{proof}
To prove 1), note that, since $\irrel(A,Y,\mathcal{G})=\emptyset$, there exists a
directed path between $W_{J}$ and $Y$ that does not intersect $A.$ Let $W$
be a child of $W_{J}$ in that path. Then $W$ cannot be in the set $\left\{
W_{1},\dots,W_{J}\right\} $ because $W_{J}$ is the last element in the
topolocally ordered sequence $W_{1},\dots,W_{J}$ of non-descendants of $A.$
Then $W\in \mathbf{M}\cup \left\{ Y\right\} $ which implies that $%
W_{J}\in \mathbf{O}$ and, since $\left( O_{1},\dots,O_{T}\right) $ is ordered
topologically, we conclude that $W_{J}=O_{T}.$ 

In the following proofs we will assume $J\geq 2$.

Turn now to the proof of part 2). Suppose that $W_{J-1}\not\in $pa$_{%
\mathcal{G}}\left( W_{J}\right) .$ Then, $W_{J-1}\notin \mathbf{O}$ because
by assumption, $\mathbf{O}\backslash O_{T}\subset \pa_{\mathcal{G}}\left(
W_{J}\right) .$ This implies that $W_{J-1}$ is either an ancestor of $%
Y$ such that all the directed paths between $W_{J-1}$ and $Y$ intersect $A,$
or $W_{J-1}$ is not an ancestor of $Y.$ Both possibilities are impossible
because they contradict that $\irrel\left( A,Y,\mathcal{G}\right) =\emptyset
.$

Turn now to the proof of part 3). 
We will first show \eqref{eq:padres_W_2} by reverse induction on $j^{\ast }$. Suppose $j^{\ast
}=J-1.$ We want to show that $W_{J-2}\in \pa_{\mathcal{G}}\left(
W_{J-1}\right) .$ If $W_{J-2}\in \mathbf{O}$ then by $\mathbf{O\backslash }%
O_{T}\mathbf{\subset }\pa_{\mathcal{G}}(O_{T})$ and part 1) of this lemma, $W_{J-2}\in $pa%
$_{\mathcal{G}}\left( W_{J}\right) ,$ which then implies by $\left( \ref%
{eq_new_assertion}\right) $ applied to $j=J-1$ that $W_{J-2}\in $pa$_{\mathcal{G}}\left(
W_{J-1}\right) .$ Suppose next that $W_{J-2}\not\in \mathbf{O}$ and $%
W_{J-2}\not\in \pa_{\mathcal{G}}\left( W_{J-1}\right) $, then by $\left( \ref%
{eq_new_assertion}\right) ,$ $W_{J-2}\not\in \pa_{\mathcal{G}}\left(
W_{J}\right) .$ Consequently, $W_{J-2}$ is either an ancestor of $Y$ such
that all the directed paths between $W_{J-2}$ and $Y$ intersect $A$ or $%
W_{J-2}$ is not an ancestor of $Y.$ Both possibilities are impossible
because they contradict that $\irrel\left( A,Y,\mathcal{G}\right) =\emptyset
.$ This shows that \eqref{eq:padres_W_2} is true for $j^{\ast }=J-1.$ Suppose now that the result holds for $j^{\ast }=m,\dots,J-1,$ for some $%
2<m\leq J-1.$ We will show that it also holds for $j^{\ast }=m-1.$
Henceforth suppose that $\left( \ref{eq_new_assertion}\right) $ holds for $%
j\in \left\{ m-1,\dots,J-1\right\} .$ Then, $\left( \ref{eq_new_assertion}%
\right) $ holds for $j\in \left\{ m,\dots,J-1\right\} $ and consequently, by
the inductive hypothesis, $\left( \ref{eq:padres_W_2}\right) $ holds for $%
j\in \left\{ m-1,m,\dots,J-1\right\} .$ It remains to show that $W_{m-2}\in \pa%
_{\mathcal{G}}\left( W_{m-1}\right) .$ Suppose that $W_{m-2}\in \mathbf{O,}$
then by $\mathbf{O\backslash }O_{T}\mathbf{\subset }\pa_{\mathcal{G}}(O_{T})$
and part 1), $W_{m-2}\in $pa$_{\mathcal{G}}\left( W_{J}\right) ,$ which
then implies, by $\left( \ref{eq_new_assertion}\right) $ being valid for all 
$j\in \left\{ m-1,\dots,J-1\right\} ,$ that 
\[
W_{m-2}\in \pa_{\mathcal{G}}\left( W_{J}\right) \backslash \left\{
W_{m-1},\dots,W_{J-1}\right\} \subset \pa_{\mathcal{G}}\left( W_{J-1}\right)
\backslash \left\{ W_{m-1},\dots,W_{J-2}\right\} \subset \dots\subset \pa_{%
\mathcal{G}}\left( W_{m-1}\right) .
\]
On the other hand, if $W_{m-2}\notin \mathbf{O}$, since $\irrel(A,Y, \mathcal{G})=\emptyset$, necessarily $W_{m-2}\in \pa_{\mathcal{G}}(W_{j})$ for some $j>m-2$. Arguing as before, this implies that $W_{m-2}\in \pa_{%
\mathcal{G}}\left( W_{m-1}\right)$.

Next we prove \eqref{eq:vale}. Suppose that for $j\in \left\{ j^{\ast
},\dots,J-1\right\} ,$ $\left( \ref{eq_new_assertion}\right) $ holds. Then,
for $j\in \left\{ j^{\ast },\dots,J-1\right\} $ we have 
\begin{eqnarray*}
\mathbf{I}_{j} &\equiv &\left[ \text{pa}_{\mathcal{G}}\left( W_{j}\right)
\cup \left\{ W_{j}\right\} \right] \cap \text{pa}_{\mathcal{G}}\left(
W_{j+1}\right)  \\
&=&\left[ \text{pa}_{\mathcal{G}}\left( W_{j}\right) \cup \left\{
W_{j}\right\} \right] \cap \left[ \left[ \text{pa}_{\mathcal{G}}\left(
W_{j+1}\right) \backslash \left\{ W_{j}\right\} \right] \cup \left\{
W_{j}\right\} \right]  \\
&=&\left[ \left[ \text{pa}_{\mathcal{G}}\left( W_{j+1}\right) \backslash
\left\{ W_{j}\right\} \right] \cup W_{j}\right]  \\
&=&\text{pa}_{\mathcal{G}}\left( W_{j+1}\right) 
\end{eqnarray*}%
where the second and forth equalities follow by $\left( \ref{eq:padres_W_2}%
\right) $ and the third follows by $\left( \ref{eq_new_assertion}\right)$. On the other hand, because by assumption $\mathbf{O}\backslash O_{T}\subset $%
pa$_{\mathcal{G}}\left( W_{J}\right) ,$ then $\mathbf{O\backslash }\left(
W_{j+1},\dots,W_{J}\right) \subset $pa$_{\mathcal{G}}\left( W_{j+1}\right) .$
Consequently, $\left( \ref{eq:vale}\right) $ holds if and only if  
\begin{equation}
\mathbf{O\cap }\left( W_{j+1},\dots,W_{J}\right) \perp \!\!\!\perp _{\mathcal{G%
}}\text{pa}_{\mathcal{G}}\left( W_{j}\right) \backslash \text{pa}_{\mathcal{G%
}}\left( W_{j+1}\right) |\text{pa}_{\mathcal{G}}\left( W_{j+1}\right) .
\label{eq:vale_equiv}
\end{equation}%
We will show by contradiction that $\left( \ref{eq:vale_equiv}\right) $
holds, and consequently that $\left( \ref{eq:vale}\right) $ holds, for $j\in
\left\{ j^{\ast },j^{\ast }+1,\dots,J-1\right\} .$ Suppose that $\left( \ref%
{eq:vale_equiv}\right) $ were not true for some $j\in \left\{ j^{\ast
},j^{\ast }+1,\dots,J-1\right\} .$ Then there would exist $u\geq j+1$ and $l<j$
such that $W_{u}\not\ort_{\mathcal{G}}W_{l}|$pa$_{\mathcal{G}}\left(
W_{j+1}\right) $ with $W_{u}\in \mathbf{O}$ and 
\begin{equation}
W_{l}\in \text{pa}_{\mathcal{G}}\left( W_{j}\right) \backslash 
\text{pa}_{\mathcal{G}}\left( W_{j+1}\right) .  \label{eq:pasado}
\end{equation}
Because, by $\left( \ref{eq:padres_W_2}\right) ,$ $W_{j}\in $pa$_{\mathcal{G}%
}\left( W_{j+1}\right) ,$ the path between $W_{l}$ and $W_{u}$ that would be
open given pa$_{\mathcal{G}}\left( W_{j+1}\right) $ would necessarily have
to include an edge $W_{l^{\ast }}\rightarrow W_{u^{\ast }}$ for some $%
l^{\ast }<j$ and $u^{\ast }\geq j+1.$ If $u^{\ast }=j+1,$ then this implies
that $W_{l^{\ast }}\in $pa$_{\mathcal{G}}\left( W_{j+1}\right) $ which is
impossible because it contradicts $W_{u}\not\ort_{\mathcal{G}}W_{l}|$pa$_{\mathcal{G}}\left(
W_{j+1}\right) $. If $%
u^{\ast }>j+1,$ then by $W_{l^{\ast }}\in $pa$_{\mathcal{G}}\left(
W_{u^{\ast }}\right) $ we have $W_{l^{\ast }}\in $pa$_{\mathcal{G}}\left(
W_{u^{\ast }}\right) \backslash \left\{ W_{j+1},\dots,W_{u^{\ast }-1}\right\} $
because $l^{\ast }<j.$ However,  by $\left( \ref{eq_new_assertion}\right) ,$%
\begin{align*}
\text{pa}_{\mathcal{G}}\left( W_{u^{\ast }}\right) \backslash \left\{
W_{j+1},\dots,W_{u^{\ast }-1}\right\} \subset \text{pa}_{\mathcal{G}}\left(
W_{u^{\ast }-1}\right) \backslash \left\{ W_{j+1},\dots,W_{u^{\ast
}-2}\right\} &\subset \dots \subset\text{pa}_{\mathcal{G}}\left( W_{j+3}\right)
\backslash \left\{ W_{j+1},W_{j+2}\right\} 
&\\
&\subset \text{pa}_{\mathcal{G}%
}\left( W_{j+2}\right) \backslash \left\{ W_{j+1}\right\} \subset \text{pa}_{%
\mathcal{G}}\left( W_{j+1}\right) 
\end{align*}
which then implies that $W_{l^{\ast }}\in $pa$_{\mathcal{G}}\left(
W_{j+1}\right) $ again contradicting $W_{u}\not\ort_{\mathcal{G}}W_{l}|$pa$_{\mathcal{G}}\left(
W_{j+1}\right) $. This
proves \eqref{eq:vale}.

Turn now to the proof of part 4). Suppose that pa$_{\mathcal{G}}\left(
W_{j^{\ast }+1}\right) \backslash \left\{ W_{j^{\ast }}\right\} \not\subset $%
pa$_{\mathcal{G}}\left( W_{j^{\ast }}\right) $ and that $\left( \ref%
{eq_new_assertion}\right) $ holds for all $j\in \left\{ j^{\ast
}+1,\dots,J-1\right\}$ if $j^{\ast }<J-1$. Then there exists $l<j^{\ast }$
such that $W_{l}\in $pa$_{\mathcal{G}}\left( W_{j^{\ast }+1}\right)
\backslash $pa$_{\mathcal{G}}\left( W_{j^{\ast }}\right) .$ By \eqref{eq:padres_W_2}, $%
W_{j}\in \pa_{\mathcal{G}}\left( W_{j+1}\right) $ for all $j=j^{\ast
},j^{\ast }+1,\dots,J-1.$ Consequently, the path $W_{l}\rightarrow W_{j^{\ast
}+1}\rightarrow W_{j^{\ast }+2}\rightarrow \circ \dots\circ \rightarrow W_{J}$
is open in $\mathcal{G}$ when conditioning on $\mathbf{I}_{j^{\ast
}}$.
By part 1), $W_{J}=O_{T}\in \mathbf{O\cap }\left( W_{j+1},\dots,W_{J}\right) ,$
and $W_{l}\in \left[ \pa_{\mathcal{G}}\left( W_{j^{\ast }}\right) \cup
W_{j^{\ast }}\right] \bigtriangleup \pa_{\mathcal{G}}\left( W_{j^{\ast
}+1}\right) ,$ thus the aforementioned open path shows that $\left( \ref%
{eq:no_vale_2}\right) $ holds. This concludes the proof of \eqref{eq:ultima_3}.
\end{proof}

\begin{lemma}\label{lemma:indep_med}
Let $\mathcal{G}$ be a DAG with vertex set $\mathbf{V}$ and
let $A$ and $Y$ be two distinct vertices in $\mathbf{V.}$ Suppose that $\irrel
\left( A,Y;\mathcal{G}\right) =\emptyset .$ Suppose $\mathbf{M}\mathbf{%
\equiv }\de_{\mathcal{G}}\left( A\right) \backslash \left\{ A,Y\right\}
\not=\emptyset $ and let $\left( M_{1},\dots,M_{K}\right) $ be the elements of
$\mathbf{M}$ sorted topologically. Let $M_{0}\equiv A$ and $M_{K+1}\mathbf{%
\equiv }Y.$ 

Suppose that for some $k^{\ast }\geq 2,$ the following inclussion holds for $%
k\in \left\{ k^{\ast },\dots,K+1\right\} $%
\begin{equation}
\pa_{\mathcal{G}}\left( M_{k}\right) \subset \pa_{\mathcal{G}%
}\left( M_{k-1}\right) \cup \left\{ M_{k-1}\right\}.   \label{eq:inclusion_2_lemma}
\end{equation}
Then, $M_{K}\in \pa_{\mathcal{G}}\left( Y\right) $ and for all $k\in
\left\{ k^{\ast },\dots,K+1\right\} $

(i)  $M_{k-2}\in \pa\left( M_{k-1}\right) $ and

(ii)   
\begin{equation}
Y\ort _{\mathcal{G}}\left[ M_{k-1},\pa_{\mathcal{G}}\left(
M_{k-1}\right) \right] \backslash \pa_{\mathcal{G}}\left( M_{k}\right)
|\pa_{\mathcal{G}}\left( M_{k}\right) .  \label{eq:d_sep_mediators_lemma}
\end{equation}
\end{lemma}
\begin{proof}

That $M_{K}\in $pa$_{\mathcal{G}}\left( Y\right) $ follows from irrel$\left(
A,Y;\mathcal{G}\right) =\emptyset $ and the fact that $M_{K}$ is last in the
topological order of $\mathbf{M}$.

To show (i), assume that for some $k^{\ast }\geq 2,\left( \ref%
{eq:inclusion_2_lemma}\right) $ holds for all $k\in \left\{ k^{\ast
},\dots,K+1\right\} $ $.$ Let $k\in \left\{ k^{\ast },\dots,K+1\right\} .$ If $%
k=k^{\ast }=2,$ then $M_{k-2}=A\in $pa$_{\mathcal{G}}\left( M_{1}\right) $
for otherwise $M_{1}$ would not be a descendant of $A.$ Next assume $k>2.$
The assumption that irrel$\left( A,Y;\mathcal{G}\right) =\emptyset $ and the
topological order of $\left( M_{1},\dots,M_{K}\right) $ implies that $%
M_{k-2}\in $pa$_{\mathcal{G}}\left( M_{r}\right) $ for some $r\in \left\{
k-1,k,\dots,K+1\right\} .$ If $r=k-1$ we are done. If $r\geq k,$ then $r\in
\left\{ k^{\ast },\dots,K+1\right\} $ and consequently $\left( \ref%
{eq:inclusion_2_lemma}\right) $ implies that  
\[
\text{pa}_{\mathcal{G}}\left( M_{r}\right) \subset \text{pa}_{\mathcal{G}%
}\left( M_{r-1}\right) \cup \left\{ M_{r-1}\right\} \subset \dots\subset \text{%
pa}_{\mathcal{G}}\left( M_{k-1}\right) \cup \left\{
M_{r-1},M_{r-2},\dots,M_{k-1}\right\} 
\]%
Consequently, $M_{k-2}\in $pa$_{\mathcal{G}}\left( M_{k-1}\right) \cup
\left\{ M_{r-1},M_{r-2},\dots,M_{k-1}\right\} $ and since $M_{k-2}\not\in
\left\{ M_{r-1},M_{r-2},\dots,M_{k-1}\right\} $ then $M_{k-2}\in $pa$_{%
\mathcal{G}}\left( M_{k-1}\right) .$

To show (ii), assume that for some $k^{\ast }\geq 2,$ $\left( \ref%
{eq:inclusion_2_lemma}\right) $ holds for all $k\in \left\{ k^{\ast
},\dots,K+1\right\} .$ Let $k\in \left\{ k^{\ast },\dots,K+1\right\} .$
Assumption $\left( \ref{eq:inclusion_2_lemma}\right) $ implies that  
\begin{eqnarray}
\text{pa}_{\mathcal{G}}\left( Y\right)  &\subset &\text{pa}_{\mathcal{G}%
}\left( M_{K}\right) \cup \left\{ M_{K}\right\} \subset \text{pa}_{\mathcal{G%
}}\left( M_{K-1}\right) \cup \left\{ M_{K},M_{K-1}\right\} 
\label{eq:inclusion_10} \\
&\subset &\cdots\subset \text{pa}_{\mathcal{G}}\left( M_{k}\right) \cup \left\{
M_{K},M_{K-1},\dots,M_{k}\right\} \subset \text{pa}_{\mathcal{G}}\left(
M_{k-1}\right) \cup \left\{ M_{K},M_{K-1},\dots,M_{k-1}\right\}   \nonumber
\end{eqnarray}%
By part (i) we have $M_{k-1}\in $pa$_{\mathcal{G}}\left( M_{k}\right) .$
Then, $\left( \ref{eq:d_sep_mediators_lemma}\right) $ is the same as 
\begin{equation}
Y\ort _{\mathcal{G}}\text{pa}_{\mathcal{G}}\left( M_{k-1}\right)
\backslash \text{pa}_{\mathcal{G}}\left( M_{k}\right) |\text{pa}_{\mathcal{G}%
}\left( M_{k}\right)   \label{eq:d_sep_mediators_2}
\end{equation}%
Suppose $\left( \ref{eq:d_sep_mediators_2}\right) $ is false. Let $M_{j}\in $%
pa$_{\mathcal{G}}\left( M_{k-1}\right) \backslash $pa$_{\mathcal{G}}\left(
M_{k}\right) $ such that $Y\not \ort_{\mathcal{G}}M_{j}|$pa$_{\mathcal{G}%
}\left( M_{k}\right) .$ Because $Y$ has no descendants in the DAG, then any
open path between $M_{j}$ and $Y$ must end with an edge pointing into $Y.$
If such path is open when we condition on pa$_{\mathcal{G}}\left(
M_{k}\right) =\left[ \text{pa}_{\mathcal{G}}\left( Y\right) \backslash
\left\{ M_{K},M_{K-1},\dots,M_{k}\right\} \right] \cup \left[ \text{pa}_{%
\mathcal{G}}\left( M_{k}\right) \backslash \text{pa}_{\mathcal{G}}\left(
Y\right) \right] ,$ then this edge must connect a vertex $M_{t}\in \left\{
M_{K},M_{K-1},\dots,M_{k}\right\} $ with $Y.$ This is because any other vertex
would be in pa$_{\mathcal{G}}\left( Y\right) \backslash \left\{
M_{K},M_{K-1},\dots,M_{k}\right\} $ and the path would then be closed because
we are conditioning on 
$$pa_{\mathcal{G}}\left( Y\right) \backslash \left\{
M_{K},M_{K-1},\dots,M_{k}\right\} .
$$
Then the path between $M_{j}$ and $Y$
that is open when we condition on pa$_{\mathcal{G}}\left( M_{k}\right) $
must be of the form 
\begin{equation}
M_{j}-\circ -\circ \dots\circ -V\rightarrow M_{t}\rightarrow Y\text{  }
\label{path:1}
\end{equation}%
or%
\begin{equation}
M_{j}-\circ -\circ \dots\circ -V\leftarrow M_{t}\rightarrow Y  \label{path:2}
\end{equation}%
for some $t\in \left\{ k,k+1,\dots,K\right\} $ and some $V\in \mathbf{V}.$ We
now argue that it cannot be of the form $\left( \ref{path:2}\right) .$
Suppose it was of the form $\left( \ref{path:2}\right) .$ Then, $V$ would
belong to $\mathbf{M}$ because $V$ is a child of a descendant of $A$ and
consequently it is itself a descendant of $A.$ By the topological order of $%
\left( M_{1},\dots,M_{K}\right) $ this would imply that $V=M_{h}$ for some $%
h>t.$ But in such case the path between $M_{j}$ and $M_{h}$ would eventually
intersect a collider $M_{r}$ for some $r>h,$ i.e. it would be of the form  
\[
M_{j}-\circ -\circ \dots\circ -\rightarrow M_{r}\leftarrow \circ \dots\leftarrow
\circ \leftarrow M_{h}\leftarrow M_{t}\rightarrow Y
\]%
However, this is impossible because by $r>h>t\in \left\{ k,k+1,\dots,K\right\} 
$ we have that neither $M_{r}$ nor its descendants are in pa$_{\mathcal{G}%
}\left( M_{k}\right) $, so the path is closed at the collider $M_{r}$ when
we condition on pa$_{\mathcal{G}}\left( M_{k}\right) .$ 

We thus conclude that if an open path exists it must be of the form $\left( %
\ref{path:1}\right) $ for some $t\in \left\{ k,k+1,\dots,K\right\} .$ However,
we will now show that this is also impossible. First we note that the
assumption that the path is open when we condition on pa$_{\mathcal{G}%
}\left( M_{k}\right) $ implies that  
\[
V\not\in \text{pa}_{\mathcal{G}}\left( M_{k}\right) .
\]%
This implies that 
\[
k+1\leq t\leq K.
\]%
Next, note that because $V\in $pa$_{\mathcal{G}}\left( M_{t}\right) $ and  
\[
\text{pa}_{\mathcal{G}}\left( M_{t}\right) \subset \text{pa}_{\mathcal{G}%
}\left( M_{k}\right) \cup \left\{ M_{k},\dots,M_{t-1}\right\} 
\]%
this implies that $V\in \left\{ M_{k},\dots,M_{t-1}\right\} .$ So, we conclude
that the open path must be of the form 
\begin{equation}
M_{j}-\circ -\circ \dots\circ -V^{\prime }\rightarrow M_{h}\rightarrow
M_{t}\rightarrow Y\text{ }  \label{path:3}
\end{equation}%
or%
\begin{equation}
M_{j}-\circ -\circ \dots\circ -V^{\prime }\leftarrow M_{h}\rightarrow
M_{t}\rightarrow Y  \label{path:4}
\end{equation}%
for some 
\[
k\leq h<t\leq K.
\]
However, reasoning as above we rule out the path $\left( \ref{path:4}\right) 
$ and conclude that the path must be of the form $\left( \ref{path:3}\right) 
$ for $V^{\prime }=M_{r}$ with $r$ such that
\[
k\leq r<h<t\leq K.
\]%
Continuing in this fashion we arrive at the conclusion that the path must
be of the form 
\[
M_{j}-\circ -\circ \dots\circ -V^{\ast }\rightarrow M_{k}\dots M_{r}\rightarrow
M_{h}\rightarrow M_{t}\rightarrow Y
\]%
But this contradicts the assumption that the path is open when we condition
on pa$_{\mathcal{G}}\left( M_{k}\right) $ since $V^{\ast }\in $ pa$_{%
\mathcal{G}}\left( M_{k}\right) $. This concludes the proof.
\end{proof}

\begin{lemma}
\label{lemma:unrestricted} 
Assume that $\mathcal{G}$ is a DAG and $A$ and $Y$ are two distinct vertices in $%
\mathcal{G}$ such that $A\in \an_{\mathcal{G}}\left( Y\right)$ and $\irrel%
\left( A,Y,\mathcal{G}\right) =\emptyset .$  Let $%
\mathbf{W}\equiv \de_{\mathcal{G}}^{c}\left( A\right) $ and $\mathbf{O\equiv
O}\left( A,Y,\mathcal{G}\right) $. Write $\mathbf{W}\equiv \left(
W_{1},\dots,W_{J}\right) $ and $\mathbf{O\equiv }\left( O_{1},\dots,O_{T}\right) 
$ in topological order relative to $\mathcal{G}.$
Then, under $\mathcal{M}\left( \mathcal{G}\right) ,$ the law of $Y$
given $\mathbf{W}$ is the same as the law of $Y$ given $\left( A,\mathbf{O}%
\right) $ and the law of $Y$ given $\left( A,\mathbf{O}\right) $ is
unrestricted. In particular, the conditional expectation $E\left( Y|\mathbf{W%
}\right) =E\left( Y|A,\mathbf{O}\right) $ is unrestricted. Furthermore, the
law of $Y$ given $\left( A,\mathbf{O}\right) $ and the law of $\mathbf{W} \cup \left\{ A\right\} $ are variation independent.
\end{lemma}

\begin{proof}[Proof of Lemma \protect\ref{lemma:unrestricted}]
That the law of $Y$ given $\mathbf{W}$ is the same as the law of $Y$ given $A,\mathbf{O}$
under any $P\in \mathcal{M}\left( \mathcal{G}\right) $ follows because $%
Y\perp \!\!\!\perp _{\mathcal{G}}\mathbf{W}\backslash \left( A,\mathbf{O}%
\right) |\left( A,\mathbf{O}\right) $.

Next, assume $\de_{\mathcal{G}}\left( A\right)\setminus \lbrace Y \rbrace \neq \emptyset$. 
Let $\mathcal{G}^{\prime}=\mathcal{G}_{\mathbf{V}%
\backslash \text{an}_{\mathcal{G}}^{c}\left( A,Y\right) }$ and
$\mathbf{V}^{\prime}= \mathbf{V}\setminus  \text{an}_{\mathcal{G}}^{c}\left( A,Y\right)  $.
Since $\mathbf{V}%
\backslash \an_{\mathcal{G}}^{c}\left( A,Y\right) $ is ancestral, $\mathcal{M}\left( 
\mathcal{G}^{\prime }\right) =\mathcal{M}\left( \mathcal{G},\mathbf{V}^{\prime} \right) $ (see Proposition 1 from \cite{evans-marginal}). 
Let $\mathbf{M\equiv }\left(
M_{1},\dots,M_{K}\right) \mathbf{\equiv }\de_{\mathcal{G}}\left( A\right)\setminus \lbrace Y \rbrace$ be topologically ordered relative
to $\mathcal{G}^{\prime}$. Now, define $\mathcal{G}_{K}\mathbf{\equiv }\tau
\left( \mathcal{G}^{\prime},M_{K}\right) $ and recursively for $%
k=K-1,K-2,\dots,1$ define $\mathcal{G}_{k}\mathbf{\equiv }\tau \left( \mathcal{%
G}_{k+1},M_{k}\right) .$  Now, since ch$_{\mathcal{G}^{\prime}}\left(
M_{K}\right) =\left\{ Y\right\} $, by Lemma 3 of \cite{evans-marginal}, $\mathcal{M}\left( \mathcal{G}
_{K}\right) =\mathcal{M}\left( \mathcal{G}^{\prime},\mathbf{V}^{\prime} \setminus \lbrace M_{K}\rbrace\right) .$ Furthermore, $$\pa_{\mathcal{G}_{K}}\left( Y\right)
= \pa_{\mathcal{G}^{\prime}}\left( Y\right) \cup \pa_{\mathcal{G}^{\prime}}\left(
M_{K}\right) .$$
Likewise, since for $k=K-1,K-2,\dots,1,$ ch$_{\mathcal{G}%
_{k+1}}\left( M_{k}\right) =\left\{ Y\right\} ,$ then we can recursively
show that for $k=K-1,K-2,\dots,1$, $\mathcal{M}\left( \mathcal{G}_{k}\right) =%
\mathcal{M}\left( \mathcal{G}_{k+1},\mathbf{V}^{\prime}\backslash \left\{ M_{K},M_{K-1},\dots,M_{k}\right\} \right) $
and pa$_{\mathcal{G}_{k}}\left( Y\right) =$pa$_{\mathcal{G}^{\prime}}\left( Y\right)
\cup \left[ \cup _{l=k}^{K}\text{pa}_{\mathcal{G}^{\prime}}\left( M_{l}\right) \right]
.$ In particular, $\mathcal{M}\left( \mathcal{G}_{1}\right) =\mathcal{M}%
\left( \mathcal{G}_{2},\mathbf{V}^{\prime}\backslash \mathbf{M} \right) $ and pa$_{\mathcal{G}%
_{1}}\left( Y\right) =$pa$_{\mathcal{G}}\left( Y\right) \cup \left[ \cup
_{l=1}^{K}\text{pa}_{\mathcal{G}}\left( M_{l}\right) \right] .$ Applying
repeatedly the property $\left( \ref{eq:p1}\right) $ we arrive at $\mathcal{M%
}\left( \mathcal{G}_{1}\right) =\mathcal{M}\left( \mathcal{G},\mathbf{V}^{\prime}%
\backslash \mathbf{M} \right) .$ But $\left( A,\mathbf{O}%
\right) = \pa_{\mathcal{G}_{1}}(Y)$ and in $\mathcal{M}\left( \mathcal{%
G}_{1}\right) $ the law of $Y$ given pa$_{\mathcal{G}_{1}}\left( Y\right) $
is unrestricted. This implies that the law of $Y$ given $\left( A,\mathbf{O}%
\right) $ is unrestricted under $\mathcal{M}\left( \mathcal{G}_{1}\right) $. Then, $\mathcal{M}\left( \mathcal{G}_{1}\right) =\mathcal{M}%
\left( \mathcal{G},\mathbf{V}^{\prime}\backslash\mathbf{M}\right) $ implies
that the law of $Y$ given $\left( A,\mathbf{O}\right) $ is unrestricted
under $\mathcal{M}\left( \mathcal{G}\right).$ Finally, in model $\mathcal{M}%
\left( \mathcal{G}_{1}\right) $ (and consequently in model $\mathcal{M}%
\left( \mathcal{G}\right) )$ the law of $
\de_{\mathcal{G}}^{c}\left( A\right)
\cup \left\{ A\right\} $ and the law of $Y$ given pa$_{\mathcal{G}%
_{1}}\left( Y\right) $ are variation independent, and therefore so are the
laws of de$_{\mathcal{G}}^{c}\left( A\right) \cup \left\{ A\right\} $ and of 
$Y$ given $\left( A,\mathbf{O}\right) $ under model $\mathcal{M}%
\left( \mathcal{G}\right)$.

If  $\de_{\mathcal{G}}\left( A\right)\setminus \lbrace Y \rbrace = \emptyset$ then 
$\left( A,\mathbf{O}\right) =\pa_{\mathcal{G}}(Y)$ and the result follows immediately arguing as above.
\end{proof}

\medskip

\begin{lemma}
\label{lemma:non_desc} Assume that $\mathcal{G}$ is a DAG with vertex set $%
\mathbf{V}$ and $A$ and $Y$ are two distinct vertices in $\mathcal{G}$ such
that $A\in \an_{\mathcal{G}}\left( Y\right) $. Let $\mathbf{W}\equiv \de_{%
\mathcal{G}}^{c}\left( A\right) $ and $\mathbf{O\equiv O}\left( A,Y,\mathcal{%
G}\right) $. Write $\mathbf{W}\equiv \left( W_{1},\dots ,W_{J}\right) $ and $%
\mathbf{O\equiv }\left( O_{1},\dots ,O_{T}\right) $ in topological order
relative to $\mathcal{G}.$ Assume $\mathbf{O\backslash }O_{T}\mathbf{\subset 
}\pa_{\mathcal{G}}(O_{T})$ and $\irrel\left( A,Y,\mathcal{G}\right)
=\emptyset .$ Let 
\[
\mathbf{I}_{j}\equiv \left[ \pa_{\mathcal{G}}\left( W_{j}\right) \cup
\left\{ W_{j}\right\} \right] \cap \pa_{\mathcal{G}}\left( W_{j+1}\right) .
\]%
Assume that $J>1$ and for some $j\in \left\{ 1,\dots ,J-1\right\} $ it holds
that for $k=j+1,\dots ,J-1,$ 
\begin{equation}
\mathbf{O\backslash I}_{k}\perp \!\!\!\perp _{\mathcal{G}}\left[ \pa_{%
\mathcal{G}}\left( W_{k}\right) \cup W_{k}\right] \bigtriangleup \pa_{%
\mathcal{G}}\left( W_{k+1}\right) |\mathbf{I}_{k}  \label{eq:ind1}
\end{equation}%
and 
\begin{equation}
\mathbf{O\backslash I}_{j}\not\ort_{\mathcal{G}}\left[ \pa_{\mathcal{G}%
}\left( W_{j}\right) \cup W_{j}\right] \bigtriangleup \pa_{\mathcal{G}%
}\left( W_{j+1}\right) |\mathbf{I}_{j}  \label{eq:ind2}
\end{equation}%
where the assertion $\left( \ref{eq:ind1}\right) $ is inexistant if $j=J-1$.
Then there exists $P^{\ast }\in \mathcal{M(G)}$ such that 
\begin{equation}
E_{P^{\ast }}\left[ b_{a}\left( \mathbf{O};P^{\ast }\right) |W_{j},\pa_{%
\mathcal{G}}\left( W_{j}\right) \right] -E_{P^{\ast }}\left[ b_{a}\left( 
\mathbf{O};P^{\ast }\right) |\pa_{\mathcal{G}}\left( W_{j+1}\right) \right] 
\label{eq:dif_non_desc}
\end{equation}%
is a non-constant function of $W_{j}$.
\end{lemma}

\begin{proof}[Proof of Lemma \protect\ref{lemma:non_desc}]

First we show that if $\left( \ref{eq:ind1}\right) $ holds for some $k\in
\left\{ 1,...,J-1\right\} ,$ then for such $k\,\ $it holds that $W_{k}\in \pa%
_{\mathcal{G}}\left( W_{k+1}\right) .$ Assume for the sake of contradiction
that $W_{k}\notin \pa_{\mathcal{G}}(W_{k+1})$. Since $\irrel(A,Y,\mathcal{G}%
)=\emptyset $, there exists a directed path between $W_{k}$ and $Y$ that
does not intersect $A$. Such a path must intersect $\mathbf{O}$. Since $%
\mathbf{O}\subset \an_{\mathcal{G}}(O_{T})$ we conclude that $W_{k}\in \an_{%
\mathcal{G}}(O_{T})$. Note also that $\mathbf{I}_{k}\cap \de_{\mathcal{G}%
}(W_{k})=\emptyset $. Then 
\begin{equation}
O_{T}\not\ort_{\mathcal{G}}W_{k}\mid \mathbf{I}_{k}.  \label{eq:zk_O}
\end{equation}%
Now $O_{T}\in \mathbf{O}\setminus \mathbf{I}_{k}$ because $O_{T}=W_{J}$.
Since $W_{k}\notin \pa_{\mathcal{G}}(W_{k+1})$ then 
\[
W_{k}\in \left[ \pa_{\mathcal{G}}\left( W_{k}\right) \cup W_{k}\right]
\bigtriangleup \pa_{\mathcal{G}}(W_{k+1}),
\]%
which together with $\left( \ref{eq:zk_O}\right) $ implies 
\begin{equation}
\mathbf{O\backslash I}_{k}\not\ort_{\mathcal{G}}\left[ \pa_{\mathcal{G}%
}\left( W_{k}\right) \cup W_{k}\right] \bigtriangleup \pa_{\mathcal{G}%
}\left( W_{k+1}\right) |\mathbf{I}_{k},  \nonumber
\end{equation}%
The last display contradicts $\left( \ref{eq:ind1}\right) $, thus proving
that $W_{k}\in \pa_{\mathcal{G}}(W_{k+1})$.

We will show that for some $P^{\ast }\in \mathcal{M(G)}$, $\left( \ref%
{eq:dif_non_desc}\right) $ is a non-constant function of $W_{j}$ by
considering separately the cases $W_{j}\not\in $pa$_{\mathcal{G}}\left(
W_{j+1}\right) $ and $W_{j}\in $pa$_{\mathcal{G}}\left( W_{j+1}\right) .$

Suppose first that $W_{j}\not\in $pa$_{\mathcal{G}}\left( W_{j+1}\right) .$
Then, since $E_{P}\left[ b_{a}\left( \mathbf{O};P\right) |\text{pa}_{%
\mathcal{G}}\left( W_{j+1}\right) \right] $ does not depend on $W_{j}$ for
all $P\in \mathcal{M(G)}$, it suffices to prove that there exists $P^{\ast
}\in \mathcal{M(G)}$ such that 
\[
E_{P^{\ast }}\left[ b_{a}\left( \mathbf{O};P^{\ast }\right) |W_{j},\pa_{%
\mathcal{G}}\left( W_{j}\right) \right] 
\]%
is a non-constant function of $W_{j}$. To show this, first note that since $%
\irrel\left( A,Y,\mathcal{G}\right) =\emptyset $ there exists a directed
path between $W_{j}$ and $Y$ that does not intersect $A$. Such a path must
intersect $\mathbf{O}$. Since $\mathbf{O\backslash }O_{T}\mathbf{\subset }\pa%
_{\mathcal{G}}(O_{T}),$ then $W_{j}\in \an_{\mathcal{G}}(O_{T})$.
Consequently,  
\begin{equation}
W_{j}\not\ort_{\mathcal{G}}O_{T}\mid \pa_{\mathcal{G}}\left( W_{j}\right) .
\label{eq:d-sep-lemma16}
\end{equation}%
Thus, there exists a law $P^{\ast }\in \mathcal{M}\left( \mathcal{G}\right) $
such that under $P^{\ast },$ $W_{j}\not\ort O_{T}\mid \pa_{\mathcal{G}%
}\left( W_{j}\right) .$ In particular, there exists a function $b^{\ast
}\left( O_{T}\right) $ such that $E_{P^{\ast }}\left[ b^{\ast }\left(
O_{T}\right) |W_{j},\text{pa}_{\mathcal{G}}\left( W_{j}\right) \right] $ is
a non-constant function of $W_{j}.$  Lemma \ref{lemma:unrestricted} implies
that we can choose the law $P^{\ast }$ so that $b\left( \mathbf{O;}P^{\ast
}\right) =b^{\ast }\left( O_{T}\right) $ thus showing that for such law $%
P^{\ast },$ $E_{P^{\ast }}\left[ b\left( \mathbf{O};P^{\ast }\right) |W_{j},%
\text{pa}_{\mathcal{G}}\left( W_{j}\right) \right] =E_{P^{\ast }}\left[
b^{\ast }\left( O_{T}\right) |W_{j},\text{pa}_{\mathcal{G}}\left(
W_{j}\right) \right] $ is a non-constant function of $W_{j}$, and
consequently $E_{P^{\ast }}\left[ b\left( \mathbf{O};P^{\ast }\right) |W_{j},%
\text{pa}_{\mathcal{G}}\left( W_{j}\right) \right] -E_{P^{\ast }}\left[
b\left( \mathbf{O};P^{\ast }\right) |\text{pa}_{\mathcal{G}}\left(
W_{j+1}\right) \right] ,$ depends on $W_{j}.$

Suppose next that $W_{j}\in $pa$_{\mathcal{G}}\left( W_{j+1}\right) .$ For
each $i=1,\dots ,J-1,$ define 
\[
\mathbf{O}_{f}^{i}\equiv \left( W_{i+1},\dots ,W_{J}\right) \cap \mathbf{O}
\]
and 
\[
\mathbf{O}_{p}^{i}\equiv \mathbf{O\backslash O}_{f}^{i}
\]
The vertex set $\mathbf{O}_{f}^{i}$ $\ $is not empty because $W_{J}=O_{T}.$
Write $\mathbf{O}_{f}^{i}=\left( O_{f,1}^{i},\dots ,O_{f,h}^{i}\right) $ and 
$\mathbf{O}_{f}^{i}=\left( O_{p,1}^{i},\dots ,O_{p,m}^{i}\right) $ in
topological order relative to $\mathcal{G}$. \ The validity of $\left( \ref%
{eq:ind2}\right) $ is equivalent to the existence of $O\in \mathbf{%
O\backslash I}_{j}$ and of $W\in \left[ \pa_{\mathcal{G}}\left( W_{j}\right)
\cup W_{j}\right] \bigtriangleup \pa_{\mathcal{G}}\left( W_{j+1}\right) $
such that 
\begin{equation}
O\not\ort_{\mathcal{G}}W|\mathbf{I}_{j}  \label{eq:no_sep}
\end{equation}%
We will next show that if $W$ is in $\pa_{\mathcal{G}}\left( W_{j+1}\right)
\backslash \left[ W_{j}\cup \pa_{\mathcal{G}}\left( W_{j}\right) \right] ,$
then $\left( \ref{eq:no_sep}\right) $ holds for $O=O_{f,1}^{1}.$ So we will
consider separately the following three cases
\[
\begin{tabular}{lll}
Case & $O$ in & $W$ in \\ \hline  
1 & $\left\{ O_{f,1}^{1}\right\} $ & $\pa_{\mathcal{G}}\left( W_{j+1}\right)
\backslash \left[ W_{j}\cup \pa_{\mathcal{G}}\left( W_{j}\right) \right] $
\\ 
2 & $\mathbf{O}_{f}^{j}$ & $\pa_{\mathcal{G}}\left( W_{j}\right) \backslash %
\pa_{\mathcal{G}}\left( W_{j+1}\right) $ \\ 
3 & $\mathbf{O}_{p}^{j}\backslash \mathbf{I}_{j}$ & $\pa_{\mathcal{G}}\left(
W_{j}\right) \backslash \pa_{\mathcal{G}}\left( W_{j+1}\right) $%
\end{tabular}%
\]%
Notice that $\mathbf{O}_{f}^{j}\subset \mathbf{O}\backslash \mathbf{I}_{j}$
and that $\pa_{\mathcal{G}}\left( W_{j}\right) \backslash \pa_{\mathcal{G}%
}\left( W_{j+1}\right) =\left[ W_{j}\cup \pa_{\mathcal{G}}\left(
W_{j}\right) \right] \backslash \pa_{\mathcal{G}}\left( W_{j+1}\right) $
because we have assumed that $W_{j}\in $pa$_{\mathcal{G}}\left(
W_{j+1}\right) .$

In the subsequent analysis we will use the fact that $W_{j}$ and $W_{j+1}$
belong to an$_{\mathcal{G}}\left( O_{f,1}^{j}\right) $.  To see why this is
true, first note that if $j=J-1,$ then $W_{j}=W_{J-1}\in $pa$_{\mathcal{G}%
}\left( W_{j+1}\right) =$pa$_{\mathcal{G}}\left( W_{J}\right) =$pa$_{%
\mathcal{G}}\left( O_{T}\right) $ by assumption$.$ On the other hand, $%
\mathbf{O}_{f}^{J-1}=O_{f,1}^{J-1}=O_{T}.$ Then, $W_{j}=W_{J-1}$ and $%
W_{j+1}=W_{J}$ belong to an$_{\mathcal{G}}\left( O_{T}\right) =$an$_{%
\mathcal{G}}\left( O_{f,1}^{J-1}\right) =$an$_{\mathcal{G}}\left(
O_{f,1}^{j}\right) .$ If $j<J-1,$ then $W_{j}$ and $W_{j+1}$ also belong to
an$_{\mathcal{G}}\left( O_{f,1}^{j}\right) $ because we have already shown
that $W_{k}\in $pa$_{\mathcal{G}}\left( W_{k+1}\right) $ for all $%
k=j+1,\dots ,J-1$ and by definition $O_{f,1}^{j}\in \left\{
W_{j+1},...,W_{J}\right\} .$

Consider the case (1). The vertex $W$  belongs to $\mathbf{O}\left(
W_{j},O_{f,1}^{j},\mathcal{G}\right) $ by virtue of being an element of the
parent set of the child $W_{j+1}$ of $W_{j}$ and the facts that (i) $W\in \de%
_{\mathcal{G}}^{c}\left( W_{j}\right) $ because it belongs to $\pa_{\mathcal{%
G}}\left( W_{j+1}\right) \setminus \{W_{j}\}$ and (ii) $W_{j}$ and $W_{j+1}$
belong to an$_{\mathcal{G}}\left( O_{f,1}^{j}\right) .$ Note that this
implies that $W\in $an$_{\mathcal{G}}\left( O_{f,1}^{j}\right) $ and
consequently that $\left( \ref{eq:no_sep}\right) $ holds with $O=O_{f,1}^{j}$
because the path $W\rightarrow W_{j+1}\rightarrow W_{j+2}\rightarrow
...\rightarrow O_{f,1}^{j}$ is open given $\mathbf{I}_{j}$ since $\mathbf{I}%
_{j}$ does not include any of the nodes in the set $\left\{
W_{j+1},W_{j+2},...,O_{f,1}^{j}\right\} .$  Now, Lemma \ref%
{lemma:unrestricted} implies that $E_{P}\left[ O_{f,1}^{j}|W_{j},\mathbf{O}%
\left( W_{j},O_{f,1}^{j};\mathcal{G}\right) \right] $ is unrestricted in
model $\mathcal{M}\left( \mathcal{G}\right) $ so we can choose $P^{\ast }$
such that $E_{P^{\ast }}\left[ O_{f,1}^{j}|W_{j},\mathbf{O}\left(
W_{j},O_{f,1}^{j};\mathcal{G}\right) \right] =W_{j}W.$ We can also choose
such $P^{\ast }$ so as to also satisfy that $b\left( \mathbf{O;}P^{\ast
}\right) =O_{f,1}^{j}$ where $\mathbf{O\equiv O}\left( A,Y;\mathcal{G}%
\right) .$ This can be done because, as established in Lemma \ref%
{lemma:unrestricted}, the conditional law of $Y$ given $\left( A,\mathbf{O}%
\right) $ is variation independent with the joint law of $\mathbf{W,}$ and
in particular, with the joint law of the subvector $\left( O_{f,1}^{j},W_{j},%
\mathbf{O}_{W_{j}}\right) $ of $\mathbf{W}$. Then, 
\begin{eqnarray*}
E_{P^{\ast }}\left[ b\left( \mathbf{O;}P^{\ast }\right) |\text{pa}_{\mathcal{%
G}}\left( W_{j+1}\right) \right]  &=&E_{P^{\ast }}\left[ O_{f,1}^{j}|\text{pa%
}_{\mathcal{G}}\left( W_{j+1}\right) \right]  \\
&=&E_{P^{\ast }}\left\{ \left. E_{P}\left[ O_{f,1}^{j}|W_{j},\mathbf{O}%
_{W_{j}},\text{pa}_{\mathcal{G}}\left( W_{j+1}\right) \right] \right\vert 
\text{pa}_{\mathcal{G}}\left( W_{j+1}\right) \right\}  \\
&=&E_{P^{\ast }}\left\{ \left. E_{P^{\ast }}\left[ O_{f,1}^{j}|W_{j},\mathbf{%
O}_{W_{j}}\right] \right\vert \text{pa}_{\mathcal{G}}\left( W_{j+1}\right)
\right\}  \\
&=&E_{P^{\ast }}\left\{ \left. W_{j}W\right\vert \text{pa}_{\mathcal{G}%
}\left( W_{j+1}\right) \right\}  \\
&=&W_{j}W
\end{eqnarray*}%
Consequently, $E_{P^{\ast }}\left[ b\left( \mathbf{O};P^{\ast }\right)
|W_{j},\text{pa}_{\mathcal{G}}\left( W_{j}\right) \right] -E_{P^{\ast }}%
\left[ b\left( \mathbf{O};P\right) |\text{pa}_{\mathcal{G}}\left(
W_{j+1}\right) \right] =E_{P^{\ast }}\left[ b\left( \mathbf{O};P^{\ast
}\right) |W_{j},\text{pa}_{\mathcal{G}}\left( W_{j}\right) \right] -W_{j}W$
which depends on $W_{j}$ because $W\in $pa$_{\mathcal{G}}\left(
W_{j+1}\right) \backslash \left[ W_{j}\cup \pa_{\mathcal{G}}\left(
W_{j}\right) \right] .$%

Consider now case (2). Let $W\in  \pa_{\mathcal{G%
}}\left( W_{j}\right) \backslash \pa_{\mathcal{G}}\left(
W_{j+1}\right) $ and $O_{f,l}^{j}\in \mathbf{O\backslash I}_{j}$ such that 
\[
O_{f,l}^{j}\not\ort_{\mathcal{G}}W|\mathbf{I}_{j}.
\]%
Let $\tau $ denote a path between $%
O_{f,l}^{j}$ and $W$ that is open given $\mathbf{I}_{j}.$ In $\tau $
the edge with one endpoint equal to $O_{f,l}^{j}$ must point into $%
O_{f,l}^{j}.$ Suppose this was not the case, then $\tau $ would intersect a
collider, say $C,$ that is a descendant of $O_{f,l}^{j}.$ However, by the
definitions of $\mathbf{I}_{j}$ and $O_{f,l}^{j}$ we know that $\mathbf{I}%
_{j}\cap $de$_{\mathcal{G}}\left( O_{f,l}^{j}\right) =\emptyset .$
Consequently $C$ cannot have a descendant in $\mathbf{I}_{j}.\,\ $So, the
path $\tau $ would be blocked at $C$ given $\mathbf{I}_{j}$ contradicting
that $\tau $ is open given $\mathbf{I}_{j}$. Because the path $\tau $ is
open, then it must intersect an element of the set $\mathbf{O}\left(
W_{j},O_{f,l}^{j},\mathcal{G}\right) ,$ and consequently, $\mathbf{O}\left(
W_{j},O_{f,l}^{j},\mathcal{G}\right) \not\ort_{\mathcal{G}}W|%
\mathbf{I}_{j}.$ Now, $W\in $pa$\left( W_{j}\right) \backslash 
\mathbf{I}_{j}$ because $\left[ \{W_{j}\}\cup \pa_{\mathcal{G}}\left(
W_{j}\right) \right] \backslash \pa_{\mathcal{G}}\left( W_{j+1}\right) =$pa$%
_{\mathcal{G}}\left( W_{j}\right) \backslash \mathbf{I}_{j}$ since we have
assumed that $W_{j}\in \pa_{\mathcal{G}}\left( W_{j+1}\right) .$ We then
conclude that 
\begin{equation}
\mathbf{O}\left( W_{j},O_{f,l}^{j},\mathcal{G}\right) \not\ort_{\mathcal{G}%
}\left[ \text{pa}_{\mathcal{G}}\left( W_{j}\right) \backslash \mathbf{I}_{j}%
\right] |\mathbf{I}_{j}.  \label{eq:main1}
\end{equation}%
So, there exists $P^{\ast }\in \mathcal{M}\left( \mathcal{G}\right) $ such
that  
\begin{equation}
\mathbf{O}\left( W_{j},O_{f,l}^{j},\mathcal{G}\right) \not\ort\left[ \text{%
pa}_{\mathcal{G}}\left( W_{j}\right) \backslash \mathbf{I}_{j}\right] |%
\mathbf{I}_{j}\text{ under }P^{\ast }.  \label{eq:P_estrella}
\end{equation}%
Now, $\left( \ref{eq:P_estrella}\right) $ implies that there exists $h^{\ast
}\left[ \mathbf{O}\left( W_{j},O_{f,l}^{j},\mathcal{G}\right) \right] $ such
that $E_{P^{\ast }}\left\{ \left. h^{\ast }\left[ \mathbf{O}\left(
W_{j},O_{f,l}^{j},\mathcal{G}\right) \right] \right\vert \left[ \text{pa}%
\left( W_{j}\right) \backslash \mathbf{I}_{j}\right] ,\mathbf{I}_{j}\right\} 
$ is a non-constant function of pa$\left( W_{j}\right) \backslash \mathbf{I}%
_{j}.$ Then, since $\left[ \text{pa}\left( W_{j}\right) \backslash \mathbf{I}%
_{j}\right] \cup \mathbf{I}_{j}=$pa$_{\mathcal{G}}\left( W_{j}\right) \cup
W_{j}$ we conclude that 
$$
E_{P^{\ast }}\left\{ \left. h^{\ast }\left[ \mathbf{%
O}\left( W_{j},O_{f,l}^{j},\mathcal{G}\right) \right] \right\vert \text{pa}_{%
\mathcal{G}}\left( W_{j}\right) ,W_{j}\right\} 
$$
is a non-constant function
of pa$\left( W_{j}\right) \backslash \mathbf{I}_{j}.$ Furthermore, by the
Local Markov property, 
$$
E_{P^{\ast }}\left\{ \left. h^{\ast }\left[ \mathbf{O%
}\left( W_{j},O_{f,l}^{j},\mathcal{G}\right) \right] \right\vert \text{pa}_{%
\mathcal{G}}\left( W_{j}\right) ,W_{j}\right\} 
$$
does not depend on $W_{j}.$
So, we conclude that $E_{P^{\ast }}\left\{ \left. h^{\ast }\left[ \mathbf{O}%
\left( W_{j},O_{f,l}^{j},\mathcal{G}\right) \right] \right\vert \text{pa}_{%
\mathcal{G}}\left( W_{j}\right) ,W_{j}\right\} =g\left[ \pa_{\mathcal{G}%
}\left( W_{j}\right) \right] $ where $g\left[ \pa_{\mathcal{G}}\left(
W_{j}\right) \right] $ is a non-constant function of  pa$\left( W_{j}\right)
\backslash \mathbf{I}_{j}.$

Now, by the variation independence of the conditional law of $Y$ given $%
\left( A,\mathbf{O}\right) $ with the joint law of $\mathbf{W}$, which holds
as established in Lemma \ref{lemma:unrestricted}$\mathbf{,}$ we can take $%
P^{\ast }$ to also satisfy  $b\left( \mathbf{O};P^{\ast }\right)
=O_{f,l}^{j}.$ Furthermore, we can take $P^{\ast }$ to additionally satisfy that $%
E_{P^{\ast }}\left[ \left. O_{f,l}^{j}\right\vert W_{j},\mathbf{O}\left(
W_{j},O_{f,l}^{j},\mathcal{G}\right) \right] =W_{j}h^{\ast }\left[ \mathbf{O}%
\left( W_{j},O_{f,l}^{j},\mathcal{G}\right) \right] $ because, again by
Lemma \ref{lemma:unrestricted},  the conditional law of $O_{f,l}^{j}$ given $%
W_{j},\mathbf{O}\left( W_{j},O_{f,l}^{j},\mathcal{G}\right) $ is variation
independent with the law of de$_{\mathcal{G}}^{c}\left( W_{j}\right) \cup
W_{j},$ and in particular, with the joint law of  law of $\mathbf{O}\left(
W_{j},O_{f,l}^{j},\mathcal{G}\right) $ and $\left[  W_{j} \cup \text{pa}_{\mathcal{G}}\left( W_{j}\right) \right]
.$ For such $P^{\ast }$ we then have 
\begin{eqnarray*}
E_{P^{\ast }}\left[ b\left( \mathbf{O};P^{\ast }\right) |W_{j},\pa_{\mathcal{%
G}}\left( W_{j}\right) \right]  &=&E_{P^{\ast }}\left[ O_{f,l}^{j}|W_{j},\pa%
_{\mathcal{G}}\left( W_{j}\right) \right]  \\
&=&E_{P^{\ast }}\left[ \left. E_{P^{\ast}}\left[ O_{f,l}^{j}|W_{j},\mathbf{O}\left(
W_{j},O_{f,l}^{j},\mathcal{G}\right) \right] \right\vert W_{j},\pa_{\mathcal{%
G}}\left( W_{j}\right) \right]  \\
&=&W_{j}E_{P^{\ast }}\left[ \left. h^{\ast}\left[ \mathbf{O}\left(
W_{j},O_{f,l}^{j},\mathcal{G}\right)\right] \right\vert W_{j},\pa_{%
\mathcal{G}}\left( W_{j}\right) \right]  \\
&=&W_{j}g\left[ \pa_{\mathcal{G}}\left( W_{j}\right) \right] .
\end{eqnarray*}%
Then, $E_{P^{\ast }}\left[ b\left( \mathbf{O};P^{\ast }\right) |W_{j},\pa_{%
\mathcal{G}}\left( W_{j}\right) \right] -E_{P^{\ast }}\left[ b\left( \mathbf{%
O};P^{\ast }\right) |\pa_{\mathcal{G}}\left( W_{j+1}\right) \right] =W_{j}g%
\left[ \pa_{\mathcal{G}}\left( W_{j}\right) \right] -E_{P^{\ast }}\left[
b\left( \mathbf{O};P^{\ast }\right) |\pa_{\mathcal{G}}\left( W_{j+1}\right) %
\right] $ is a non-constant function of $W_{j}$ because $g\left[ \pa_{%
\mathcal{G}}\left( W_{j}\right) \right] $ is a non-constant function of pa$%
\left( W_{j}\right) \backslash \mathbf{I}_{j}$ and 
$$
\left[ \pa_{\mathcal{G}}\left(
W_{j}\right) \right]\backslash \mathbf{I}_{j}\cap \pa_{\mathcal{G}}\left(
W_{j+1}\right) =\emptyset .
$$

Finally, consider case (3). Let $W\in  \pa_{%
\mathcal{G}}\left( W_{j}\right) \backslash \pa_{\mathcal{G}}\left(
W_{j+1}\right) $ and $O_{p,l}^{j}\in \mathbf{O}_{p}^{j}\backslash \mathbf{I}%
_{j}$ such that 
\[
O_{p,l}^{j}\not\ort_{\mathcal{G}}W|\mathbf{I}_{j}.
\]
We then
have that $O_{p,l}^{j}\not\ort_{\mathcal{G}}\left[ \text{pa}_{\mathcal{%
G}}\left( W_{j}\right) \backslash \mathbf{I}_{j}\right] |\mathbf{I}_{j}$,
which then implies that  there exists $P^{\ast }\in \mathcal{M}\left( 
\mathcal{G}\right) $ such that 
\begin{equation}
O_{p,l}^{j}\not\ort\left[ \text{pa}_{\mathcal{G}}\left( W_{j}\right)
\backslash \mathbf{I}_{j}\right] |\mathbf{I}_{j}\text{ under }P^{\ast }.
\end{equation}%
The last display implies that there exists $h^{\ast }\left(
O_{p,l}^{j}\right) $ such that $E_{P^{\ast }}\left\{ \left. h^{\ast }\left(
O_{p,l}^{j}\right) \right\vert \left[ \text{pa}\left( W_{j}\right)
\backslash \mathbf{I}_{j}\right] ,\mathbf{I}_{j}\right\} $ is a non-constant
function of pa$\left( W_{j}\right) \backslash \mathbf{I}_{j}.$ Then, since $%
\left[ \text{pa}\left( W_{j}\right) \backslash \mathbf{I}_{j}\right] \cup 
\mathbf{I}_{j}=$pa$_{\mathcal{G}}\left( W_{j}\right) \cup W_{j}$ we conclude
that 
$$
E_{P^{\ast }}\left\{ \left. h^{\ast }\left( O_{p,l}^{j}\right)
\right\vert \text{pa}_{\mathcal{G}}\left( W_{j}\right) ,W_{j}\right\} 
$$ is a
non-constant function of pa$\left( W_{j}\right) \backslash \mathbf{I}_{j}.$
Furthermore, by the Local Markov property, 
$$
E_{P^{\ast }}\left\{ \left.
h^{\ast }\left( O_{p,l}^{j}\right) \right\vert \text{pa}_{\mathcal{G}}\left(
W_{j}\right) ,W_{j}\right\} 
$$
does not depend on $W_{j}.$ So, we conclude
that $E_{P^{\ast }}\left\{ \left. h^{\ast }\left( O_{p,l}^{j}\right)
\right\vert \text{pa}_{\mathcal{G}}\left( W_{j}\right) ,W_{j}\right\} =g%
\left[ \pa_{\mathcal{G}}\left( W_{j}\right) \right] $ where $g\left[ \pa_{%
\mathcal{G}}\left( W_{j}\right) \right] $ is a non-constant function of pa$%
\left( W_{j}\right) \backslash \mathbf{I}_{j}.$  By Lemma \ref%
{lemma:unrestricted} \ we can take $P^{\ast }$ to also satisfy that $b\left( 
\mathbf{O};P^{\ast }\right) =h^{\ast }\left( O_{p,l}^{j}\right) O_{f,1}^{j}$
and $E_{P^{\ast }}\left[ \left. O_{f,1}^{j}\right\vert W_{j},\mathbf{O}%
\left( W_{j},O_{f,1}^{j},\mathcal{G}\right) \right] =W_{j}.$ Then%
\begin{eqnarray*}
E_{P^{\ast }}\left[ b\left( \mathbf{O};P^{\ast }\right) |W_{j},\pa_{\mathcal{%
G}}\left( W_{j}\right) \right]  &=&E_{P^{\ast }}\left\{ \left. h^{\ast
}\left( O_{p,l}^{j}\right) O_{f,1}^{j}\right\vert W_{j},\pa_{\mathcal{G}%
}\left( W_{j}\right) \right\}  \\
&=&E_{P^{\ast }}\left\{ \left. h^{\ast }\left( O_{p,l}^{j}\right) E_{P}\left[
\left. O_{f,1}^{j}\right\vert W_{j},\mathbf{O}\left( W_{j},O_{f,1}^{j},%
\mathcal{G}\right) \right] \right\vert W_{j},\pa_{\mathcal{G}}\left(
W_{j}\right) \right\}  \\
&=&W_{j}E_{P^{\ast }}\left\{ \left. h^{\ast }\left( O_{p,l}^{j}\right)
\right\vert W_{j},\pa_{\mathcal{G}}\left( W_{j}\right) \right\}  \\
&=&W_{j}g\left[ \pa_{\mathcal{G}}\left( W_{j}\right) \right]. 
\end{eqnarray*}%
Consequently, 
$$
E_{P^{\ast }}\left[ b\left( \mathbf{O};P^{\ast }\right)
|W_{j},\text{pa}_{\mathcal{G}}\left( W_{j}\right) \right] -E_{P^{\ast}}\left[
b\left( \mathbf{O};P^{\ast }\right) |\text{pa}_{\mathcal{G}}\left(
W_{j+1}\right) \right] =W_{j}g\left[ \pa_{\mathcal{G}}\left( W_{j}\right) %
\right] -E_{P}\left[ b\left( \mathbf{O};P^{\ast }\right) |\text{pa}_{%
\mathcal{G}}\left( W_{j+1}\right) \right] 
$$
depends on $W_{j}.$ This
concludes the proof of the lemma.
\end{proof}

\begin{lemma}
\label{lemma:main_Z} 
Let $\mathcal{G}$ be a DAG with vertex set $\mathbf{V}$ and
let $A$ and $Y$ be two distinct vertices in $\mathbf{V.}$ Suppose that $\irrel
\left( A,Y;\mathcal{G}\right) =\emptyset .$ Suppose $\mathbf{M}\mathbf{%
\equiv }\de_{\mathcal{G}}\left( A\right) \backslash \left\{ A,Y\right\}
\not=\emptyset $ and let $\left( M_{1},\dots,M_{K}\right) $ be the elements of
$\mathbf{M}$ sorted topologically. Let $M_{0}\equiv A$ and $M_{K+1}\mathbf{%
\equiv }Y.$ Let $\mathbf{O}\equiv \mathbf{O}(A,Y,\mathcal{G})$. Let $\mathbf{O}_{min}$ be the
smallest among the subsets $\mathbf{O}_{sub}$ of $\mathbf{O}$ such that $A\perp \!\!\!\perp _{\mathcal{G}%
}\left( \mathbf{O\backslash O}_{sub}\right) |\mathbf{O}_{sub}$.
\begin{enumerate}

\item $E_{P}\left[ T_{P,a,\mathcal{G}}|Y,\pa_{\mathcal{G}}\left( Y\right) %
\right] =T_{P,a,\mathcal{G}} $ for all $P\in \mathcal{M}%
\left( \mathcal{G}\right) $ if and only if $\lbrace A \rbrace \cup \mathbf{O}%
_{\min }\subseteq \pa_{\mathcal{G}}\left( Y\right) $.

\item Suppose $\left\{ A\right\} \cup \mathbf{O}_{\min }\subseteq \pa_{%
\mathcal{G}}\left( Y\right) $. If $\pa_{\mathcal{G}}\left( Y\right)
\backslash \left\{ M_{K}\right\} \not\subset \pa_{\mathcal{G}}\left(
M_{K}\right) $ then there exists $P\in \mathcal{M}\left( \mathcal{G}\right) $
such that 
\begin{equation*}
E_{P}\left[ T_{P,a,\mathcal{G}}|M_{K},\pa_{\mathcal{G}}\left( M_{K}\right) \right]
-E_{P}\left[ T_{P,a,\mathcal{G}}\mid \pa_{\mathcal{G}}\left( Y\right) \right]
\end{equation*}%
is a non-constant function of $M_{K}.$

\item Suppose $\left\{ A\right\} \cup \mathbf{O}_{\min }\subseteq \pa_{%
\mathcal{G}}\left( Y\right) ,\pa_{\mathcal{G}}\left( Y\right) \backslash
\left\{ M_{K}\right\} \subset \pa_{\mathcal{G}}\left( M_{K}\right) $ and
there exists $j\geq 1$ such that for all $k=K-1,\dots,j+1,\pa_{\mathcal{G}%
}\left( M_{k+1}\right) \backslash \left\{ M_{k}\right\} \subset \pa_{%
\mathcal{G}}\left( M_{k}\right) $ but $\pa_{\mathcal{G}}\left(
M_{j+1}\right) \backslash \left\{ M_{j}\right\} \not\subset \pa_{\mathcal{G}%
}\left( M_{j}\right) .$ Then, there exists $P\in \mathcal{M}\left( \mathcal{G%
}\right) $ such that $E_{P}\left[ T_{P,a,\mathcal{G}}|M_{j},\pa_{\mathcal{G}%
}\left( M_{j}\right) \right] -E_{P}\left[ T_{P,a,\mathcal{G}}|\pa_{\mathcal{G}%
}\left( M_{j+1}\right) \right] $ is a non-constant of function of $M_{j}$.
\end{enumerate}
\end{lemma}

\begin{proof}[Proof of Lemma \protect\ref{lemma:main_Z}]

1) If $\lbrace A \rbrace \cup \mathbf{O}%
_{\min }\subseteq \pa_{\mathcal{G}}\left( Y\right)$, then $E_{P}\left[ T_{P,a,\mathcal{G}}|Y,\pa_{\mathcal{G}}\left( Y\right) %
\right] =T_{P,a,\mathcal{G}} $ for all $P\in \mathcal{M}%
\left( \mathcal{G}\right) $  holds trivially by the definition of $T_{P,a,\mathcal{G}}$.

Now suppose that $A\not\in \pa_{\mathcal{G}}\left( Y\right) $ or $\mathbf{%
O}_{\min }\not\subset \pa_{\mathcal{G}}\left( Y\right) .$ If $A\not\in $
pa$_{\mathcal{G}}\left( Y\right) $ then $E_{P}\left[ T_{P,a,\mathcal{G}}|Y,\text{%
pa}_{\mathcal{G}}\left( Y\right) \right] $ is not a function of $A$ and
consequently, it cannot be equal to $I_{a}(A)Y/\pi_{a} \left( \mathbf{O};P\right) .$
Next, suppose $\mathbf{O}_{min }\not\subseteq \pa_{\mathcal{G}}\left(
Y\right) $ because for some $O_{j}\in \mathbf{O}_{min },O_{j}\not\in \pa%
_{\mathcal{G}}\left( Y\right)$. Now, because $\mathbf{O}_{min}$ is the
smallest among the subsets $\mathbf{O}_{sub}$ of $\mathbf{O}$ such that $A\perp \!\!\!\perp _{\mathcal{G}%
}\left( \mathbf{O\backslash O}_{sub}\right) |\mathbf{O}_{sub},$ then
there exists a law $P^{\ast}\in \mathcal{M}\left( \mathcal{G}\right) $ such that $%
I_{a}(A)Y/\pi_{a} \left( \mathbf{O}_{min};P^{\ast}\right) $ is a non-constant function of $O_{j}.$
For such $P^{\ast}$, $E_{P^{\ast}}\left[ T_{P^{\ast},a,\mathcal{G}}|Y,\text{pa}_{\mathcal{G}}\left(
Y\right) \right] $ cannot be equal to $I_{a}(A) Y/\pi_{a} \left( \mathbf{O}_{min};P^{\ast}\right) .$

\medskip

2) 
Suppose that $\left\{ A\right\} \cup \mathbf{O}_{\min }\subset $pa$_{%
\mathcal{G}}\left( Y\right) $ but pa$_{\mathcal{G}}\left( Y\right)
\backslash \left\{ M_{K}\right\} \not\subset $ pa$_{\mathcal{G}}\left(
M_{K}\right) .$ Let 
$$
M^{\ast }\in \pa_{\mathcal{G}}\left( Y\right)
\backslash \left\{ M_{K}\cup \text{ pa}_{\mathcal{G}}\left( M_{K}\right)
\right\} .
$$ 
Since $M_{K}$ is the last element in the topological order of $\mathbf{M}$ and the assumptions that $\irrel(A,Y, \mathcal{G})$, $M_{K}\in \pa_{\mathcal{G}}(Y)$. Then there exists $P^{\ast}\in \mathcal{M}\left( G\right) $ be such that $E_{P^{\ast}}\left[
\left. Y\right\vert \text{pa}_{\mathcal{G}}\left( Y\right) \right] =M^{\ast
}M_{K}$. For such $P^{\ast}$, 
$$
E_{P^{\ast}}\left[ \left. T_{P^{\ast},a,\mathcal{G}}\right\vert 
\text{pa}_{\mathcal{G}}\left( Y\right) \right] =\frac{A}{\pi \left( \mathbf{O%
}_{\min };P^{\ast}\right) }M^{\ast } M_{K}.
$$
Furthermore, 
\begin{eqnarray*}
&&E_{P^{\ast}}\left[ T_{P^{\ast},a,\mathcal{G}}|M_{K},\text{pa}_{\mathcal{G}}\left(
M_{K}\right) \right] 
\\
&=&E_{P^{\ast}}\left[ \left. \frac{I_{a}(A)}{\pi_{a} \left( 
\mathbf{O}_{min} ;P^{\ast}\right) }E_{P^{\ast}}\left[ Y|A,\mathbf{O}_{min},M_{K},\text{pa}_{\mathcal{G}}\left(
M_{K}\right) ,\text{pa}_{\mathcal{G}}\left( Y\right) \right] \right\vert
M_{K},\text{pa}_{\mathcal{G}}\left( M_{K}\right) \right] \\
&=&E_{P^{\ast}}\left[ \left. \frac{I_{a}(A)}{\pi_{a} \left( 
\mathbf{O}_{min} ;P^{\ast}\right) }E_{P^{\ast}}\left[ Y|\text{pa}_{\mathcal{G}}\left(
Y\right) \right] \right\vert M_{K},\text{pa}_{\mathcal{G}}\left(
M_{K}\right) \right] \\
&=&E_{P^{\ast}}\left[ \left. \frac{I_{a}(A)}{\pi_{a} \left( 
\mathbf{O}_{min} ;P^{\ast}\right) }M^{\ast }M_{K}\right\vert M_{K},\text{pa}_{%
\mathcal{G}}\left( M_{K}\right) \right] \\
&=&M_{K}E_{P^{\ast}}\left[ \left. \frac{I_{a}(A)}{\pi_{a} \left( 
\mathbf{O}_{min} ;P^{\ast}\right) }M^{\ast }\right\vert M_{K},\text{pa}_{%
\mathcal{G}}\left( M_{K}\right) \right] .
\end{eqnarray*}%
Then, 
\begin{eqnarray*}
&&E_{P^{\ast}}\left[ T_{P^{\ast},a,\mathcal{G}}|M_{K},\text{pa}_{\mathcal{G}}\left(
M_{K}\right) \right] -E_{P^{\ast}}\left[ \left. T_{P^{\ast},a,\mathcal{G}}\right\vert 
\text{pa}_{\mathcal{G}}\left( Y\right) \right] \\
&=&M_{K}\left\{ E_{P^{\ast}}\left[ \left. \frac{I_{a}(A)}{\pi_{a} \left( \mathbf{O}_{min} ;P^{\ast}\right) }M^{\ast }\right\vert M_{K},\text{pa}_{%
\mathcal{G}}\left( M_{K}\right) \right] -\frac{A}{\pi \left( \mathbf{O%
}_{\min };P^{\ast}\right) }M^{\ast }\right\}.
\end{eqnarray*}%
The right hand side is a non-constant function of $M_{K}$ because $M^{\ast
}\notin \left\{ M_{K}\right\} \cup $pa$_{\mathcal{G}}\left( M_{K}\right) .$

\medskip

3) Suppose that $\left\{ A\right\} \cup \mathbf{O}_{\min }\subset \pa_{%
\mathcal{G}}\left( Y\right) $ and $\pa_{\mathcal{G}}\left( Y\right)
\backslash \left\{ M_{K}\right\} \subset $ pa$_{\mathcal{G}}\left(
M_{K}\right) $ and that pa$_{\mathcal{G}}\left( M_{k+1}\right) \backslash
\left\{ M_{k}\right\} \subset $ pa$_{\mathcal{G}}\left( M_{k}\right) $ for
all $k=K-1,\dots,j+1,$ but pa$_{\mathcal{G}}\left( M_{j+1}\right) \backslash
\left\{ M_{j}\right\} \not\subset $ pa$_{\mathcal{G}}\left( M_{j}\right) .$

Now pa$_{\mathcal{G}}\left( M_{j+1}\right) \backslash \left\{ M_{j}\right\}
\not\subset $ pa$_{\mathcal{G}}\left( M_{j}\right) $ implies that there
exists $M^{\ast \ast }\in $pa$_{\mathcal{G}}\left( M_{j+1}\right) \backslash
\left\{ M_{j}\cup \text{pa}_{\mathcal{G}}\left( M_{j}\right) \right\} .$
On the other hand, by part (i) of Lemma \ref{lemma:indep_med} we know that $M_{k}\in \pa_{\mathcal{G}}(M_{k+1})$ for $k=j,\dots,K$. 
Now, consider a law $P^{\ast}$ such that
$$
E_{P^{\ast}}\left[Y \mid \text{pa}_{\mathcal{G}}\left( Y\right)  \right]= M_{K}
$$
and
$$
E_{P^{\ast}}\left[M_{k} \mid \text{pa}_{\mathcal{G}}\left( M_{k}\right)  \right]= M_{k-1}
$$
for all $k=j+2,\dots,K$ and such that
$$
E_{P^{\ast}}\left[M_{j+1} \mid \text{pa}_{\mathcal{G}}\left( M_{j+1}\right)  \right]= M^{\ast\ast}M_{j}.
$$
Since
\begin{align*}
\lbrace A \rbrace \cup \mathbf{O}_{min} \subset \pa_{\mathcal{G}}(Y) \subset \lbrace M_{K}\rbrace  \cup \pa_{\mathcal{G}}(M_{K})&\subset \lbrace M_{K}, M_{K-1}\rbrace  \cup \pa_{\mathcal{G}}(M_{K-1}) 
\\
&\subset \dots \subset  \lbrace M_{K}, M_{K-1},\dots, M_{j+1}\rbrace  \cup \pa_{\mathcal{G}}(M_{j+1})
\end{align*}
then
\begin{eqnarray*}
E_{P^{\ast}}\left[ T_{P^{\ast},a,\mathcal{G}}|\text{pa}_{\mathcal{G}}\left( M_{j+1}\right) %
\right] &=&  \frac{I_{a}(A)}{\pi_{a} \left( \mathbf{O}_{min} ;P^{\ast}\right) }  E_{P^{\ast}}\left[ \left. E_{P^{\ast}}\left[ Y|\text{pa}_{\mathcal{G}}\left(
Y\right) ,\text{pa}_{\mathcal{G}}\left( M_{j+1}\right)  \right] \right\vert \text{pa}_{\mathcal{G}}\left(
M_{j+1}\right) \right] \\
&=& \frac{I_{a}(A)}{\pi_{a} \left( \mathbf{O}_{min} ;P^{\ast}\right) }  E_{P^{\ast}}\left[ \left. E_{P^{\ast}}\left[ Y|\text{pa}_{\mathcal{G}}\left( Y\right) %
\right] \right\vert \text{pa}_{\mathcal{G}}\left( M_{j+1}\right) \right]  \\
&=& \frac{I_{a}(A)}{\pi_{a} \left( \mathbf{O}_{min} ;P^{\ast}\right) }  E_{P^{\ast}}\left[ M_{K} %
\mid \text{pa}_{\mathcal{G}}\left( M_{j+1}\right) \right]  \\
&=& \frac{I_{a}(A)}{\pi_{a} \left( \mathbf{O}_{min} ;P^{\ast}\right) } E_{P^{\ast}}\left[ \left. E_{P^{\ast}}\left[ M_{K}|\text{pa}_{\mathcal{G}}\left(
M_{K}\right) ,\text{pa}_{\mathcal{G}%
}\left( M_{j+1}\right) \right] \right\vert \text{pa}_{\mathcal{G}}\left(
M_{j+1}\right) \right] \\
&=& \frac{I_{a}(A)}{\pi_{a} \left( \mathbf{O}_{min} ;P^{\ast}\right) } E_{P^{\ast}}\left[ \left. E_{P^{\ast}}\left[ M_{K}|\text{pa}_{\mathcal{G}}\left(
M_{K}\right) \right] \right\vert \text{pa}_{\mathcal{G}}\left(
M_{j+1}\right) \right] \\
&=& \frac{I_{a}(A)}{\pi_{a} \left( \mathbf{O}_{min} ;P^{\ast}\right) } E_{P^{\ast}}\left[ \left. M_{K-1} \right\vert \text{pa}_{\mathcal{G}}\left(
M_{j+1}\right) \right] \\
&=&\dots \\
&=& \frac{I_{a}(A)}{\pi_{a} \left( \mathbf{O}_{min} ;P^{\ast}\right) } E_{P^{\ast}}\left[ \left. M_{j+1}\right\vert \text{pa}_{\mathcal{G}}\left(
M_{j+1}\right) \right] \\
&=&\frac{I_{a}(A)}{\pi_{a} \left( \mathbf{O}_{min} ;P^{\ast}\right) } M^{\ast \ast}M_{j}.
\end{eqnarray*}%

On the other hand, 
\begin{eqnarray*}
&&E_{P^{\ast}}\left[ T_{P^{\ast},a,\mathcal{G}}|M_{j}, \text{pa}_{\mathcal{G}}\left( M_{j}\right) %
\right] 
\\
&=& E_{P^{\ast}}\left[ \frac{I_{a}(A)}{\pi_{a} \left( \mathbf{O}_{min} ;P^{\ast}\right) } \left. E_{P^{\ast}}\left[ Y| A,\mathbf{O}_{min},\text{pa}_{\mathcal{G}}\left(
Y\right) ,M_{j}, \text{pa}_{\mathcal{G}}\left( M_{j}\right)  \right] \right\vert M_{j}, \text{pa}_{\mathcal{G}}\left(
M_{j}\right) \right] \\
&=&  E_{P^{\ast}}\left[\frac{I_{a}(A)}{\pi_{a} \left( \mathbf{O}_{min} ;P^{\ast}\right) }  \left. E_{P^{\ast}}\left[ Y|\text{pa}_{\mathcal{G}}\left( Y\right) %
\right] \right\vert M_{j}, \text{pa}_{\mathcal{G}}\left( M_{j}\right) \right]  \\
&=& E_{P^{\ast}}\left[ \frac{I_{a}(A)}{\pi_{a} \left( \mathbf{O}_{min} ;P^{\ast}\right) }  M_{K} %
\mid M_{j}, \text{pa}_{\mathcal{G}}\left( M_{j}\right) \right]  \\
&=&  E_{P^{\ast}}\left[\frac{I_{a}(A)}{\pi_{a} \left( \mathbf{O}_{min} ;P^{\ast}\right) } \left. E_{P^{\ast}}\left[ M_{K}|A,\mathbf{O}_{min},\text{pa}_{\mathcal{G}}\left(
M_{K}\right) ,M_{j},\text{pa}_{\mathcal{G}%
}\left( M_{j}\right) \right] \right\vert M_{j}, \text{pa}_{\mathcal{G}}\left(
M_{j}\right) \right] \\
&=&  E_{P^{\ast}}\left[ \frac{I_{a}(A)}{\pi_{a} \left( \mathbf{O}_{min} ;P^{\ast}\right) }\left. E_{P^{\ast}}\left[ M_{K}|\text{pa}_{\mathcal{G}}\left(
M_{K}\right) \right] \right\vert M_{j}, \text{pa}_{\mathcal{G}}\left(
M_{j}\right) \right] \\
&=&  E_{P^{\ast}}\left[\frac{I_{a}(A)}{\pi_{a} \left( \mathbf{O}_{min} ;P^{\ast}\right) } \left. M_{K-1} \right\vert M_{j}, \text{pa}_{\mathcal{G}}\left(
M_{j}\right) \right] \\
&=&\dots \\
&=& E_{P^{\ast}}\left[ \frac{I_{a}(A)}{\pi_{a} \left( \mathbf{O}_{min} ;P^{\ast}\right) } \left. M_{j+1}\right\vert M_{j}, \text{pa}_{\mathcal{G}}\left(
M_{j}\right) \right] \\
&=& E_{P^{\ast}}\left[  \frac{I_{a}(A)}{\pi_{a} \left( \mathbf{O}_{min} ;P^{\ast}\right) }\left. E_{P^{\ast}}\left[ M_{j+1}\mid A,\mathbf{O}_{min},M_{j}, \text{pa}_{\mathcal{G}}\left(
M_{j}\right), \text{pa}_{\mathcal{G}}\left(
M_{j+1}\right)  \right]\right\vert M_{j}, \text{pa}_{\mathcal{G}}\left(
M_{j}\right) \right] 
\\
&=&  E_{P^{\ast}}\left[ \frac{I_{a}(A)}{\pi_{a} \left( \mathbf{O}_{min} ;P^{\ast}\right) }\left. E_{P^{\ast}}\left[ M_{j+1}\mid  \text{pa}_{\mathcal{G}}\left(
M_{j+1}\right)  \right]\right\vert M_{j}, \text{pa}_{\mathcal{G}}\left(
M_{j}\right) \right] 
\\
&=& E_{P^{\ast}}\left[  \frac{I_{a}(A)}{\pi_{a} \left( \mathbf{O}_{min} ;P^{\ast}\right) }\left. M^{\ast\ast} M_{j}\right\vert M_{j},\text{pa}_{\mathcal{G}}\left(
M_{j}\right) \right] \\
&=&M_{j} E_{P^{\ast}}\left[ \frac{I_{a}(A)}{\pi_{a} \left( \mathbf{O}_{min} ;P^{\ast}\right) }  \left. M^{\ast\ast} \right\vert M_{j},\text{pa}_{\mathcal{G}}\left(
M_{j}\right) \right].
\end{eqnarray*}%
Consequently,
\begin{align*}
&E_{P^{\ast}}\left[ T_{P^{\ast},a,\mathcal{G}}|M_{j},\text{pa}_{\mathcal{G}}\left(
M_{j}\right) \right] -E_{P^{\ast}}\left[ T_{P^{\ast},a,\mathcal{G}}|\text{pa}_{\mathcal{G}%
}\left( M_{j+1}\right) \right]
\\
&=  \frac{I_{a}(A)}{\pi_{a} \left( \mathbf{O}_{min} ;P^{\ast}\right) } M_{j} \left(E_{P^{\ast}}\left[ \frac{I_{a}(A)}{\pi_{a} \left( \mathbf{O}_{min} ;P^{\ast}\right) }  \left. M^{\ast\ast} \right\vert M_{j},\text{pa}_{\mathcal{G}}\left(
M_{j}\right) \right] - M^{\ast\ast} \right),
\end{align*}%
which is a non-constant function of $M_{j}$.

\end{proof}

\bibliographystyle{apalike}
\bibliography{efficient}

\end{document}